\documentclass[a4paper,11pt]{smfbook}
\usepackage{amssymb,amsmath,mathrsfs,graphics,graphicx,mathptm}
\usepackage[T1]{fontenc}
\usepackage[english, francais]{babel}
\usepackage{float,lscape}
\usepackage[all]{xy}
\usepackage{etoolbox}

\theoremstyle{plain}
\newtheorem{thm}{Theorem}[section]
\newtheorem{pro}[thm]{Proposition}
\newtheorem{lem}[thm]{Lemma}
\newtheorem{cor}[thm]{Corollary}
\newtheorem{thmdef}[thm]{Theorem-Definition}

\theoremstyle{definition}
\newtheorem*{defi}{Definition}
\newtheorem*{defis}{Definitions}
\newtheorem{eg}[thm]{Example}
\newtheorem{egs}[thm]{Examples}
\newtheorem{rem}[thm]{Remark}
\newtheorem{rems}[thm]{Remarks}

\def\noteB#1#2{{\begin{small}#1\end{small}} \hfill {\begin{small}\pageref{#2}\end{small}} \\}

\def\transp #1{\vphantom{#1}^{\mathrm t}\! {#1}}

\def\og{\leavevmode\raise.3ex\hbox{$\scriptscriptstyle\langle\!\langle$~}}
\def\fg{\leavevmode\raise.3ex\hbox{~$\!\scriptscriptstyle\,\rangle\!\rangle$}}

\addtocounter{section}{0}             
\numberwithin{equation}{section}       

\begin{document}

\selectlanguage{english}

\frontmatter

\begin{Large}
\title{\textsc{Some properties of the Cremona group}}
\end{Large}

\author{Julie \textsc{D\'eserti}}
\address{Universit\"{a}t Basel, Mathematisches Institut, Rheinsprung $21$, CH-$4051$ Basel, Switzerland}
\medskip
\address{On leave from Institut de Math\'ematiques de Jussieu, Universit\'e Paris~$7$, Projet G\'eom\'etrie 
et Dynamique, Site Chevaleret, Case $7012$, $75205$ Paris Cedex~$13$, France}
\email{deserti@math.jussieu.fr}

\address{Author supported by the Swiss National Science Foundation grant no PP00P2\_128422 /1}
\maketitle

\begin{abstract}
We recall some properties, unfortunately not all, of the Cremona group. 

We first begin by presenting a nice proof of the amalgamated product structure of the well-known subgroup of 
the Cremona group made up of the polynomial automorphisms of $\mathbb{C}^2$. Then we deal with the classification 
of birational maps and some applications (Tits alternative, non-simplicity...) Since any 
birational map can be written as a composition of quadratic birational maps up to an automorphism of the complex 
projective plane, we spend time on these special maps. Some questions of group theory are evoked: the classification
of the finite subgroups of the Cremona group and related problems, the description of the 
automorphisms of the Cremona group and the representations of some lattices in the Cremona group. The description of the centralizers 
of discrete dynamical systems is an important problem in real and complex dynamic, we make a state of art of 
this problem in the Cremona group.

Let $Z$ be a compact complex surface which carries an automorphism $f$ of positive topolo\-gical entropy. Either the 
Kodaira dimension of $Z$ is zero and $f$ is conjugate to an automorphism on the unique minimal model of $Z$ which is 
either a torus, or a K$3$ surface, or an Enriques surface, or $Z$ is a non-minimal rational surface and $f$ is conjugate 
to a birational map of the complex projective plane. We deal with results obtained in this last case: construction of 
such automorphisms, dynamical properties (rotation domains...) are touched on. 
\end{abstract}

\bigskip


\bigskip


\newpage\hspace{1cm}\newpage

\vspace*{3cm}
\begin{flushright}
$\begin{array}{l}
\text{\sl \hspace{0.3cm}Dear Pat,}\\

\text{\sl \hspace{0.3cm}You came upon me carving some kind of little figure out of wood}\\

\text{\sl and you said: "Why don't you make something for me ?"}\\

\text{\sl \hspace{0.3cm}I asked you what you wanted, and you said, "A box."}\\

\text{\sl \hspace{0.3cm}"What for ?"}\\

\text{\sl \hspace{0.3cm}"To put things in."}\\

\text{\sl \hspace{0.3cm}"What things ?"}\\

\text{\sl \hspace{0.3cm}"Whatever you have," you said.}\\

\text{\sl \hspace{0.3cm}Well, here's your box. Nearly everything I have is in it, and it is}\\

\text{\sl not full. Pain and excitement are in it, and feeling good or bad and}\\

\text{\sl evil thoughts and good thoughts -- the pleasure of design and some}\\

\text{\sl despair and the indescribable joy of creation.}\\

\text{\sl \hspace{0.3cm}And on top of these are all the gratitude and love I have for you.}\\

\text{\sl \hspace{0.3cm}And still the box is not full.}\\

\text{\sl John}

\\
\\
\\

\hspace{8.2cm}\text{J. Steinbeck}
\end{array}$
\end{flushright}

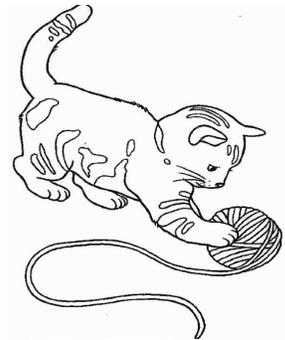
\vspace*{2cm}

\begin{figure}[H]
\begin{center}
\input{ptitchat.pstex_t}
\end{center}
\end{figure}

\newpage\hspace{1cm}\newpage

\chapter*{Introduction}\label{Chap:intro}

The study of the Cremona group $\mathrm{Bir}(\mathbb{P}^2)$, {\it i.e.} the group of birational maps 
from $\mathbb{P}^2(\mathbb{C})$ into itself, started in the XIXth century. The subject has known a lot 
of developments since the beginning of the XXIth century; we will deal with these most recent results. 
Unfortunately we will not be exhaustive.

\medskip

We introduce a special subgroup of the Cremona group: the group $\mathrm{Aut}(\mathbb{C}^2)$ of 
polynomial automorphisms of the plane. This subgroup has been the object of many studies along 
the XXth century. It is more rigid and so, in some sense, easier to understand. Indeed $\mathrm{Aut}
(\mathbb{C}^2)$ has a structure of amalgamated product so acts non trivially on a tree (Bass-Serre theory); 
this allows to give properties satisfied by polynomial automorphisms. There are a lot of 
different proofs of the structure of amalgamated product. We present one of them due to Lamy in 
Chapter~\ref{Chap:jung}; this one is particularly interesting for us because Lamy considers 
$\mathrm{Aut}(\mathbb{C}^2)$ as a subgroup of the Cremona group and works in $\mathrm{Bir}
(\mathbb{P}^2)$ (\emph{see} \cite{La2}).

\medskip

A lot of dynamical aspects of a birational map are controlled by its action on the cohomo\-logy 
$\mathrm{H}^2(X,\mathbb{R})$ of a "good" birational model $X$ of $\mathbb{P}^2(\mathbb{C})$. The 
construction of such a model is not canonical; so Manin has introduced the space of infinite dimension 
of all cohomological classes of all birational models of $\mathbb{P}^2(\mathbb{C})$. Its completion 
for the bilinear form induced by the cup product defines a real Hilbert space $\overline{\mathcal{Z}}
(\mathbb{P}^2)$ endowed with an intersection form. One of the two sheets of the hyperboloid $\{[D]
\in\overline{\mathcal{Z}}(\mathbb{P}^2)\,\vert\, [D]^2=1\}$ owns a metric which yields a hyperbolic 
space (Gromov sense); let us denote it by $\mathbb{H}_{\overline{\mathcal{Z}}}$. We get a faithful 
representation of $\mathrm{Bir}(\mathbb{P}^2)$ into $\mathrm{Isom}(\mathbb{H}_{\overline{\mathcal{Z}}})$. 
The classification of isometries into three types has an algrebraic-geometric meaning and induces a 
classification of birational maps (\cite{Can3}); it is strongly related to the classification of Diller 
and Favre (\cite{DiFa}) built on the degree growth of the sequence $\{\deg f^n\}_{n\in\mathbb{N}}$. 
Such a sequence either is bounded (elliptic maps), or grows linearly (de Jonqui\`eres twists), or grows
 quadratically (Halphen twists), or grows exponentially (hyperbolic maps). We give some applications
of this construction: $\mathrm{Bir}(\mathbb{P}^2)$ satisfies the Tits alternative (\cite{Can3}) and 
is not simple~(\cite{CL}).

\medskip

One of the oldest results about the Cremona group is that any birational map of the complex projective 
plane is a product of quadratic birational maps up to an automorphism of the complex projective plane.
 It is thus natural to study the quadratic birational maps and also the cubic ones in order to make in 
evidence some direct differences (\cite{CD}). In Chapter~\ref{Chap:quadcub} we present a stratification 
of the set of quadratic birational maps. We recall that this set is smooth. We also give a geometric 
description of the quadratic birational maps and a criterion of birationality for quadratic rational maps. 
We then deal with cubic birational maps; the set of such maps is not smooth anymore.

\medskip

While N\oe ther was interested in the decomposition of the birational maps, some people stu\-died 
finite subgroups of the Cremona group (\cite{Ber, Ka, Wiman}). A strongly related problem is the characterization of
the birational maps that preserve curves of positive genus. In Chapter \ref{Chap:folinv} we
give some statements and ideas of proof on this subject; people recently went back to this domain
\cite{BaBe, Beauville, BeauvilleBlanc, Blanc, dF, DolgachevIskovskikh, BPV2, Pan, DJS}, 
providing new results about the number of conjugacy classes in $\mathrm{Bir}(\mathbb{P}^2)$ of 
birational maps of order $n$ for example (\cite{dF, Bl}). We also present another construction of birational
involutions related to holomorphic foliations of degree $2$ on~$\mathbb{P}^2(\mathbb{C})$ 
(\emph{see} \cite{CD2}).

\medskip

A classical question in group theory is the following: let $\mathrm{G}$ be a group, what is the 
automorphisms group $\mathrm{Aut}(\mathrm{G})$ of $\mathrm{G}$ ? For example, the automorphisms of 
$\mathrm{PGL}_n(\mathbb{C})$ are, for $n\geq 3$, obtained from the inner automorphisms, the involution
$u\mapsto\transp u^{-1}$ and the 
automorphisms of the field $\mathbb{C}$. A similar result holds for the affine group of the complex 
line $\mathbb{C}$; we give a proof of it in Chapter \ref{Chap:aut}. We also give an idea of the 
description of the automorphisms group of $\mathrm{Aut}(\mathbb{C}^2)$, resp. $\mathrm{Bir}
(\mathbb{P}^2)$ (\emph{see}~\cite{De, De2}). 

\medskip

Margulis studies linear representations of the lattices of simple, real Lie groups of real rank strictly 
greater than $1$; Zimmer suggests to generalize it to non-linear ones. In that spirit we expose the 
representations of the classical lattices $\mathrm{SL}_n(\mathbb{Z})$ into the Cremona group (\cite{De4}). 
We see, in Chapter \ref{Chap:Zimmer}, that there is a description of embeddings of $\mathrm{SL}_3
(\mathbb{Z})$ into $\mathrm{Bir}(\mathbb{P}^2)$ (up to conjugation such an embedding is the canonical
 embedding or the involution $u\mapsto \transp u^{-1}$); therefore $\mathrm{SL}_n(\mathbb{Z})$ cannot 
be embedded as soon as $n\geq 4$.

\medskip

The description of the centralizers of discrete dynamical systems is an important problem in dynamic; it 
allows to measure algebraically the chaos of such a system. In Chapter \ref{Chap:centralizer} we describe the centralizer 
of birational maps. Methods are different for elliptic maps of infinite order, de Jonqui\`eres twists, Halphen 
twists and hyperbolic maps. In the first case, we can give explicit formulas~(\cite{BlDe}); in particular 
the centralizer is uncountable. In the second case, we do not always have explicit formulas (\cite{CD3})... 
When $f$ is an Halphen twist, the situation is different: the centralizer contains a 
subgroup of finite index which is abelian, free and of rank~$\leq~8$ (\emph{see}~\cite{Can3, [Gi]}). Finally 
for a hyperbolic map $f$ the centralizer is an extension of a cyclic group by a finite group (\cite{Can3}).

\bigskip

The study of automorphisms of compact complex surfaces with positive entropy is strongly related with birational 
maps of the complex projective plane. Let $f$ be an automorphism of a compact complex surface $\mathrm{S}$ with positive 
entropy; then either~$f$ is birationally conjugate to a birational map of the complex projective plane, or~the
 Kodaira dimension of $\mathrm{S}$ is zero and then $f$ is conjugate to an automorphism of the unique minimal model of 
$\mathrm{S}$ which has to be a torus, a K$3$ surface or an Enriques surface (\cite{Can1}). The case of K$3$ surfaces has 
been studied in \cite{Can2, Mc2, Og, Si, Wa}. One of the first example given in the context of rational surfaces 
is due to Coble (\cite{Co}). Let us mention another well-known example: let us consider $\Lambda=\mathbb{Z}
[\mathrm{i}]$ and $E=\mathbb{C}/\Lambda.$ The group $\mathrm{SL}_2(\Lambda)$ acts linearly on~$\mathbb{C}^2$ 
and preserves the lattice $\Lambda\times\Lambda;$ then any element $A$ of $\mathrm{SL}_2(\Lambda)$ induces an 
automorphism $f_A$ on~$E\times~E$ which commutes with $\iota(x,y)=(\mathrm{i}x,\mathrm{i}y).$ The automorphism $f_A$ 
lifts to an automorphism~$\widetilde{f_A}$ on the desingularization of the quotient $(E\times E)/\iota,$ which is a Kummer 
surface. This surface is rational and the entropy of~$\widetilde{f_A}$ is positive as soon as one of the eigenvalues 
of $A$ has modulus $>1$. 

\medskip

  We deal with surfaces obtained by blowing up the complex projective plane in a finite number of points. This is 
justified by Nagata theorem (\emph{see} \cite[Theorem 5]{Na}): let~$\mathrm{S}$ be a rational 
surface and let $f$ be an automorphism on $\mathrm{S}$ such that $f_*$ is of infinite order; then there exists a sequence of 
holomorphic applications~$\pi_{j+1}\colon  \mathrm{S}_{j+1}\to~\mathrm{S}_j$ such that $\mathrm{S}_1=\mathbb{P}^2(\mathbb{C}),$ $\mathrm{S}_{N+1}=~\mathrm{S}$ 
and $\pi_{j+1}$ is the blow-up of $p_j\in \mathrm{S}_j.$ Such surfaces are called {\it basic surfaces}\label{Chap13:ind2}. 
Nevertheless a surface obtained from $\mathbb{P}^2(\mathbb{C})$ by generic blow-ups has no non trivial automorphism~(\cite{Hi, Ko}). 

\medskip

Using Nagata and Harbourne works McMullen gives an analogous result of Torelli's Theorem for K$3$ surfaces 
(\cite{Mc}): he constructs automorphisms on rational surfaces prescribing the action of the automorphisms 
on the cohomological groups of the surface. These surfaces are rational ones having, up to a multiplicative 
factor, a unique $2$-form $\Omega$ such that $\Omega$ is meromorphic and $\Omega$ does not vanish. If $f$ 
is an automorphism on $\mathrm{S}$ obtained via this construction, $f^*\Omega$ is proportional to $\Omega$ and $f$ 
preserves the poles of~$\Omega$. We also have the following property: when we project $\mathrm{S}$ on the complex 
projective plane, $f$ induces a birational map which preserves a cubic (Chapter~\ref{Chap:mcmdg}).

\medskip

  In \cite{BK1, BK2, BK3} the authors consider birational maps of $\mathbb{P}^2(\mathbb{C})$ and adjust 
the coefficients in order to find, for any of these maps $f$, a finite sequence of blow-ups $\pi\colon  
Z\to\mathbb{P}^2 (\mathbb{C})$ such that the induced map~$f_ Z=\pi^{-1}f\pi$ is an automorphism of $Z.$ 
Some of their works are inspired by \cite{HV, HV2, Ta1, Ta2, Ta3}. More precisely Bedford and Kim produce 
examples which preserve no curve and also non trivial continuous families (Chapter \ref{chapbedkim1}). 
They prove dynamical properties such as coexistence of rotation domains of rank $1$ and~$2$ (Chapter \ref{chapbedkim1}). 

\medskip

In \cite{DeGr} the authors study a family of birational maps $(\Phi_n)_{n \geq 2};$ they construct, for any 
$n$, two points infinitely near $\widehat{P}_1$ and $\widehat{P}_2$  having the follo\-wing property: 
$\Phi_n$ induces an isomorphism between $\mathbb{P}^2(\mathbb{C})$ blown up in $\widehat{P}_1$ and~$\mathbb{P}^2(
 \mathbb{C})$ blown up in $\widehat{P}_2.$ Then they give general conditions on $\Phi_n$ allowing them to give 
automorphisms $\varphi$ of $\mathbb{P}^2(\mathbb{C})$ such that $\varphi \, \Phi_n$ is an automorphism 
of~$\mathbb{P}^2(\mathbb{C})$ blown up in $\widehat{P}_1,$ $\varphi(\widehat{P}_2),$ $(\varphi \, \Phi_n) 
\, \varphi(\widehat{P}_2),$ $\ldots,$ $(\varphi \, \Phi_n)^k \, \varphi(\widehat{P}_2)=\widehat{P}_1.$ This 
construction does not work only for~$\Phi_n$, they apply it to other maps (Chapter \ref{Chap:juju}). They 
use the theory of deformations of complex manifolds to describe explicitely the small deformations of rational 
surfaces; this allows them to give a simple criterion to determine the number of parameters of the deformation 
of a given basic surface (\cite{DeGr}). We end by a short scholium about the construction of
automorphisms with positive entropy on rational non-minimal surfaces obtained from birational maps of
the complex projective plane.

\subsection*{Acknowledgement}\,

Just few words in french. 
\selectlanguage{french}
Un grand merci au rapporteur pour ses judicieux conseils, remarques et suggestions.
Je tiens \`a remercier Dominique Cerveau pour sa g\'en\'erosit\'e, ses encouragements 
et son enthousiasme permanents. Merci \`a Julien Grivaux pour sa pr\'ecieuse aide, \`a Charles Favre pour 
sa constante pr\'esence et ses conseils depuis quelques ann\'ees d\'ej\`a, \`a Paulo Sad pour ses invitations 
au sud de l'\'equateur, les s\'eminaires bis etc. Je remercie Serge Cantat, en particulier pour nos discussions 
concernant le Chapitre~\ref{Chap:centralizer}. Merci \`a Jan-Li Lin pour ses commentaires et r\'ef\'erences au 
sujet de la Remarque~\ref{Rem:janli} et du Chapitre \ref{Chap:introaut}. J\'er\'emy Blanc m'a propos\'e de donner 
un cours sur le groupe de Cremona \`a B\^ale, c'est ce qui m'a d\'ecid\'ee \`a \'ecrire ces notes, je l'en remercie.
 Merci \`a Philippe Goutet pour ses incessantes solutions \`a mes probl\`emes LaTeX.

Enfin merci \`a l'Universit\'e de B\^ale, \`a l'Universit\'e Paris $7$ et \`a l'IMPA pour leur accueil.

\mainmatter
\selectlanguage{english}
\chapter{First steps}\label{Chap:firststep}

\section{Divisors and intersection theory}

  Let $X$ be an algebraic variety. A \textbf{\textit{prime divisor}}\label{Chap2:ind1} on $X$ is an irreducible 
closed subset of $X$ of codimension~$1$. 

\begin{egs}
\begin{itemize}
\item[$\bullet$] If $X$ is a surface, the prime divisors of $X$ are the irreducible curves that lie on it.

\item[$\bullet$] If $X=\mathbb{P}^n(\mathbb{C})$ then prime divisors are given by the zero locus of irreducible homogeneous polynomials.
\end{itemize}
\end{egs}

  A \textbf{\textit{Weil divisor}}\label{Chap2:ind2} on $X$ is a formal finite sum of prime divisors with integer coefficients
\begin{align*}
& \sum_{i=1}^{m}a_iD_i, && m\in\mathbb{N},\, a_i\in\mathbb{Z},\, D_i \text{ prime divisor of } X.
\end{align*}
Let us denote by $\mathrm{Div}(X)$ the set of all Weil divisors on $X$.

  If $f\in\mathbb{C}(X)^*$ is a rational function and $D$ a prime divisor we can define the 
\textbf{\textit{multipli\-city}}\label{Chap2:ind3}~$\nu_f(D)$ of $f$ at $D$ as follows: 
\begin{itemize}
\item[$\bullet$] $\nu_f(D)=k>0$ if $f$ vanishes on $D$ at the order $k$;
\item[$\bullet$] $\nu_f(D)=-k$ if $f$ has a pole of order $k$ on $D$;
\item[$\bullet$] and $\nu_f(D)=0$ otherwise.
\end{itemize}

  To any rational function $f\in\mathbb{C}(X)^*$ we associate a divisor $\mathrm{div}(f)\in~\mathrm{Div}(X)$ defined by 
$$\mathrm{div}(f)=\sum_{\substack{\text{$D$ prime} \\ \text{divisor}}}\nu_f(D)\, D.$$

  Note that $\mathrm{div}(f)\in\mathrm{Div}(X)$ since $\nu_f(D)$ is zero for all but finitely many $D$. Divisors obtained
 like that are called \textbf{\textit{principal divisors}}\label{Chap2:ind4}. As $\mathrm{div}(fg)=\mathrm{div}(f)+
\mathrm{div}(g)$ the set of principal divisors is a subgroup of~$\mathrm{Div}(X)$.

  Two divisors $D$, $D'$ on an algebraic variety are \textbf{\textit{linearly equivalent}}\label{Chap2:ind5} if~$D-D'$ is a 
principal divisor. The set of equivalence classes corresponds to the quotient of $\mathrm{Div}(X)$ by the subgroup of 
principal divisors; when~$X$ is smooth this quotient is isomorphic to the \textbf{\textit{Picard group}}\label{Chap2:ind6}~$\mathrm{Pic}(X)$.
\footnote{The \textbf{\textit{Picard group}} of $X$ is the group of isomorphism classes of line bundles on $X$.}

\begin{eg}\label{picproj}
Let us see that $\mathrm{Pic}(\mathbb{P}^n)=\mathbb{Z}H$ where $H$ is the divisor of an hyperplane.

  Consider the homorphism of groups given by
\begin{align*}
&\Theta\colon\mathrm{Div}(\mathbb{P}^n)\to\mathbb{Z}, && D\text{ of degree $d$}\mapsto d.
\end{align*}

  Let us first remark that its kernel is the subgroup of principal divisors. Let $D=\sum a_iD_i$ be a divisor in the kernel, 
where each $D_i$ is a prime divisor given by an homogeneous polynomial $f_i\in \mathbb{C}[x_0,\ldots,x_n]$ of some degree 
$d_i$. Since $\sum a_id_i=0$, $f=\prod f_i^{a_i}$ belongs to $\mathbb{C}(\mathbb{P}^n)^*$. We have by 
construction~$D=\mathrm{div}(f)$ so $D$ is a principal divisor. Conversely any principal divisor is equal to $\mathrm{div}(f)$ where 
$f=g/h$ for some homogeneous polynomials $g$, $h$ of the same degree. Thus any principal divisor belongs to the kernel.

  Since $\mathrm{Pic}(\mathbb{P}^n)$ is the quotient of $\mathrm{Div}(\mathbb{P}^n)$ by the subgroup of principal divisors, 
we get, by restricting $\Theta$ to the quotient, an isomorphism $\mathrm{Pic}(\mathbb{P}^n)\to\mathbb{Z}$. We conclude 
by noting that an hyperplane is sent on $1$.
\end{eg}

\bigskip

  We can define the notion of intersection.

\begin{pro}[\cite{Ha}]
Let $\mathrm{S}$ be a smooth projective surface. There exists a unique bilinear symmetric form 
\begin{align*}
&\mathrm{Div}(\mathrm{S})\times\mathrm{Div}(\mathrm{S})\to\mathbb{Z}, && (C,D)\mapsto C\cdot D
\end{align*}
having the following properties:
\begin{itemize}
\item[$\bullet$] if $C$ and $D$ are smooth curves meeting transversally then $C\cdot D=\# (C\cap D)$;

\item[$\bullet$] if $C$ and $C'$ are linearly equivalent then $C\cdot D=C'\cdot D$.
\end{itemize}
In particular this yields an intersection form
\begin{align*}
&\mathrm{Pic}(\mathrm{S})\times\mathrm{Pic}(\mathrm{S})\to\mathbb{Z}, && (C,D)\mapsto C\cdot D.
\end{align*}
\end{pro}

  Given a point $p$ in a smooth algebraic variety $X$ of dimension $n$ we say that $\pi\colon Y \to~X$ is a 
\textbf{\textit{blow-up}}\label{Chap2:ind17} of $p\in X$ if $Y$ is a smooth variety, if  $$\pi_{\vert Y\setminus\{\pi^{-1}(p)\}}
\colon Y\setminus\{\pi^{-1}(p)\}\to X\setminus\{p\}$$ is an isomorphism and if $\pi^{-1}(p)\simeq\mathbb{P}^{n-1}(\mathbb{C})$. Set 
$E=\pi^{-1}(p)$; $E$ is called the \textbf{\textit{exceptional divisor}}\label{Chap2:ind18}.

  If $\pi\colon Y\to X$ and $\pi'\colon Y'\to X$ are two blow-ups of the same point $p$ then there exists an isomorphism 
$\varphi\colon Y\to Y'$ such that $\pi=\pi'\varphi$. So we can speak about \textbf{\textit{the}} blow-up of $p\in X$.

\begin{rem}
When $n=1$, $\pi$ is an isomorphism but when $n\geq 2$ it is not: it contracts $E=\pi^{-1}(p)\simeq\mathbb{P}^{n-1}(\mathbb{C})$ 
onto the point $p$.
\end{rem}

\begin{eg}
We now describe the blow-up of $(0:0:1)$ in $\mathbb{P}^2(\mathbb{C})$. Let us work in the affine chart $z=1$, {\it i.e.} in 
$\mathbb{C}^2$ with coordinates $(x,y)$. Set $$\mathrm{Bl}_{(0,0)}\mathbb{P}^2=\Big\{\big((x,y),(u:v)\big)\in\mathbb{C}^2\times
\mathbb{P}^1\,\big\vert\, xv=yu\Big\}.$$ The morphism $\pi\colon\mathrm{Bl}_{(0,0)}\mathbb{P}^2\to\mathbb{C}^2$ given by the first 
projection is the blow-up of $(0,0)$:
\begin{itemize}
\item[$\bullet$] First we can note that $\pi^{-1}(0,0)=\Big\{\big((0,0),(u:v)\big)\,\big\vert\, (u:v)\in\mathbb{P}^1\Big\}$ so 
$E=~\pi^{-1}(0,0)$ is isomorphic to~$\mathbb{P}^1$;

\item[$\bullet$] Let $q=(x,y)$ be a point of $\mathbb{C}^2\setminus\{(0,0)\}$. We have $$\pi^{-1}(q)=\Big\{\big((x,y),(x:y)\big)\Big\}
\in \mathrm{Bl}_{(0,0)}\mathbb{P}^2\setminus~E$$ so $\pi_{\vert\mathrm{Bl}_{(0,0)}\mathbb{P}^2\setminus E}$ is an isomorphism, the inverse 
being $$(x,y)\mapsto\big((x,y),(x:y)\big).$$
\end{itemize}

  How to compute ? In affine charts: let $U$ $($resp. $V)$ be the open subset of $\mathrm{Bl}_{(0,0)}\mathbb{P}^2$ where $v\not=0$ $($resp. 
$u\not=0)$. The open subset $U$ is isomorphic to $\mathbb{C}^2$ via the map
\begin{align*}
&\mathbb{C}^2\to U , && (y,u)\mapsto\big((yu,y),(u:1)\big);
\end{align*}
we can see that $V$ is also isomorphic to $\mathbb{C}^2$. In local coordinates we can define the blow-up by
\begin{align*}
& \mathbb{C}^2\to \mathbb{C}^2, && (y,u)\mapsto (yu,y), && E\text{ is described by } \{y=0\}
\end{align*}
\begin{align*}
& \mathbb{C}^2\to \mathbb{C}^2, && (x,v)\mapsto (x,xv), && E\text{ is described by } \{x=0\}
\end{align*}
\end{eg}

  Let $\pi\colon\mathrm{Bl}_p\mathrm{S}\to \mathrm{S}$ be the blow-up of the point $p\in \mathrm{S}$. The morphism~$\pi$ 
induces a map $\pi^*$ from~$\mathrm{Pic}(\mathrm{S})$ to $\mathrm{Pic}(\mathrm{Bl}_p\mathrm{S})$ which sends a curve $C$ 
on~$\pi^{-1}(C)$. If $C\subset \mathrm{S}$ is irreducible, the \textbf{\textit{strict transform}}\label{Chap2:ind7} 
$\widetilde{C}$ of $C$ is $\widetilde{C}=\overline{\pi^{-1}(C\setminus\{p\})}$.

 We now recall what is the \textbf{\textit{multiplicity of a curve at a point}}\label{Chap2:ind8}. If $C\subset \mathrm{S}$ 
is a curve and $p$ is a point of $\mathrm{S}$, we can define the multiplicity $m_p(C)$ of~$C$ at $p$. Let $\mathfrak{m}$ be 
the maximal ideal of the ring of functions $\mathcal{O}_{p,\mathrm{S}}$\footnote{Let us recall that if $X$ is a quasi-projective 
variety and if $x$ is a point of $X$, then $\mathcal{O}_{p,X}$ is the set of equivalence classes of pairs $(U,f)$ where 
$U\subset X$ is an open subset $x\in U$ and $f\in\mathbb{C}[U]$.}. Let~$f$ be a local equation of $C$; then $m_p(C)$ can be 
defined as the integer~$k$ such that $f\in\mathfrak{m}^k\setminus \mathfrak{m}^{k+1}$. For example if $\mathrm{S}$ is rational, 
we can find a neighborhood~$U$ of $p$ in $\mathrm{S}$ with $U\subset\mathbb{C}^2$, we can assume that $p=(0,0)$ in this affine 
neighborhood, and $C$ is described by the equation 
\begin{align*}
&\sum_{i=1}^nP_i(x,y)=0, && P_i \text{ homogeneous polynomials of degree $i$ in two variables}.
\end{align*}
The multiplicity $m_p(C)$ is equal to the lowest $i$ such that $P_i$ is not equal to $0$. We have
\begin{itemize}
\item[$\bullet$] $m_p(C)\geq 0$;

\item[$\bullet$] $m_p(C)=0$ if and only if $p\not \in C$;

\item[$\bullet$] $m_p(C)=1$ if and only if $p$ is a smooth point of $C$.
\end{itemize}

  Assume that $C$ and $D$ are distinct curves with no common component then we define an integer $(C\cdot D)_p$ which counts 
the intersection of $C$ and~$D$ at $p$:
\begin{itemize}
\item[$\bullet$] it is equal to $0$ if either $C$ or $D$ does not pass through $p$;

\item[$\bullet$] otherwise let $f$, resp. $g$ be some local equations of $C$, resp. $D$ in a neighborhood of $p$ and define 
$(C\cdot D)_p$ to be the dimension of $\mathcal{O}_{p,\mathrm{S}}/(f,g)$.
\end{itemize}

  This number is related to $C\cdot D$ by the following statement.

\begin{pro}[\cite{Ha}, Chapter V, Proposition 1.4]
If $C$ and $D$ are distinct curves without any common irreducible component on a smooth surface, we have $$C\cdot D=\sum_{p\in C\cap D}
(C\cdot D)_p;$$ in particular $C\cdot D\geq 0$.
\end{pro}

  Let $C$ be a curve in $\mathrm{S}$, $p=(0,0)\in\mathrm{S}$. Let us take local coordinates $x$, $y$ at $p$ and let us set 
$k=m_p(C)$; the curve $C$ is thus given by
\begin{align*}
&P_k(x,y)+P_{k+1}(x,y)+\ldots+P_r(x,y)=0,
\end{align*}
where $P_i$ denotes a homogeneous polynomial of degree $i$.
The blow-up of~$p$ can be viewed as $(u,v)\mapsto(uv,v)$; the pull-back of $C$ is given by $$v^k\big(p_k(u,1)+vp_{k+1}(u,1)
+\ldots+v^{r-k}p_r(x,y)\big)=0,$$ {\it i.e.} it decomposes into $k$ times the exceptional divisor $E=\pi^{-1}(0,0)=(v=0)$ and 
the strict transform. So we have the following statement:

\begin{lem}\label{mult}
Let $\pi\colon\mathrm{Bl}_p\mathrm{S}\to\mathrm{S}$ be the blow-up of a point $p\in\mathrm{S}$. We have in $\mathrm{Pic}
(\mathrm{Bl}_p\mathrm{S})$ $$\pi^*(C)=\widetilde{C}+m_p(C)E$$ where $\widetilde{C}$ is the strict transform of $C$ and $E=\pi^{-1}(p)$.
\end{lem}

  We also have the following statement.

\begin{pro}[\cite{Ha}, Chapter V, Proposition 3.2]\label{hart}
Let $\mathrm{S}$ be a smooth surface, let $p$ be a point of $\mathrm{S}$ and let $\pi\colon\mathrm{Bl}_p\mathrm{S}\to \mathrm{S}$ 
be the blow-up of $p$. We denote by $E\subset\mathrm{Bl}_p\mathrm{S}$ the curve $\pi^{-1}(p)\simeq\mathbb{P}^1$. We have
$$\mathrm{Pic}(\mathrm{Bl}_p\mathrm{S})=\pi^*\mathrm{Pic}(\mathrm{S})+\mathbb{Z}E.$$
The intersection form on  $\mathrm{Bl}_p\mathrm{S}$ is induced by the intersection form on $\mathrm{S}$ via the following formulas
\begin{itemize}
\item[$\bullet$] $\pi^*C\cdot\pi^*D=C\cdot D$ for any $C,\,D\in\mathrm{Pic}(\mathrm{S})$;
\item[$\bullet$] $\pi^*C\cdot E=0$ for any $C \in\mathrm{Pic}(\mathrm{S})$;
\item[$\bullet$] $E^2=E\cdot E=-1$;
\item[$\bullet$] $\widetilde{C}^2=C^2-1$ for any smooth curve $C$ passing through $p$ and where $\widetilde{C}$ is the strict transform 
of $C$. 
\end{itemize}
\end{pro}

If $X$ is an algebraic variety, the \textbf{\textit{nef cone}}\label{Chap2:ind8c} $\mathrm{Nef}(X)$ is the cone of divisors~$D$ such 
that $D\cdot C\geq 0$ for any curve $C$ in $X$. 

\section{Birational maps}\label{Sec:firstdef}

  A \textbf{\textit{rational map}}\label{Chap2:ind9} from $\mathbb{P}^2(\mathbb{C})$ into itself is a map of the following type
\begin{align*}
& f\colon\mathbb{P}^2(\mathbb{C})\dashrightarrow\mathbb{P}^2(\mathbb{C}), && (x:y:z)\dashrightarrow ( f_0(x,y,z): f_1(x,y,z): f_2(x,y,z))
\end{align*}
where the $ f_i$'s are homogeneous polynomials of the same degree without common factor.

  A \textbf{\textit{birational map}}\label{Chap2:ind10} from $\mathbb{P}^2(\mathbb{C})$ into itself is a rational map $$f\colon
\mathbb{P}^2(\mathbb{C})\dashrightarrow\mathbb{P}^2(\mathbb{C})$$ such that there exists a rational map $\psi$ from $\mathbb{P}^2(\mathbb{C})$ 
into itself satisfying $ f\circ \psi=\psi\circ  f=\mathrm{id}.$ 

  The \textbf{\textit{Cremona group}}\label{Chap2:ind11} $\mathrm{Bir}(\mathbb{P}^2)$ is the group of birational maps from~$\mathbb{P}^2
(\mathbb{C})$ into itself. The elements of the Cremona group are also called \textbf{\textit{Cremona transformations}}\label{Chap2:ind12}. 
An element $ f$ of~$\mathrm{Bir}(\mathbb{P}^2)$ is equivalently given by~$(x,y)\mapsto(f_1(x,y),f_2(x,y))$ where $\mathbb{C}( f_1, f_2)=
\mathbb{C}(x_1,x_2)$, {\it i.e.} $$\mathrm{Bir}(\mathbb{P}^2)\simeq~\mathrm{Aut}_\mathbb{C}(\mathbb{C}(x,y)).$$

  The \textbf{\textit{degree}}\label{Chap2:ind13} of $ f\colon(x:y:x)\dashrightarrow( f_0(x,y,z): f_1(x,y,z): f_2(x,y,z))\in\mathrm{Bir}
(\mathbb{P}^2)$ is equal to the degree of the $ f_i$'s: $\deg  f=\deg  f_i.$

\begin{egs}
\begin{itemize}
\item[$\bullet$] Every automorphism 
\begin{align*}
&  f\colon(x:y:z)\dashrightarrow(a_0x+a_1y+a_2z:a_3x+a_4y+a_5z:a_6x+a_7y+a_8z),\\ & \det(a_i)\not=0
\end{align*}
of the complex projective plane is a birational map. The degree of $ f$ is equal to $1.$ In other words $\mathrm{Aut}(\mathbb{P}^2)=
\mathrm{PGL}_3(\mathbb{C})\subset\mathrm{Bir}(\mathbb{P}^2).$

\item[$\bullet$] The map $\sigma\colon(x:y:z)\dashrightarrow(yz:xz:xy)$ is rational; we can verify that $\sigma\circ\sigma=\mathrm{id},$ 
{\it i.e.} $\sigma$ is an involution so $\sigma$ is birational. We have: $\deg\sigma=2.$
\end{itemize}
\end{egs}

\begin{defis}
  Let $ f\colon(x:y:z)\dashrightarrow( f_0(x,y,z): f_1(x,y,z): f_2(x,y,z))$ be a birational map of~$\mathbb{P}^2(\mathbb{C});$ then:
\begin{itemize}
 \item[$\bullet$] the \textbf{\textit{indeterminacy locus}} of $f$, denoted by $\mathrm{Ind}\, f$, is the set
$$\Big\{m\in\mathbb{P}^2(\mathbb{C})\,\big\vert\,  f_0(m)= f_1(m)= f_2(m)=0\Big\}$$\label{Chap2:ind14}

\item[$\bullet$] and the \textbf{\textit{exceptional locus}} $\mathrm{Exc}\, f$ of $f$ is given by
$$\Big\{m\in\mathbb{P}^2(\mathbb{C})\,\big\vert\,\det\mathrm{jac}( f)_{(m)}=0\Big\}.$$\label{Chap2:ind15}
\end{itemize}
\end{defis}

\begin{egs}
\begin{itemize}
\item[$\bullet$] For any $ f$ in $\mathrm{PGL}_3(\mathbb{C})=\mathrm{Aut}(\mathbb{P}^2)$ we have
$\mathrm{Ind}\, f=\mathrm{Exc}\, f=\emptyset.$

\medskip

\item[$\bullet$] Let us denote by $\sigma$ the map defined by
$\sigma\colon(x:y:z)\dashrightarrow(yz:xz:xy);$ we note that
\begin{align*}
&\mathrm{Exc}\, \sigma=\big\{x=0,\, y=0, z=0\big\},\\
& \mathrm{Ind}\,\sigma=\big\{(1:0:0),\,(0:1:0),\,(0:0:1)\big\}.
\end{align*}

\medskip

\item[$\bullet$] If $\rho$ is the following map
$\rho\colon(x:y:z)\dashrightarrow(xy:z^2:yz),$ then
\begin{align*}
&\mathrm{Exc}\, \rho=\big\{y=0,\, z=0\big\}&&\&
&&\mathrm{Ind}\,\rho=\big\{(1:0:0),\,(0:1: 0)\big\}.
\end{align*}
\end{itemize}
\end{egs}

\begin{defi}
Let us recall that if $X$ is an irreducible variety and $Y$ a variety, a \textbf{\textit{rational map}}\label{Chap2:ind16} 
$f\colon X\dashrightarrow Y$ is a morphism from a non-empty open subset $U$ of $X$ to $Y$.
\end{defi}

  Let $f\colon\mathbb{P}^2(\mathbb{C})\dashrightarrow\mathbb{P}^2(\mathbb{C})$ be the birational map given by $$(x:y:z)\dashrightarrow 
(f_0(x,y,z):f_1(x:y:z):f_2(x,y,z))$$ where the $f_i$'s are homogeneous polynomials of the same degree $\nu$, and without common factor. 
The  \textbf{\textit{linear system}}\label{Chap2:ind19} $\Lambda_f$ of $f$ is the pre-image of the linear system of lines of 
$\mathbb{P}^2(\mathbb{C})$; it is the system of curves given by $\sum a_if_i=0$ for $(a_0:a_1:a_2)$ in $\mathbb{P}^2(\mathbb{C})$. 
Let us remark that if $A$ is an automorphism of $\mathbb{P}^2(\mathbb{C})$, then $\Lambda_f=\Lambda_{Af}$. The degree of the curves of 
$\Lambda_f$ is~$\nu$, {\it i.e.} it coincides with the degree of $f$. If $f$ has one point of 
indeterminacy~$p_1$, let us denote by $\pi_1\colon\mathrm{Bl}_{p_1}\mathbb{P}^2\to \mathbb{P}^2(\mathbb{C})$ the blow-up of $p_1$ and 
$\mathcal{E}_1$ the exceptional divisor. The map $\varphi_1=f\circ \pi_1$ is a birational map from $\mathrm{Bl}_{p_1}\mathbb{P}^2$ into 
$\mathbb{P}^2(\mathbb{C})$. If $\varphi_1$ is not defined at one point $p_2$ then we blow it up via $\pi_2\colon\mathrm{Bl}_{p_1,p_2}
\mathbb{P}^2\to\mathbb{P}^2(\mathbb{C})$; set $\mathcal{E}_2=\pi_2^{-1}(p_2)$. Again the map $\varphi_2=\varphi_1\circ\pi_1\colon
\mathrm{Bl}_{p_1,p_2}\mathbb{P}^2\dashrightarrow\mathbb{P}^2(\mathbb{C})$ is a birational map. We continue the same processus 
until~$\varphi_r$ becomes a morphism. The $p_i$'s are called \textbf{\textit{base-points of $f$}}\label{Chap2:ind20} or \textbf{\textit{base-points 
of $\Lambda_f$}}\label{Chap2:ind21}. Let us descri\-be~$\mathrm{Pic}(\mathrm{Bl}_{p_1,\ldots,p_r}\mathbb{P}^2)$. First $\mathrm{Pic}
(\mathbb{P}^2)=\mathbb{Z}L$ where $L$ is the divisor of a line (Example \ref{picproj}). Set~$E_i=(\pi_{i+1}\ldots\pi_r)^*\mathcal{E}_i$ 
and $\ell=(\pi_1\ldots\pi_r)^*(L)$. Applying $r$ times Proposition \ref{hart} we get $$\mathrm{Pic}(\mathrm{Bl}_{p_1,\ldots,p_r}
\mathbb{P}^2)=\mathbb{Z}\ell\oplus\mathbb{Z}E_1\oplus\ldots\oplus\mathbb{Z}E_r.$$ Moreover all elements of the basis $(\ell,E_1,\ldots,
E_r)$ satisfy the following relations
\begin{align*}
& \ell^2=\ell\cdot\ell=1,&& E_i^2=-1,\\ 
& E_i\cdot E_j=0 \,\,\,\,\,\,\forall\,\,1\leq i\not=j\leq r, && E_i\cdot\ell=0\,\,\,\,\,\,\forall 
1\leq i\leq r.
\end{align*}
The linear system $\Lambda_f$ consists of curves of degree $\nu$ all passing through the $p_i$'s with multipli\-city~$m_i$. Set 
$E_i=(\pi_{i+1}\ldots\pi_r)^*\mathcal{E}_i$. Applying $r$ times Lemma \ref{mult} the elements of $\Lambda_{\varphi_r}$ are equivalent to 
$\nu L-\sum_{i=1}^rm_iE_i$ where~$L$ is a generic line. Remark that these curves have self-intersection $$\nu^2-\sum_{i=1}^rm_i^2.$$ All 
members of a linear system are linearly equivalent and the dimension of~$\Lambda_{\varphi_r}$ is $2$ so the self-intersection has to be 
non-negative. This implies that the number $r$ exists, {\it i.e.} the number of base-points of $f$ is finite. Let us note that by 
construction the map $\varphi_r$ is a birational morphism from~$\mathrm{Bl}_{p_1,\ldots,p_r}\mathbb{P}^2$ to $\mathbb{P}^2(\mathbb{C})$ 
which is the blow-up of the points of $f^{-1}$; we have the following diagram
\begin{align*}
\xymatrix{& \mathrm{S}'\ar[dl]_{\pi_r\circ\ldots\circ\pi_1}\ar[dr]^{\varphi_r} &\\
\mathrm{S}\ar@{-->}[rr]_f & & \widetilde{\mathrm{S}} }
\end{align*}
The linear system $\Lambda_f$ of $f$ corresponds to the strict pull-back of the system $\mathcal{O}_{\mathbb{P}^2}(1)$ of lines 
of~$\mathbb{P}^2(\mathbb{C})$ by $\varphi$. The system $\Lambda_{\varphi_r}$ which is its image on $\mathrm{Bl}_{p_1,\ldots,p_r}\mathbb{P}^2$ 
is the strict pull-back of the system $\mathcal{O}_{\mathbb{P}^2}(1)$. Let us consider a general line $L$ of $\mathbb{P}^2(\mathbb{C})$ which 
does not pass through the $p_i$'s; its pull-back $\varphi_r^{-1}(L)$ corresponds to a smooth curve on $\mathrm{Bl}_{p_1,\ldots,p_r}
\mathbb{P}^2$ which has self-intersection~$-1$ and genus $0$. We thus have $(\varphi_r^{-1}(L))^2=1$ and by adjunction formula $$\varphi_r^{-1}
(L)\cdot~\mathrm{K}_{\mathrm{Bl}_{p_1,\ldots,p_r}\mathbb{P}^2}=-3.$$ Since the elements of $\Lambda_{\varphi_r}$ are equivalent to $$\nu L-\sum_{i=1}^r
m_iE_i$$ and since $\mathrm{K}_{\mathrm{Bl}_{p_1,\ldots,p_r}\mathbb{P}^2}=-3L+\sum_{i=1}^rE_i$ we have
\begin{align*}
& \sum_{i=1}^rm_i=3(\nu-1), && \sum_{i=1}^rm_i^2=\nu^2-1.
\end{align*}
In particular if $\nu=1$ the map $f$ has no base-points. If $\nu=2$ then $r=3$ and $m_1=m_2=m_3=1$. As we will see later (Chapter 
\ref{Chap:quadcub}) it doesn't mean that "there is one quadratic birational map".

\medskip

  So there are three standard ways to describe a Cremona map
\begin{itemize}
\item[$\bullet$] the explicit formula $(x:y:z)\dashrightarrow(f_0(x,y,z):f_1(x,y,z):f_2(x,yz))$ where the $f_i$'s are homogeneous polynomials 
of the same degree and without common factor;

\item[$\bullet$] the data of the degree of the map, the base-points of the map and their multiplicity (it defines a map up to an automorphism);

\item[$\bullet$] the base-points of $\pi$ and the curves contracted by $\eta$ with the notations of Theorem \ref{Zariski} (it defines a map up 
to an automorphism).
\end{itemize}

\section{Zariski's theorem}

  Let us recall the following statement.

\begin{thm}[Zariski, 1944]\label{Zariski}
Let $\mathrm{S}$, $\widetilde{\mathrm{S}}$ be two smooth projective surfaces and let $f\colon\mathrm{S}\dashrightarrow~\widetilde{\mathrm{S}}$ be a 
birational map. There exists a smooth projective surface $\mathrm{S}'$ and two sequences of blow-ups $\pi_1\colon \mathrm{S}'\to \mathrm{S}$, $\pi_2
\colon \mathrm{S}'\to \widetilde{\mathrm{S}}$ such that  
$f=\pi_2\pi^{-1}_1$
\begin{align*}
\xymatrix{& \mathrm{S}'\ar[dl]_{\pi_1}\ar[dr]^{\pi_2} &\\
\mathrm{S}\ar@{-->}[rr]_f & & \widetilde{\mathrm{S}} }
\end{align*}
\end{thm}

\begin{eg}
The involution 
\begin{align*}
&\sigma\colon\mathbb{P}^2(\mathbb{C})\dashrightarrow\mathbb{P}^2(\mathbb{C}), && (x:y:z)\dashrightarrow(yz:xz:xy)
\end{align*}
is the composition of two sequences of blow-ups
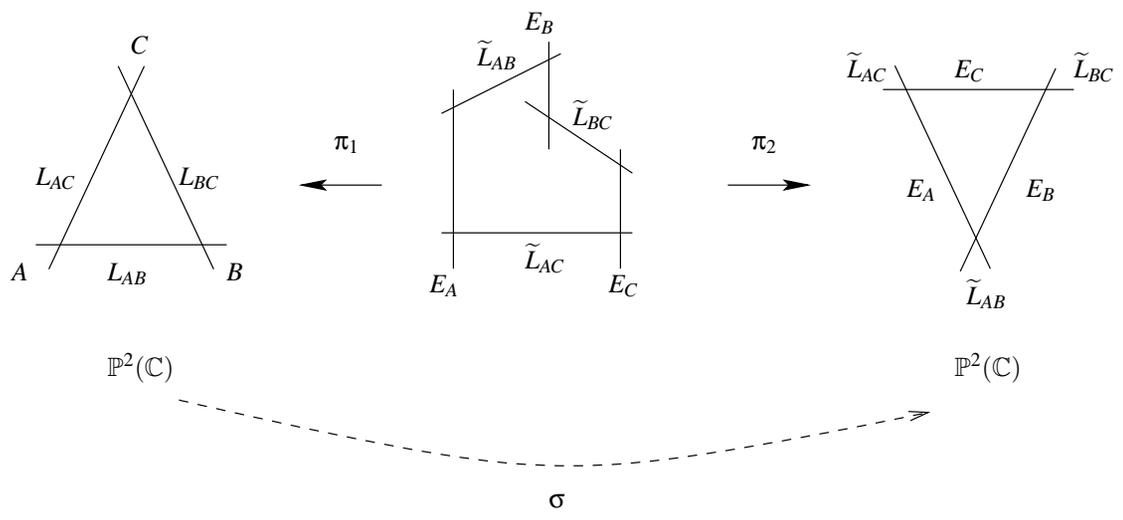
\begin{landscape}
\begin{figure}[H]
\begin{center}
\input{decomp.pstex_t}
\end{center}
\end{figure}
\end{landscape}

\noindent with
\begin{align*}
 & A=(1:0:0), && B=(0:1:0), && C=(0:0:1),
\end{align*}
$L_{AB}$ $($resp. $L_{AC},$ resp. $L_{BC})$ the line passing through $A$ and $B$ $($resp. $A$ and $C,$ resp. $B$ and $C)$ $E_A$ $($resp. $E_B,$ resp. $E_C)$ the 
exceptional divisor obtained by blowing up 
$A$ $($resp. $B,$ resp. $C)$ and~$\widetilde{L}_{AB}$ $($resp. $\widetilde{L}_{AC},$ 
resp. $\widetilde{L}_{BC})$ the strict transform of $L_{AB}$ $($resp. 
$L_{AC},$ resp. $L_{BC})$.
\end{eg}

  There are two steps in the proof of Theorem \ref{Zariski}. The first one is to compose $f$ with a sequence of blow-ups in order to remove all the points of 
indeterminacy (remark that this step is also possible with a rational map and can be adapted in higher dimension); we thus have
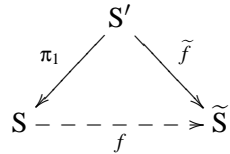
\begin{align*}
\xymatrix{& \mathrm{S}'\ar[dl]_{\pi_1}\ar[dr]^{\widetilde{f}} &\\
\mathrm{S}\ar@{-->}[rr]_f & & \widetilde{\mathrm{S}} }
\end{align*}
The second step is specific to the case of birational map between two surfaces and can be stated as follows.

\begin{pro}[\cite{La2}]\label{bobo2}
Let $f\colon \mathrm{S}\to \mathrm{S}'$ be a birational morphism between two surfaces~$\mathrm{S}$ and~$\mathrm{S}'.$ Assume that  $f^{-1}$ is not defined at 
a point $p$ of~$\mathrm{S}'$; then~$f$ can be written $\pi\phi$ where $\pi\colon\mathrm{Bl}_p\mathrm{S}'\to\mathrm{S}'$ is the blow-up of $p\in\mathrm{S}'$  
and~$\phi$ a birational morphism from~$\mathrm{S}$ to~$\mathrm{Bl}_p\mathrm{S}'$
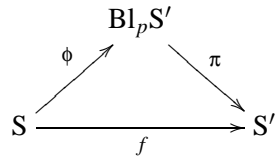
\begin{align*}
\xymatrix{&\mathrm{Bl}_p\mathrm{S}'\ar[dr]^{\pi} &\\
\mathrm{S}\ar[ru]^\phi\ar[rr]_f & & \mathrm{S}' }
\end{align*}
\end{pro}

  Before giving the proof of this result let us give a useful Lemma.

\begin{lem}[\cite{Be}]\label{bobo}
Let $f\colon \mathrm{S}\dashrightarrow \mathrm{S}'$ be a birational map between two surfaces~$\mathrm{S}$ and~$\mathrm{S}'$. If there exists a point $p
\in\mathrm{S}$ such that $f$ is not defined at $p$ there exists a curve~$\mathcal{C}$ on~$\mathrm{S}'$ such that~$f^{-1}(\mathcal{C})=p.$
\end{lem}

\begin{proof}[Proof of the Proposition \ref{bobo2}]
Assume that $\phi=\pi^{-1}f$ is not a morphism. Let $m$ be a point of~$\mathrm{S}$ such that~$\phi$ is not defined at $m$. On the one hand $f(m)=p$ and 
$f$ is not locally invertible at~$m$, on the other hand there exists a curve in $\mathrm{Bl}_p\mathrm{S}'$ contracted on $m$ by $\phi^{-1}$ (Lemma~\ref{bobo}). 
This curve is necessarily the exceptional divisor~$E$ obtained by blowing up.

  Let $q_1$, $q_2$ be two different points of $E$ at which $\phi^{-1}$ is well defined and let $\mathcal{C}_1$, $\mathcal{C}_2$ be two germs of smooth 
curves transverse to $E$. Then $\pi(\mathcal{C}_1)$ and~$\pi(\mathcal{C}_2)$ are two germs of smooth curve transverse at $p$ which are the image by~$f$ of 
two germs of curves at $m$. The differential of $f$ at~$m$ is thus of rank~$2$: contradiction with the fact that~$f$ is not locally invertible at $m$.

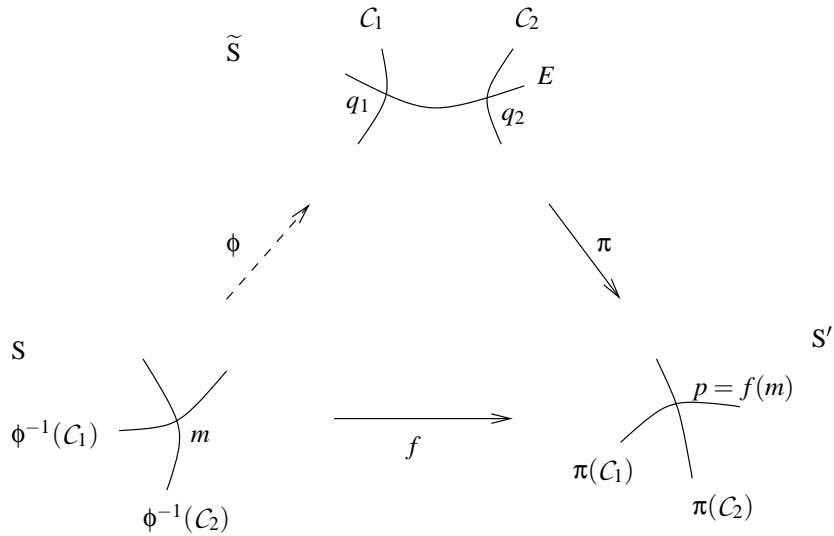
\begin{figure}[H]
\begin{center}
\input{zariski.pstex_t}
\end{center} 
\end{figure}

\end{proof}

  We say that $f\colon\mathrm{S}\dashrightarrow\mathbb{P}^2(\mathbb{C})$ is \textbf{\textit{induced by a polynomial automorphism\footnote{A 
polynomial automorphism of $\mathbb{C}^2$ is a bijective application of the following type
\begin{align*}
&  f\colon\mathbb{C}^2\to\mathbb{C}^2, && (x,y)\mapsto( f_1(x,y), f_2(x,y)), &&  f_i\in\mathbb{C}[x,y].
\end{align*}} of $\mathbb{C}^2$}} if
\begin{itemize}
\item[$\bullet$] $\mathrm{S}=\mathbb{C}^2\cup D$ where $D$ is a union of irreducible curves, $D$ is called \textbf{\textit{divisor at infinity}};

\item[$\bullet$] $\mathbb{P}^2(\mathbb{C})=\mathbb{C}^2\cup L$ where $L$ is a line, $L$ is called \textbf{\textit{line at infinity}};

\item[$\bullet$] $f$ induces an isomorphism between $\mathrm{S}\setminus D$ and $\mathbb{P}^2(\mathbb{C})\setminus L.$
\end{itemize}

  If $f\colon\mathrm{S}\dashrightarrow\mathbb{P}^2(\mathbb{C})$ is induced by a polynomial automorphism of $\mathbb{C}^2$ it satisfies some properties:

\begin{lem}\label{applzar}
Let $\mathrm{S}$ be a surface. Let $f$ be a birational map from $\mathrm{S}$ to~$\mathbb{P}^2(\mathbb{C})$ induced by a polynomial automorphism of 
$\mathbb{C}^2$. Assume that $f$ is not a morphism. Then 
\begin{itemize}
\item[$\bullet$] $f$ has a unique point of indeterminacy $p_1$ on the divisor at infinity;

\item[$\bullet$] $f$ has base-points $p_2$, $\ldots$, $p_s$ and for all $i=2,\ldots,s$ the point $p_i$ is on the exceptional divisor obtained by blowing up $p_{i-1}$;

\item[$\bullet$] each irreducible curve contained in the divisor at infinity is contracted on a point by $f$;
 
\item[$\bullet$] the first curve contracted by $\pi_2$ is the strict transform of a curve contained in the divisor at infinity;
 
\item[$\bullet$] in particular if $\mathrm{S}=\mathbb{P}^2(\mathbb{C})$ the first curve contracted by $\pi_2$ is the transform of the line at infinity $($in the domain$)$.
\end{itemize}
\end{lem}

\begin{proof}
According to Lemma \ref{bobo} if $p$ is a point of indeterminacy of $f$ there exists a curve contracted by~$f^{-1}$ on $p$. As~$f$ is induced by an automorphism 
of $\mathbb{C}^2$ the unique curve on $\mathbb{P}^2(\mathbb{C})$ which can be blown down is the line at infinity so $f$ has at most one point of indeterminacy. 
As~$f$ is not a morphism, it has exactly one.

  The second assertion is obtained by induction.

  Each irreducible curve contained in the divisor at infinity is either contracted on a point, or sent on the line at infinity in $\mathbb{P}^2(\mathbb{C})$. Since 
$f^{-1}$ contracts the line at infinity on a point the second eventuality is excluded.

  According to Theorem \ref{Zariski} we have
\begin{align*}
\xymatrix{& \mathrm{S}'\ar[dl]_{\pi_1}\ar[dr]^{\pi_2} &\\
\mathrm{S}\ar@{-->}[rr]_f & & \mathbb{P}^2(\mathbb{C}) }
\end{align*}
where $\mathrm{S}'$ is a smooth projective surface and $\pi_1\colon \mathrm{S}'\to \mathrm{S}$, $\pi_2\colon \mathrm{S}'\to \mathbb{P}^2(\mathbb{C})$ are two sequences 
of blow-ups. The divisor at infinity in $\mathrm{S}'$ is the union of 
\begin{itemize}
\item[$\bullet$] a divisor of self-intersection $-1$ obtained by blowing-up $p_s$, 
\item[$\bullet$] the other divisors, all of self-intersection $\leq -2$, produced in the sequence of blow-ups,
\item[$\bullet$] and the strict transform of  the divisor at infinity in $\mathrm{S}'$. 
\end{itemize}
The first curve contracted by $\pi_2$ is of self-intersection $-1$ and cannot be the last curve produced by $\pi_1$ (otherwise $p_s$ is not a point of indeterminacy); 
so the first curve contracted by $\pi_2$ is the strict transform of a curve contained in the divisor at infinity.

  The last assertion follows from the previous one.
\end{proof}

\chapter{Some subgroups of the Cremona group}\label{Chap:jung}

\section{A special subgroup: the group of polynomial automorphisms of the plane}\label{Sec:autpoly}

  A \textbf{\textit{polynomial automorphism}}\label{Chap2:ind22} of $\mathbb{C}^2$ is a bijective application of the following type
\begin{align*}
&  f\colon\mathbb{C}^2\to\mathbb{C}^2, && (x,y)\mapsto( f_1(x,y), f_2(x,y)), &&  f_i\in\mathbb{C}[x,y].
\end{align*}
The \textbf{\textit{degree}}\label{Chap2:ind23} of $ f=( f_1, f_2)$ is defined by $\deg  f=\max(\deg  f_1,\deg  f_2).$ Note that $\deg \psi 
f\psi^{-1}\not=\deg  f$ in general so we define the \textbf{\textit{first dynamical degree}}\label{Chap2:ind24} of $f$ $$d( f)=\lim(\deg  
f^n)^{1/n}$$ which is invariant under conjugacy\footnote{The limit exists since the sequence $\{\deg f^n\}_{n\in\mathbb{N}}$ is submultiplicative}. 
The set of the polynomial automorphisms is a group denoted by $\mathrm{Aut}(\mathbb{C}^2).$

\begin{egs}
\begin{itemize}
\item[$\bullet$] The map 
\begin{align*}
&\mathbb{C}^2\to\mathbb{C}^2,&&(x,y)\mapsto(a_1x+b_1y+c_1,a_2x+b_2y+c_2),\\
\end{align*}
\begin{align*}
& a_i,\,b_i,\,c_i\in \mathbb{C},\,a_1b_2-a_2b_1\not=0
\end{align*}
is an automorphism of $\mathbb{C}^2.$ The set of all these maps is the \textbf{\textit{affine group}}\label{Chap2:ind25} $\mathtt{A}.$

\item[$\bullet$] The map
\begin{align*}
&\mathbb{C}^2\to\mathbb{C}^2, && (x,y)\mapsto(\alpha x+P(y),\beta y+\gamma), 
\end{align*}
\begin{align*}
& \alpha,\,\beta,\, \gamma\in\mathbb{C},\,\alpha\beta\not=0,\, 
P\in\mathbb{C}[y]
\end{align*}
is an automorphism of $\mathbb{C}^2.$ The set of all these maps is a group, the \textbf{\textit{elementary group}}\label{Chap2:ind26}~$\mathtt{E}.$

\item[$\bullet$] Of course $$\mathtt{S}=\mathtt{A}\cap\mathtt{E}=\big\{(a_1x+b_1y+c_1,b_2y+c_2)\,\big\vert\, a_i,\,b_i,\,c_i\in\mathbb{C},
\,a_1b_2\not=0\big\}$$ is a subgroup of $\mathrm{Aut}(\mathbb{C}^2).$
\end{itemize}
\end{egs}

  The group $\mathrm{Aut}(\mathbb{C}^2)$ has a very special structure.

\begin{thm}[\cite{Ju}, Jung's Theorem]\label{jung}
The group $\mathrm{Aut}(\mathbb{C}^2)$ is the amalgamated product of~$\mathtt{A}$ and $\mathtt{E}$ along~$\mathtt{S}:$ $$\mathrm{Aut}
(\mathbb{C}^2)=\mathtt{A}\ast_\mathtt{S}\mathtt{E}.$$ 
In other words $\mathtt{A}$ and $\mathtt{E}$ generate $\mathrm{Aut}(\mathbb{C}^2)$ and each element $f$ in $\mathrm{Aut}(\mathbb{C}^2)\setminus
\mathtt{S}$ can be written as follows
\begin{align*}
& f=(a_1)e_1\ldots a_n(e_n), && e_i\in\mathtt{E}\setminus\mathtt{A},\, a_i\in\mathtt{A}\setminus\mathtt{E}.
\end{align*}
Moreover this decomposition is unique modulo the following relations
\begin{align*}
& a_ie_i=(a_is)(s^{-1}e_i), && e_{i-1}a_i=(e_{i-1}s')(s'^{-1}a_i), && s,\, s' \in\mathtt{S}.
\end{align*}
\end{thm}

\begin{rem}
The Cremona group is not an amalgam $($\cite{Yves}$)$. Nevertheless we know generators for $\mathrm{Bir}(\mathbb{P}^2):$
\begin{thm}[\cite{No, No2, No3, Cas}]\label{nono}
The Cremona group is generated by~$\mathrm{Aut}(\mathbb{P}^2)=\mathrm{PGL}_3(\mathbb{C})$ and the involution~$\left(\frac{1}{x},\frac{1}{y}\right)$.
\end{thm}
\end{rem}

  There are many proofs of Theorem \ref{jung}; you can find a "historical review" in \cite{La2}. We will now give an idea of the proof 
done in \cite{La2} and give details in \S \ref{Sec:jung}. Let $$\widetilde{ f}\colon(x,y)\mapsto(\widetilde{ f}_1(x,y),\widetilde{f}_2
(x,y))$$ be a polynomial automorphism of $\mathbb{C}^2$ of degree $\nu.$ We can view $\widetilde{ f}$ as a birational map:
\begin{align*}
&  f\colon\mathbb{P}^2(\mathbb{C})\dashrightarrow\mathbb{P}^2(\mathbb{C}), && (x:y:z)\dashrightarrow \left(z^\nu\widetilde{ f}_1\left(\frac{x}
{z},\frac{y}{z}\right):z^\nu\widetilde{ f}_2\left(\frac{x}{z},\frac{y}{z}\right):z^\nu\right).
\end{align*}
Lamy proved there exists $\varphi\in\mathrm{Bir}(\mathbb{P}^2)$ induced by a polynomial automorphism of $\mathbb{C}^2$ such that $\#\,\mathrm{Ind}\, 
 f\varphi^{-1}<\#\,\mathrm{Ind}\, f$; more precisely "$\varphi$ comes from an elementary automorphism". Proceeding recursively we obtain a map $g$ such 
that $\#\mathrm{Ind}\, f=0$, in other words an automorphism of $\mathbb{P}^2(\mathbb{C})$ which gives an affine automorphism.

  According to Bass-Serre theory (\cite{Se}) we can canonically associate a tree to any amalgamated product. Let~$\mathcal{T}$ be the tree associated 
to $\mathrm{Aut}(\mathbb{C}^2)$: 
\begin{itemize}
\item[$\bullet$] the disjoint union of~$\mathrm{Aut}(\mathbb{C}^2)/\mathtt{E}$ and $\mathrm{Aut}(\mathbb{C}^2)/\mathtt{A}$ 
is the set of vertices, 
\item[$\bullet$] $\mathrm{Aut}(\mathbb{C}^2)/\mathtt{S}$ is the set of edges. 
\end{itemize}
All these quotients must be understood as being left cosets; 
the cosets of~$f\in\mathrm{Aut}(\mathbb{C}^2)$ are noted respectively $ f\mathtt{E},$ $ f\mathtt{A},$ and $ f\mathtt{S}.$ By definition the edge 
$h\mathtt{S}$ links the vertices $ f\mathtt{A}$ and $g\mathtt{E}$ if $h\mathtt{S}\subset f\mathtt{A}$ and $h\mathtt{S}\subset g\mathtt{E}$ (and 
so~$f\mathtt{A}=h\mathtt{A}$ and $g\mathtt{E}=h\mathtt{E}$). In this way we obtain a graph; the fact that $\mathtt{A}$ and $\mathtt{E}$ are amalgamated 
along $\mathtt{S}$ is equivalent to the fact that~$\mathcal{T}$ is a tree (\cite{Se}). This tree is uniquely characterized (up to isomorphism) by 
the following pro\-perty: there exists an action of $\mathrm{Aut}(\mathbb{C}^2)$ on $\mathcal{T},$ such that the fundamental domain of this action 
is a segment, {\it i.e.} an edge and two vertices, with $\mathtt{E}$ and~$\mathtt{A}$ equal to the stabilizers of the vertices of this segment (and 
so $\mathtt{S}$ is the stabilizer of the entire segment). This action is simply the left translation: $g(h\mathtt{S})=(g\circ h)\mathtt{S}.$

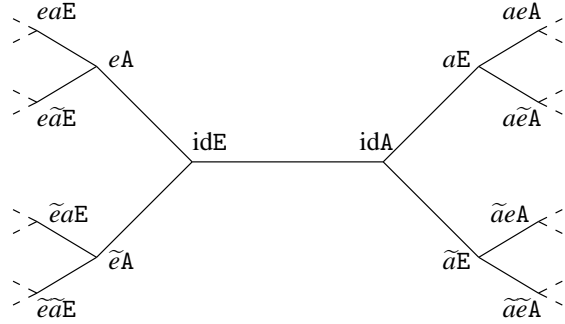
\begin{figure}[H]
\begin{center}
\input{arbre.pstex_t}
\end{center}
\end{figure}

  From a dynamical point of view affine automorphisms and elementary automorphisms are simple. Nevertheless there exist some elements in $\mathrm{Aut}(\mathbb{C}^2)$ with a rich dynamic; 
this is the case of \textbf{\textit{H\'enon automorphisms}}\label{Chap2:ind27}, automorphisms of the type $\varphi g_1\ldots g_p\varphi^{-1}$ with
\begin{align*}
& \varphi\in\mathrm{Aut}(\mathbb{C}^2),\, g_i=(y,P_i(y)-\delta_i x),\,P_i\in\mathbb{C}[y],\,\deg P_i\geq 2,\,\delta_i\in\mathbb{C}^*.
\end{align*}
Note that $g_i=\stackrel{\in\,\mathtt{A}\setminus\mathtt{E}}{\overbrace{(y,x)}}\,\,\,\,\,\stackrel{\in\,\mathtt{E}\setminus\mathtt{A}}{\overbrace{(-\delta_ix+P_i(y),y)}}.$

  Using Jung's theorem, Friedland and Milnor proved the following statement.

\begin{pro}[\cite{FM}]
Let $ f$ be an element of $\mathrm{Aut}(\mathbb{C}^2).$

  Either $ f$ is conjugate to an element of $\mathtt{E},$ or $ f$ is a H\'enon automorphism.
\end{pro}

  If $ f$ belongs to $\mathtt{E},$ then $d( f)=1.$ If $ f=g_1\ldots g_p$ with $g_i=(y,P_i(y)-\delta_i x),$ then $d( f)=\displaystyle\prod_{i=1}^p\deg g_i\geq 2.$ Then we have 
\begin{itemize}
\item[$\bullet$] $d( f)=1$ if and only if $f$ is conjugate to an element of $\mathtt{E};$

\item[$\bullet$] $d( f)>1$ if and only if $ f$ is a H\'enon automorphism.
\end{itemize}

\smallskip

  H\'enon automorphisms and elementary automorphisms are very different:

\smallskip
\begin{itemize}
 \item[$\bullet$] H\'enon automorphisms:

no invariant rational fibration (\cite{Br}),

countable centralizer (\cite{La}),

infinite number of hyperbolic periodic points;

\item[$\bullet$] Elementary automorphisms:

invariant rational fibration,

uncountable centralizer.
\end{itemize}

\section{Proof of \textsc{Jung's} theorem}\label{Sec:jung}

  Assume that $\Phi$ is a polynomial automorphism of $\mathbb{C}^2$ of degree $n$ 
\begin{align*}
&\Phi\colon(x,y)\mapsto (\Phi_1(x,y),\Phi_2(x,y)), && \Phi_i\in\mathbb{C}[x,y];
\end{align*}
we can extend $\Phi$ to a birational map still denoted by $\Phi$ $$\Phi\colon(x:y:z)\dashrightarrow\left(z^n\Phi_1\left(\frac{x}{z},\frac{y}{z}\right):z^n\Phi_2\left(\frac{x}{z},
\frac{y}{z}\right):z^n\right).$$ The line at infinity in $\mathbb{P}^2(\mathbb{C})$ is $z=0$. The map $\Phi\colon\mathbb{P}^2(\mathbb{C})\dashrightarrow\mathbb{P}^2
(\mathbb{C})$ has a unique point of indeterminacy which is on the line at infinity (Lemma~\ref{applzar}). We can assume, up to conjugation by an affine automorphism, 
that this point is $(1:0:0)$ (of course this conjugacy doesn't change the number of base-points of $\Phi$). We will show that there exists $\varphi\colon\mathbb{P}^2
(\mathbb{C})\dashrightarrow\mathbb{P}^2(\mathbb{C})$ a birational map induced by a polynomial automorphism of $\mathbb{C}^2$ such that $$
\xymatrix{& \mathbb{P}^2(\mathbb{C})\ar@{-->}[dr]^{\Phi\circ\varphi^{-1}} &\\
\mathbb{P}^2(\mathbb{C})\ar@{-->}[ur]^{\varphi}\ar@{-->}[rr]_\Phi & & \mathbb{P}^2(\mathbb{C})}$$ and $\#\text{ base-points of } \Phi\varphi^{-1}<\#\text{ base-points of } \Phi$. 
To do this we will rearrange the blow-ups of the sequences $\pi_1$ and $\pi_2$ appearing when we apply Zariski's Theorem: the map $\varphi$ is constructed by realising 
some blow-ups of~$\pi_1$ and some blow-ups of $\pi_2$.

\subsection{Hirzebruch surfaces}\label{hirzebruch}\,

Let us consider the surface $\mathbb{F}_1$ obtained by blowing-up $(1:0:0)\in\mathbb{P}^2(\mathbb{C})$. This surface is a compactification of $\mathbb{C}^2$ which has 
a natural rational fibration corresponding to the lines $y=$ constant. The divisor at infinity is the union of two rational curves ({\it i.e.} curves isomorphic to 
$\mathbb{P}^1(\mathbb{C})$) which intersect in one point. One of them is the strict transform of the line at infinity in $\mathbb{P}^2(\mathbb{C})$, it is a fiber 
denoted by $f_1$; the other one, denoted by~$s_1$ is the exceptional divisor which is a section for the fibration. We have: $f_1^2=0$ and $s_1^2=~-1$ (Proposition \ref{hart}). 
More generally for any $n$ we denote by $\mathbb{F}_n$ a compactification of $\mathbb{C}^2$ with a rational fibration and such that the divisor at infinity is the union 
of two transversal rational curves: a fiber $f_\infty$ and a section $s_\infty$ of self-intersection $-n$. These surfaces are called 
\textbf{\textit{Hirzebruch surfaces}}\label{ind1272}: 
$$\mathbb{P}_{\mathbb{P}^1(\mathbb{C})}\big(\mathcal{O}_{\mathbb{P}^1(\mathbb{C})}\oplus\mathcal{O}_{\mathbb{P}^1(\mathbb{C})}(n)\big).$$ Let us consider the surface 
$\mathbb{F}_n$. Let $p$ be the intersection of $s_n$ and $f_n$, where~$f_n$ is a fiber. Let $\pi_1$ be the blow-up of 
$p\in\mathbb{F}_n$ and let $\pi_2$ be the contraction of the strict transform $\widetilde{f_n}$ of $f_n$. We can go from $\mathbb{F}_n$ to $\mathbb{F}_{n+1}$ via~$\pi_2\pi_1^{-1}$:

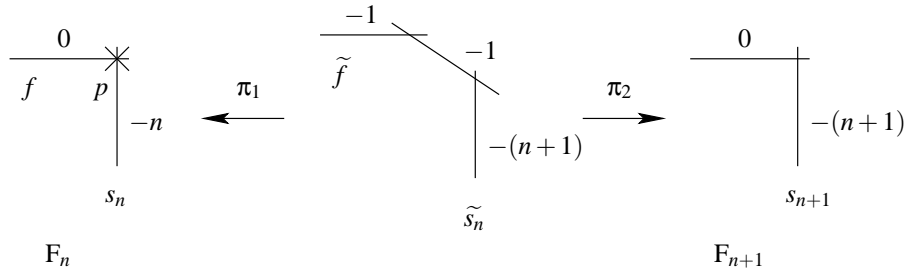
\begin{figure}[H]
\begin{center}
\input{el1.pstex_t}
\end{center}
\end{figure}

  We can also go from $\mathbb{F}_{n+1}$ to $\mathbb{F}_n$ via $\pi_2\pi_1^{-1}$ where 
\begin{itemize}
\item[$\bullet$] $\pi_1$ is the blow-up of a point $p\in\mathbb{F}_{n+1}$ which belongs to the fiber $f_n$ and not to the section~$s_{n+1}$, 
\item[$\bullet$] $\pi_2$ the contraction of the strict transform~$\widetilde{f_n}$ of $f_n:$
\end{itemize}

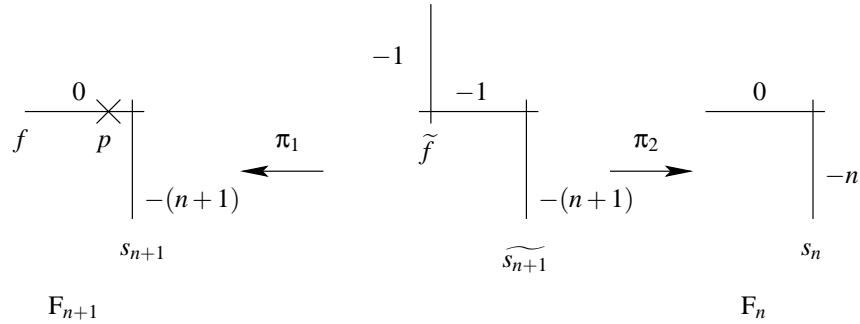
\begin{figure}[H]
\begin{center}
\input{el2.pstex_t}
\end{center}
\end{figure}

\subsection{First step: blow-up of $(1:0:0)$}\,

The point $(1:0:0)$ is the first blown-up point in the sequence $\pi_1$. Let us denote by~$\varphi_1$ the blow-up of $(1:0:0)\in\mathbb{P}^2(\mathbb{C})$, we have  
$$
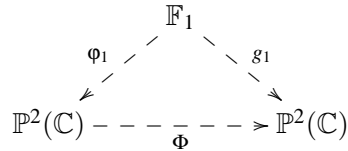
\xymatrix{& \mathbb{F}_1\ar@{-->}[dl]_{\varphi_1}\ar@{-->}[dr]^{g_1} &\\
\mathbb{P}^2(\mathbb{C})\ar@{-->}[rr]_\Phi & & \mathbb{P}^2(\mathbb{C})}$$ Note that $\#\text{ base-points of }g_1=\#\text{ base-points of }\Phi-1$. Let us come back to the 
diagram given by Zariski's theorem. The first curve contracted by~$\pi_2$ which is a curve of self-intersection~$-1$ is the strict transform of the line at infinity 
(Lemma \ref{applzar}, last assertion); it corresponds to the fiber $f_1$ in $\mathbb{F}_1$. But in $\mathbb{F}_1$ we have $f_1^2=0$; the self-intersection of this curve 
has thus to decrease so the point of indeterminacy $p$ of $g_1$ has to belong to~$f_1$. But~$p$ also belongs to the curve produced by the blow-up (Lemma \ref{applzar}, 
second assertion); in other words~$p=f_1\cap s_1$.

\subsection{Second step: Upward induction}\,

\begin{lem}\label{upward}
Let $n\geq 1$ and let $h\colon\mathbb{F}_n\dashrightarrow\mathbb{P}^2(\mathbb{C})$ be a birational map induced by a polynomial automorphism of $\mathbb{C}^2$. Suppose 
that $h$ has only one point of indeterminacy $p$ such that~$p=~f_n\cap s_n$. Let $\varphi\colon\mathbb{F}_n\dashrightarrow\mathbb{F}_{n+1}$ be the birational map which 
is the blow-up of $p$ composed with the contraction of the strict transform of $f_n$. Let us consider the birational map~$h'=h\circ\varphi^{-1}$: 
$$
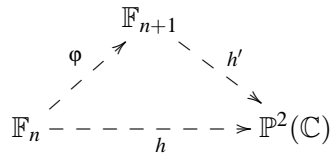
\xymatrix{& \mathbb{F}_{n+1}\ar@{-->}[dr]^{h'} &\\
\mathbb{F}_n\ar@{-->}[ur]^{\varphi}\ar@{-->}[rr]_h & & \mathbb{P}^2(\mathbb{C})}$$ 
Then 
\begin{itemize}
\item[$\bullet$] $\#\text{ base-points of }h'=\#\text{ base-points of }h-1$;

\item[$\bullet$] the point of indeterminacy of $h'$ belongs to $f_{n+1}$.
\end{itemize}
\end{lem}

\begin{proof}
Let us apply Zariski Theorem to $h$; we obtain 
$$
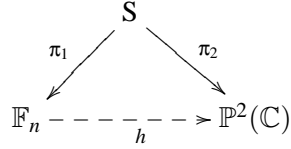
\xymatrix{& \mathrm{S}\ar[dl]_{\pi_1} \ar[dr]^{\pi_2}&\\
\mathbb{F}_n\ar@{-->}[rr]_h & & \mathbb{P}^2(\mathbb{C})}$$ where $\mathrm{S}$ is a smooth projective surface and $\pi_1$, $\pi_2$ are two sequences of blow-ups. 

Since $\widetilde{s_n}^2\leq -2$ (where $\widetilde{s_n}$ is the strict transform of $s_n$) the first curve contracted by $\pi_2$ is the transform of~$f_n$ (Lemma 
\ref{applzar}). So the transform of $f_n$ in~$\mathrm{S}$ is of self-intersection
$-1$; we also have $f_n^2=0$ in $\mathbb{F}_n$. This implies that after the blow-up of $p$ the points appearing in $\pi_1$ are not on~$f_n$. Instead of realising these 
blow-ups and then contracting the transform of $f_n$ we first contract and then realise the blow-ups. In other words we have the following diagram
$$
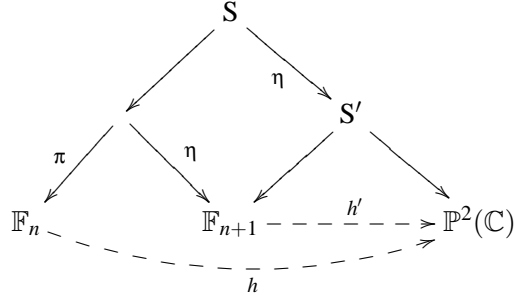
\xymatrix{& &\mathrm{S}\ar[dl] \ar[dr]_{\eta}&&\\
&\ar[dr]^\eta\ar[dl]_\pi  & &  \mathrm{S}' \ar[dr]\ar[dl]&\\
\mathbb{F}_n\ar@{-->}@/_0.6cm/[rrrr]_h & &\mathbb{F}_{n+1}\ar@{-->}[rr]^{h'} & &\mathbb{P}^2(\mathbb{C})
}$$
\smallskip
where $\pi$ is the blow-up of $p$ and $\eta$ the contraction of $f_n$. The map $\eta\pi^{-1}$ is exactly the first link mentioned in~\S\ref{hirzebruch}. We can see that to 
blow-up $p$ allows us to decrease the number of points of indeterminacy and to contract~$f_n$ does not create some point of indeterminacy. So 
$$\#\text{ base-points of $h'$}=\#\text{ base-points of $h$ $-1$}$$ Moreover the point of indeterminacy of $h'$ is on the curve obtained by the blow-up of $p$, {\it i.e.} $f_n$.

\end{proof}

  After the first step we are under the assumptions of the Lemma \ref{upward} with $n=1$. The Lemma gives an application $h'\colon\mathbb{F}_2\dashrightarrow\mathbb{P}^2(\mathbb{C})$ such 
that the point of indeterminacy is on $f_2$. If this point also belongs to $s_2$ we can apply the Lemma again. Repeating this as long as the assumptions of the Lemma \ref{upward} are 
satisfied, we obtain the following diagram
$$
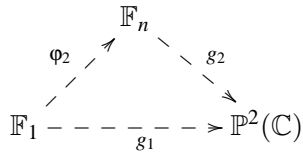
\xymatrix{& \mathbb{F}_{n}\ar@{-->}[dr]^{g_2} &\\
\mathbb{F}_1\ar@{-->}[ur]^{\varphi_2}\ar@{-->}[rr]_{g_1} & & \mathbb{P}^2(\mathbb{C})}$$ 
where $\varphi_2$ is obtained by applying $n-1$ times Lemma \ref{upward}. Moreover $$\#\text{ base-points of }g_2=\#\text{ base-points of }g_1-n+1$$ and the point of indeterminacy of 
$g_2$ is on $f_n$ but not on $s_n$ (remark: as, for $n\geq 2$, there is no morphism from~$\mathbb{F}_n$ to $\mathbb{P}^2(\mathbb{C})$, the map $g_2$ has a point of indeterminacy).

\subsection{Third step: Downward induction}\,

\begin{lem}\label{downward}
Let $n\geq 2$ and let $h\colon\mathbb{F}_n\dashrightarrow\mathbb{P}^2(\mathbb{C})$ be a birational map induced by a polynomial automorphism of $\mathbb{C}^2$. Assume that $h$ has only 
one point of indeterminacy $p$, and that $p$ belongs to $f_n$ but not to $s_n$. Let~$\varphi\colon\mathbb{F}_n\dashrightarrow\mathbb{F}_{n-1}$ be the birational map which is the blow-up 
of~$p$ composed with the contraction of the strict transform of $f_n$. Let us consider the birational map~$h'=h\circ\varphi^{-1}$: 
$$
\xymatrix{& \mathbb{F}_{n-1}\ar@{-->}[dr]^{h'} &\\
\mathbb{F}_n\ar@{-->}[ur]^{\varphi}\ar@{-->}[rr]_h & & \mathbb{P}^2(\mathbb{C})}$$ 
Then 
\begin{itemize}
\item[$\bullet$] $\#\text{ base-points of }h'=\#\text{ base-points of }h-1$;

\item[$\bullet$] if $h'$ has a point of indeterminacy, it belongs to $f_{n-1}$ and not to $s_{n-1}$.
\end{itemize}
\end{lem}

\begin{proof}
Let us consider the Zariski decomposition of $h$
$$
\xymatrix{& \mathrm{S}\ar[dl]_{\pi_1}\ar[dr]^{\pi_2} &\\
\mathbb{F}_n\ar@{-->}[rr]_h & & \mathbb{P}^2(\mathbb{C})}$$
Since $\widetilde{s_n}^2=-n$ with $n\geq 2$, the first curve 
blown down by $\pi_2$ is the transform of $f_n$ (Lemma \ref{applzar}). Like in the proof of Lemma \ref{upward} we obtain 
the following commutative diagram
$$
\xymatrix{& &\mathrm{S}\ar[dl] \ar[dr]_{\eta}&&\\
&\ar[dr]^\eta\ar[dl]_\pi  & &  \mathrm{S}' \ar[dr]\ar[dl]&\\
\mathbb{F}_n\ar@{-->}@/_0.6cm/[rrrr]_h & &\mathbb{F}_{n-1}\ar@{-->}[rr]^{h'} & &\mathbb{P}^2(\mathbb{C})
}$$
where $\pi$ is the blow-up of $p$ and $\eta$ the contraction of $f_n$.
We immediately have: $$\#\text{ base-points of } h'=\#\text{ base-points of } h-1.$$
Let $F'$ be the exceptional divisor associated to $\pi$; the map $h$ has a base-point on $F'$.
Assume that this point is $F'\cap \widetilde{f_n}$, then $(\pi_1^{-1}(f_n))^2\leq -2$: contradiction
with the fact that it is the first curve blown down by $\pi_2$. So the base-point of $h$ is not
 $F'\cap \widetilde{f_n}$ and so it is the point of indeterminacy of~$h'$ that is on $f_{n-1}$
but not on $s_{n-1}$.
\end{proof}

 After the second step the assumptions in Lemma \ref{downward} are satisfied. Let us remark that if~$n\geq 3$ then the map $h'$ given by Lemma \ref{downward} still satisfies 
the assumptions in this Lemma. After applying $n-1$ times Lemma~\ref{downward} we have the following diagram
$$
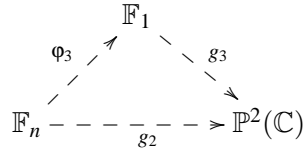
\xymatrix{& \mathbb{F}_{1}\ar@{-->}[dr]^{g_3} &\\
\mathbb{F}_n\ar@{-->}[ur]^{\varphi_3}\ar@{-->}[rr]_{g_2} & & \mathbb{P}^2(\mathbb{C})}$$ 

\subsection{Last contraction}\,

Applying Zariski's theorem to $g_3$ we obtain 
$$
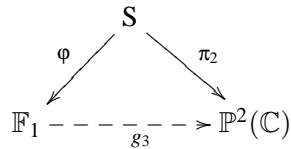
\xymatrix{& \mathrm{S}\ar[dl]_{\varphi}\ar[dr]^{\pi_2} &\\
\mathbb{F}_1\ar@{-->}[rr]_{g_3} & & \mathbb{P}^2(\mathbb{C})}$$ 
The fourth assertion of the Lemma \ref{applzar} implies that the first curve contracted by $\pi_2$ is either the strict transform of $f_1$ by $\pi_1$, or the strict transform of $s_1$ 
by $\pi_1$. Assume that we are in the first case; then after realising the sequence of blow-ups $\pi_1$ and contracting this curve the transform of $s_1$ is of self-intersection $0$ 
and so cannot be contracted: contradiction with the third assertion of Lemma \ref{applzar}. So the first curve contracted is the strict transform of $s_1$ which can be done and we 
obtain 
$$
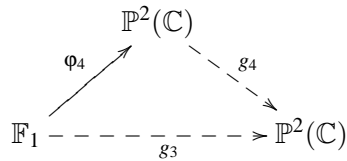
\xymatrix{&  \mathbb{P}^2(\mathbb{C})\ar@{-->}[dr]^{g_4} &\\
\mathbb{F}_1\ar[ur]^{\varphi_4}\ar@{-->}[rr]_{g_3} & & \mathbb{P}^2(\mathbb{C})}$$ 
The morphism $\varphi_4$ is the blow-up of a point and the exceptional divisor associated to its blow-up is $s_1$; up to an automorphism we can assume that~$s_1$ is contracted 
on $(1:0:0)$. Moreover $$\#\text{ base-points of }g_3=\#\text{ base-points of }g_4.$$

\subsection{Conclusion}\,

  After all these steps we have 
$$
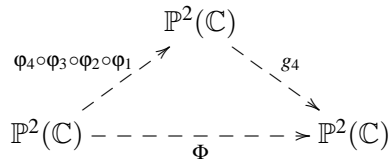
\xymatrix{& \mathbb{P}^2(\mathbb{C})\ar@{-->}[dr]^{g_4} &\\
\mathbb{P}^2(\mathbb{C})\ar@{-->}[ur]^{\varphi_4\circ\varphi_3\circ\varphi_2\circ\varphi_1}\ar@{-->}[rr]_{\Phi} & & \mathbb{P}^2(\mathbb{C})}$$ 
where $\#\text{ base-points of }g_4=\#\text{ base-points of }\Phi-2n+1$ (with $n\geq 2$).

  Let us check that $\varphi=\varphi_4\circ\varphi_3\circ\varphi_2\circ\varphi_1$ is induced by an element of $\mathtt{E}$. It is sufficient to prove that $\varphi$ preserves 
the fibration $y=$ constant, {\it i.e.} the pencil of curves through $(1:0:0)$; indeed
\begin{itemize}
\item[$\bullet$] the blow-up $\varphi_1$ sends lines through $(1:0:0)$ on the fibers of $\mathbb{F}_1$;

\item[$\bullet$] $\varphi_2$ and $\varphi_3$ preserve the fibrations associated to $\mathbb{F}_1$ and $\mathbb{F}_n$;

\item[$\bullet$] the morphism $\varphi_4$ sends fibers of $\mathbb{F}_1$ on lines through $(1:0:0)$. 
\end{itemize}
Finally $g_4$ is obtained by composing $\Phi$ with a birational map induced by an affine automorphism and a birational map induced by an element of $\mathtt{E}$ so $g_4$ is induced 
by a polynomial automorphism; morevoer $$\#\text{ base-points of }g_4<\#\text{ base-points of }\Phi.$$

\subsection{Example}

  Let us consider the polynomial automorphism $\Phi$ of $\mathbb{C}^2$ given by $$\Phi=\big(y+(y+x^2)^2+(y+x^2)^3,y+x^2\big).$$  

  Let us now apply to $\phi$ the method just explained above. The point of indeterminacy of~$\Phi$ is $(0:1:0)$. Let us compose $\Phi$ with $(y,x)$ to deal with an automorphism 
whose point of indeterminacy is $(1:0:0)$. Let us blow up this point
  
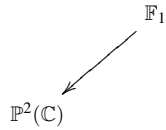
\begin{scriptsize}
$$
\xymatrix{& \mathbb{F}_1\ar[dl]\\
\mathbb{P}^2(\mathbb{C})& }$$ 
\end{scriptsize}

  Then we apply Lemma \ref{upward}
  \begin{scriptsize}
$$\xymatrix{&&\ar[dl]\ar[dr]&\\
& \mathbb{F}_1\ar[dl] & &\mathbb{F}_2\\
\mathbb{P}^2(\mathbb{C})& && }$$ 
\end{scriptsize}

 On $\mathbb{F}_2$ the point of indeterminacy is on the fiber, we thus apply Lemma~\ref{downward}
\begin{scriptsize}
$$\xymatrix{&&\ar[dl]\ar[dr]& & \ar[dl]\ar[dr]&\\
& \mathbb{F}_1\ar[dl] & &\mathbb{F}_2& &\mathbb{F}_1\\
\mathbb{P}^2(\mathbb{C})& && && }$$ 
\end{scriptsize}
and contracts $s_1$
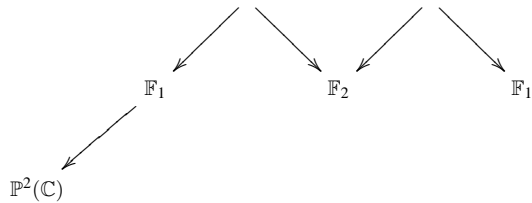
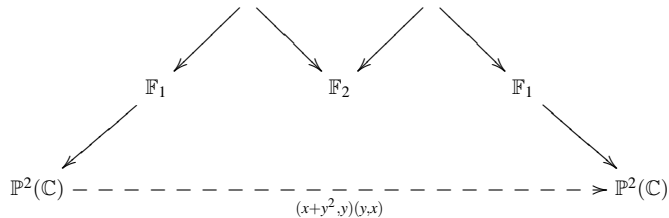
\begin{scriptsize}
$$\xymatrix{&&\ar[dl]\ar[dr]& & \ar[dl]\ar[dr]& & \\
& \mathbb{F}_1\ar[dl] & &\mathbb{F}_2& &\mathbb{F}_1 \ar[dr]\\
\mathbb{P}^2(\mathbb{C})\ar@{-->}[rrrrrr]_{(x+y^2,y)(y,x)}& && &&& \mathbb{P}^2(\mathbb{C})}$$ 
\end{scriptsize}
We get the decomposition $\Phi=\Phi'(x+y^2,y)(y,x)$ with $$\Phi'=(y+x^2+x^3,x)=(x+y^2+y^3,y)(y,x).$$ We can check that $\Phi'$ has a unique point of indeterminacy $(0:1:0)$. 
Let us blow up the point $(1:0:0)$
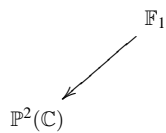
\begin{scriptsize}
$$\xymatrix{& \mathbb{F}_1\ar[dl]\\
\mathbb{P}^2(\mathbb{C})& }$$ 
\end{scriptsize}
and then apply two times Lemma \ref{upward} 
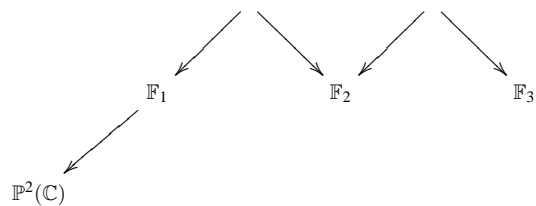
\begin{scriptsize}
$$\xymatrix{&&\ar[dl]\ar[dr]& & \ar[dl]\ar[dr]&\\
& \mathbb{F}_1\ar[dl] & &\mathbb{F}_2& &\mathbb{F}_3\\
\mathbb{P}^2(\mathbb{C})& && && }$$
\end{scriptsize} 
then two times Lemma \ref{downward} 
\begin{scriptsize}
$$\xymatrix{&&\ar[dl]\ar[dr]& & \ar[dl]\ar[dr]&& \ar[dl]\ar[dr]&& \ar[dl]\ar[dr]&\\
& \mathbb{F}_1\ar[dl] & &\mathbb{F}_2& &\mathbb{F}_3& &\mathbb{F}_2& &\mathbb{F}_1\\
\mathbb{P}^2(\mathbb{C})& && && && &&}$$
\end{scriptsize} 
\begin{landscape}
Finally we contract the section $s_1$ 
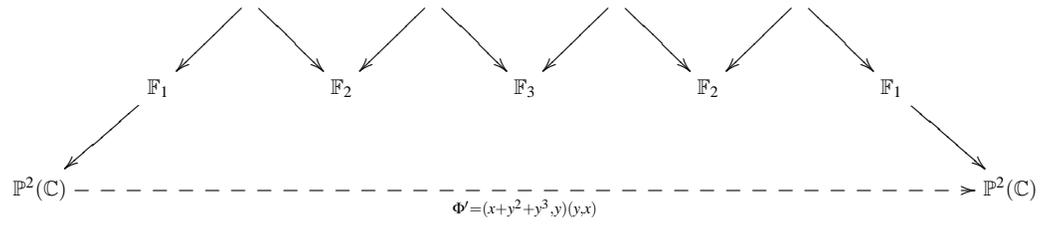
\begin{scriptsize}
$$\xymatrix{&&\ar[dl]\ar[dr]& & \ar[dl]\ar[dr]&& \ar[dl]\ar[dr]&& \ar[dl]\ar[dr]&&\\
& \mathbb{F}_1\ar[dl] & &\mathbb{F}_2& &\mathbb{F}_3& &\mathbb{F}_2& &\mathbb{F}_1\ar[dr]\\
\mathbb{P}^2(\mathbb{C})\ar@{-->}[rrrrrrrrrr]_{\Phi'=(x+y^2+y^3,y)(y,x)}& && && && &&&\mathbb{P}^2(\mathbb{C})}$$ 
\end{scriptsize}
and obtain $\Phi'=(x+y^2+y^3,y)(y,x)$.
\end{landscape}

\section{The de Jonqui\`eres group}\label{Sec:Jonq}

The \textbf{\textit{de Jonqui\`eres maps}}\label{Chap2:ind28} are, up to birational conjugacy, of the 
following type 
$$\left(\frac{a(y)x+b(y)}{c(y)x+d(y)},\frac{\alpha y+\beta}{\gamma y+\delta}\right),$$
\begin{align*}
& \left[\begin{array}{cc} a(y) & b(y) \\ c(y) & d(y) \end{array}\right]\in\mathrm{PGL}_2(
\mathbb{C}(y)), && \left[\begin{array}{cc} \alpha & \beta \\ \gamma & \delta
\end{array}\right]\in\mathrm{PGL}_2(\mathbb{C});
\end{align*}
let us remark that the family of lines $y=$ constant is preserved by such a Cremona transformation. De Jonqui\`eres maps are exactly the Cremona 
maps which preserve a rational fibration\footnote{Here a rational fibration is a rational application from $\mathbb{P}^2(\mathbb{C})$ into
$\mathbb{P}^1(\mathbb{C})$ whose fibers are rational curves.}. The de Jonqui\`eres maps form a group, called \textbf{\textit{de Jonqui\`eres 
group}}\label{Chap2:ind29} and denoted by~$\mathrm{dJ}$. Remark that the exceptional set of $\phi$ is reduced to a finite number of fibers $y=$ cte 
and possibly the line at infinity. 

  In some sense $\mathrm{dJ}\subset\mathrm{Bir}(\mathbb{P}^2)$ is the analogue of $\mathtt{E}\subset\mathrm{Aut}(\mathbb{C}^2).$   In the $80$'s 
Gizatullin and Iskovskikh give a presentation of $\mathrm{Bir}(\mathbb{P}^2)$ (\emph{see} \cite{Gi, Is}); let us state the result of Iskovskikh 
presented in $\mathbb{P}^1(\mathbb{C})\times\mathbb{P}^1(\mathbb{C})$ which is birationally isomorphic to $\mathbb{P}^2(\mathbb{C})$.

\begin{thm}[\cite{Is}]\label{isk}
The group of birational maps of $\mathbb{P}^1(\mathbb{C})\times\mathbb{P}^1(\mathbb{C})$ is generated by $\mathrm{dJ}$ and~$\mathrm{Aut}(\mathbb{P}^1
(\mathbb{C})\times\mathbb{P}^1(\mathbb{C}))$ \footnote{The de 
ui\`eres group is birationally isomorphic to the subgroup of  $\mathrm{Bir}
(\mathbb{P}^1(\mathbb{C})\times\mathbb{P}^1(\mathbb{C}))$ which preserves the first projection $p\colon\mathbb{P}^1(\mathbb{C})\times\mathbb{P}^1(
\mathbb{C})\to\mathbb{P}^1(\mathbb{C}).$}.

  Moreover the relations in $\mathrm{Bir}(\mathbb{P}^1(\mathbb{C})\times\mathbb{P}^1(\mathbb{C}))$ are the relations of $\mathrm{dJ},$ of~$\mathrm{Aut}
(\mathbb{P}^1(\mathbb{C})\times\mathbb{P}^1(\mathbb{C}))$ and the relation 
\begin{align*}
&(\eta e)^3=\left(\frac{1}{x},\frac{1}{y}\right)&& \text{where} && \eta\colon(x,y)\mapsto(y,x) && \& && e\colon(x,y)\mapsto\left(x,\frac{x} {y}\right).
\end{align*}
\end{thm}

  Let $f$ be a birational map of $\mathbb{P}^2(\mathbb{C})$ of degree $\nu$. Assume that $f$ has a base-point $p_1$ of multiplicity $m_1=\nu-1$. Then we have
\begin{align*}
& \nu^2-(\nu-1)^2-\sum_{i=2}^r m_i^2=1, && 3\nu-(\nu-1)-\sum_{i=2}^r m_i=3
\end{align*}
where $p_2$, $\ldots$, $p_r$ are the other base-points of $f$ and $m_i$ the multiplicity of~$p_i$. This implies that~$\sum_{i=2}^rm_i(m_i-1)=0$, hence 
$m_2=\ldots=m_r=1$ and $r=2\nu-1$. For simplicity let us assume that the $p_i$'s are in $\mathbb{P}^2(\mathbb{C})$. The homaloidal system $\Lambda_f$ 
consists of curves of degree $\nu$ with singular point $p_1$ of multiplicity $\nu-1$ passing simply to $2\nu-2$ points $p_2$, $\ldots$, $p_{2\nu-1}$. 
The corresponding Cremona transformation is a de Jonqui\`eres transformation. Indeed let $\Gamma$ be an element of $\Lambda_f$. Let $\Xi$ be the pencil 
of curves of $\Lambda_f$ that have in common with $\Gamma$ a point $m$ distinct from $p_1$, $\ldots$, $p_{2\nu-1}$. The number of intersections of $\Gamma$ 
with a generic curve of $\Xi$ that are absorbed by the $p_i$'s is at least $$(\nu-1)(\nu-2)+2\nu-2+1=\nu(\nu-1)+1$$ one more than the number given by Bezout's 
theorem. The curves of $\Xi$ are thus all split into $\Gamma$ and a line of the pencil centered in $p_1$. Let us assume that $p_1=(1:0:0)$; then $\Gamma$ is given by 
\begin{align*}
&x\psi_{\nu-2}(y,z)+\psi_{\nu-1}(y,z), && \deg\psi_i=i.
\end{align*}
To describe $\Lambda_f$ we need an arbitrary curve taken from $\Lambda_f$ and outside $\Xi$ which gives 
\begin{align*}
&(x\psi_{\nu-2}+\psi_{\nu-1})(a_0y+a_1z)+x\varphi_{\nu-1}(y,z)+\varphi_\nu(y,z), && \deg\varphi_i=i.
\end{align*}
Therefore $f$ can be represented by  
\begin{align*} 
&(x:y:z)\dashrightarrow\big(x\varphi_{\nu-1}+\varphi_\nu:(x\psi_{\nu-2}+\psi_{\nu-1})(ay+bz):(x\psi_{\nu-2}+\psi_{\nu-1})(cy+dz)\big) 
\end{align*} 
with $ad-bc\not=0$.
We can easily check that $f$ is invertible and that $\Lambda_f$ and~$\Lambda_{f^{-1}}$ have the same type. At last we have in the affine chart $z=1$ 
$$\left(\frac{x\varphi_{\nu-1}(y)+\varphi_\nu(y)}{x\psi_{\nu-2}(y)+\psi_{\nu-1}(y)},\frac{ay+b}{cy+d}\right).$$

\section{No dichotomy in the Cremona group}

There is a strong dichotomy in $\mathrm{Aut}(\mathbb{C}^2)$ (\emph{see} \S\ref{Sec:autpoly}); we will see that there is 
no such dichotomy in $\mathrm{Bir}(\mathbb{P}^2)$. Let us consider the family of birational maps $(f_{\alpha,\beta})$ given by 
\begin{align*}
&\mathbb{P}^2(\mathbb{C})\dashrightarrow\mathbb{P}^2(\mathbb{C}),
&&(x:y:z)\mapsto((\alpha x+y)z:\beta y(x+z):z(x+z)),
\end{align*}
$$\alpha,\,\beta\in\mathbb{C}^*,\,\vert\alpha\vert=\vert\beta\vert=1$$
so in the affine chart $z=1$
$$f_{\alpha,\beta}(x,y)=\left(\frac{\alpha x+y}{x+1},\beta y\right).$$

\begin{thm}[\cite{De}]\label{linearisation}
The first dynamical degree\footnote{For a birational map $f$ of $\mathbb{P}^2(\mathbb{C})$ the \textbf{\textit{first dynamical degree}}\label{Chap2:ind27a}
is given by $\lambda(f)=\displaystyle\lim_{n\to +\infty}(\deg f^n)^{1/n}$.} of $f_{\alpha,\beta}$ is equal to $1;$ more precisely~$\deg f_{\alpha,\beta}^n\sim~n.$

Assume that $\alpha$ and $\beta$ are generic and have modulus $1$. If $g$ commutes with~$f_{\alpha,\beta},$ then $g$ coincides with an iterate of
 $f_{\alpha,\beta};$ in particular the centralizer of $f_{\alpha,\beta}$ is countable.

The elements $f_{\alpha,\beta}^2$ have two fixed points $m_1,$ $m_2$ and
\smallskip 
\begin{itemize}
\item[$\bullet$] there exists a neighborhood $\mathcal{V}_1$ of $m_1$ on which $f_{\alpha,\beta}$ is conjugate to $(\alpha x,\beta y);$ in 
particular the closure of the orbit of a point of $\mathcal{V}_1$ $($under~$f_{\alpha,\beta})$ is a torus of dimension $2$; \smallskip

\item[$\bullet$] there exists a neighborhood $\mathcal{V}_2$ of $m_2$ such that $f_{\alpha,\beta}^2$ is locally lineari\-zable on $\mathcal{V}_2;$ 
the closure of a generic orbit of a point of $\mathcal{V}_2$ $($under $f_{\alpha,\beta}^2)$ is a circle.
\end{itemize}
\end{thm}

In the affine chart $(x,y)$ the maps $f_{\alpha,\beta}$ preserve the $3$-manifolds $\vert y\vert=$~cte. The orbits presented below are bounded 
in a copy of $\mathbb{R}^2\times\mathbb{S}^1.$ The dynamic happens essentially in dimension $3$; different projections allow us to have a good 
representation of the orbit of a point. In the affine chart $z=1$ let us denote by $p_1$ and $p_2$ the two standard projections. The given 
pictures are representations (in perspective) of the following projections. 

\begin{itemize}
\item[$\bullet$] Let us first consider the set
\begin{align*}
\Omega_1(m,\alpha,\beta)=\big\{(p_1(f_{\alpha,\beta}^n(m)),\mathrm{Im} (p_2(f_{
\alpha, \beta}^n(m))))\,\big\vert\, n=1..30000\big\};
\end{align*} 
this set is contained in the product of $\mathbb{R}^2$ with an interval. The orbit of a point under the action of $f_{\alpha,\beta}$ is 
compressed by the double covering  $(x,\rho e^{\mathbf{i}\theta})\to(x,\rho\sin\theta).$

\item[$\bullet$] Let us introduce
\begin{align*}
\Omega_2(m,\alpha,\beta)=\big\{(\mathrm{Re}(p_1(f_{\alpha,\beta}^n(m))),p_2(f_{\alpha,\beta}^n(m)))\,\big\vert\, n=1..30000\big\}
\end{align*} 
which is contained in a cylinder $\mathbb{R}\times\mathbb{S}^1;$ this second projection shows how to \og decompress\fg\hspace{1mm} $\Omega_1$ 
to have the picture of the orbit.
\end{itemize}

\bigskip

Let us assume that $\alpha=\exp(2\mathbf{i}\sqrt{3})$ and $\beta=\exp(2\mathbf{i}\sqrt{2})$; let us denote by~$\Omega_i(m)$ instead of~$\Omega_i(m,\alpha,\beta).$

The following pictures illustrate Theorem \ref{linearisation}.

\begin{tabular}{ll}
\hspace{8mm}\begin{picture}(0,0)%
\includegraphics{des1scan.pstex}%
\end{picture}%
\setlength{\unitlength}{3947sp}%
\begingroup\makeatletter\ifx\SetFigFont\undefined%
\gdef\SetFigFont#1#2#3#4#5{%
  \reset@font\fontsize{#1}{#2pt}%
  \fontfamily{#3}\fontseries{#4}\fontshape{#5}%
  \selectfont}%
\fi\endgroup%
\begin{picture}(1200,1200)(1201,-1561)
\end{picture}%
\hspace{16mm} &\hspace{16mm}
\begin{picture}(0,0)%
\includegraphics{des2scan.pstex}%
\end{picture}%
\setlength{\unitlength}{3947sp}%
\begingroup\makeatletter\ifx\SetFigFont\undefined%
\gdef\SetFigFont#1#2#3#4#5{%
  \reset@font\fontsize{#1}{#2pt}%
  \fontfamily{#3}\fontseries{#4}\fontshape{#5}%
  \selectfont}%
\fi\endgroup%
\begin{picture}(1200,1200)(1201,-1561)
\end{picture}%

\hspace{3mm}\\
\hspace{8mm}$\Omega_1(10^{-4}\mathbf{i},10^{-4} \mathbf{i})$\hspace{16mm} &\hspace{16mm}
$\Omega_2(10^{-4}\mathbf{i},10^{-4} \mathbf{i})$
\hspace{3mm}\\
\end{tabular}

\vspace{4mm}

It is "the orbit" of a point in the linearization domain of $(0:0:1);$ we note that the closure of an orbit is a torus. 
\vspace{8mm}

\begin{tabular}{ll}
\hspace{6mm}\begin{picture}(0,0)%
\includegraphics{des3scan.pstex}%
\end{picture}%
\setlength{\unitlength}{3947sp}%
\begingroup\makeatletter\ifx\SetFigFont\undefined%
\gdef\SetFigFont#1#2#3#4#5{%
  \reset@font\fontsize{#1}{#2pt}%
  \fontfamily{#3}\fontseries{#4}\fontshape{#5}%
  \selectfont}%
\fi\endgroup%
\begin{picture}(1200,1200)(1201,-1561)
\end{picture}%
\hspace{6mm} & \hspace{12mm}
\begin{picture}(0,0)%
\includegraphics{des4scan.pstex}%
\end{picture}%
\setlength{\unitlength}{3947sp}%
\begingroup\makeatletter\ifx\SetFigFont\undefined%
\gdef\SetFigFont#1#2#3#4#5{%
  \reset@font\fontsize{#1}{#2pt}%
  \fontfamily{#3}\fontseries{#4}\fontshape{#5}%
  \selectfont}%
\fi\endgroup%
\begin{picture}(1200,1200)(1201,-1561)
\end{picture}%
\\
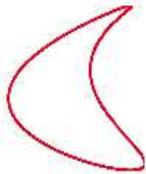
\hspace{-4mm}$\Omega_1(10000+10^{-4}\mathbf{i},10000+10^{-4}
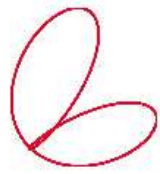
\mathbf{i})$\hspace{2mm} & \hspace{2mm}$ \Omega_2(10000+10^{-4}\mathbf{i},10000+10^{-4}
\mathbf{i})$
\end{tabular}

\vspace{4mm}

It is \og the orbit\fg\, under $f_{\alpha,\beta}^2$ of a point in the linearization domain of $(0:1:0);$ the closure of an \og orbit\fg\, is a 
topological circle. The singularities are artifacts of projection. 

\begin{rem}
The line $z=0$ is contracted by $f_{\alpha,\beta}$ on $(0:1:0)$ which is blow up on~$z=0:$ the map~$f_{\alpha,\beta}$ is not algebraically stable 
$($\emph{see} Chapter \ref{Chap:as}$)$ that's why we consider  $f_{\alpha,\beta}^2$ instead of~$f_{\alpha,\beta}.$
\end{rem}

The theory does not explain what happens outside the linearization domains. Between   $\mathcal{V}_1$ and $\mathcal{V}_2$ the experiences suggest
 a chaotic dynamic as we can see below. 

\bigskip

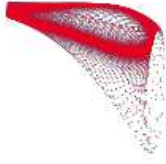
\begin{tabular}{cc}
\hspace{-3mm}\begin{picture}(0,0)%
\includegraphics{des9scan.pstex}%
\end{picture}%
\setlength{\unitlength}{3947sp}%
\begingroup\makeatletter\ifx\SetFigFont\undefined%
\gdef\SetFigFont#1#2#3#4#5{%
  \reset@font\fontsize{#1}{#2pt}%
  \fontfamily{#3}\fontseries{#4}\fontshape{#5}%
  \selectfont}%
\fi\endgroup%
\begin{picture}(1200,1200)(1201,-1561)
\end{picture}%
\hspace{7mm} & \hspace{7mm}
\begin{picture}(0,0)%
\includegraphics{des10scan.pstex}%
\end{picture}%
\setlength{\unitlength}{3947sp}%
\begingroup\makeatletter\ifx\SetFigFont\undefined%
\gdef\SetFigFont#1#2#3#4#5{%
  \reset@font\fontsize{#1}{#2pt}%
  \fontfamily{#3}\fontseries{#4}\fontshape{#5}%
  \selectfont}%
\fi\endgroup%
\begin{picture}(1200,1200)(1201,-1561)
\end{picture}%
\\
\hspace{-1mm}$\Omega_1(0.4+10^{-4}\mathbf{i},0.4+10^{-4} \mathbf{i})$\hspace{7mm} & \hspace{7mm}
$\Omega_2(0.4+10^{-4}\mathbf{i},0.4+10^{-4} \mathbf{i})$
\end{tabular}

\vspace{4mm}

We note a deformation of the invariant tori.

\vspace{10mm}

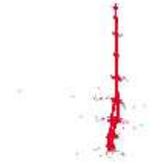
\begin{tabular}{cc}
 \begin{picture}(0,0)%
\includegraphics{des5scan.pstex}%
\end{picture}%
\setlength{\unitlength}{3947sp}%
\begingroup\makeatletter\ifx\SetFigFont\undefined%
\gdef\SetFigFont#1#2#3#4#5{%
  \reset@font\fontsize{#1}{#2pt}%
  \fontfamily{#3}\fontseries{#4}\fontshape{#5}%
  \selectfont}%
\fi\endgroup%
\begin{picture}(1200,1200)(1201,-1561)
\end{picture}%
 \hspace{4mm} &
\hspace{4mm}\begin{picture}(0,0)%
\includegraphics{des6scan.pstex}%
\end{picture}%
\setlength{\unitlength}{3947sp}%
\begingroup\makeatletter\ifx\SetFigFont\undefined%
\gdef\SetFigFont#1#2#3#4#5{%
  \reset@font\fontsize{#1}{#2pt}%
  \fontfamily{#3}\fontseries{#4}\fontshape{#5}%
  \selectfont}%
\fi\endgroup%
\begin{picture}(1200,1200)(1201,-1561)
\end{picture}%
\\
& \\
$\Omega_1(0.9+10^{-4}\mathbf{i},0.9 +10^{-4}
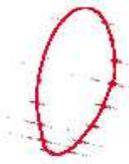
\mathbf{i})$\hspace{4mm} &
\hspace{4mm}$\Omega_2(0.9+10^{-4}\mathbf{i},0.9 +10^{-4}
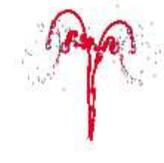
\mathbf{i})$\\
& \\
\begin{picture}(0,0)%
\includegraphics{des7scan.pstex}%
\end{picture}%
\setlength{\unitlength}{3947sp}%
\begingroup\makeatletter\ifx\SetFigFont\undefined%
\gdef\SetFigFont#1#2#3#4#5{%
  \reset@font\fontsize{#1}{#2pt}%
  \fontfamily{#3}\fontseries{#4}\fontshape{#5}%
  \selectfont}%
\fi\endgroup%
\begin{picture}(1200,1200)(1201,-1561)
\end{picture}%
 \hspace{4mm} &
\hspace{4mm} \begin{picture}(0,0)%
\includegraphics{des8scan.pstex}%
\end{picture}%
\setlength{\unitlength}{3947sp}%
\begingroup\makeatletter\ifx\SetFigFont\undefined%
\gdef\SetFigFont#1#2#3#4#5{%
  \reset@font\fontsize{#1}{#2pt}%
  \fontfamily{#3}\fontseries{#4}\fontshape{#5}%
  \selectfont}%
\fi\endgroup%
\begin{picture}(1200,1200)(1201,-1561)
\end{picture}%
\\
& \\
$\Omega_1(1+10^{-4}\mathbf{i},1 +10^{-4}\mathbf{i})$\hspace{4mm} &
\hspace{4mm}$\Omega_2(1+10^{-4}\mathbf{i},1 +10^{-4}\mathbf{i})$\\
\end{tabular}

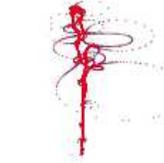
\begin{tabular}{cc}
\begin{picture}(0,0)%
\includegraphics{etrange6scan.pstex}%
\end{picture}%
\setlength{\unitlength}{3947sp}%
\begingroup\makeatletter\ifx\SetFigFont\undefined%
\gdef\SetFigFont#1#2#3#4#5{%
  \reset@font\fontsize{#1}{#2pt}%
  \fontfamily{#3}\fontseries{#4}\fontshape{#5}%
  \selectfont}%
\fi\endgroup%
\begin{picture}(1200,1200)(1201,-1561)
\end{picture}%

\hspace{4mm} &\hspace{4mm} \begin{picture}(0,0)%
\includegraphics{f15scan.pstex}%
\end{picture}%
\setlength{\unitlength}{3947sp}%
\begingroup\makeatletter\ifx\SetFigFont\undefined%
\gdef\SetFigFont#1#2#3#4#5{%
  \reset@font\fontsize{#1}{#2pt}%
  \fontfamily{#3}\fontseries{#4}\fontshape{#5}%
  \selectfont}%
\fi\endgroup%
\begin{picture}(1200,1200)(1201,-1561)
\end{picture}%
\\
& \\
$\Omega_1(1.08+10^{-4}\mathbf{i},1.08+10^{-4}\mathbf{i})$\hspace{4mm} &
\hspace{4mm}$\Omega_2(1.08+10^{-4}\mathbf{i},1.08+10^{-4}\mathbf{i})$\\
\end{tabular}

\bigskip

The invariant tori finally disappear; nevertheless the pictures seem to organize themselves around a closed curve. 

\medskip

So if there is no equivalence between first dynamical degree strictly greater than $1$ and countable centraliser we have an implication; more 
precisely we have the following statement. 

\begin{thm}[\cite{Can3}]
Let $f$ be a birational map of the complex projective plane with first dynamical degree $\lambda(f)$ strictly greater than $1.$ If $\psi$ is an 
element of $\mathrm{Bir}(\mathbb{P}^2)$ which commutes with~$f,$ there exist two integers $m$ in $\mathbb{N}^*$ and $n$ in $\mathbb{Z}$ such that~$\psi^m=f^n.$
\end{thm}

\chapter{Classification and applications}\label{Chap:as}

\section{Notions of stability and dynamical degree}\label{Subsec:rap} 

Let $X$, $Y$ be two compact complex surfaces and let $f\,\colon\, X\dashrightarrow Y$ be a dominant meromorphic map. Let $\Gamma_f$ be the graph of $f$ and let $\pi_1\,\colon\, 
\Gamma_f\to~X$, $\pi_2\,\colon\,\Gamma_f\to Y$ be the natural projections. If~$\Gamma_f$ is a singular submanifold of~$X\times Y$, we consider a desingularization of $\Gamma_f$ 
without chan\-ging the notation. If $\beta$ is a diffe\-rential form of bidegree $(1,1)$ on $Y$, then $\pi_2^*\beta$ determines a form of bidegree $(1,1)$ on $\Gamma_f$ which 
can be pushed forward as a current $f^*\beta:=\pi_{1_*}\pi_2^*\beta$ on~$X$ thanks to the first projection. Let us note that~$f^*$ induces an operator between $\mathrm{H}^{1,1}
(Y,\mathbb{R})$ and~$\mathrm{H}^{1,1}(X, \mathbb{R})$~: if $\beta$ and~$\gamma$ are homologous, then $f^*\beta$ and $f^*\gamma$ are homologous. In a similar way we can define 
the push-forward $f_*:=\pi_{2*}\pi_1^*\colon\mathrm{H}^{p,q}(X)\to\mathrm{H}^{p,q}(Y)$. Note that when~$f$ is bimeromorphic $f_*=(f^{-1})^*$.

  Assume that $X=Y$. The map $f$ is \textbf{\textit{algebraically stable}} if there exists no curve $V$ in $X$ such that~$f^k(V)$ belongs to $\mathrm{Ind}\,f$ for some integer 
$k \geq 0$. 

\begin{thmdef}[\cite{DiFa}]
 Let $f\colon\mathrm{S}\to\mathrm{S}$ be a dominating meromorphic map on a K\"{a}hler surface and let $\omega$ be a K\"{a}hler form. Then $f$ is \textbf{\textit{algebraically 
stable}}\label{Chap8:ind17} if and only if any of the following holds:
\begin{itemize}
\item[$\bullet$] for any $\alpha\in\mathrm{H}^{1,1}(\mathrm{S})$ and any $k$ in $\mathbb{N},$ we have $(f^*)^k\alpha=(f^k)^*\alpha;$

\item[$\bullet$] there is no curve $\mathcal{C}$ in $\mathrm{S}$ such that $f^k(\mathcal{C})\subset\mathrm{Ind}\, f$ for some integer $k\geq 0;$

\item[$\bullet$] for all $k\geq 0$ we have $(f^k)^*\omega=(f^*)^k\omega.$
\end{itemize}
\end{thmdef}

  In other words for an algebraically stable map the following does not happen

\begin{figure}[H]
\begin{center}
\input{as.pstex_t}
\end{center}
\end{figure}
{\it i.e.} the positive orbit\footnote{The positive orbit of $p_1$ under the action of $f$
is the set $\{f^n(p_1)\,\vert\,n\geq 0\}$.} of $p_1\in\mathrm{Ind}\,f^{-1}$ intersects $\mathrm{Ind}\,f$.

\begin{rem}
Let $f$ be a Cremona transformation. The map $f$ is not algebraically stable if and only if there exists an integer $k$ such that $$\deg f^k<(\deg f)^k.$$ So if $f$ is 
algebraically stable, then $\lambda(f)=\deg f.$
\end{rem}

\begin{egs}
\begin{itemize}
\item[$\bullet$] An automorphism of $\mathbb{P}^2(\mathbb{C})$ is algebraically stable.

\item[$\bullet$] The involution $\sigma\colon\mathbb{P}^2(\mathbb{C})\dashrightarrow\mathbb{P}^2(\mathbb{C}),$ $(x:y:z)\mapsto(yz:xz:xy)$ is not algebraically stable: 
$\mathrm{Ind}\,\sigma^{-1}=\mathrm{Ind}\,\sigma^{-1};$ moreover $\deg\sigma^2=1$ and $(\deg\sigma)^2=4.$
\end{itemize}
\end{egs}

\begin{egs}
Let $A$ be an automorphism of the complex projective plane and let $\sigma$ be the birational map given by 
\begin{align*}
\sigma\colon\mathbb{P}^2(\mathbb{C})\dashrightarrow 
\mathbb{P}^2(\mathbb{C}), &&(x:y:z)\dashrightarrow(yz:xz:xy).
\end{align*}
Assume that the coefficients of $A$ are positive real numbers. The map $A\sigma$ is algebraically stable $($\cite{CD}$)$.

  Let $A$ be an automorphism of the complex projective plane and let $\rho$ be the birational map given by 
\begin{align*}
&\rho\colon\mathbb{P}^2(\mathbb{C})\dashrightarrow \mathbb{P}^2(\mathbb{C}),
&&(x:y:z)\dashrightarrow(xy:z^2:yz).
\end{align*}
Assume that the coefficients of $A$ are positive real numbers. We can verify that $A\rho$ is algebraically stable. The same holds with
\begin{align*}
&\tau\colon\mathbb{P}^2(\mathbb{C})\dashrightarrow \mathbb{P}^2(\mathbb{C}), &&(x:y:z)\dashrightarrow(x^2:xy:y^2-xz).
\end{align*}

  Let us say that the coefficients of an automorphism $A$ of $\mathbb{P}^2(\mathbb{C})$ are algebraically independent if $A$ has a representative in $\mathrm{GL}_3
(\mathbb{C})$ whose coefficients are algebraically independent over~$\mathbb{Q}.$ We can deduce the following: let $A$ be an automorphism of the projective plane whose 
coefficients are algebraically independent over $\mathbb{Q},$ then $A\sigma$ and $(A\sigma)^{-1}$ are algebraically stable. 
\end{egs}
  
Diller and Favre prove the following statement.

\begin{thm}[\cite{DiFa}, theorem 0.1]
Let $\mathrm{S}$ be a rational surface and let $f\,\colon\, \mathrm{S}\dasharrow \mathrm{S}$ be a birational map. There exists a birational morphism $\varepsilon \,
\colon\,\widetilde{\mathrm{S}}\to \mathrm{S}$ such that $\varepsilon f \varepsilon^{-1}$ is algebraically stable. 
\end{thm}

\begin{proof}[Idea of the proof]
Let us assume that $f$ is not algebraically stable; hence there exists a curve $\mathcal{C}$ and an integer $k$ such that $\mathcal{C}$ is blown down
onto $p_1$ and~$p_k=~f^{k-1}(p_1)$ is an indeterminacy point of $f$.

The idea of Diller and Favre is the following: after blowing up the points~$p_i$ the image of~$\mathcal{C}$ is, for $i=1,$ $\ldots,$ $k,$ a curve. Doing this for any 
element of $\mathrm{Exc}\, f$ whose an iterate belongs to~$\mathrm{Ind}\,f$ 
we get the statement. 
\end{proof}

\begin{rem}\label{Rem:janli}
There is no similar result in higher dimension. Let us recall the following statement due to Lin $($\cite[Theorem 5.7]{Lin}$)$: suppose that $A=(a_{ij})\in\mathrm{M}_n
(\mathbb{Z})$ is an integer matrix with $\det\, A=1$. If $\lambda$ and $\overline{\lambda}$ are the only eigenvalues of $A$ of maximal modulus, also with algebraic 
multiplicity one, and if $\lambda=\vert\lambda\vert \mathrm{e}^{2\mathbf{i}\pi\vartheta}$ with $\vartheta\in\mathbb{Q}$; then there is no toric birational model which 
makes the corresponding monomial map 
\begin{align*}
&f_A \colon\mathbb{C}^n\to\mathbb{C}^n, && (x_1,\ldots,x_n)\mapsto\left(\prod_j x_j^{a_{1j}},\ldots,\prod_j x_j^{a_{nj}}\right)
\end{align*}
algebraically stable. A $3\times 3$ example is $($\cite{HaPr}$)$ $$A=\left[\begin{array}{ccc}-1 & 1 & 0 \\ -1 & 0 & 1\\ 1& 0& 0\end{array}\right];$$ in higher dimension 
$\left[\begin{array}{cc}A & 0 \\ 0 & \mathrm{Id}\end{array}\right]$ where $0$ is the zero matrix and $\mathrm{Id}$ is the identity matrix works.
\end{rem}

  The \textbf{\textit{first dynamical degree}}\label{Chap8:ind17a} of $f$ is defined by $$\lambda(f)=\limsup_{n\to+\infty} \vert(f^n)^*\vert^{1/n}$$ where $\vert\, .\,\vert$ denotes a 
norm on $\mathrm{End}( \mathrm{H}^{1,1}(X,\mathbb{R}))~;$ this number is greater or equal to $1$ (\emph{see} \cite{[RS], [Fr]}). Let us remark that for all birational 
maps $f$ we have the inequality $$\lambda(f)^n\leq~\deg f^n$$ where $\deg f$ is the algebraic degree of $f$ (the algebraic degree of $f=(f_0:f_1:f_2)$ is the degree 
of the homogeneous polynomials $f_i$).

\begin{egs}
\begin{itemize}
\item[$\bullet$] The first dynamical degree of a birational map of the complex projective plane of finite order is equal to $1$.

\item[$\bullet$] The first dynamical degree of an automorphism of $\mathbb{P}^2(\mathbb{C})$ is equal to $1$.

\item[$\bullet$] The first dynamical degree of an elementary automorphism $($resp. a de Jonqui\`eres map$)$  is equal to $1$.

\item[$\bullet$] The first dynamical degree of a H\'enon automorphism of degree $d$ is equal to $d$.

\item[$\bullet$] The first dynamical degree of the monomial map 
\begin{align*}
& f_B\colon(x,y)\mapsto(x^ay^b,x^cy^d)
\end{align*}
is the largest eigenvalue of $B=\left[\begin{array}{cc} a & b\\ c & d\end{array}\right]$.

\item[$\bullet$] Let us set $E=\mathbb{C}/\mathbb{Z}[\mathbf{i}]$, $Y=E\times E=\mathbb{C}^2/\mathbb{Z}[\mathbf{i}]\times
\mathbb{Z}[\mathbf{i}]$ and $B=\left[\begin{array}{cc} a & b\\ c & d\end{array}\right]$. The matrix $B$ acts linearly
on $\mathbb{C}^2$ and preserves $\mathbb{Z}[\mathbf{i}]\times\mathbb{Z}[\mathbf{i}]$ so $B$ induces a map $G_B\colon
E\times E\to E\times E$. The surface $E\times E$ is not rational whereas $X=Y/(x,y)\sim(\mathbf{i}x,\mathbf{i}y)$ is.
The matrix $B$ induces a map $G_B\colon E\times E\to E\times E$ that commutes with $(\mathbf{i}x,\mathbf{i}y)$ so
$G_B$ induces a map $g_B\colon X\to X$ birationally conjugate to an element of $\mathrm{Bir}(\mathbb{P}^2)$. 
The first dynamical degree of $g_B$ is equal to the square of the largest eigenvalue of $B$.
\end{itemize}
\end{egs}

Let us give some properties about the first dynamical degree. Let us recall that a \textbf{\textit{Pisot number}}\label{Chap8:ind17c}  
is a positive algebraic integer greater than~$1$ all of whose conjugate elements have absolute value less than $1$.
A real algebraic integer is a \textbf{\textit{Salem number}}\label{Chap8:ind17d}  if all its conjugate roots have absolute value no greater 
than $1$, and at least one has absolute value exactly~$1$.

\begin{thm}[\cite{DiFa}]
The set $$\big\{\lambda(f)\,\vert\, f\in\mathrm{Bir}(\mathbb{P}^2)\big\}$$ is contained in $\{1\}\cup\mathcal{P}\cup\mathcal{S}$ 
where $\mathcal{P}$ $($resp. $\mathcal{S})$ denotes the set of Pisot $($resp. Salem$)$ numbers. 

In particular it is a subset of algebraic numbers.
\end{thm}

\section{Classification of birational maps}\label{Sec:twist}

\begin{thm}[\cite{[Gi], DiFa, BlDe}]\label{dillerfavre}
Let $f$ be an element of $\mathrm{Bir}(\mathbb{P}^2)$; up to birational conjugation, exactly one of the following holds. 
\begin{itemize}
\item[$\bullet$] The sequence $\vert(f^n)^*\vert$ is bounded, the map $f$ is conjugate either to $(\alpha x:\beta y:z)$ or to $(\alpha x:y+z:z)$;

\item[$\bullet$] the sequence $\vert(f^n)^*\vert$ grows linearly, and $f$ preserves a rational fibration. In this case $f$ cannot be conjugate to an automorphism of a 
projective surface;

\item[$\bullet$] the sequence $\vert(f^n)^*\vert$ grows quadratically, and $f$ is conjugate to an automorphism preserving an elliptic fibration.

\item[$\bullet$] the sequence $\vert(f^n)^*\vert$ grows exponentially; the spectrum of $f^*$ outside the unit disk consists of the single simple eigenvalue $\lambda(f)$, 
the eigenspace associated to $\lambda(f)$ is generated by a nef class $\theta_+\in\mathrm{H}^{1,1}(\mathbb{P}^2(\mathbb{C}))$. Moreover $f$ is conjugate to an automorphism 
if and only if~$(\theta_+,\theta_+)=0$. 
\end{itemize}
\medskip

In the second and third cases, the invariant fibration is unique. 
\end{thm}

\begin{defi}
Let $f$ be an element of $\mathrm{Bir}(\mathbb{P}^2)$. 
\begin{itemize}
\item[$\bullet$] If $\left\{\deg f^k\right\}_{k\in \mathbb{N}}$ is bounded, $f$ is \textbf{\textit{elliptic}};\label{Chap2:ind29a}

\item[$\bullet$] if $\left\{\deg f^k\right\}_{k\in \mathbb{N}}$ grows linearly $($resp. quadratically$)$, then $f$ is a \textbf{\textit{de Jonqui\`eres twist}} $($resp. an 
\textbf{\textit{Halphen twist}}$);$\label{Chap2:ind29aa}

\item[$\bullet$]  if $\left\{\deg f^k\right\}_{k\in \mathbb{N}}$ grows exponentially, $f$ is \textbf{\textit{hyperbolic}}\label{Chap2:ind29aaa}.
\end{itemize}
\end{defi}

\begin{rem}
 If $\left\{\deg f^k\right\}_{k\in \mathbb{N}}$ grows linearly $($resp. quadratically$)$ then $f$ preserves a pencil of rational curves $($resp. elliptic curves$)$;
up to birational conjugacy $f$ preserves a pencil of lines, {\it i.e.} belongs to the de Jonqui\`eres group $($resp. preserves an Halphen pencil, \it{i.e.} a pencil 
of $($elliptic$)$ curves of degree $3n$ passing through $9$ points with multiplicity $n)$.
\end{rem}

\section{Picard-Manin space}

Manin describes in \cite[Chapter 5]{Manin} the inductive limit of the Picard group of any surface obtained by blowing up any point of a surface $\mathrm{S}$. Then he 
shows that the group $\mathrm{Bir}(\mathrm{S})$ linearly acts on this limit group.

$\bullet$ Let $\mathrm{S}$ be a K\"{a}hler  compact complex surface. Let $\mathrm{Pic}(\mathrm{S})$ be the Picard group of $\mathrm{S}$ and let $\mathrm{NS}(\mathrm{S})$ 
be its N\'eron-Severi group\footnote{The N\'eron-Severi group of a variety is the group of divisors modulo algebraic equivalence.}. Let us consider the morphism from 
$\mathrm{Pic}(\mathrm{S})$  into $\mathrm{NS}(\mathrm{S})$ which associates to any line bundle~$L$ its Chern class $c_1(L)$; its kernel is denoted by $\mathrm{Pic}^0(\mathrm{S})$. 
The dimension of~$\mathrm{NS}(\mathrm{R})\otimes~\mathbb{R}$ is called the \textbf{\textit{Picard number}}\label{Chap2:ind31} of $\mathrm{S}$ and is denoted by $\rho(\mathrm{S})$. 
There is an intersection form on the Picard group, there is also one on the N\'eron-Severi group; when $\mathrm{S}$ is projective, its signature is $(1,\rho(\mathrm{S})-1)$. 
The nef cone is denoted by $\mathrm{NS}^+(\mathrm{S})$ or $\mathrm{Pic}^+(\mathrm{S})$ when $\mathrm{NS}(\mathrm{S})=\mathrm{Pic}(\mathrm{S})$. Let $\mathrm{S}$ and~$\mathrm{S}'$ 
be two surfaces and let $\pi\colon\mathrm{S}\to\mathrm{S}'$ be a birational morphism. The morphism $\pi^*$ is injective and preserves the nef cone: $\pi^*(\mathrm{NS}^+
(\mathrm{S}'))\subset\mathrm{NS}^+(\mathrm{S})$. Moreover for any~$\ell$, $\ell'$ in $\mathrm{Pic}(\mathrm{S})$, we have $(\pi^*\ell,\pi^*\ell')=(\ell,\ell')$.

$\bullet$ Let $\mathrm{S}$ be a K\"{a}hler compact complex surface. Let $\mathcal{B}(\mathrm{S})$ be the category which objects are the birational morphisms $\pi'\colon\mathrm{S}'
\to\mathrm{S}$. A morphism between two objects $\pi_1\colon\mathrm{S}'_1\to\mathrm{S}$ and $\pi_2\colon\mathrm{S}'_2\to~\mathrm{S}$ of this category is a birational morphism 
$\varepsilon\colon\mathrm{S}'_1\to\mathrm{S}'_2$ such that $\pi_2\varepsilon=\pi_1$. In particular the set of morphisms between two objects in either empty, or reduced to a unique 
element. 

This set of objects is ordered as follows: $\pi_1\geq\pi_2$ if and only if there exists a morphism from~$\pi_1$ to $\pi_2$; we thus say that $\pi_1$ (resp. $\mathrm{S}'_1$) 
dominates~$\pi_2$ (resp. $\mathrm{S}'_2$). Geometrically this means that the set of base-points of $\pi_1^{-1}$ contains the set of base-points of $\pi_2^{-1}$. If $\pi_1$ and 
$\pi_2$ are two objects of~$\mathcal{B}(\mathrm{S})$ there always exists another one which simultaneously dominates $\pi_1$ and~$\pi_2$. Let us set $$\mathcal{Z}(\mathrm{S})=
\displaystyle\lim_{\rightarrow}\mathrm{NS}(\mathrm{S}')$$ the inductive limit is taken following the injective morphism $\pi^*$.

The group $\mathcal{Z}(\mathrm{S})$ is called \textbf{\textit{Picard-Manin space}}\label{Chap2:ind33} space of $\mathrm{S}$. The inva\-riant structures of $\pi^*$ induce invariant 
structures for $\mathcal{Z}(\mathrm{S})$:
\begin{itemize}
\item[$\bullet$] an intersection form $(,)\colon\mathcal{Z}(\mathrm{S})\times\mathcal{Z}(\mathrm{S})\to\mathbb{Z}$;

\item[$\bullet$] a nef cone $\mathcal{Z}^+(\mathrm{S})=\displaystyle\lim_{\rightarrow}\mathrm{NS}^+(\mathrm{S})$;

\item[$\bullet$] a canonical class, viewed as a linear form $\Omega\colon\mathcal{Z}(\mathrm{S})\to\mathbb{Z}$.
\end{itemize} 

Note that $\mathrm{NS}(\mathrm{S}')$ embeds into $\mathcal{Z}(\mathrm{S})$ so we can identify $\mathrm{NS}(\mathrm{S}')$ and its image in $\mathcal{Z}(\mathrm{S})$. 

Let us now describe the action of birational maps of $\mathrm{S}$ on $\mathcal{Z}(\mathrm{S})$. Let~$\mathrm{S}_1$ and $\mathrm{S}_2$ be two surfaces and let $f$ be a birational 
map from $\mathrm{S}_1$ to $\mathrm{S}_2$. Accor\-ding to Zariski Theorem we can remove the indeterminacy of $f$ thanks to two birational morphisms $\pi_1\colon\mathrm{S}'\to
\mathrm{S}_1$ and $\pi_2\colon\mathrm{S}'\to\mathrm{S}_2$ such that $f=\pi_2\pi_1^{-1}$. The map $\pi_1$ (resp. $\pi_2$) embeds $\mathcal{B}(\mathrm{S}')$ into $\mathcal{B}
(\mathrm{S}_1)$ (resp.~$\mathcal{B}(\mathrm{S}_2)$). Since any object of $\mathcal{B}(\mathrm{S}_1)$ (resp. $\mathcal{B}(\mathrm{S}_2)$) is dominated by an object of~$\pi_{1*}
(\mathcal{B}(\mathrm{S}))$ (resp. $\pi_{2*}(\mathcal{B}(\mathrm{S}))$) we get two isomorphisms
\begin{align*}
& \pi_{1*}\colon\mathcal{Z}(\mathrm{S}')\to\mathcal{Z}(\mathrm{S}_1), &&\pi_{2*}\colon\mathcal{Z}(\mathrm{S}')\to\mathcal{Z}(\mathrm{S}_2).
\end{align*}
Then we set $f_*=\pi_{2*}\pi_{1*}^{-1}$.

\begin{thm}[\cite{Manin}, page 192]
The map $f\mapsto f_*$ induces an injective morphism from~$\mathrm{Bir}(\mathrm{S})$ into $\mathrm{GL}(\mathcal{Z}(\mathrm{S}))$. 

If $f$ belongs to $\mathrm{Bir}(\mathrm{S})$, the linear map $f_*$ preserves the intersection form and the nef cone.
\end{thm}

Let us denote by $\mathrm{Eclat}(\mathrm{S})$ the union of the surfaces endowed with a birational morphism $\pi\colon\mathrm{S}'\to~\mathrm{S}$ modulo the following equivalence 
relation: $\mathrm{S}\ni p_1\sim p_2\in\mathrm{S}$ if and only if $\varepsilon_2^{-1}\varepsilon_1$ sends $p_1$ onto~$p_2$ and is a local isomorphism between a neighborhood 
of $p_1$ and a neighborhood of $p_2$. A point of~$\mathrm{Eclat}(\mathrm{S})$ corresponds either to a point of $\mathrm{S}$, or to a point on an exceptional divisor of a 
blow-up of $\mathrm{S}$ etc. Any surface $\mathrm{S}'$ which dominates $\mathrm{S}$ embeds into $\mathrm{Eclat}(\mathrm{S})$. Let us consider the free abelian group 
$\mathrm{Ec}(\mathrm{S})$ generated by the points of $\mathrm{Eclat}(\mathrm{S})$; we have a scalar product on~$\mathrm{Ec}(\mathrm{S})$
\begin{align*}
& (p,p)_E=-1, && (p,q)=0 \text{ if } p\not=q.
\end{align*}
The group $\mathrm{Ec}(\mathrm{S})$ can be embedded in $\mathcal{Z}(\mathrm{S})$ (\emph{see} \cite{Can3}). If $p$ is a point of~$\mathrm{Eclat}(\mathrm{S})$ let us denote by 
$e_p$ the point of $\mathcal{Z}(\mathrm{S})$ associated to $p$, {\it i.e.}~$e_p$ is the class of the exceptional divisor obtained by blowing up $p$.  This determines the image 
of the basis of $\mathrm{Ec}(\mathrm{S})$ in $\mathcal{Z}(\mathrm{S})$ so we have the morphism defined by
\begin{align*}
& \mathrm{Ec}(\mathrm{S})\to\mathcal{Z}(\mathrm{S}), && \sum a(p)p\mapsto \sum a(p)e_p.
\end{align*}
Using this morphism and the canonical embedding from $\mathrm{NS}(\mathrm{S})$ into $\mathcal{Z}(\mathrm{S})$ we can consider the morphism $$\mathrm{NS}(\mathrm{S})\times 
\mathrm{Ec}(\mathrm{S})\to \mathcal{Z}(\mathrm{S}).$$

\begin{pro}[\cite{Manin}, p.197]
The morphism $\mathrm{NS}(\mathrm{S})\times \mathrm{Ec}(\mathrm{S})\to \mathcal{Z}(\mathrm{S})$ induces an isome\-try between $(\mathrm{NS}(\mathrm{S}),(\cdot,\cdot))\oplus
(\mathrm{Ec}(\mathrm{S}),(\cdot,\cdot)_E)$ and $(\mathcal{Z}(\mathrm{S}),(\cdot,\cdot))$.
\end{pro}

\begin{eg}
Let us consider a point $p$ of $\mathbb{P}^2(\mathbb{C})$, $\mathrm{Bl}_p\mathbb{P}^2$ the blow-up of $p$ and let us denote by~$E_p$ the exceptional divisor. Let us now consider 
$q\in\mathrm{Bl}_p\mathbb{P}^2$ and as previously we define~$\mathrm{Bl}_{p,q}\mathbb{P}^2$ and $E_q$. The elements $e_p$ and $e_q$ belong to the image of $\mathrm{NS}
(\mathrm{Bl}_{p,q}\mathbb{P}^2)$ in $\mathcal{Z}(\mathbb{P}^2)$. If $\widetilde{E_p}$ is the strict transform of $E_p$ in $\mathrm{Bl}_{p,q}\mathbb{P}^2$ the element $e_p$ 
$($resp. $e_q)$ corresponds to $\widetilde{E_p}+E_q$ $($resp. $E_q)$. We can check that $(e_p,e_q)=0$ and $(e_p,e_p)=1$. 
\end{eg}

$\bullet$ The completed Picard-Manin space $\overline{\mathcal{Z}}(\mathrm{S})$ of $\mathrm{S}$ is the $L^2$-completion of~$\mathcal{Z}(\mathrm{S})$; in other words 
$$\overline{\mathcal{Z}}(\mathrm{S})=\big\{[D]+\sum a_p[E_p]\,\big\vert\,[D]\in\mathrm{NS}(\mathrm{S}),\,a_p\in\mathbb{R},\,\sum a_p^2<\infty\big\}.$$ Note that $\mathcal{Z}
(\mathrm{S})$ corresponds to the case where the $a_p$ vanishes for all but a finite number of~$p\in~\mathrm{Eclat}(\mathrm{S})$.

\begin{eg}
For $\mathrm{S}=\mathbb{P}^2(\mathbb{C})$ the N\'eron-Severi group $\mathrm{NS}(\mathrm{S})$ is isomorphic to $\mathbb{Z}[H]$ where~$H$ is a line. Thus the elements of 
$\overline{\mathcal{Z}}(\mathrm{S})$ are given by 
\begin{align*}
&a_0[H]+\sum_{p\in\mathrm{Eclat}(\mathrm{S})}a_p[E_p], && \text{with }\sum a_p^2<\infty.
\end{align*}
\end{eg}

The group $\mathrm{Bir}(\mathrm{S})$ acts on $\overline{\mathcal{Z}}(\mathrm{S})$; let us give details when $\mathrm{S}=\mathbb{P}^2(\mathbb{C})$. Let $f$ be a birational map 
from $\mathbb{P}^2(\mathbb{C})$ into itself. According to Zariski Theorem there exist two morphisms $\pi_1,\,\pi_2\colon\mathrm{S}\to~\mathbb{P}^2(\mathbb{C})$ such that 
$f=\pi_2\pi_1^{-1}$. Defining $f^*$ by $f^*=(\pi_1^*)^{-1}\pi_2^*$ and $f_*$ by $f_*=(f^*)^{-1}$ we get the representation $f\mapsto f_*$ of the Cremona group in the 
orthogonal group of $\mathcal{Z}(\mathbb{P}^2)$ (resp. $\overline{\mathcal{Z}}(\mathbb{P}^2)$) with respect to the intersection form. Since for any $p$ in $\mathbb{P}^2
(\mathbb{C})$ such that $f$ is defined at $p$ we have $f_*(e_p)=e_{f(p)}$ this representation is faithful; it also preserves the integral structure of $\mathcal{Z}
(\mathbb{P}^2)$ and the nef cone.

$\bullet$ Only one of the two sheets of the hyperboloid $\big\{[D]\in\mathcal{Z}(\mathbb{P}^2)\,\big\vert\,[D]^2=1\big\}$ intersects the nef cone $\mathcal{Z}(\mathbb{P}^2)$; 
let us denote it by $\mathbb{H}_{\overline{\mathcal{Z}}}$. In other words $$\mathbb{H}_{\overline{\mathcal{Z}}}=\big\{[D] \in \overline{\mathcal{Z}}(\mathbb{P}^2) \, \big\vert 
\,[D]^2=1, \,[H] \cdot [D] >0\big\}.$$ We can define a distance on $\mathbb{H}_{\overline{\mathcal{Z}}}$: $$\cosh(\mathrm{dist}([D_1],[D_2]))=[D_1] \cdot[D_2].$$ The space 
$\mathbb{H}_{\overline{Z}}$ is a model of the "hyperbolic space of infinite dimension"; its isometry group is denoted by $\mathrm{Isom}(\mathbb{H}_{\overline{\mathcal{Z}}})$ 
(\emph{see} \cite{Gr3}, \S 6). As the action of~$\mathrm{Bir}(\mathbb{P}^2)$ on $\mathcal{Z}(\mathbb{P}^2)$ preserves the two-sheeted hyperboloid and as the action also 
preserves the nef cone we get a faithful representation of $\mathrm{Bir}(\mathbb{P}^2)$ into $\mathrm{Isom}(\mathbb{H}_{\overline{\mathcal{Z}}})$. In the context of the 
Cremona group we will see that the classification of isometries into three types has an algebraic-geometric meaning.

$\bullet$ As $\mathbb{H}_{\overline{\mathcal{Z}}}$ is a complete $\mathrm{cat}(-1)$ metric space, its isometries are either elliptic, or parabolic, or hyperbolic (\emph{see} 
\cite{DG}). In the case of hyperbolic case we can characterize these isometries as follows:
\begin{itemize}
\item elliptic isometry: there exists an element $\ell$ in $\mathcal{Z}(\mathrm{S})$ such that $f^*(\ell)=\ell$ and $(\ell,\ell)>0$ then~$f_*$ is a rotation around $\ell$ and 
the orbit of any $p$ in~$\mathcal{Z}(\mathrm{S})$ (resp. any $p$ in $\mathbb{H}_{\overline{\mathcal{Z}}}$) is bounded;

\item parabolic isometry: there exists a non zero element $\ell$ in $\mathcal{Z}^+(\mathrm{S})$ such that $f_*(\ell)=\ell$. Moreover $(\ell,\ell)=0$ and $\mathbb{R}\ell$ is 
the unique inva\-riant line by $f_*$ which intersects $\mathcal{Z}^+(\mathrm{S})$. If $p$ belongs to $\mathcal{Z}^+(\mathrm{S})$, then $\displaystyle\lim_{n\to\infty}f_*^n(
\mathbb{R}p)=\mathbb{R}\ell$.

\item hyperbolic isometry: there exists a real number $\lambda>1$ and two ele\-ments~$\ell_+$ and $\ell_-$ in~$\mathcal{Z}(\mathrm{S})$ such that $f_*(\ell_+)=\lambda\ell_+$ and 
$f_*(\ell_-)=(1/\lambda)\ell_-$. If $p$ is a point of $\mathcal{Z}^+(\mathrm{S})\setminus\mathbb{R}\ell_-$, then $$\displaystyle\lim_{n\to\infty}\left(\frac{1}{\lambda}\right)^n
f_*^n(p)=v\in\mathbb{R}\ell_+\setminus\{0\},$$ We have a similar property for $\ell_-$ and $f^{-1}$.
\end{itemize}

This classification and Diller-Favre classification (Theorem \ref{dillerfavre}) are related by the follo\-wing statement.

\begin{thm}[\cite{Can3}]\label{Thm:cafa}
Let $f$ be a birational map of a compact complex surface $\mathrm{S}$. Let $f_*$ be the action induced by $f$ on $\mathcal{Z}(\mathrm{S})$.
\begin{itemize}
\item[$\bullet$] $f_*$ is elliptic if and only if $f$ is an elliptic map: there exists an element~$\ell$ in $\mathcal{Z}^+(\mathrm{S})$ such that $f(\ell)=\ell$ and $(\ell,\ell)>0$, then 
$f_*$ is a rotation around~$\ell$ and the orbit of any $p$ in $\mathcal{Z}(\mathrm{S})$ $($resp. any $p$ in $\mathbb{H}_{\overline{\mathcal{Z}}})$ is bounded.

\item[$\bullet$] $f_*$ is parabolic if and only if $f$ is a parabolic map: there exists a non zero $\ell$ in  $\mathcal{Z}^*(\mathrm{S})$ such that $f(\ell)=\ell$. Moreover $(\ell,\ell)=0$ 
and $\mathbb{R}\ell$ is the unique invariant line by $f_*$ which intersects $\mathcal{Z}^+(\mathrm{S})$. If $p$ belongs to~$\mathcal{Z}^*(\mathrm{S})$, then 
$\displaystyle\lim_{n\to+\infty}(f_*)^n(\mathbb{R}p)= \mathbb{R}\ell$.

\item[$\bullet$] $f_*$ is hyperbolic if and only if $f$ is a hyperbolic map: there exists a real number $\lambda>1$ and two elements $\ell_+$ and $\ell_-$ in $\mathcal{Z}(\mathrm{S})$ such 
that $f_*(\ell_+)=\lambda\ell_+$ and $f_*(\ell_-)=(1/\lambda)\ell_-$. If $p$ belongs to $\mathcal{Z}^+\setminus \mathbb{R}\ell_-$ then $$\displaystyle \lim_{n\to+\infty}
\left(\frac{1}{\lambda}\right)^nf_*^n(p)=v\in\mathbb{R}\ell_+\setminus\{0\};$$ there is a similar property for $\ell_-$ and $f^{-1}$.
\end{itemize}
\end{thm}

\section{Applications}

\subsection{Tits alternative}\,

Linear groups satisfy Tits alternative. 

\begin{thm}[\cite{Ti}]\label{tits}
Let $\Bbbk$ be a field of characteristic zero. Let $\Gamma$ be a finitely generated subgroup of $\mathrm{GL}_n(\Bbbk).$ Then
\begin{itemize}
\item[$\bullet$] either $\Gamma$ contains a non abelian, free group; 

\item[$\bullet$] or $\Gamma$ contains a solvable\footnote{Let
$\mathrm{G}$ be a group; let us set $\mathrm{G}^{(0)}=\mathrm{G}$
et $\mathrm{G}^{(k)}=[\mathrm{G}^{(k-1)},\mathrm{G}^{(k-1)}]=\langle 
aba^{-1}b^{-1}\,\vert\, a,\, b\in \mathrm{G}^{(k-1)}\rangle$ for $k\geq 1.$ The group $\mathrm{G}$ is solvable if there exists an integer $k$ such that $\mathrm{G}^{(k)}=
\{\mathrm{id}\}.$} subgroup of finite index.
\end{itemize}
\end{thm}

Let us mention that the group of diffeomorphisms of a real manifold of dimension $\geq 1$ does not satisfy Tits alternative (\emph{see} \cite{Gh3} and references therein). 
Nevertheless the group of polynomial automorphisms of $\mathbb{C}^2$ satisfies Tits alternative~(\cite{La}); Lamy obtains this property from the classification of subgroups  
of $\mathrm{Aut}(\mathbb{C}^2),$ classification established by using the action of this group on $\mathcal{T}$:

\begin{thm}[\cite{La}]\label{autsteph}
Let $\mathrm{G}$ be a subgroup of $\mathrm{Aut}(\mathbb{C}^2).$ Exactly one of the followings holds: 
\begin{itemize}
\item[$\bullet$] any element of $\mathrm{G}$ is conjugate to an element of $\mathtt{E}$, then 
\begin{itemize}
\item either $\mathrm{G}$ is conjugate to a subgroup of $\mathtt{E}$;

\item or $\mathrm{G}$ is conjugate to a subgroup of $\mathtt{A}$;

\item or $\mathrm{G}$ is abelian, $\mathrm{G}=\bigcup_{i\in\mathbb{N}}\mathrm{G}_i$ with $\mathrm{G}_i\subset\mathrm{G}_{i+1}$ and
any $\mathrm{G}_i$ is conjugate to a finite cyclic group of the form $\langle(\alpha x,\beta y)\rangle$ with $\alpha,$ $\beta$
roots of unicity of the same order. Any element of $\mathrm{G}$ has a unique fixe point\footnote{as polynomial automorphism of $\mathbb{C}^2$} and this fixe point is the same for 
any element of $\mathrm{G}.$ 
Finally the action of $\mathrm{G}$ fixes a piece of the tree $\mathcal{T}.$
\end{itemize}

\item[$\bullet$] $\mathrm{G}$ contains H\'enon automorphisms, all having the same geodesic, in this case $\mathrm{G}$ is solvable and contains a subgroup of finite index isomorphic 
to $\mathbb{Z}.$

\item[$\bullet$] $\mathrm{G}$ contains two H\'enon automorphisms with distinct geodesics, $\mathrm{G}$ thus contains a free subgroup on two generators.
\end{itemize}
\end{thm}

One of the common ingredients of the proofs of Theorems
\ref{tits}, \ref{autsteph} \ref{clastypefini} is the following statement, a criterion used by Klein to construct free pro\-ducts. 

\begin{lem}\label{pingpong}
Let $\mathrm{G}$ be a group acting on a set $\mathrm{X}.$ Let us consider $\Gamma_1$ and~$\Gamma_2$ two subgroups of $\mathrm{G}$, 
and set $\Gamma=\langle\Gamma_1,\Gamma_2\rangle.$ Assume that
\begin{itemize}
\item[$\bullet$] $\Gamma_1$ $($resp. $\Gamma_2)$ has only $3$ $($resp. $2)$ elements,\marginpar{bizarre cette hypoth\`ese}

\item[$\bullet$] there exist $X_1$ and $X_2$ two non empty subsets of $X$
such that
\begin{small}
\begin{align*}
&\mathrm{X}_2 \nsubseteq\mathrm{X}_1;
&&\forall\, \alpha \in \Gamma_1 \setminus \{\mathrm{id}\}, \,
\alpha(\mathrm{X}_2) \subset\mathrm{X}_1;
&&\forall\, \beta \in \Gamma_2 \setminus \{\mathrm{id}\}, \,
\beta(\mathrm{X}_1) \subset\mathrm{X}_2.
\end{align*}
\end{small}
\end{itemize}
Then $\Gamma$ is isomorphic to the free product $\Gamma_1 \ast \Gamma_2$ of $\Gamma_1$ and $\Gamma_2.$
\end{lem}

\begin{eg}\label{egpingpong}
The matrices $\left[%
\begin{array}{cc}
  1 & 2 \\
  0 & 1 \\
\end{array}%
\right]$ and $\left[%
\begin{array}{cc}
  1 & 0 \\
  2 & 1 \\
\end{array}%
\right]$ generate a free subgroup of rank~$2$ in $\mathrm{SL}_2(\mathbb{Z})$. Indeed let us set
\begin{align*}
&\Gamma_1=\left\{\left[%
\begin{array}{cc}
  1 & 2 \\
  0 & 1 \\
\end{array}%
\right]^n\,\big\vert \, n \in \mathbb{Z} \right\},
&&\Gamma_2=\left\{\left[%
\begin{array}{cc}
  1 & 0 \\
  2 & 1 \\
\end{array}%
\right]^n\,\big\vert\, n\in \mathbb{Z} \right\},
\end{align*}
\begin{align*}
&\mathrm{X}_1=\left\{(x,y)\in\mathbb{R}^2 \,\big\vert\, \vert x\vert>\vert y\vert\right\} &&\&
 &&\mathrm{X}_2=\left\{(x,y)\in\mathbb{R}^2
\,\big\vert\, \vert x\vert<\vert y\vert\right\}.
\end{align*}
Let us consider an element $\gamma$ of $\Gamma_1\setminus \{\mathrm{id}\}$ and $(x,y)$ an element of $\mathrm{X}_2,$ we note that $\gamma(x,y)$ is of the form $(x+my,y),$ with 
$\vert m\vert \geq
2;$ therefore $\gamma(x,y)$ belongs to $\mathrm{X}_1$. If $\gamma$ belongs to $\Gamma_2\setminus \{\mathrm{id}\}$ and if $(x,y)$ belongs to $\mathrm{X}_1,$ the image of $(x,y)$ by 
$\gamma$ belongs to $\mathrm{X}_2.$ According to Lemma \ref{pingpong} we have
\begin{align*}
\langle\left[%
\begin{array}{cc}
  1 & 2 \\
  0 & 1 \\
\end{array}%
\right] ,\,\left[%
\begin{array}{cc}
  1 & 0 \\
  2 & 1 \\
\end{array}%
\right]\rangle \simeq \mathrm{F}_2= \mathbb{Z}\ast
\mathbb{Z}=\Gamma_1\ast\Gamma_2.
\end{align*}
\smallskip

We also obtain that
\begin{align*}
&\left[%
\begin{array}{cc}
  1 & k \\
  0 & 1 \\
\end{array}%
\right] && \text{and} &&\left[%
\begin{array}{cc}
  1 & 0 \\
  k & 1 \\
\end{array}%
\right] 
\end{align*}
\noindent generate a free group of rank $2$ in $\mathrm{SL}_2(\mathbb{Z})$ for any $k \geq 2.$ Nevertheless it is not true for~$k=~1,$ the matrices 
\begin{align*}
&\left[%
\begin{array}{cc}
  1 & 1 \\
  0 & 1 \\
\end{array}%
\right]  && \text{and} &&\left[%
\begin{array}{cc}
  1 & 0 \\
  1 & 1 \\
\end{array}%
\right]
\end{align*}
generate $\mathrm{SL}_2(\mathbb{Z})$.
\end{eg}

\begin{eg}
Two generic matrices in $\mathrm{SL}_\nu(\mathbb{C}),$ with $\nu\geq 2$, generate a free group isomorphic to $\mathrm{F}_2.$ 
\end{eg}

In \cite{Can3} Cantat characterizes the finitely generated subgroups of $\mathrm{Bir}(\mathbb{P}^2);$ Favre reformulates, in \cite{Fa}, this classification:

\begin{thm}[\cite{Can3}]\label{clastypefini}
Let $\mathrm{G}$ be a finitely generated subgroup of the Cremona group. Exactly one of the following holds:
\begin{itemize}
\item[$\bullet$] Any element of $\mathrm{G}$ is elliptic thus 
\begin{itemize}
\item either $\mathrm{G}$ is, up to finite index and up to birational conjugacy, contained in the connected component of $\mathrm{Aut}(\mathrm{S})$ where $\mathrm{S}$ 
denotes a mi\-nimal rational surface;

\item or $\mathrm{G}$ preserves a rational fibration.
\end{itemize}

\item[$\bullet$] $\mathrm{G}$ contains a $($de Jonqui\`eres or Halphen$)$ twist and does not contain hyperbolic map, thus $\mathrm{G}$ preserves a rational or elliptic fibration.

\item[$\bullet$] $\mathrm{G}$ contains two hyperbolic maps $f$ and $g$ 
such that $\langle f, g\rangle$ is free.

\item[$\bullet$] $\mathrm{G}$ contains a hyperbolic map and for any pair
$(f,g)$ of hyperbolic maps, $\langle f, g\rangle$ is not a free group, then
\begin{align*}
& 1\longrightarrow\ker\rho\longrightarrow\mathrm{G}\stackrel{\rho}{\longrightarrow}
\mathbb{Z}\longrightarrow 1
\end{align*}
and $\ker\rho$ contains only elliptic maps. 
\end{itemize}
\end{thm}

\noindent One consequence is the following statement. 

\begin{thm}[\cite{Can3}] 
The Cremona group $\mathrm{Bir}(\mathbb{P}^2)$ satisfies Tits alternative.
\end{thm}

\subsection{Simplicity}\,

Let us recall that a simple group has no non trivial normal subgroup. We first remark that $\mathrm{Aut}(\mathbb{C}^2)$ is not simple; let $\phi$ be the morphism defined by
\begin{align*}
&\mathrm{Aut}(\mathbb{C}^2)\to\mathbb{C}^*, && f\mapsto\det\mathrm{jac}\,f. 
\end{align*}
The kernel of $\phi$ is a proper normal subgroup of $\mathrm{Aut}(\mathbb{C}^2).$ In the seventies Danilov has established that $\ker\,\phi$ is not simple (\cite{Da}). Thanks to 
some results of Schupp (\cite{Sc}) he proved that the normal subgroup\footnote{\hspace{1mm}Let $\mathrm{G}$ be a group and let   $f$ be an element of $\mathrm{G};$ the normal 
subgroup generated by $f$ in $\mathrm{G}$ is~$\langle hfh^{-1}\hspace{1mm}\vert\hspace{1mm} h\in~\mathrm{G}\rangle.$} generated by  
\begin{align*}
&(ea)^{13}, && a=(y,-x), && e=(x,y+3x^5-5x^4)
\end{align*}
is strictly contained in $\mathrm{Aut}(\mathbb{C}^2).$

\medskip

More recently Furter and Lamy gave a more precise statement. Before giving it let us introduce a length $\ell(.)$ for the elements of $\mathrm{Aut}(\mathbb{C}^2).$
\begin{itemize}
\item[$\bullet$] If $f$ belongs to $\mathtt{A}\cap\mathtt{E},$ then $\ell(f)=0;$ 

\item[$\bullet$] otherwise $\ell(f)$ is the minimal integer $n$ such that $f=g_1\ldots g_n$ with $g_i$ in $\mathtt{A}$ or $\mathtt{E}.$
\end{itemize}
The length of the element given by Danilov is $26.$

We note that $\ell(.)$ is invariant by inner conjugacy, we can thus assume that $f$ has minimal length in its conjugacy class.

\begin{thm}[\cite{FL}]
Let $f$ be an element of $\mathrm{Aut}(\mathbb{C}^2)$. Assume that $\det\mathrm{jac}\,f=1$ and that~$f$ has minimal length in its conjugacy class.
\begin{itemize}
\item[$\bullet$] If $f$ is non trivial and if $\ell(f)\leq 8,$ the normal subgroup generated by $f$ coincides with the group of polynomial automorphisms $f$ of $\mathbb{C}^2$ 
with~$\det\mathrm{jac}\,f=1$; 

\item[$\bullet$] if $f$ is generic\footnote{\hspace{1mm}See \cite{FL} for more details.} and if $\ell(f)\geq 14,$ the normal subgroup generated by~$f$ is strictly contained in the 
subgroup $\big\{f\in\mathrm{Aut}(\mathbb{C}^2)\,\big\vert\,\det\mathrm{jac}\,f=1\big\}$ of~$\mathrm{Aut}
(\mathbb{C}^2)$.
\end{itemize} 
\end{thm}

Is the Cremona group simple ? 

Cantat and Lamy study the general situation of a group $\mathrm{G}$ acting by isometries 
on a $\delta$-hyperbolic space and apply it to the particular case of the Cremona 
group acting by isometries on the hyperbolic space $\mathbb{H}_{\overline{\mathcal{Z}}}$.
Let us recall that a birational map $f$ induces a hyperbolic isometry $f_*\in
\mathbb{H}_{\overline{\mathcal{Z}}}$ if and only if $\{\deg f^k\}_{k\in\mathbb{N}}$
grows exponentially (Theorem \ref{Thm:cafa}). Another characterization given in \cite{CL}
is the following: $f$ induces a hyperbolic isome\-try~$f_*\in\mathbb{H}_{\overline{\mathcal{
Z}}}$ if and only if there is a $f_*$-invariant plane in the Picard-Manin space
that intersects $\mathbb{H}_{\overline{\mathcal{Z}}}$ on a curve $\mathrm{Ax}(f_*)$
(a geodesic line) on which $f_*$ acts by a translation: 
\begin{align*}
&\mathrm{dist}(x,f_*(x))=\log\lambda(f), && \forall x\in\mathrm{Ax}(f_*).
\end{align*}
The curve $\mathrm{Ax}(f_*)$ is uniquely determined and is called the axis of $f_*$.
A birational map $f$ is \textbf{\textit{tight}}\label{Chap2:ind33z} if
\begin{itemize}
\item[$\bullet$] $f_*\in\mathrm{Isom}(\mathbb{H}_{\overline{\mathcal{Z}}})$ is hyperbolic;

\item[$\bullet$] there exists a positive number $\varepsilon$ such that: if $g$ is a birational map
and if $g_*(\mathrm{Ax}(f_*))$ contains two points at distance $\varepsilon$ which are at distance
at most $1$ from $\mathrm{Ax}(f_*)$ then $g_*(\mathrm{Ax}(f_*))=\mathrm{Ax}(f_*)$;

\item[$\bullet$] if $g$ is a birational map and $g_*(\mathrm{Ax}(f_*))=\mathrm{Ax}(f_*)$ 
then $gfg^{-1}=f$ or~$f^{-1}$. 
\end{itemize}
 Applying their results on group acting by isometries on $\delta$-hyperbolic space to
the Cremona group, Cantat and Lamy obtain the following statement.

\begin{thm}[\cite{CL}]
 Let $f$ be a birational map of the complex projective plane. If $f$ is
tight, then $f^k$ generates a non trivial normal subgroup of~$\mathrm{Bir}(\mathbb{P}^2)$
for some positive interger $k$.
\end{thm}

They exhibit tight elements by working with the unique irreducible component of maximal 
dimension $$\mathrm{V}_d=\big\{\phi\psi\varphi^{-1}\,\vert\,\phi,\,\varphi\in\mathrm{Aut}
(\mathbb{P}^2),\,\psi\in\mathrm{dJ},\,\deg\psi=d\big\}$$ of $\mathrm{Bir}_d$. 

\begin{cor}[\cite{CL}]
The Cremona group contains an uncountable number of normal subgroups.

In particular $\mathrm{Bir}(\mathbb{P}^2)$ is not simple.
\end{cor}

\subsection{Representations of cocompact lattices of $\mathrm{SU}(n,1)$ in the Cremona group} \,
In \cite{DP} Delzant and Py study actions of K\"{a}hler groups on infinite dimensional real hyperbolic spaces, describe some exotic actions of $\mathrm{PSL}_2(\mathbb{R})$ on these 
spaces, and give an application to the study of the Cremona group. In particular they give a partial answer to a question of Cantat (\cite{Can3}):

\begin{thm}[\cite{DP}]
Let $\Gamma$ be a cocompact lattice in the group $\mathrm{SU}(n,1)$ with $n\geq 2$. If $\rho\colon\Gamma\to\mathrm{Bir}(\mathbb{P}^2)$ is an injective homomorphism, then one of the 
following two possibilities holds:
\begin{itemize}
\item[$\bullet$] the group $\rho(\Gamma)$ fixes a point in the Picard-Manin space;

\item[$\bullet$] the group $\rho(\Gamma)$ fixes a unique point in the boundary of the Picard-Manin space.
\end{itemize}
\end{thm}

\chapter{Quadratic and cubic birational maps}\label{Chap:quadcub}

\section{Some definitions and notations}\,

Let ~$\mathrm{Rat}_k$ be the projectivization of the space of triplets of homogeneous polynomials of degree~$k$ in $3$ variables:
\begin{align*}
\mathrm{Rat}_k=\mathbb{P}\big\{(f_0,f_1,f_2)\,\big\vert \, f_i\in\mathbb{C}[x,y,z]_k\big\}.
\end{align*}

An element of $\mathrm{Rat}_k$ has \textbf{\textit{degree}} $\leq k$.

We associate to $f=(f_0:f_1:f_2)\in\mathrm{Rat}_k$ the rational map
\begin{align*}
&f^\bullet\colon(x:y:z)\dashrightarrow\delta(f_0(x,y,z):f_1(x,y,z):f_2(x,y,z)), 
\end{align*}
where $\delta=\frac{1}{\text{pgcd}(f_0,f_1,f_2)}$.

Let $f$ be in $\mathrm{Rat}_k$; we say that $f=(f_0:f_1:f_2)$ is \textbf{\textit{purely of degree $k$}} if the $f_i$'s have no common factor.
Let us denote by $\mathring{\mathrm{R}} \mathrm{at}_k$ the set of rational maps purely of degree $k$. Whereas~$\mathrm{Rat}_k$ can be identified to a projective space, 
$\mathring{\mathrm{R}}\mathrm{at}_k$ is an open Zariski subset of it. An element of~$\mathrm{Rat}_k \setminus~\mathring{\mathrm{R}}\mathrm{at}_k$ can be written 
$\psi f=(\psi f_0:\psi f_1:\psi f_2)$ where $f$ belongs to $\mathrm{Rat}_\ell,$ where $\ell<~k,$ and $\psi$ is a homogeneous polynomial of degree $k-\ell.$ Let us 
denote by $\mathrm{Rat}$ the set of all rational maps from $\mathbb{P}^2(\mathbb{C})$ into itself: it is $\displaystyle\bigcup_{k\geq 1}\mathring{\mathrm{R}}\mathrm{at}_k.$
 It's also the injective limite of the $\mathrm{Rat}^\bullet_k$'s where
\begin{align*}
\mathrm{Rat}_k^\bullet=\big\{f^\bullet\,\big\vert\, f\in\mathrm{Rat}_k\big\}.
\end{align*}
Note that if $f\in\mathrm{Rat}_k$ is purely of degree $k$ then $f$ can be identified to $f^\bullet.$ This means that the application
\begin{align*}
\mathring{\mathrm{R}}\mathrm{at}_k\to\mathrm{Rat}_k^\bullet
\end{align*}
is injective. Henceforth when there is no ambiguity we use the notation $f$ for the elements of~$\mathrm{Rat}_k$ and for those of $\mathrm{Rat}_k^\bullet$. We will also say that 
an element of $\mathrm{Rat}_k$ \og is\fg\, a rational map.

The space $\mathrm{Rat}$ contains the group of birational maps of $\mathbb{P}^2(\mathbb{C}).$ 
Let $\mathrm{Bir}_k\subset \mathrm{Rat}_k$ be the set of birational maps $f$ of $\mathrm{Rat}_k$ such that $f^\bullet$ is invertible, and let us 
deno\-te by~$\mathring{\mathrm{B}}\mathrm{ir}_k\subset \mathrm{Bir}_k$ the set of birational maps purely of degree $k.$ Set
\begin{align*}
\mathrm{Bir}_k^\bullet=\big\{f^\bullet\,\big\vert\,f\in\mathrm{Bir}_k\big\}.
\end{align*}
The Cremona group can be identified to $\displaystyle\bigcup_{k\geq 1}\mathring{\mathrm{B}}\mathrm{ir}_k.$ Note that $\mathring{\mathrm{B}}\mathrm{ir}_1\simeq \mathrm{PGL}_3(\mathbb{C})$ 
is the group of automorphisms of $\mathbb{P}^2(\mathbb{C})$; we have $\mathring{\mathrm{B}}\mathrm{ir}_1\simeq\mathrm{Bir}_1^\bullet =\mathrm{Bir}_1.$ The set $\mathrm{Rat}_1$ can be 
identified to $\mathbb{P}^8(\mathbb{C})$ and $\mathring{\mathrm{R}} \mathrm{at}_1$ is the projectivization of the space of matrices of rank greater than $2.$

For $k=2$ the inclusion $\mathring{\mathrm{B}}\mathrm{ir}_2\subset\mathrm{Bir}_2$ is strict. Indeed if $A$ is in $\mathrm{PGL}_3(\mathbb{C})$ and if $\ell$ is a linear 
form, $\ell A$ is in $\mathrm{Bir}_2$ but not in $\mathring{\mathrm{B}}\mathrm{ir}_2.$

There are two "natural" actions on $\mathrm{Rat}_k$. The first one is the action of~$\mathrm{PGL}_3(\mathbb{C})$ by \textbf{\textit{dynamic conjugation}}
\begin{align*}
&\mathrm{PGL}_3(\mathbb{C})\times\mathrm{Rat}_k\to\mathrm{Rat}_k, &&(A,Q) \mapsto AQA^{-1}
\end{align*}
and the second one is the action of $\mathrm{PGL}_3(\mathbb{C})^2$ by \textbf{\textit{left-right composition}} (l.r.)
\begin{align*}
&\mathrm{PGL}_3(\mathbb{C})\times\mathrm{Rat}_k\times\mathrm{PGL}_3(\mathbb{C})\to\mathrm{Rat}_k,
&&(A,Q,B)\mapsto AQB^{-1}.
\end{align*}
Remark that $\mathring{\mathrm{R}}\mathrm{at}_k,$ $\mathrm{Bir}_k$ and $\mathring{\mathrm{B}}\mathrm{ir}_k$ are invariant under these two actions. Let us denote by $\mathcal{O}_{dyn}(Q)$ 
(resp. $\mathcal{O}_{l.r.} (Q)$) the orbit of $Q\in\mathrm{Rat}_k$ under the action of $\mathrm{PGL}_3(\mathbb{C})$ by dynamic conjugation (resp. under the l.r. action). 

\begin{egs}
Let $\sigma$ be the birational map given by 
\begin{align*}
& \mathbb{P}^2(\mathbb{C})\dashrightarrow\mathbb{P}^2(\mathbb{C}), && (x:y:z)\dashrightarrow(yz:xz:xy).
\end{align*}
The map $\sigma$ is an involution whose indeterminacy and exceptional sets are given by:
\begin{align*}
&\mathrm{Ind}\, \sigma=\big\{(1:0:0),\, (0:1:0),\, (0:0:1)\big\}, && \mathrm{Exc}\, \sigma=\big\{x=0,\, y=0,\, z=0\big\}.
\end{align*}

  The Cremona transformation $\rho\colon(x:y:z)\dashrightarrow(xy:z^2:yz)$ has two points of indeterminacy which are $(0:1:0)$ and $(1:0:0)$; the curves contracted by $\rho$ are $z=0$, 
resp. $y=0$. Let $\tau$ be the map defined by $(x:y:z)\dashrightarrow(x^2:xy:y^2-xz)$; we have
\begin{align*}
&\mathrm{Ind}\,\tau=\big\{(0:0:1)\big\},&& \mathrm{Exc}\,\tau=\big\{x=0\big\}.
\end{align*}

Notice that $\rho$ and $\tau$ are also involutions.

The Cremona transformations $f$ and $\psi$ are \textbf{\textit{birationally conjugate}} if there exists a birational map $\eta$ such that $f=\psi\eta\psi^{-1}.$ The three maps $\sigma$, 
$\rho$ and $\tau$ are birationally conjugate to some involutions of $\mathrm{PGL}_3(\mathbb{C})$ $($\emph{see for example} \cite{DI}$)$.
\end{egs}

Let us continue with quadratic rational maps.

Let $\mathbb{C}[x,y,z]_\nu$ be the set of homogeneous polynomials of degree $\nu$ in $\mathbb{C}^3$. Let us consider the rational map $\mathrm{det}\,\mathrm{jac}$ defined by
\begin{eqnarray*}
\mathrm{det}\,\mathrm{jac}\,\colon\,\mathrm{Rat}_2&\dashrightarrow&\mathbb{P}(\mathbb{C}[x,y,z]_3)
\simeq\big\{\text{curves of degree $3$}\big\}\\
\, [Q] &\dashrightarrow&[\mathrm{det}\,\mathrm{jac}\, Q=0].
\end{eqnarray*}

\begin{rem}\label{padef}
The map $\mathrm{det}\,\mathrm{jac}$ is not defined for maps $[Q]$ such that $\mathrm{det}\,\mathrm{jac}\, Q\equiv~0$; such a map is up to l.r. conjugacy $(Q_0:Q_1:0)$ or $(x^2:y^2:xy)$.
\end{rem}

\begin{pro}[\cite{CD}]\label{imdelta}
The map $\mathrm{det}\,\mathrm{jac}$ is surjective.
\end{pro}

\begin{proof}
For the map $\sigma$ we obtain three lines in general position, for $\rho$ the union of a "double line" and a line, for $\tau$ one "triple line" and for $(x^2:y^2:(x-y)z)$ the union of 
three concurrent lines. 

With \begin{align*}
\mathrm{det}\,\mathrm{jac}\,\left(-\frac{1}{\alpha}x^2+z^2:-\frac{\alpha}{2}xz+\frac{1+\alpha}{4}x^2-\frac{1}{4}y^2:xy\right)=[y^2z=x(x-z)(x-\alpha z)]
\end{align*}
we get all cubics having a Weierstrass normal form.

If $Q\colon(x:y:z)\dashrightarrow(xy:xz:x^2+yz),$ then $\mathrm{det}\,\mathrm{jac}\, Q=[x(x^2-yz)=0]$ is the union of a conic and a line in generic position. 

We have $\mathrm{det}\,\mathrm{jac}\, (y^2:x^2+2xz:x^2+xy+yz)=[y(2x^2-yz)=0]$ which is the union of a conic and a line tangent to this conic. 

We use an argument of dimension to show that the cuspidal cubic belongs to the image of $\mathrm{det}\,\mathrm{jac}$. 

Up to conjugation we obtain all plane cubics, we conclude by using the l.r. action. 
\end{proof}

\section{Criterion of birationality}\label{critbir}\,

We will give a presentation of the classification of the quadratic birational maps. Let us recall that if $\phi$ is a rational map and $P$ a homogeneous polynomial in three variables we 
say that $\phi$ contracts $P$ if the image by $\phi$ of the curve $[P=0]\setminus\mathrm{Ind}\, \phi$ is a finite set.

\begin{rem}
In general a rational map doesn't contract $\mathrm{det}\,\mathrm{jac}\,f$ $($it is the case for $f\colon(x:y:z)\dashrightarrow(x^2:y^2:z^2))$. Buts if $f$ is a birational map 
of~$\mathbb{P}^2(\mathbb{C})$ into itself, then $\mathrm{det}\,\mathrm{jac}\ f$ is contracted by $f$.
\end{rem}

Let $A$ and $B$ be two elements of $\mathrm{PGL}_3(\mathbb{C})$. Set $Q=A\sigma B$ (resp. $Q=A\rho B,$ resp. $Q=A\tau B$). Then  $\mathrm{det}\,\mathrm{jac}\, Q$ is the union of three 
lines in general position (resp. the union of a "double" line and a "simple" line, resp. a triple line). We will give a criterion which allows us to determine if a quadratic rational 
map is birational or not.

\begin{thm}[\cite{CD}]\label{cri}
Let $Q$ be a rational map; assume that $Q$ is purely quadratic and non degenerate $(${\it i.e.} $\mathrm{det}\,\mathrm{jac}\, Q\not\equiv 0)$. Assume that~$Q$ contracts $\mathrm{det}\,
\mathrm{jac}\, Q$; then $\mathrm{det}\,\mathrm{jac}\, Q$ is the union of three lines $($non-concurrent when they are distincts$)$ and $Q$ is birational.

Moreover:
\begin{itemize}
\item[$\bullet$] if $\mathrm{det}\,\mathrm{jac}\, Q$ is the union of three lines in general position, $Q$ is, up to l.r. equivalence, the involution $\sigma$;

\item[$\bullet$] if $\mathrm{det}\,\mathrm{jac}\, Q$ is the union of a "double" line and a "simple" line, $Q=\rho$ up to l.r. conjugation.

\item[$\bullet$] if $\mathrm{det}\,\mathrm{jac}\, Q$ is a "triple" line, $Q$ belongs to $\mathcal{O}_{l.r.} (\tau).$
\end{itemize}
\end{thm}

\begin{cor}[\cite{CD}]
A quadratic rational map from $\mathbb{P}^2(\mathbb{C})$ into itself belongs to $\mathcal{O}_{l.r.}(\sigma)$ if and only if it has three points of indeterminacy.
\end{cor}

\begin{rem}
A birational map $Q$ of $\mathbb{P}^2(\mathbb{C})$ into itself contracts $\mathrm{det}\,\mathrm{jac}\, Q$ and doesn't contract any other curve. Is the Theorem \ref{cri} avalaible in 
degree strictly larger than $2$ ? No, as soon as the degree is $3$ we can exhibit elements~$Q$ contracting $\mathrm{det}\,\mathrm{jac}\, Q$ but which are not birational: $$Q\colon(x:y:z)
\dashrightarrow(x^2y:xz^2:y^2z).$$
\end{rem}

\begin{rem}
We don't know if there is an analogue to Theorem \ref{cri} in any dimension; \cite{PRV} can maybe help to find an answer in dimension $3$.
\end{rem}

\begin{rem}
In \cite[Chapter 1, \S 6]{CD} we can find another criterion which allows us to determine if a quadratic rational map is rational or not.
\end{rem}

\begin{proof}[Proof of Theorem \ref{cri}]
Let us see that $\mathrm{det}\,\mathrm{jac}\, Q$ is the union of three lines.

Assume that $\mathrm{det}\,\mathrm{jac}\, Q$ is irreducible. Let us set $Q\colon(x:y:z)\dashrightarrow(Q_0:Q_1:Q_2)$. Up to l.r. conjugacy we can assume that $\mathrm{det}\,\mathrm{jac}\,
 Q$ is contracted on $(1:0:0)$; then $\mathrm{det}\mathrm{jac}\, Q$ divides $Q_1$ and~$Q_2$ which is impossible.

In the same way if $\det\,\mathrm{jac}\, Q=Lq$ where $L$ is linear and $q$ non degenerate and quadratic, we can assume that  $q=0$ is contracted on $(1:0:0)$; then $Q\colon(x:y:z)
\dashrightarrow(q_1:q:\alpha q)$ and so is degenerate.

Therefore $\det \mathrm{jac}\,Q$ is the product of three linear forms.

\bigskip

First of all let us consider the case where, up to conjugacy, $\mathrm{det}\,\mathrm{jac}\, Q=xyz$. If the lines $x=0$ and $y=0$
are contracted on the same point, for example $(1:0:0)$, then $Q\colon(x:y:z)\dashrightarrow(q:xy:\alpha xy)$ which is degenerate. The lines $x=0,$ $y=0$ and $z=0$ are thus contracted on 
three distinct points. A computation shows that they cannot be aligned. We can assume that~$x=0$ (resp. $y=0$, resp. $z=0$) is contracted on $(1:0:0)$ (resp. $(0:1:0),$ resp. $(0:0:1)$);
 let us note that $Q$ is the involution $(x:y:z)\dashrightarrow(yz:xz:xy)$ up to l.r. conjugacy. 

\bigskip

Now let us consider the case when $\mathrm{det}\,\mathrm{jac}\, Q$ has two branches $x=0$ and $z=0$. As we just see, the lines $x=0$ and $z=0$ are contracted on two distinct points, for 
example~$(1:0:0)$ and $(0:1:0)$. The map $Q$ is up to l.r. conjugacy $Q\colon(x:y:z)\dashrightarrow(z(\alpha y+\beta z):x(\gamma x+\delta y):xz)$. A direct computation shows that $Q$ is 
birational as soon as $\beta\delta-\alpha\gamma\not=0$ and in fact l.r. equivalent to $\rho$.

\bigskip

Then assume that $\mathrm{det}(\mathrm{jac}\, Q)=z^3$. We can suppose that $z=0$ is contracted on $(1:0:0)$; then $Q\colon(x:y:z)\dashrightarrow(q: z\ell_1:z\ell_2)$ where~$q$ is a 
quadratic form and the $\ell_i$'s are linear forms. 

\begin{itemize}
\item[$\bullet$] If $(z,\ell_1,\ell_2)$ is a system of coordinates we can write up to conjugacy
\begin{align*}
&Q\colon(x:y:z)\dashrightarrow(q:xz:yz), &&q=ax^2+by^2+cz^2+dxy.
\end{align*}
The explicit computation of $\mathrm{det}(\mathrm{jac}\, Q)$ implies: $a=b=d=0$, {\it i.e.} either $Q$ is degenerate, or $Q$ represents a linear map which is impossible.

\item[$\bullet$] Assume that $(z,\ell_1,\ell_2)$ is not a system of coordinates, {\it i.e.}
\begin{align*}
& \ell_1=az+\ell(x,y), &&\ell_2=bz+\varepsilon\ell(x,y).
\end{align*}
Let us remark that $\ell$ is nonzero (otherwise $Q$ is degenerate), thus we can assume that $\ell=x$. Up to l.r. equivalence  $$Q\colon(x:y:z)\dashrightarrow(q:xz:z^2).$$ An explicit 
computation implies the following equality: $\mathrm{det}\mathrm{jac}\,Q =-2z^2\frac{\partial q} {\partial y}$; thus $z$ divides $\frac{\partial q}{\partial y}$. In other words 
$q=\alpha z^2+\beta xz+\gamma x^2+\delta yz.$ Up to l.r. equivalence, we obtain $Q=\tau.$
\end{itemize}

\medskip

Finally let us consider the case: $\mathrm{det}(\mathrm{jac}\, Q)=xy(x-y)$. As we just see the lines $x=0$ and $y=0$ are contracted on two distinct points, for example $(1:0:0)$ 
and $(0:1:0).$ So
\begin{align*}
&Q\colon(x:y:z)\dashrightarrow(y(ax+by+cz): x(\alpha x+\beta y+\gamma z):xy)
\end{align*}
with $a,\, b,\, c,\, \alpha,\,\beta,\,\gamma\in\mathbb{C}$.
Let us note that the image of the line $x=y$ by~$Q$ is $((a+b)x+cz:(\alpha+\beta)x+\gamma z:x)$; it is a point if and only if $c$ and $\gamma$ are zero, then $Q$ does not depend 
on $z$.
\end{proof}

Set
\begin{align*}
& \Sigma^3:=\mathcal{O}_{l.r.}(\sigma),&& \Sigma^2:=\mathcal{O}_{l.r.}(\rho),&& \Sigma^1:=\mathcal{O}_{l.r.}(\tau).
\end{align*}

Let us consider a birational map represented by $$Q\colon(x:y:z)\dashrightarrow\ell(\ell_0:\ell_1:\ell_2)$$ where $\ell$ and the~$\ell_i$'s are linear forms, the $\ell_i$'s being 
independent. The line given by $\ell=0$ is an apparent contracted line; indeed the action of~$Q$ on~$\mathbb{P}^2(\mathbb{C})$ is obviously the action of the 
automorphism $(\ell_0:\ell_1:\ell_2)$ of $\mathbb{P}^2(\mathbb{C}).$ Let us denote by $\Sigma^0$ the set of these maps
\begin{align*}
\Sigma^0=\big\{\ell(\ell_0:\ell_1:\ell_2)\,\big\vert\,\ell,\,\ell_i \text{ linear forms, the $\ell_i$'s being independent}\big\}.
\end{align*}

We will abusively call the elements of $\Sigma^0$ linear elements; in fact the set $$(\Sigma^0)^\bullet=\big\{f^\bullet\,\big\vert\, f\in\Sigma^0\big\}$$ can be identified to 
$\mathrm{PGL}_3(\mathbb{C})$. We have $\Sigma^0=\mathcal{O}_{l.r.}(x (x:y:z))$: up to l.r. conjugacy a map $\ell A$ can be written $xA'$ then $x\mathrm{id}$. This approach 
allows us to see degenerations of quadratic maps on linear maps. 

Let us remark that an element of $\Sigma^i$ has $i$ points of indeterminacy and $i$ contracted curves.

An element of $\Sigma^i$ cannot be linearly conjugate to an element of $\Sigma^j$ where $j\not=i$; nevertheless they can be birationally conjugate: the involutions $\sigma$, $\rho$ and 
$\tau$ are birationally conjugate to involutions of $\mathrm{PGL}_3(\mathbb{C})$. Let us mention that a generic element of $\Sigma^i,$ $i\geq 1,$ is not birationally conjugate to a 
linear map. 

\begin{cor}[\cite{CD}]
We have
\begin{align*}
& \mathring{\mathrm{B}}\mathrm{ir}_2=\Sigma^1\cup\Sigma^2\cup\Sigma^3, && \mathrm{Bir}_2=\Sigma^0\cup\Sigma^1\cup\Sigma^2\cup \Sigma^3.
\end{align*}
\end{cor}

\begin{rems}
 \textbf{\textit{i.}} A N\oe ther decomposition of $\rho$ is $$(z-y:y-x:y)\sigma(y+z:z:x)\sigma(x+z:y-z:z).$$

We recover the classic fact already mentioned in \cite{Hu, AC}: for any birational quadratic map~$Q$ with two points of indeterminacy there exist $\ell_1$, $\ell_2$ and $\ell_3$ in 
$\mathrm{PGL}_3(\mathbb{C})$ such that $Q=\ell_1\sigma\ell_2 \sigma\ell_3$. 

 \textbf{\textit{ii.}} The map $\tau=(x^2:xy:y^2-xz)$ of $\Sigma^1$ can be written 
$\ell_1\sigma\ell_2\sigma \ell_3\sigma\ell_4\sigma\ell_5$ where
\begin{align*}
& \ell_1=(y-x:2y-x:z-y+x), && \ell_2=(x+z:x:y), \\
& \ell_3=(-y:x+z-3y:x), && \ell_4=(x+z:x:y),\\
& \ell_5= (y-x:-2x+z:2x-y). &&
\end{align*}
Therefore each element of $\Sigma^1$ is of the following type $\ell_1\sigma\ell_2\sigma\ell_3 \sigma \ell_4 \sigma\ell_5$ where~$\ell_i$ 
is in $\mathrm{PGL}_3(\mathbb{C})$ $($\emph{see} \cite{Hu, AC}$)$. The converse is false: if the $\ell_i$'s are generic then $\ell_1\sigma\ell_2 \sigma\ell_3 \sigma \ell_4 
\sigma\ell_5$ is of degree~$16$.
\end{rems}

\section{Some orbits under the left-right action}\,

As we saw $\mathrm{Bir}_2$ is a finite union of l.r. orbits but it is not a closed algebraic subset of $\mathrm{Rat}_2:$ the "constant" map $(yz:0:0)$ is in the closure of 
$\mathcal{O}_{l.r.}(\sigma)$ but not in $\mathrm{Bir}_2.$ To precise the nature of $\mathrm{Bir}_2$ we will study the orbits of $\sigma$, $\rho$, $\tau$ and $x(x:y:z)$.

\begin{pro}[\cite{CD}]
The dimension of $\Sigma^3=\mathcal{O}_{l.r.}(\sigma)$ is $14$.
\end{pro}

\begin{proof}
Let us denote by \textbf{\textit{${\rm Isot}\,\sigma$}} the isotropy group of $\sigma$. Let $(A,B)$ be an element of $(\mathrm{SL}_3(\mathbb{C}))^2$ such that $A\sigma=\sigma B$; a 
computation shows that $(A,B)$ belongs to
\begin{align*}
\langle\left(\left(\frac{x}{\alpha}:\frac{y}{\beta}:\alpha\beta z\right),\,\left(\alpha x:\beta y:\frac{z}{\alpha\beta}\right)\right), \, \mathscr{S}_6\times \mathscr{S}_6\,
\big\vert\,\alpha,\,\beta\in\mathbb{C}^*\rangle
\end{align*}
where $$\mathscr{S}_6=\big\{\mathrm{id},\,(x:z:y),\,(z:y:x),\,(y:x:z), \,(y:z:x),\,(z:x:y)\big\}.$$ This implies that $\dim{\rm Isot}\,\sigma=~2.$
\end{proof}

\begin{pro}[\cite{CD}]\label{iso}
The dimension of $\Sigma^2=\mathcal{O}_{l.r.}(\rho)$ is $13$.
\end{pro}

\begin{proof}
We will compute ${\rm Isot}\,\rho$, {\it i.e.} let $A$ and $C$ be two elements of $\mathrm{SL}_3(\mathbb{C})$ such that $A\rho=~\eta \rho C$ where $\eta$ is in $\mathbb{C}^*$. Let us 
recall that $$\mathrm{Ind}\,\rho=\big\{(0:1:0), \, (1:0:0)\big\};$$ the equa\-lity $A\rho=\eta\rho C$ implies that $C$ preserves $\mathrm{Ind}\,\rho$. But the points of indetermincay of 
$\rho$ "are not the same", they don't have the same multiplicity so~$C$ fixes $(0:1:0)$ and $(1:0:0)$; thus $C=(ax+bz:cy+dz:ez)$, where $ace\not=0.$ A computation shows that
\begin{align*}
&A=(\eta\gamma\delta x+\eta\beta\delta z:\eta\alpha^2 y:\eta\alpha\delta z), && C=(\gamma x+\beta z:\delta y:\alpha z)
\end{align*}
with $\eta^3\alpha^2\delta= \alpha \gamma\delta=1$. The dimension of the isotropy group is then $3$. 
\end{proof}

Notice that the computation of ${\rm Isot}\,\rho$ shows that we have the following relations
\begin{align*}
&(\gamma\delta x+\beta\delta z:\alpha^2y:\alpha\delta z)\rho=\rho(\gamma x+\beta z:\delta y:\alpha z), &&\alpha,\,\gamma,\,
\delta\in\mathbb{C}^*,\,\beta\in\mathbb{C}.
\end{align*}

We can compute the isotropy group of $\tau$ and show that:

\begin{pro}[\cite{CD}]\label{iso2}
The dimension of $\Sigma^1$ is $12.$
\end{pro}

In particular we obtain the following relations: $A\tau=\tau B$ when
\begin{align*}
&A=\left[
\begin{array}{ccc}
\alpha\varepsilon & 0 & \beta\varepsilon \\
\varepsilon\gamma+2\alpha\beta & \alpha^2 & (\varepsilon
\delta+ \beta^2)\\
0 & 0 & \varepsilon^2
\end{array}
\right],&& B=\left[
\begin{array}{ccc}
\alpha & \beta & 0\\
0 & \varepsilon & 0\\
\gamma & \delta & \alpha/\varepsilon
\end{array}
\right], 
\end{align*}
where $\beta,\,\gamma,\,\delta\in\mathbb{C},\,\alpha,\,\varepsilon
\in\mathbb{C}^*.$

A similar computation allows us to state the following result.

\begin{pro}[\cite{CD}]
The dimension of $\Sigma^0=\mathcal{O}_{l.r.}(x(x:y:z))$ is~$10.$
\end{pro}

\section{Incidence conditions; smoothness of $\mathrm{Bir}_2$
and non-smoothness of $\overline{\mathrm{Bir}_2}$}\,

Let us study the incidence conditions between the $\Sigma^i$'s and the smoothness of $\mathrm{Bir}_2:$

\begin{pro}[\cite{CD}]\label{lili}
We have
\begin{align*}
&\Sigma^0\subset\overline{\Sigma^1},&&\Sigma^1\subset \overline{
\Sigma^2},&&\Sigma^2\subset \overline{\Sigma^3}
\end{align*}
$($the closures are taken in $\mathrm{Bir}_2)$; in particular $\Sigma^3$ is dense in $\mathrm{Bir}_2.$
\end{pro}

\begin{proof}
By composing $\sigma$ with $(z:y:\varepsilon x+z)$ we obtain $$\sigma_1^\varepsilon=\big(y(\varepsilon x+z):z(\varepsilon
x+z):yz\big)$$ which is for $\varepsilon\not=0$ in $\mathcal{O}_{l.r.}(\sigma)$. But $\sigma_1^\varepsilon$ is l.r. conjugate to
$$\sigma_2^\varepsilon= \big(xy:(\varepsilon x+z)z:yz\big).$$
Let us note that $\displaystyle\lim_{\varepsilon\to 0}\sigma_2^\varepsilon=(xy:z^2:yz)=\rho$; so $\Sigma^2\subset\overline{\Sigma^3}$.\vspace{0.3cm}

If we compose $\rho$ with $(z:x+y:x),$ we have up to l.r. equivalence $(yz+xz:x^2:xy).$
Composing with $(x:y:y+z),$ we obtain up to l.r. conjugation the map $f=(yz+y^2+xz:x^2:xy)$. Set $g_\varepsilon:=f(x/\varepsilon:y:-\varepsilon z)$; up to l.r. conjugation 
$g_\varepsilon$ can be written $(-\varepsilon
yz+y^2-xz:x^2:xy)$. For $\varepsilon=0$ we have the map $\tau$. Therefore $\Sigma^1$ is contained in 
$\overline{\Sigma^2}$.

If $\varepsilon$ is nonzero, then $\tau$ can be written up to l.r. conjugation: $$(x^2:xy:\varepsilon^2y^2+xz);$$ for $\varepsilon=0$ we obtain $x(x:y:z)$ which is in $\Sigma^0$. 
Hence $\Sigma^0\subset\overline{ \Sigma^1}$.
\end{proof}

Thus we can state the following result.

\begin{thm}[\cite{CD}]
The closures being taken in $\mathrm{Bir}_2$ we have
\begin{align*}
& \overline{\Sigma^0}=\Sigma^0,&&\overline{\Sigma^1}=\Sigma^0
\cup\Sigma^1,&&\overline{\Sigma^2}=\Sigma^0\cup\Sigma^1\cup
\Sigma^2,
\end{align*}
\begin{align*}
&\mathring{\mathrm{B}}\mathrm{ir}_2=\Sigma^1\cup\Sigma^2\cup\Sigma^3, && \mathrm{Bir}_2=\overline{
\Sigma^3}=\Sigma^0\cup\Sigma^1\cup \Sigma^2\cup\Sigma^3
\end{align*}
with
\begin{align*}
&\dim\Sigma^0=10,&&\dim\Sigma^1=12,&&\dim\Sigma^2=13 &&\text{and}
&&\dim\Sigma^3=14.
\end{align*}
\end{thm}
\vspace{0.3cm}

\begin{thm}[\cite{CD}]\label{lis}
The set of quadratic birational maps is smooth in the set of rational maps. 
\end{thm}

\begin{proof}
Because any $\Sigma^i$ is one orbit and because of the incidence conditions it is sufficient to prove that the closure of $\Sigma^3$ is smooth along $\Sigma^0$.

The tangent space to $\Sigma^0$ in $x(x:y:z)$ is given by:
\begin{small}
\begin{eqnarray}
{\rm T}_{x(x:y:z)}\Sigma^0&=&\big\{(\alpha_1x^2+\alpha_4xy+\alpha_5
xz:\beta_1x^2+\beta_2y^2+\beta_4xy+\beta_5xz+\beta_6yz:\nonumber\\
&
&\hspace{3mm}\gamma_1x^2+\beta_6z^2+\gamma_4xy+\gamma_5xz+\beta_2yz
)\,\big\vert\,\alpha_i,\,\beta_i,\,\gamma_i\in\mathbb{C}\big\}. \nonumber
\end{eqnarray}
\end{small}

The vector space $S$ generated by 
\begin{align*}
& (y^2:0:0), && (z^2:0:0), && (yz:0:0), && (0:z^2:0),\\
 & (0:0:y^2), && (0:0:z^2), &&(0:0:yz)
\end{align*}
  is a supplementary of ${\rm T}_{x(x:y:z)}\Sigma^0$ in $\mathrm{Rat}_2.$ Let $f$ be an element of $\Sigma^3\cap\big\{x(x:y:z)+S\big\},$ it can be written
\begin{align*}
(x^2+Ay^2+Bz^2+Cyz:xy+az^2:xz+\alpha y^2+\beta z^2+\gamma yz).
\end{align*}
Necessarily $f$ has three points of indeterminacy. 

Assume that $a\not=0;$ let us remark that the second component of a point of indeterminacy of~$f$ is nonzero. If $(x:y:z)$ belongs to $\mathrm{Ind}\, f,$ then $x=-az^2/y.$ We have
\begin{small}
\begin{eqnarray}
f(-az^2/y:y:z)&=&(a^2z^4+Ay^4+By^2z^2+Cy^3z:0:-az^3+\alpha y^3+ \beta yz^2+\gamma y^2z)\nonumber\\
&=&(P:0:Q).\nonumber
\end{eqnarray}
\end{small}
As $f$ has three points of indeterminacy, the polynomials $P$ and $Q$ have to vanish on three distinct lines. In particular $Q$ divides $P$:
\begin{align*}
a^2z^4+Ay^4+By^2z^2+Cy^3z=(My+Nz)(-az^3+\alpha y^3+\beta yz^2+\gamma y^2z).
\end{align*}
Thus
\begin{equation}\label{trois}
B=-\beta^2-a\gamma,\hspace{6mm}C=-\beta\gamma-a\alpha,\hspace{6mm}A=-\alpha\beta.
\end{equation}
These three equations define a smooth graph through $f$ and $x(x:y:z),$ of codimension~$3$ as $\Sigma^3.$

Assume now that $a$ is zero; a point of indeterminacy $(x:y:z)$ of $f$ satisfies $xy=0.$ If~$x=0$ we have
\begin{align*}
 &f(0:y:z)=(Ay^2+Bz^2+Cyz:0:\alpha y^2+\beta x^2+\gamma yz)
\end{align*}
and if $y=0$ we have $f(x:0:z)=(x^2+Bz^2:0:xz+\beta z^2).$ The map $f$ has a point of indeterminacy of the form $(x:0:z)$ if and only if $B=-\beta^2$. If it happens, $f$ has only one 
such point of indeterminacy. Since $f$ has three points of indeterminacy, two of them are of the form $(0:y:z)$ and the polynomials $Ay^2+Bz^2+Cyz$ and $\alpha y^2+\beta z^2+\gamma yz$ 
are $\mathbb{C}$-colinear. We obtain the conditions
\smallskip

\begin{itemize}
\item[$\bullet$] $a=0,$ $B=-\beta^2,$ $A=-\alpha\beta$ and $C=-\beta\gamma $ if $\beta$ is nonzero;

\item[$\bullet$] $a=B=\beta=A\gamma-\alpha C=0$ otherwise.
\end{itemize}
Let us remark that in this last case $f$ cannot have three points of indeterminacy. Finally we note that $\Sigma^3\cap\big\{x(x:y:z)+S\big\}$ is contained in the graph defined by 
the equations~(\ref{trois}). The same holds for the closure $\overline{\Sigma^3}\cap\big\{x(x:y:z)+S\big\}$ which, for some reason of dimension, coincides thus with this graph. Then
 $\overline{\Sigma^3}$ is smooth along $\Sigma^0.$
\end{proof}

\begin{rem}
Since $\overline{\Sigma^3}$ is smooth along $\Sigma^0$ and since we have incidence conditions, $\overline{\Sigma^3}$ is smooth along $\Sigma^2$ and $\Sigma^1.$ Nevertheless we can 
show these two statements by constructing linear families of birational maps $($\emph{see} \cite{CD}$)$.
\end{rem}

\begin{pro}[\cite{CD}]\label{rapp1}
The closure of $\mathrm{Bir}_2$ in $\mathbb{P}^{17}\simeq\mathrm{Rat}_2$ is not smooth.
\end{pro}

\begin{proof}
Let $\phi$ be a degenerate birational map given by $z(x:y:0).$ The tangent space to~$\mathcal{O}_{l.r.}(\phi)$ in $\phi$ is given by

\begin{small}
\begin{eqnarray}
{\rm T}_\phi\mathcal{O}_{l.r.}(\phi)&=&\big\{(\alpha_1x^2+\alpha_3z^2+
\alpha_4xy+\alpha_5xz+\alpha_6yz:\alpha_4y^2+\beta_3z^2\nonumber\\
& &\hspace{3mm}+
\alpha_1xy+\beta_5xz+\beta_6yz:\gamma_5xz+\gamma_6yz)\,\big\vert\,\alpha_i,
\,\beta_i,\,\gamma_i\in\mathbb{C}\big\}.\nonumber
\end{eqnarray}
\end{small}

A supplementary $S$ of ${\rm T}_\phi\mathcal{O}_{l.r.}(\phi)$ is the space of dimension $8$ generated by
\begin{align*}
&(y^2:0:0),&&(0:x^2:0),&&(0:y^2:0),&&(0:xy:0),\\
& (0:0:x^2), && (0:0:y^2), && (0:0:z^2), && (0:0:xy).
\end{align*}
We will prove that $\big\{\phi+S\big\}\cap\overline{\Sigma^3}$ contains a singular analytic subset of codimension $3$. Since $\overline{\Sigma^3}$ is also of codimension $3$ we will 
obtain, using the l.r. action, the non-smoothness of $\overline{\Sigma^3}$ along the orbit of $\phi$. An element $Q$ of $\big\{\phi+S\big\}$ can be writen
\begin{align*}
(xz+ay^2:yz+bx^2+cy^2+dxy:ex^2+fy^2+gz^2+hxy).
\end{align*}
The points of indeterminacy are given by the three following equations 
\begin{align*}
& xz+ay^2=0, && yz+bx^2+cy^2+dxy=0, &&
ex^2+fy^2+hxy=0;
\end{align*}
after eliminating $z$ this yields to $P_1=P_2=0$ where
\begin{align*}
& P_1=-ay^3+bx^3+cxy^2+dx^2y, && P_2=ex^4+fx^2y^2+a^2gy^4+hx^3y.
\end{align*}
Let us remark that if, for some values of the parameters, $P_1$ vanishes on three distinct lines and divides $P_2,$ then the corresponding map $Q$ has three points of indeterminacy and 
is birational, more precisely $Q$ is in $\Sigma^3.$ The fact that  $P_1$ divides $P_2$ gives 
\begin{equation}\label{sys}
P_2=(Ax+By) P_1\hspace{5mm}\Leftrightarrow\hspace{5mm}\left\{\begin{array}{lllll}
e=bA \\
f=cA+dB\\
a^2g=-aB\\
h=dA+bB\\
aA=cB
\end{array}\right.
\end{equation}
Let us note that the set $\Lambda$ of parameters such that 
\begin{align*}
& a=0,&& bf-ce=0,&& bh-de=0
\end{align*}
satisfies the system (\ref{sys}) (with $A=e/b$ and $B=0$). The set $\Lambda$ is of codimension $3$ and is not smooth.
The intersection $\Lambda'$ of quadrics $bf-ce=0$ and $bh-de=0$ is not smooth. Indeed $\Lambda'$ contains the linear space~$E$ given by $b=e=0$ but is not reduced to $E$: for example 
the space defined by $b=c=d=e=f=h$ is contained in $\Lambda'$ and not in $E.$ Since $\text{ codim }E=\text{codim }\Lambda'$ the set $\Lambda'$ is thus reducible and then not smooth; 
it is the same for $\Lambda.$ If $a=b=e=0$ (resp. $b=c=d=e =f=h=1,$ $a=0$) the polynomial $P_1$ is equal to $cxy^2+dx^2y$ (resp. $x^3+xy^2 +x^2y$) and in general vanishes on three 
distinct lines. So we have constructed in $\overline{\Sigma^3}\cap\big\{\phi+S\big\}$ a singular analytic set of codimension $3.$
\end{proof}

\section{A geometric description of quadratic birational maps}

\subsection{First definitions and first properties}\,

In a plane $\mathcal{P}$ let us consider a net of conics, {\it i.e.} a $2$-dimensional linear system $\Lambda$ of conics. Such a system is a \textsl{\textbf{homaloidal net}} if it 
possesses three base-points, that is three points through which all the elements of $\Lambda$ pass. There are three different such nets
\begin{itemize}
\item[$\bullet$] the nets $\Lambda_3$ of conics with three distinct base-points;

\item[$\bullet$] the nets $\Lambda_2$ of conics passing through two points, all having at one of them the same tangent; 

\item[$\bullet$] the nets $\Lambda_1$ of conics mutually osculating at a point.
\end{itemize}

In order to have three conics that generate a homaloidal net $\Lambda$ it suffices to annihilate the minors of a matrix $$\left[\begin{array}{ccc}\ell_0 & \ell_1&\ell_2\\ \ell'_0 & 
\ell'_1 & \ell'_2\end{array}\right]$$ whose elements are linear forms in the indeterminates $x$, $y$ and $z$. Indeed the two conics described by
\begin{equation}\label{twoconics}
\ell_0\ell_1'-\ell_0'\ell_1=0, \hspace{2cm} \ell_0\ell'_2-\ell_2\ell'_0=0
\end{equation}
have four points in common. One of them ($(\ell_0=0)\cap(\ell_0'=0)$) doesn't belong to the third conic $\ell_1\ell'_2-\ell'_1\ell_2=0$ obtained from (\ref{twoconics}) by eliminating 
$\ell_0/\ell_0'$. So $\Lambda$ is given by 
\begin{align*}
& a_0(\ell_0\ell_1'-\ell_0'\ell_1)+a_1(\ell_0\ell'_2-\ell_2\ell'_0)+a_2(\ell_1\ell'_2-\ell'_1\ell_2)=0
\end{align*}
with $(a_0:a_1:a_2)\in\mathbb{P}^2(\mathbb{C})$.

Let $x$, $y$, $z$ be some projective coordinates in $\mathcal{P}$ and let $u$, $v$,  $w$ be some projective coordinates in $\mathcal{P}'$, another plane which coincides with 
$\mathcal{P}$. Let $f$ be the algebraic correspondance between these two planes; it is defined by $$\left\{\begin{array}{ll}\varphi(x,y,z;u,v,w)=0\\\psi(x,y,z;u,v,w)=0\end{array}
\right.$$ As $f$ is a birational isomorphism we can write $\varphi$ and $\psi$ as follows $$\left\{\begin{array}{ll}\varphi(x,y,z;u,v,w)=u\ell_0(x,y,z)+v\ell_1(x,y,z)+w\ell_2(x,y,z),
\\\psi(x,y,z;u,v,w)=u\ell_0'(x,y,z)+v\ell_1'(x,y,z)+w\ell_2'(x,y,z)\end{array}\right.$$ and also $$\left\{\begin{array}{ll}\varphi(x,y,z;u,v,w)=xL_0(u,v,w)+yL_1(u,v,w)+zL_2(u,v,w),\\ 
\psi(x,y,z;u,v,w)=xL_0'(u,v,w)+yL_1'(u,v,w)+zL_2'(u,v,w)\end{array}\right.$$ where $\ell_i$, $\ell'_i$, $L_i$ and $L'_i$ are some linear forms. This implies in particular that 
\begin{equation}\label{AA}
(u:v:w)=(\ell_1\ell'_2-\ell_2\ell'_1:\ell_2\ell'_0-\ell_0\ell'_2:\ell_0\ell'_1-\ell_1\ell'_0)
\end{equation}
{\it i.e.} $u$ (resp. $v$, resp. $w$) is a quadratic form in $x$, $y$, $z$. 

On can note that if $m=(u:v:w)\in\mathcal{P}'$ belongs to the line $\mathcal{D}$ given by $a_0u+a_1v+a_2w=0$ the point $(x:y:z)$ corresponding to it via (\ref{AA}) belongs to the conic 
given by $$a_0(\ell_1\ell'_2-\ell_2\ell'_1)+a_1(\ell_2\ell'_0-\ell_0\ell'_2)+a_2(\ell_0\ell'_1-\ell_1\ell'_0)=0.$$ So the lines of a plane thus correspond to the conics of a homaloidal 
net of the other plane.

Conversely we can associate a quadratic map between two planes to a homaloidal net of conics in one of them. Let $\Lambda$ be an arbitrary homaloidal net of conics in $\mathcal{P}$ and 
let us consider a projectivity $\theta$ between $\Lambda$ and the net of lines in $\mathcal{P}'$. Let $m$ be a point of $\mathcal{P}$ and let us assume that $m$ is not a base-point of 
$\Lambda$. The elements of $\Lambda$ passing through $m$ is a pencil of conics with four base-points: the three base-points of $\Lambda$ and $m$. To this pencil corresponds a pencil of 
lines whose base-point $\widetilde{m}$ is determined by $m$. To a point $m'\in\mathcal{P}'$ corresponds a pencil of conics in~$\mathcal{P}$, the image of the pencil of lines centered in 
$m$. Therefore the map which sends $m$ to $\widetilde{m}$ gives rise to a Cremona map from $\mathcal{P}$ into $\mathcal{P}'$ which sends the conics of $\mathcal{P}$ into the lines of 
$\mathcal{P}'$.

So we have the following statement.

\begin{pro}
To give a quadratic birational map between two planes is, up to an automorphism, the same as giving a homaloidal net of conics in one of them.
\end{pro}

\begin{rem}
To a base-point of one of the two nets is associated a line in the other plane which is an exceptional line.
\end{rem}

\subsection{Classification of the quadratic birational maps between planes}\,

We can deduce the classification of the quadratic birational maps between planes from the description of the homaloidal nets $\Lambda$ of conics in $\mathcal{P}$.

\begin{itemize}
\item[$\bullet$] If $\Lambda$ has three distinct base-points we can assume that these points are $p_0=(1:0:0)$, $p_1=(0:1:0)$, $p_2=(0:0:1)$ and $\Lambda$ is thus given by
\begin{align*}
& a_0yz+a_1xz+a_2xy=0, && (a_0:a_1:a_2)\in\mathbb{P}^2(\mathbb{C}).
\end{align*}
The map $f$ is defined by $(x:y:z)\dashrightarrow(yz:xz:xy)$ and can easily be inverted ($f$ is an involution).

\item[$\bullet$] If $\Lambda$ has two distinct base-points, we can assume that the conics of $\Lambda$ are tangent at $p_2=(0:0:1)$ to the line $x=0$ and also pass through $p_0=(1:0:0)$.
 Then $\Lambda$ is given by 
\begin{align*}
&a_0xz+a_1xy+a_2y^2=0, && (a_0:a_1:a_2)\in\mathbb{P}^2(\mathbb{C}).
\end{align*}
The map $f$ is defined by $(x:y:z)\dashrightarrow(xz:xy:y^2)$ and its inverse is $(u:v:w)\dashrightarrow (v^2:vw:uw)$.

\item[$\bullet$] If the conics of $\Lambda$ are mutually osculating at $p_2=(0:0:1)$, we can assume that $\Lambda$ contains the two degenerated conics $x^2=0$ and $xy=0$. Let 
$\mathcal{C}$ be an irreducible conic in~$\Lambda$; assume that $\mathcal{C}\cap(y=0)=p_0$ and that $p_1=(0:1:0)$ is the pole of $y=0$ with respect to $\mathcal{C}$. Assume finally 
that $(1:1:1)$ belongs to $\mathcal{C}$ then $\mathcal{C}$ is given by $xz+y^2=0$ and $\Lambda$ is defined by 
\begin{align*}
& a_0(xz+y^2)+a_1x^2+a_2xy=0, && (a_0:a_1:a_2)\in\mathbb{P}^2(\mathbb{C}).
\end{align*}

The map $f$ is $(x:y:z)\dashrightarrow (xz-y^2:x^2:xy)$ and its inverse is $(u:v:w)\dashrightarrow(v^2:vw:uv+w^2)$.
\end{itemize}

\begin{rem}
We can see that $f$ and $f^{-1}$ have the same type. So the homaloidal nets associated to $f$ and $f^{-1}$ have the same type.
\end{rem}

\section{Cubic birational maps}

The space of birational maps which are purely of degree $2$ is smooth and connected. Is it the case in any degree ? Let us see what happens in degree $3$. In the old texts we can find 
a description of cubic birational maps which is based on enumerative geometry. In \cite[Chapter~6]{CD} we give a list of normal forms up to l.r. conjugation, the connectedness appearing 
as a consequence of this classification. The methods are classical: topo\-logy of the complement of some plane curves, contraction of the jacobian determinant... Unfortunately, as soon as 
the degree is greater than $3$ we have no criterion as in degree $2$: if $f$ is the map $(x^2y:xz^2:y^2z),$ the zeroes of $\mathrm{det}\,\mathrm{jac}\, f$ are contracted but $f$ is not 
invertible. Nevertheless if $f$ is birational, the curve $\mathrm{det}\,\mathrm{jac}\, f=0$ is contracted and it helps in a lot of cases. We show that in degree $3$ the possible 
configurations of contracted curves are the following unions of lines and conics:
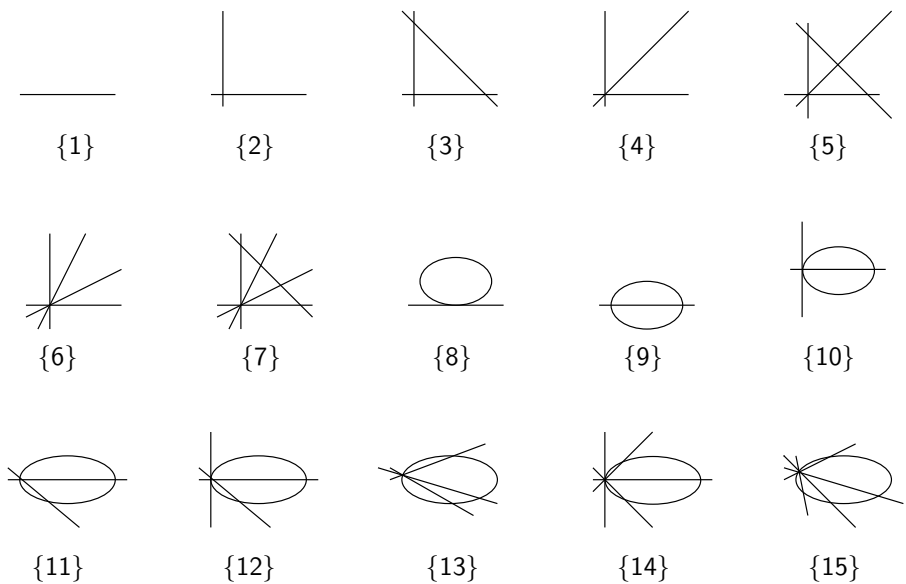
\begin{figure}[H]
\begin{center}
\input{gal.pstex_t}
\end{center}
\end{figure}

\begin{figure}[H]
\begin{center}
\input{gal2.pstex_t}
\end{center}
\end{figure}

\begin{figure}[H]
\begin{center}
\input{gal3.pstex_t}
\end{center}
\end{figure}

The following table gives the classification of cubic birational maps up to conjugation:

\begin{scriptsize}
\begin{landscape}
\begin{center}
\begin{tabular}{|*{3}{c}*{5}{c|}l r|}
   \hline
    & & & & & & &\\
    & \hspace*{2mm} & $(xz^2+y^3:yz^2:z^3)$  & \hspace*{2mm}&\hspace*{2mm}$\mathsf{\{1\}}$\hspace*{2mm}
    &\hspace*{2mm}$\mathsf{\{1\}}$\hspace*{2mm}&\hspace*{2mm}13\hspace*{2mm}&\\
    & \hspace*{2mm} &  $(xz^2:x^2y:z^3)$ & \hspace*{2mm}&\hspace*{2mm}$\mathsf{\{2\}}$\hspace*{2mm}
    &\hspace*{2mm}$\mathsf{\{2\}}$\hspace*{2mm}&\hspace*{2mm}15\hspace*{2mm}&\\
    & \hspace*{2mm} & $(xz^2:x^3+xyz:z^3)$ & \hspace*{2mm}&\hspace*{2mm}$\mathsf{\{2\}}$\hspace*{2mm}
    &\hspace*{2mm}$\mathsf{\{2\}}$\hspace*{2mm}&\hspace*{2mm}15\hspace*{2mm}&\\
    & \hspace*{2mm} & $(x^2z:x^3+z^3+xyz:xz^2)$ & \hspace*{2mm}&\hspace*{2mm}$\mathsf{\{2\}}$\hspace*{2mm}
    &\hspace*{2mm}$\mathsf{\{2\}}$\hspace*{2mm}&\hspace*{2mm}14\hspace*{2mm}&\\
    & \hspace*{2mm} & $(x^2z:x^2y+z^3:xz^2)$ & \hspace*{2mm}&\hspace*{2mm}$\mathsf{\{2\}}$\hspace*{2mm}
    &\hspace*{2mm}$\mathsf{\{2\}}$\hspace*{2mm}&\hspace*{2mm}15\hspace*{2mm}&\\
    & \hspace*{2mm} & $(xyz:yz^2:z^3-x^2y)$ & \hspace*{2mm}&\hspace*{2mm}$\mathsf{\{2\}}$\hspace*{2mm}
    &\hspace*{2mm}$\mathsf{\{8\}}$\hspace*{2mm}&\hspace*{2mm}14\hspace*{2mm}&\\
    & \hspace*{2mm} & $(x^3:y^2z:xyz)$ & \hspace*{2mm}&\hspace*{2mm}$\mathsf{\{3\}}$\hspace*{2mm}
    &\hspace*{2mm}$\mathsf{\{3\}}$\hspace*{2mm}&\hspace*{2mm}15\hspace*{2mm}&\\
    & \hspace*{2mm} & $(x^2(y-z):xy(y-z):y^2z)$ & \hspace*{2mm}&\hspace*{2mm}$\mathsf{\{3\}}$\hspace*{2mm}
    &\hspace*{2mm}$\mathsf{\{10\}}$\hspace*{2mm}&\hspace*{2mm}15\hspace*{2mm}&\\
    & \hspace*{2mm} & $(x^2z:xyz:y^2(x-z))$ & \hspace*{2mm}&\hspace*{2mm}$\mathsf{\{3\}}$\hspace*{2mm}
    &\hspace*{2mm}$\mathsf{\{10\}}$\hspace*{2mm}&\hspace*{2mm}15\hspace*{2mm}&\\
    & \hspace*{2mm} & $(xyz:y^2z:x(y^2-xz))$ & \hspace*{2mm}&\hspace*{2mm}$\mathsf{\{3\}}$\hspace*{2mm}
    &\hspace*{2mm}$\mathsf{\{10\}}$\hspace*{2mm}&\hspace*{2mm}15\hspace*{2mm}&\\
    & \hspace*{2mm} & $(x^3:x^2y:(x-y)yz)$ & \hspace*{2mm}&\hspace*{2mm}$\mathsf{\{4\}}$\hspace*{2mm}
    &\hspace*{2mm}$\mathsf{\{4\}}$\hspace*{2mm}&\hspace*{2mm}15\hspace*{2mm}&\\
    & \hspace*{2mm} & $(x^2(x-y):xy(x-y):xyz+y^3)$ & \hspace*{2mm}&\hspace*{2mm}$\mathsf{\{4\}}$\hspace*{2mm}
    &\hspace*{2mm}$\mathsf{\{4\}}$\hspace*{2mm}&\hspace*{2mm}16\hspace*{2mm}&\\
    & \hspace*{2mm} & $(xz(x+y):yz(x+y):xy^2)$ & \hspace*{2mm}&\hspace*{2mm}$\mathsf{\{5\}}$\hspace*{2mm}
    &\hspace*{2mm}$\mathsf{\{5\}}$\hspace*{2mm}&\hspace*{2mm}16\hspace*{2mm}&\\
    & \hspace*{2mm} & $(x(x+y)(y+z):y(x+y)(y+z):xyz)$ & \hspace*{2mm}&\hspace*{2mm}$\mathsf{\{5\}}$\hspace*{2mm}
    &\hspace*{2mm}$\mathsf{\{12\}}$\hspace*{2mm}&\hspace*{2mm}16\hspace*{2mm}&\\
    & \hspace*{2mm} & $(x(x+y+z)(x+y):y(x+y+z)(x+y):xyz)$ & \hspace*{2mm}&\hspace*{2mm}$\mathsf{\{5\}}$\hspace*{2mm}
    &\hspace*{2mm}$\mathsf{\{12\}}$\hspace*{2mm}&\hspace*{2mm}16\hspace*{2mm}&\\
    & \hspace*{2mm} & $(x(x^2+y^2+\gamma xy):y(x^2+y^2+\gamma xy):xyz),$ $\gamma^2\not=4$ & \hspace*{2mm}&\hspace*{2mm}$\mathsf{\{6\}}$\hspace*{2mm}
    &\hspace*{2mm}$\mathsf{\{6\}}$\hspace*{2mm}&\hspace*{2mm}15\hspace*{2mm}&\hspace*{2mm}$1$ parameter\\
    & \hspace*{2mm} & $(xz(y+x):yz(y+x):xy(x-y))$ & \hspace*{2mm}&\hspace*{2mm}$\mathsf{\{7\}}$\hspace*{2mm}
    &\hspace*{2mm}$\mathsf{\{7\}}$\hspace*{2mm}&\hspace*{2mm}16\hspace*{2mm}&\\
    & \hspace*{2mm} & $(x(x^2+y^2+\gamma xy+\gamma_+xz+yz):y(x^2+y^2+\gamma xy+\gamma_+xz+yz):xyz)$ & \hspace*{2mm}&\hspace*{2mm}$\mathsf{\{7\}}$\hspace*{2mm}
    &\hspace*{2mm}$\mathsf{\{14\}}$\hspace*{2mm}&\hspace*{2mm}16\hspace*{2mm}&\hspace*{2mm}$1$ parameter\\
    & \hspace*{2mm} & $(y(x-y)(x+z):x(x-y)(z-y):yz(x+y))$ & \hspace*{2mm}
    &\hspace*{2mm}$\mathsf{\{7\}}$\hspace*{2mm}&\hspace*{2mm}$\mathsf{\{14\}}$\hspace*{2mm}
    &\hspace*{2mm}16\hspace*{2mm}&\\
    & \hspace*{2mm} & $(x(x^2+yz):y^3:y(x^2+yz))$ & \hspace*{2mm}&\hspace*{2mm}$\mathsf{\{8\}}$\hspace*{2mm}
    &\hspace*{2mm}$\mathsf{\{2\}}$\hspace*{2mm}&\hspace*{2mm}14\hspace*{2mm}&\\
    & \hspace*{2mm} & $(y^2z:x(xz+y^2):y(xz+y^2))$ & \hspace*{2mm}&\hspace*{2mm}$\mathsf{\{9\}}$\hspace*{2mm}
    &\hspace*{2mm}$\mathsf{\{9\}}$\hspace*{2mm}&\hspace*{2mm}15\hspace*{2mm}&\\
    & \hspace*{2mm} & $(x(y^2+xz):y(y^2+xz):xyz)$ & \hspace*{2mm}&\hspace*{2mm}$\mathsf{\{10\}}$\hspace*{2mm}
    &\hspace*{2mm}$\mathsf{\{3\}}$\hspace*{2mm}&\hspace*{2mm}15\hspace*{2mm}&\\
    & \hspace*{2mm} & $(x(y^2+xz):y(y^2+xz):xy^2)$ & \hspace*{2mm}&\hspace*{2mm}$\mathsf{\{10\}}$\hspace*{2mm}
    &\hspace*{2mm}$\mathsf{\{3\}}$\hspace*{2mm}&\hspace*{2mm}15\hspace*{2mm}&\\
    & \hspace*{2mm} & $(x(x^2+yz):y(x^2+yz):xy^2)$ & \hspace*{2mm}&\hspace*{2mm}$\mathsf{\{10\}}$\hspace*{2mm}
    &\hspace*{2mm}$\mathsf{\{3\}}$\hspace*{2mm}&\hspace*{2mm}15\hspace*{2mm}&\\
    & \hspace*{2mm} & $(x(xy+xz+yz):y(xy+xz+yz):xyz)$ & \hspace*{2mm}&\hspace*{2mm}$\mathsf{\{11\}}$\hspace*{2mm}
    &\hspace*{2mm}$\mathsf{\{11\}}$\hspace*{2mm}&\hspace*{2mm}16\hspace*{2mm}&\\
    & \hspace*{2mm} & $(x(x^2+yz+xz):y(x^2+yz+xz):xyz)$ & \hspace*{2mm}&\hspace*{2mm}$\mathsf{\{11\}}$\hspace*{2mm}
    &\hspace*{2mm}$\mathsf{\{11\}}$\hspace*{2mm}&\hspace*{2mm}16\hspace*{2mm}&\\
    & \hspace*{2mm} & $(x(x^2+xy+yz):y(x^2+xy+yz):xyz)$ & \hspace*{2mm}&\hspace*{2mm}$\mathsf{\{12\}}$\hspace*{2mm}
    &\hspace*{2mm}$\mathsf{\{5\}}$\hspace*{2mm}&\hspace*{2mm}16\hspace*{2mm}&\\
    & \hspace*{2mm} & $(x(x^2+yz):y(x^2+yz):xy(x-y))$ & \hspace*{2mm}&\hspace*{2mm}$\mathsf{\{12\}}$\hspace*{2mm}
    &\hspace*{2mm}$\mathsf{\{5\}}$\hspace*{2mm}&\hspace*{2mm}16\hspace*{2mm}&\\
    & \hspace*{2mm} & $(x(y^2+\gamma xy+yz+xz):y(y^2+\gamma xy+yz+xz):xyz),$ $\gamma\not=0,\, 1$ & \hspace*{2mm}&\hspace*{2mm}$\mathsf{\{13\}}$\hspace*{2mm}
    &\hspace*{2mm}$\mathsf{\{13\}}$\hspace*{2mm}&\hspace*{2mm}16\hspace*{2mm}&\hspace*{2mm}$1$ parameter\\
    & \hspace*{2mm} & $(x(x^2+y^2+\gamma xy+xz):y(x^2+y^2+\gamma xy+xz):xyz),$ $\gamma^2\not=4,$ &\hspace*{2mm} &\hspace*{2mm}$\mathsf{\{14\}}$\hspace*{2mm}
    &\hspace*{2mm}$\mathsf{\{7\}}$\hspace*{2mm}&\hspace*{2mm}16\hspace*{2mm}&\hspace*{2mm}$1$ parameter\\
    & \hspace*{2mm} & $(x(x^2+yz+xz):y(x^2+yz+xz):xy(x-y))$ & \hspace*{2mm}&\hspace*{2mm}$\mathsf{\{14\}}$\hspace*{2mm}
    &\hspace*{2mm}$\mathsf{\{7\}}$\hspace*{2mm}&\hspace*{2mm}16\hspace*{2mm}&\\
    & \hspace*{2mm} & $(x(x^2+y^2+\gamma xy+\delta xz+yz):y(x^2+y^2+\gamma xy+\delta xz+yz):xyz),$ $\gamma^2\not=4,$
    $\delta\not=\gamma_\pm$ & \hspace*{2mm}&\hspace*{2mm}$\mathsf{\{15\}}$\hspace*{2mm}&\hspace*{2mm}$\mathsf{\{15\}}$\hspace*{2mm}
    &\hspace*{2mm}16\hspace*{2mm}&\hspace*{2mm}$2$ parameters\\
    & & & & & & &\\
\hline
\end{tabular}
\end{center}
\end{landscape}
\end{scriptsize}

  where $\gamma$ denotes a complex number and where
\begin{align*}
\gamma_+:=\frac{\gamma+\sqrt{\gamma^2-4}} {2} &&
\gamma_-:=\frac{\gamma-\sqrt{\gamma^2-4}}{2}.
\end{align*}
For any model we mention the configuration of contracted curves of the map (second column), the configuration of the curves contracted by the inverse (third column), the dimension 
of its orbit under the l.r. action (fourth column) and the parameters (fifth column).

Any cubic birational map can be written, up to dynamical conjugation,~$Af$ where $A$ denotes an element of $\mathrm{PGL}_3(\mathbb{C})$ and $f$ an element of the previous table. 
This classification allows us to prove that the \og generic\fg\hspace{0.1cm} element has the last configuration and allows us to establish that the dimension of the space 
$\mathring{\mathrm{B}}\mathrm{ir}_3$ of birational maps purely of degree $3$ is $18.$ Up to l.r. conjugation the elements having the generic configuration $\mathsf{\{15\}}$ form 
a family of $2$ parameters: in degree $2$ there are $3$ l.r. orbits, in degree $3$ an infinite number. 

Let us note that the configurations obtained by degenerescence from picture $\mathsf{\{15\}}$ do not all appear. In degree $2$ there is a similar situation: the configuration of 
three concurrent lines is not realised as the exceptional set of a quadratic birational map.

Let us denote by $\mathscr{X}$ the set of birational maps purely of degree $3$ having configuration $\mathsf{\{15\}}.$ We establish that the closure of $\mathscr{X}$ in 
$\mathring{\mathrm{B}}\mathrm{ir}_3$ is $\mathring{\mathrm{B}}\mathrm{ir}_3.$ We can show that $\mathring{\mathrm{B}}\mathrm{ir}_3$ is irreducible, in fact rationally connected 
(\cite[Chapter~6]{CD}); but if $\mathrm{Bir}_2$ is smooth and irreducible $\mathrm{Bir}_3$, viewed in $\mathbb{P}^{29}(\mathbb{C})\simeq~\mathrm{Rat}_3$, 
doesn't have the same properties (\cite[Theorem 6.38]{CD}).

\bigskip\bigskip

Let us mention another result. Let $\mathrm{dJ}_d$ be the subset of $\mathrm{dJ}$ made of birational maps of degree~$d$ and let $\mathrm{V}_d$ be the subset of 
$\mathrm{Bir}(\mathbb{P}^2)$ defined by $$\mathrm{V}_d=\big\{AfB\,\big\vert\,A,\,B\in\mathrm{PGL}_3(\mathbb{C}),\,f\in\mathrm{dJ}_d\big\}.$$  The dimension of
 $\mathrm{Bir}_d$ is equal to $4d+6$ and $\mathrm{V}_d$ its unique irreducible component of maximal dimension (\cite{Nguyen}).

\chapter{Finite subgroups of the Cremona group}\label{Chap:folinv}

The study of the finite subgroups of $\mathrm{Bir}(\mathbb{P}^2)$ began in the $1870's$ with 
Bertini, Kantor and Wiman (\cite{Ber, Ka, Wiman}). Since then, many mathematicians have been 
interested in the subject, let us for example mention \cite{BaBe, Beauville, BeauvilleBlanc, Blanc, 
dF, DolgachevIskovskikh}. In $2006$ Dolgachev and Iskovkikh 
improve the results of 
Kantor and Wiman and give the description of finite subgroups of $\mathrm{Bir}(\mathbb{P}^2)$ up
to conjugacy. 
Before stating one of the key result let us introduce some notions.

Let $\mathrm{S}$ be a smooth projective surface. A \textbf{\textit{conic bundle}}\label{Chap7:ind3a} 
$\eta\colon\mathrm{S}\to\mathbb{P}^1(\mathbb{C})$ is a morphism whose generic fibers
have genus $0$ and singular fibers are the union of two lines. A surface endowed with conic
bundles is isomorphic either to $\mathbb{F}_n$, or to $\mathbb{F}_n$ blown up in a finite 
number of points, all belonging to different fibers (the number of blow-ups is exactly
the number of singular fibers).

  A surface $\mathrm{S}$ is called a \textbf{\textit{del Pezzo surface}}\label{Chap7:ind3}
if $-\mathrm{K}_\mathrm{S}$ is ample, which means that $-\mathrm{K}_\mathrm{S}\cdot \mathcal{C}>0$ 
for any irreducible curve $\mathcal{C}\subset \mathrm{S}$. Any del Pezzo surface except $\mathbb{P}^1
(\mathbb{C})\times\mathbb{P}^1(\mathbb{C})$ is obtained by blowing up~$r$ points $p_1$, 
$\ldots$, $p_r$ of $\mathbb{P}^2(\mathbb{C})$ with $r\leq 8$ and no $3$ of $p_i$ are collinear, 
no $6$ are on the same conic and no~$8$ lie on a cubic having a singular point at one of them. 
The degree of~$\mathrm{S}$ is~$9-r$.

\begin{thm}[\cite{Ma, Is2}]\label{Thm:ssgpefini}
Let $\mathrm{G}$ be a finite subgroup of the Cremona group. There exists a smooth projective
surface $\mathrm{S}$ and a birational map $\phi\colon\mathbb{P}^2(\mathbb{C})\dashrightarrow
\mathrm{S}$ such that $\phi\mathrm{G}\phi^{-1}$ is a subgroup of $\mathrm{Aut}
(\mathrm{S})$. Moreover one can assume that
\begin{itemize}
\item[$\bullet$] either $\mathrm{S}$ is a del Pezzo surface;

\item[$\bullet$] or there exists a conic bundle $\mathrm{S}\to\mathbb{P}^1(\mathbb{C})$
invariant by $\phi\mathrm{G}\phi^{-1}$.
\end{itemize}
\end{thm}

\begin{rem}
The alternative is not exclusive: there are conic bundles on del Pezzo surfaces.
\end{rem}

Dolgachev and Iskovskikh give a characterization of pairs $(\mathrm{G},\mathrm{S})$
satisfying one of the possibilities of Theorem \ref{Thm:ssgpefini}. Then they use 
Mori theory to determine when two pairs are birationally conjugate. Let us note
that the first point was partially solved by Wiman and Kantor but not the second. 
There are still some open questions (\cite{DolgachevIskovskikh} \S 9), for example 
the description
of the algebraic varieties that parametrize the conjugacy classes of the finite 
subgroups of~$\mathrm{Bir}(\mathbb{P}^2)$. Blanc gives an answer to this question
for finite abelian subgroups of $\mathrm{Bir}(\mathbb{P}^2)$ with no elements
with an inva\-riant curve of positive genus, also for elements of finite order 
(resp. cyclic subgroups of finite order) of the Cremona group (\cite{Blanc, Bl2}).

\section{Birational involutions}

\subsection{Geiser involutions}\label{Subsec:Geiser}\,

  Let $p_1$, $\ldots,$ $p_7$ be seven points of $\mathbb{P}^2(\mathbb{C})$ in general position. Let $L$ be the linear system of cubics through the $p_i$'s. A cubic is given by a 
homogeneous polynomial of degree $3$ in three variables. The dimension of the space of homogeneous polynomials of degree $3$ in $3$ variables is $10$ thus~$\dim\{C\,\vert\, C\text{ cubic}
\}=10-1=9$; cubics have to pass through
$p_1$, $\ldots$,~$p_7$ so $\dim L=2$. Let $p$ be a generic point of $\mathbb{P}^2(\mathbb{C});$ let us consider the pencil~$L_p$ containing elements of  $L$ through $p$. 
A pencil of generic cubics 
\begin{align*}
& a_0C_0+a_1C_1, && C_0,\,C_1 \text{ two cubics } && (a_0:a_1)\in\mathbb{P}^1(\mathbb{C})
\end{align*}
has nine base-points (indeed by Bezout's theorem the intersection of two cubics is $3\times 3=9$ points); so we define by $\mathcal{I}_G(p)$ the ninth base-point of~$L_p$.

  The involution $\mathcal{I}_G=\mathcal{I}_G(p_1,\ldots,p_7)$ which sends $p$ to $\mathcal{I}_G(p)$ is a \textbf{\textit{Geiser involution}}\label{Chap7:ind1}.

  We can check that such an involution is birational, of degree $8;$ its fixed points form an hyperelliptic curve of genus $3,$ degree $6$ with $7$ ordinary double points which are 
the $p_i$'s. The exceptional locus of a Geiser involution is the union of seven cubics passing through the seven points of indeterminacy of $\mathcal{I}_G$ and singular in one of 
these seven points (cubics with double point).

  The involution $\mathcal{I}_G$ can be realized as an automorphism of a del Pezzo surface of degree~$2$.

\subsection{Bertini involutions}\,

  Let $p_1,$ $\ldots,$ $p_8$ be eight points of $\mathbb{P}^2(\mathbb{C})$ in general position. Let us consider the set of sextics $\mathcal{S}=\mathcal{S}(p_1,\ldots,p_8)$ with double 
points in  $p_1,$ $\ldots,$ $p_8.$ Let $m$ be a point of $\mathbb{P}^2(\mathbb{C}).$ The pencil given by the elements of $\mathcal{S}$ having a double point in $m$ has a tenth base double
 point $m'.$ The involution which swaps~$m$ and $m'$ is a \textbf{\textit{Bertini involution}}\label{Chap7:ind4} $\mathcal{I}_B=\mathcal{I}_B(p_1,\ldots,p_8).$ 

  Its fixed points form a non hyperelliptic curve of genus $4,$ degree $9$ with triple points in the $p_i$'s and such that the normalisation is isomorphic to a singular intersection of 
a cubic surface and a quadratic cone in $\mathbb{P}^3(\mathbb{C}).$

  The involution $\mathcal{I}_B$ can be realized as an automorphism of a del Pezzo surface of degree~$1$.

\subsection{de Jonqui\`eres involutions}\,

  Let $\mathcal{C}$ be an irreductible curve of degree $\nu\geq 3.$ Assume that $\mathcal{C}$ has a unique singular point $p$ and that $p$ is an ordinary multiple point with multiplicity 
$\nu-2.$ To $(\mathcal{C},p)$ we associate a birational involution $\mathcal{I}_J$ which fixes pointwise $\mathcal{C}$ and which preserves lines through $p.$ Let $m$ be a generic point 
of $\mathbb{P}^2(\mathbb{C})\setminus\mathcal{C};$ let $r_m,$ $q_m$ and $p$ be the intersections of the line $(mp)$ and $\mathcal{C};$ the point $\mathcal{I}_J(m)$ is defined by the 
following property: the cross ratio of $m,$ $\mathcal{I}_J(m),$ $q_m$ and $r_m$ is equal to $-1.$ The map $\mathcal{I}_J$ is a \textbf{\textit{de Jonqui\`eres 
involution}}\label{Chap7:ind5} of degree $\nu$ centered in $p$ and preserving $\mathcal{C}.$ More precisely its fixed points are the curve $\mathcal{C}$ of ge\-nus~$\nu-2$ for 
$\nu\geq 3.$ 

For $\nu=2$ the curve $\mathcal{C}$ is a smooth conic; we can do the same construction by choosing a point $p$ not on $\mathcal{C}.$ 

\subsection{Classification of birational involutions}

\begin{defi}
We say that an involution is of 
\textbf{\textit{de Jonqui\`eres type}}\label{Chap7:ind7} it is birationally conjugate to a de Jonqui\`eres involution. We can also speak about involution of 
\textbf{\textit{Geiser type}}\label{Chap7:ind8}, resp. \textbf{\textit{Bertini type}}\label{Chap7:ind9}. 
\end{defi}

\begin{thm}[\cite{Ber, BaBe}]
A non-trivial birational involution of $\mathbb{P}^2(\mathbb{C})$ is either of de Jonqui\`eres type, or Bertini type, or Geiser type. 
\end{thm}

More precisely Bayle and Beauville obtained the following statement.

\begin{thm}[\cite{BaBe}]\label{Thm:invbirconjclas}
The map which associates to a birational involution of $\mathbb{P}^2$ its
normalized fixed curve establishes a one-to-one correspondence between:
\begin{itemize}
\item[$\bullet$] conjugacy classes of de Jonqui\`eres involutions of degree $d$ and isomorphism
classes of hyperelliptic curves of genus $d-2$ $(d\geq 3)$;
\item[$\bullet$] conjugacy classes of Geiser involutions and isomorphism classes of non-hyperelliptic curves of genus $3$;
\item[$\bullet$] conjugacy classes of Bertini involutions and isomorphism classes of non-hyperelliptic curves of genus $4$ 
whose canonical model lies on a singular quadric.
\end{itemize}
The de Jonqui\`eres involutions of degree $2$ form one conjugacy class.
\end{thm}

\section{Birational involutions and foliations}

\subsection{Foliations: first definitions}

  A  \textbf{\textit{holomorphic foliation}}\label{Chap7:ind10} $\mathcal{F}$ of codimension $1$ and \textbf{\textit{degree}}\label{Chap7:ind11} $\nu$ on $\mathbb{P}^2(\mathbb{C})$ is 
given by a $1$-form $$\omega=u(x,y,z)\mathrm{d}x+v(x,y,z)\mathrm{d}y+w(x,y,z)\mathrm{d}z$$ where $u,$ $v$ and $w$ are homogeneous polynomials of degree $\nu+1$ 
without common component and satisfying the Euler identity $xu+vy+wz=0.$ The \textbf{\textit{singular locus}}\label{Chap7:ind12} $\mathrm{Sing}\,\mathcal{F}$ of $\mathcal{F}$ is the 
projectivization of the singular locus of~$\omega$ $$\mathrm{Sing} \, \omega=\big\{(x,y,z)\in\mathbb{C}^3\,\big\vert\, u(x,y,z)=v(x,y,z)=w(x,y,z)=0\big\}.$$ Let us give a geometric 
interpretation of the degree. Let $\mathcal{F}$ be a foliation of degree $\nu$ on $\mathbb{P}^2(\mathbb{C}),$ let $\mathcal{D}$ be a generic line, and let $p$ a point of $\mathcal{D}
\setminus\mathrm{Sing}\,\mathcal{F}.$ We say that $\mathcal{F}$ is \textbf{\textit{transversal}}\label{Chap7:ind13} to $\mathcal{D}$ if the leaf $\mathcal{L}_p$ of $\mathcal{F}$ in $p$
 is transversal to~$\mathcal{D}$ in $p,$ otherwise we say that $p$ is a \textbf{\textit{point of tangency}}\label{Chap7:ind14} between $\mathcal{F}$ and~$\mathcal{D}.$ The degree $\nu$ of $\mathcal{F}$ is exactly 
the number of points of tangency between $\mathcal{F}$ and $\mathcal{D}.$ Indeed, if $\omega$ be a $1$-form of degree $\nu+1$ on $\mathbb{C}^3$ defining~$\mathcal{F}$, it is of the 
following type 
\begin{align*}
& \omega=P_0\mathrm{d}x+P_1\mathrm{d}y+P_2\mathrm{d}z, && P_i \text{ homogeneous polynomial of degree }\nu+1.
\end{align*}
Let us denote by $\omega_0$ the restriction of $\omega$ to the affine chart $x=1$ $$\omega_0=\omega_{\vert x=1}=P_1(1,y,z)\mathrm{d}y+P_2(1,y,z)\mathrm{d}z.$$ Assume that the line 
$\mathcal{D}=\big\{z=0\big\}$ is a generic line. In the affine chart $x=1$ the fact that the radial vector field vanishes on $\mathcal{D}$ implies that $$P_0(1,y,0)+yP_1(1,y,0)=0.$$ 
Generically (on the choice of $\mathcal{D}$) the polynomial $P_0(1,y,0)$ is of degree $\nu+1$ so $P_1(1,y,0)$ is of degree $\nu$. Since $\omega_{0\vert \mathcal{D}}=P_1(1,y,0)
\mathrm{d}y$, the restriction of~$\omega_0$ to $\mathcal{D}$ vanishes into $\nu$ points: the number of tangencies between $\mathcal{F}$ and $\mathcal{D}$ is~$\nu$.

  The classification of foliations of degree $0$ and $1$ on $\mathbb{P}^2(\mathbb{C})$ is known since the XIXth century. A foliation of degree $0$ on $\mathbb{P}^2(\mathbb{C})$ is a 
pencil of lines, {\it i.e.} is given by $x\mathrm{d}y-y\mathrm{d}x=x^2\mathrm{d}\left(\frac{y}{x}\right)$, the pencil of lines being $\frac{y}{x}=$ cte. Each foliation of degree $1$ 
on the complex projective plane has~$3$ singularities (counting with multiplicity), has, at least, one invariant line and is given by a rational closed $1$-form (in other words there 
exists a homogeneous polynomial $P$ such that~$\omega/P$ is closed); the leaves are the connected components of the \og levels\fg\, of a primitive of this $1$-form. The possible 
$1$-forms are 
\begin{align*}
& x^{\lambda_0}y^{\lambda_1}z^{\lambda_2}, \,\,\lambda_i\in\mathbb{C},\,\sum_i\lambda_i=0, && 
\frac{x}{y}\exp\left(\frac{z}{y}\right), && \frac{Q}{x^2}
\end{align*}
where $Q$ is a quadratic form of maximal rank. More generally a foliation of degree $0$ on $\mathbb{P}^n(\mathbb{C})$ is associated to a pencil of hyperplanes, {\it i.e.} is given by 
the 
levels of $\ell_1/\ell_2$ where $\ell_1$, $\ell_2$ are two independent linear forms. Let $\mathcal{F}$ be a foliation of degree $1$ on $\mathbb{P}^n(\mathbb{C})$. Then
\begin{itemize}
\item[$\bullet$] either there exists a projection $\tau\colon\mathbb{P}^n(\mathbb{C})\dashrightarrow\mathbb{P}^2(\mathbb{C})$ and a foliation of degree $1$ on $\mathbb{P}^2(\mathbb{C})$ 
such that $\mathcal{F}=\tau^*\mathcal{F}_1$,

\item[$\bullet$] or the foliation is given by the levels of $Q/L^2$ where $Q$ (resp. $L$) is of degree $2$ (resp. $1$).
\end{itemize}

  For $\nu\geq 2$ almost nothing is known except the generic nonexistence of an invariant curve (\cite{J, CLN}). Let us mention that 
\begin{itemize}
\item[$\bullet$] there exists a description of the space of foliations of degree $2$ in $\mathbb{P}^3(\mathbb{C})$ (\emph{see} \cite{CLN2});

\item[$\bullet$] any foliation of degree $2$ is birationally conjugate to another (not necessary of degree $2$) given by a linear differential 
equation $\frac{\mathrm{d}y}{\mathrm{d}x}=P(x,y)$ where $P$ is in $\mathbb{C}(x)[y]$ (\emph{see} \cite{CLNLPT}).
\end{itemize}

  A regular point $m$ of $\mathcal{F}$ is an \textbf{\textit{inflection point}}\label{Chap7:ind15} for $\mathcal{F}$ if $\mathcal{L}_m$ has an inflection point in $m.$ Let us denote 
by $\mathrm{Flex}\,\mathcal{F}$ the closure of these points. A way to find this set has been given by Pereira in~\cite{Pe}: let $$Z=E\frac{\partial}{\partial x}+F\frac{\partial}{\partial y}
+G\frac{\partial}{\partial z}$$ be a homogeneous vector field on $\mathbb{C}^3$ non colinear to the radial vector field $R=x\frac{\partial}{\partial x}+y\frac{\partial}{\partial y}
+z\frac{\partial}{\partial z}$ describing $\mathcal{F}$ ({\it i.e.} $\omega=i_Ri_Z\mathrm{d}x\wedge\mathrm{d}y\wedge\mathrm{d}z$). Let us consider $$\mathcal{H}=\left\vert
\begin{array}{ccc} x & E & Z(E)\\ y & F & Z(F)\\ z & G & Z(G) \end{array}\right\vert;$$ the zeroes of $\mathcal{H}$ is the union of $\mathrm{Flex}\mathcal{F}$ and the lines invariant by 
$\mathcal{F}$.

\subsection{Foliations of degree $2$ and involutions}

  To any foliation $\mathcal{F}$ of degree $2$ on $\mathbb{P}^2(\mathbb{C})$ we can associate a birational involution $\mathcal{I}_\mathcal{F}:$ let us consider a generic point $m$ of 
$\mathcal{F},$ since $\mathcal{F}$ is of degree~$2,$ the tangent $\mathrm{T}_m\mathcal{L}_m$ to the leaf through $m$ is tangent to $\mathcal{F}$ at a second point~$p,$ the involution
 $\mathcal{I}_\mathcal{F}$ is the map which swaps these two points. More precisely let us assume that $\mathcal{F}$ is given by the vector field~$\chi.$ The image by
 $\mathcal{I}_\mathcal{F}$ of a generic point $m$ is the point $m+s\chi(m)$ where $s$ is the unique nonzero parameter for which $\chi(m)$ and $\chi(m+s\chi(m))$ are colinear.

  Let $q$ be a singular point of $\mathcal{F}$ and let $\mathcal{P}(q)$ be the pencil of lines through~$q.$ The curve of points of tangency $\mathrm{Tang}(\mathcal{F},\mathcal{P}(q))$ 
between $\mathcal{F}$ and $\mathcal{P}(q)$ is blown down by $\mathcal{I}_\mathcal{F}$ on $q.$ We can verify that all contracted curves are of this type.

\subsubsection{Jouanolou example}\,

  The foliation $\mathcal{F}_J$ is described in the affine chart $z=1$ by $$(x^2y-1)\mathrm{d}x-(x^3-y^2)\mathrm{d}y;$$ this example is due to Jouanolou
and is the first known foliation without invariant algebraic curve. 

  We can compute $\mathcal{I}_{\mathcal{F}_J}:$ 
\begin{small}
\begin{align*}
& (xy^7+3x^5y^2z-x^8-5x^2y^4z^2+2y^3z^5+x^3yz^4-xz^7: \\
&\hspace{6mm} 3xy^5z^2+2x^5z^3-x^7y-5x^2y^2z^4+x^4y^3z+yz^7-y^8: \\
&\hspace{6mm} xy^4z^3-5x^4y^2z^2-y^7z+2x^3y^5+3x^2yz^5-z^8+x^7z).
\end{align*}
\end{small}
its degree is $8$ and $$\mathrm{Ind}\,\mathcal{I}_{\mathcal{F}_J}=\mathrm{Sing}\,\mathcal{F}_J=\big\{(\xi^j:\xi^{-2j}:1)\,\big\vert\, j=0,\ldots,6,\,\,\,\xi^7=1\big\}.$$

  As there is no invariant algebraic curve for $\mathcal{F}_J$ we have $$\mathrm{Flex}\,\mathcal{F}_J=\mathrm{Fix}\, \mathcal{I}_{\mathcal{F}_J}=2(3x^2y^2z^2-xy^5-x^5z-yz^5);$$ this curve is irreducible.

The subgroup of $\mathrm{Aut}(\mathbb{P}^2)$ which preserves a foliation $\mathcal{F}$ of $\mathbb{P}^2(\mathbb{C})$ is called the 
\textbf{\textit{isotropy group}}\label{Chap7:ind16} of $\mathcal{F};$ it is an algebraic subgroup of $\mathrm{Aut}(\mathbb{P}^2)$ denoted by $$\mathrm{Iso}\,\mathcal{F}
=\big\{\varphi\in\mathrm{Aut}(\mathbb{P}^2)\,\big\vert\,\varphi^*\mathcal{F}=\mathcal{F}\big\}.$$

  The point $(1:1:1)$ is a singular point of $\mathrm{Flex}\,\mathcal{F}_J,$ it is an ordinary double point. If we let $\mathrm{Iso}\,\mathcal{F}_J$ act, we note that each singular
 point of $\mathcal{F}_J$ is an ordinary double point of $\mathrm{Flex}\,\mathcal{F}_J$ and that $\mathrm{Flex}\,\mathcal{F}_J$ has no other singular point. Therefore $\mathrm{Flex}\,
\mathcal{F}_J$ has genus $\frac{(6-1)(6-2)}{2}-7=3.$

  The singular points of $\mathrm{Sing}\,\mathcal{F}_J$ are in general position so $\mathcal{I}_{\mathcal{F}_J}$ is a Geiser involution.

  The group $\langle\mathcal{I}_{\mathcal{F}_J},\,\mathrm{Iso}\,\mathcal{F}_J\rangle$ is a finite subgroup of $\mathrm{Bir}(\mathbb{P}^2);$ it cannot be conjugate to a subgroup of 
$\mathrm{Aut}(\mathbb{P}^2)$ because $\mathrm{Fix}\,\mathcal{I}_J$ is of genus $3.$ This group of order  $42$ appears in the classification of finite subgroups of $\mathrm{Bir}
(\mathbb{P}^2)$ (\emph{see}~\cite{DI}).

\subsubsection{The generic case}\,

  Let us recall that if $\mathcal{F}$ is of degree $\nu,$ then $\#\,\mathrm{Sing}\,\mathcal{F}=\nu^2+\nu+1$ (let us precise that points are counted with multiplicity). Thus a quadratic 
foliation has seven singular points counted with multiplicity; moreover if we choose seven points $p_1,$ $\ldots,$ $p_7$ in general position, there exists one and only one foliation 
$\mathcal{F}$ such that $\mathrm{Sing}\,\mathcal{F}=\big\{p_1,\,\ldots,\,p_7\big\}$ (\emph{see} \cite{GMK}).

\begin{thm}[\cite{CD2}]
Let $p_1,$ $\ldots,$ $p_7$ be seven points of $\mathbb{P}^2(\mathbb{C})$ in general position. Let $\mathcal{F}$ be the quadratic foliation such that $\mathrm{Sing}\,\mathcal{F}=
\big\{p_1,\,\ldots,\, p_7\big\}$ and let $\mathcal{I}_G$ be the Geiser involution associated to the $p_i$'s. Then $\mathcal{I}_G$ and $\mathcal{I}_\mathcal{F}$ coincide.
\end{thm}

\begin{cor}[\cite{CD2}]
The involution associated to a generic quadratic foliation of $\mathbb{P}^2(\mathbb{C})$ is a Geiser involution.
\end{cor}

  This allows us to give explicit examples of Geiser involutions. Indeed we can explicitely write a generic foliation of degree $2$ of $\mathbb{P}^2(\mathbb{C}):$ we can assume that 
$(0:0:1),$ $(0:1:0),$ $(1:0:0)$ and $(1:1:1)$ are singular for~$\mathcal{F}$ and that the line at infinity is not preserved by $\mathcal{F}$ so the foliation $\mathcal{F}$
is given in the affine chart $z=1$ by the vector field
\begin{align*}
& \big(x^2y+ax^2+bxy+cx+ey\big)\frac{\partial}{\partial x}+\big(xy^2+Ay^2+Bxy+Cx+Ey\big)\frac{\partial}{\partial y}
\end{align*}
with $1+a+b+c+e=1+A+B+C+E=0.$
Then the construction detailed in \ref{Subsec:Geiser} allows us to give an explicit expression for the involution~$\mathcal{I}_\mathcal{F}$.

\begin{rem}

  Let us consider a foliation $\mathcal{F}$ of degree $3$ on $\mathbb{P}^2(\mathbb{C}).$ Every generic line of~$\mathbb{P}^2(\mathbb{C})$ is tangent to $\mathcal{F}$ in three points. 
The \og application\fg\, which switches these three points is in general multivalued; we give a criterion which says when this application is birational. This allows us to give explicit 
examples of trivolutions and finite subgroups of $\mathrm{Bir}(\mathbb{P}^2)$ $($\emph{see}~\cite{CD2}$)$.
\end{rem}

\section{Number of conjugacy classes of birational maps of finite order}

The number of conjugacy classes of birational involutions in $\mathrm{Bir}(\mathbb{P}^2)$
is infinite (Theorem~\ref{Thm:invbirconjclas}). Let $n$ be a positive integer; what is 
the number~$\nu(n)$
of conjugacy classes of birational maps of order $n$ in~$\mathrm{Bir}(\mathbb{P}^2)$ ?
De Fernex gives an answer for $n$ prime (\cite{dF}); there is a complete answer in 
\cite{Bl}.

\begin{thm}[\cite{Bl}]
For $n$ even, $\nu(n)$ is infinite; this is also true for $n=3$, $5$.

For any odd integer $n\not=3$, $5$ the number of conjugacy classes $\nu(n)$
of elements of order $n$ in~$\mathrm{Bir}(\mathbb{P}^2)$ is finite. Furthermore
\begin{itemize}
\item[$\bullet$] $\nu(9)=3$;

\item[$\bullet$] $\nu(15)=9$;

\item[$\bullet$] $\nu(n)=1$ otherwise.
\end{itemize}
\end{thm}
 
Let us give an idea of the proof. 
Assume that $n$ is even. Let us consider an element $P$ of~$\mathbb{C}[x^n]$
without multiple root. Blanc proves that there exists a birational map $f$
of order $2n$ such that~$f^n$ is the involution $(x,P(x)/y)$ that fixes
the hyperelliptic curve $y^2=P(x)$. So the case $n=2$ allows to conclude
for any even $n\geq 4$.

To any elliptic curve $\mathcal{C}$ we can associate a birational map $f_\mathcal{C}$ 
of the complex projective plane whose set of fixed points is $\mathcal{C}$. 
Indeed  let us consider the smooth cubic plane curve $\mathcal{C}=
\{(x:y:z)\in\mathbb{P}^2(\mathbb{C})\,\vert\, P(x,y,z)=0\}$ where $P$ is a non-singular 
form of degree $3$ in $3$ variables. The surface $\mathrm{S}=\{(w:x:y:z)\in\mathbb{P}^3
(\mathbb{C})\,\vert\, w^3=P(x,y,z)\}$ is a del Pezzo surface of degree $3$ (\emph{see
for example} \cite{Kollar}). The map $f_\mathcal{C}\colon w\mapsto \exp(\frac{2\mathbf{i}\pi}
{3})w$ gives rise to an 
automorphism of $\mathrm{S}$ whose set of fixed points is isomorphic to $\mathcal{C}$.
Since the number of isomorphism classes of ellitpic curves is infinite the number of 
conjugacy classes in $\mathrm{Bir}(\mathbb{P}^2)$ of elements of order $3$ 
is thus also infinite. A similar construction holds for birational maps of order $5$.

To show the last part of the statement Blanc 
applies Theorem \ref{Thm:ssgpefini} to the subgroup generated by a birational map of 
odd order $n\geq 7$.

\section{Birational maps and invariant curves}\label{Sec:courbinv}

Examining to Theorem \ref{Thm:invbirconjclas}, it is not surprising that simultaneously 
Castelnuovo was interested
in birational maps that preserve curves of positive genus. Let $\mathcal{C}$ be an 
irreducible curve of $\mathbb{P}^2(\mathbb{C})$; the \textbf{\textit{inertia group}}\label{Chap7:ind20a} 
of~$\mathcal{C}$, denoted by $\mathrm{Ine}(\mathcal{C})$, is the subgroup of $\mathrm{Bir}(\mathbb{P}^2)$
that fixes pointwise $\mathcal{C}$. Let~$\mathcal{C}\subset\mathbb{P}^2(\mathbb{C})$ be a 
curve of genus $>1$, then an element of $\mathrm{Ine}(\mathcal{C})$ is either a de Jonqui\`eres
map, or a birational map of order $2$, $3$ or $4$ (\emph{see} \cite{Cas2}). This result
has been recently precised as follows.

\begin{thm}[\cite{BPV2}]
Let $\mathcal{C}\subset\mathbb{P}^2(\mathbb{C})$ be an irreducible curve of genus $>1$.
Any $f$ of~$\mathrm{Ine}(\mathcal{C})$ is either a de Jonqui\`eres map, or a birational
map of order $2$
or $3$. In the first case, if $f$ is of finite order, it is an involution.
\end{thm}

To prove this statement Blanc, Pan and Vust follow Castelnuovo's idea; they construct the
\textbf{\textit{adjoint linear system}}\label{Chap7:ind20aa} of $\mathcal{C}$: let $\pi
\colon Y\to\mathbb{P}^2(\mathbb{C})$ be an
embedded resolution of singularities of $\mathcal{C}$ and let $\widetilde{\mathcal{C}}$
be the strict transform of $\mathcal{C}$. Let $\Delta$ be the fixed part of the linear
system $\vert\widetilde{\mathcal{C}}
+\mathrm{K}_Y\vert$. If~$\vert\widetilde{\mathcal{C}}
+\mathrm{K}_Y\vert$ is neither empty, nor reduced to a divisor, $\pi_*\vert\widetilde{\mathcal{C}}
+\mathrm{K}_Y\vert\setminus\Delta$ is the adjoint linear system. By iteration they obtain that
any element $f$ of $\mathrm{Ine}(\mathcal{C})$ preserves a fibration $\mathcal{F}$ that
is rational or elliptic. If $\mathcal{F}$ is rational, $f$ is a de Jonqui\`eres map. Let us
assume that $\mathcal{F}$ is elliptic. Since~$\mathcal{C}$ is of genus $>1$ the restriction
of $f$ to a generic fiber is an automorphism with at most two fixed points: $f$ is thus
of order $2$, $3$ or $4$. Applying some classic results about automorphisms of elliptic
curves Blanc, Pan and Vust show that $f$ is of genus $2$ or $3$. Finally they note that 
this result cannot be extended to curves of genus $\leq 1$; this eventuality
has been dealt with in~\cite{Pan, Blanc2} with different technics. 

\medskip

Let us also mention results due to Diller, Jackson and Sommese that are obtained from a more
dynamical point of view.

\begin{thm}[\cite{DJS}]
Let $\mathrm{S}$ be a projective complex surface and $f$ be a birational map
on~$\mathrm{S}$. Assume that $f$ is algebraically stable and hyperbolic. 
Let $\mathcal{C}$ be a connected invariant curve of $f$. 
Then $\mathcal{C}$ is of genus $0$ or $1$.

If $\mathcal{C}$ is of genus $1$, then, after contracting some curves in $\mathrm{S}$,
there exists a meromorphic $1$-form such that 
\begin{itemize}
 \item[$\bullet$] $f^*\omega=\alpha\omega$ with $\alpha\in\mathbb{C}$,

\item[$\bullet$] and $-\mathcal{C}$ is the divisor of poles of $\omega$.
\end{itemize}
The constant $\alpha$ is determined solely by $\mathcal{C}$ and $f_{\vert\mathcal{C}}$.
\end{thm}

They are also interested in the number of irreducible components of an invariant 
curve of a birational map $f\in\mathrm{Bir}(\mathrm{S})$ where $\mathrm{S}$ 
denotes a rational surface. They prove that except in a particular case, this 
number is bounded by a quantity that only depends on $\mathrm{S}$.

\begin{thm}[\cite{DJS}]
Let $\mathrm{S}$ be a rational surface and let $f$ be a birational map
on $\mathrm{S}$. Assume that $f$ is algebraically stable and hyperbolic.
Let $\mathcal{C}\subset\mathrm{S}$ be a curve invariant by $f$. 

If one of the connected components of $\mathcal{C}$ is of genus $1$ the 
number of irreducible components of $\mathcal{C}$ is bounded by 
$\dim \mathrm{Pic}(\mathrm{S})+2$.

If every connected component of $\mathcal{C}$ has genus $0$ then 
\begin{itemize}
 \item[$\bullet$] either $\mathcal{C}$ has at most $\dim\mathrm{Pic}
(\mathrm{S})+1$ irreducible components;

\item[$\bullet$] or there exists an holomorphic map $\pi\colon\mathrm{S}\to
\mathbb{P}^1(\mathbb{C})$, unique up to automorphisms of $\mathbb{P}^1
(\mathbb{C})$, such that $\mathcal{C}$ contains exactly $k\geq 2$ distinct
fibers of $\pi$, and $\mathcal{C}$ has at most $\dim\mathrm{Pic}
(\mathrm{S})+k-1$ irreducible components.
\end{itemize}
\end{thm}

\chapter{Automorphism groups}\label{Chap:aut}

\section{Introduction}

Several mathematicians have been interested in and are still interested in the algebraic properties of the diffeomorphisms groups of manifolds. Let us for example mention the following
result. Let $\mathrm{M}$ and $\mathrm{N}$ be two smooth manifolds without boundary and let $\mathrm{Diff}^p(\mathrm{M})$ denote the group of~$\mathcal{C}^p$-diffeomorphisms of $\mathrm{M}.$ 
In $1982$ Filipkiewicz proves that if $\mathrm{Diff}^p(\mathrm{M})$ and $\mathrm{Diff}^q(\mathrm{N})$ are isomorphic as abstract groups then $p=q$ and the isomorphism is induced by a 
$\mathcal{C}^p$-diffeomorphism from~$\mathrm{M}$ to~$\mathrm{N}.$

\begin{thm}[\cite{Fi}]\label{Thm:filip}
Let $\mathrm{M}$ and $\mathrm{N}$ be two smooth manifolds without boun\-dary. Let $\varphi$ be an isomorphism between $\mathrm{Diff}^p(\mathrm{M})$ and $\mathrm{Diff}^q(\mathrm{N}).$ 
Then $p$ is equal to $q$ and there exists $\psi\colon \mathrm{M}\to\mathrm{N}$ of class~$\mathcal{C}^p$ such that
\begin{align*} 
&\varphi(f)=\psi f \psi^{-1}, && \forall f\in \mathrm{Diff}^p(\mathrm{M}).
\end{align*}
\end{thm}

There are similar statements for diffeomorphisms which preserve a vo\-lume form, a symplectic form (\cite{Banyaga, Banyaga2})... If $\mathrm{M}$ is a Riemann surface of genus larger 
than $2$, then the group of diffeomorphisms which preserve the complex structure is finite. Thus there is no hope to obtain a similar result as Theorem \ref{Thm:filip}: we can find
two distinct curves of genus $3$ whose group of automorphisms is trivial. More generally if $\mathrm{M}$ is a complex compact manifold of general type, then $\mathrm{Aut}
(\mathrm{M})$ is finite and often trivial. On the contrary let us mention two examples of homogeneous manifolds:

\begin{itemize}
\item[$\bullet$] any automorphism of $\mathrm{Aut}(\mathbb{P}^2)$ is the composition of an inner automorphism, the action of an automorphism of the field 
$\mathbb{C}$ and the involution~$u\mapsto~\transp~u^{-1}$ (\emph{see} for example \cite{Die});

\item[$\bullet$] the automorphisms group of the torus $\mathbb{C}/\Gamma$ is the semi-direct product $\mathbb{C}/\Gamma\rtimes\mathbb{Z}/2\mathbb{Z}\simeq\mathbb{R}^2/
\mathbb{Z}^2\rtimes\mathbb{Z}/2\mathbb{Z}$ for all lattices $\Gamma\not=\mathbb{Z}[\mathbf{i}],\,\mathbb{Z}[\mathbf{j}]$.
\end{itemize}

\medskip

In the first part of the Chapter we deal with the structure of the automorphisms group of the affine group $\mathrm{Aff}(\mathbb{C})$ of the complex line
(Theorem~\ref{Thm:aff}). Let us say a few words about it.
Let~$\phi$ be an automorphism of $\mathrm{Aff}(\mathbb{C})$ and let $\mathrm{G}$ be a maximal (for the inclusion) abelian subgroup of~$\mathrm{Aff}(\mathbb{C})$;
then~$\phi(\mathrm{G})$ is still a maximal abelian subgroup of $\mathrm{Aff}(\mathbb{C})$. We get the nature of $\phi$ from the precise description of the 
maximal abelian subgroups of~$\mathrm{Aff}(\mathbb{C})$.

In the second part of the Chapter we are focused on the automorphisms group of polynomial automorphisms of $\mathbb{C}^2$. Let $\phi$ be an 
automorphism of~$\mathrm{Aut}(\mathbb{C}^2)$.
Using the structure of amalgamated product of $\mathrm{Aut}(\mathbb{C}^2)$ (Theo\-rem \ref{jung}) Lamy determines the centralisers 
of the elements of $\mathrm{Aut}(\mathbb{C}^2)$ (\emph{see} \cite{La}); we thus obtain that the set of H\'enon automorphisms is preserved by 
$\phi$ (Proposition \ref{Pro:henon}). Since the elementary group $\mathtt{E}$ is maximal among the solvable subgroups of length $3$ of 
$\mathrm{Aut}(\mathbb{C}^2)$ (Proposition \ref{Pro:Esol}) we establish a property of rigidity for $\mathtt{E}$: up to
 conjugation by a polynomial automorphism of the plane $\phi(\mathtt{E})=\mathtt{E}$ (\emph{see} Proposition \ref{Pro:rigidite}). This rigidity 
allows us to characterize $\phi$.

We finish Chapter \ref{Chap:aut} with the description of $\mathrm{Aut}(\mathrm{Bir}(\mathbb{P}^2))$. Let $\phi$ be an automorphism of $\mathrm{Bir}
(\mathbb{P}^2)$. The study of the uncountable maximal abelian subgroups $\mathrm{G}$ of $\mathrm{Bir}(\mathbb{P}^2)$ leads
to the following alternative: either $\mathrm{G}$ owns an element of finite order, or $\mathrm{G}$ preserves a rational fibration (that is
$\mathrm{G}$ is, up to conjugation, a subgroup of $\mathrm{dJ}=\mathrm{PGL}_2(\mathbb{C}(y))\rtimes\mathrm{PGL}_2(\mathbb{C})$). This allows us to prove 
that $\mathrm{PGL}_3(\mathbb{C})$ is pointwise invariant by $\phi$ up to conjugacy and up to the action of an automorphism of the field $\mathbb{C}$. The
last step is to establish that $\varphi(\sigma)=\sigma$; we then conclude with Theorem \ref{nono}.

\section{The affine group of the complex line}\label{affine}

  Let $\mathrm{Aff}(\mathbb{C})=\Big\{z\mapsto az+b\,\big\vert \, a\in\mathbb{C}^*,\,b\in\mathbb{C}\Big\}$ be the affine
group of the complex line.

\begin{thm}\label{Thm:aff}
Let $\phi$ be an automorphism of $\mathrm{Aff}(\mathbb{C}).$ Then there exist $\tau$ an automorphism of the field $\mathbb{C}$
and $\psi$ an element of $\mathrm{Aff}(\mathbb{C})$ such that 
\begin{align*}
&\phi(f)=\tau(\psi f\psi^{-1}), &&\forall\,f\in\mathrm{Aff}(\mathbb{C}).
\end{align*}
\end{thm}

\begin{proof}
If $\mathrm{G}$ is a maximal abelian subgroup of $\mathrm{Aff}(\mathbb{C})$ then $\phi(\mathrm{G})$ too. The maximal abelian subgroups of $\mathrm{Aff}(\mathbb{C})$ are
\begin{align*}
&\mathrm{T}=\Big\{z\mapsto z+\alpha\,\big\vert\,\alpha\in\mathbb{C}\Big\}
&& \text{and} &&
\mathrm{D}_{z_0}=\Big\{z\mapsto\alpha(z-z_0)+z_0\,\big\vert\,\alpha\in\mathbb{C}^*\Big\}.
\end{align*}
Note that $\mathrm{T}$ has no element of finite order so $\phi(\mathrm{T})=\mathrm{T}$ and $\phi(\mathrm{D}_{z_0})=\mathrm{D}_{z_0'}.$ Up to a conjugacy by an element of~$\mathrm{T}$ 
one can suppose that $\phi(\mathrm{D}_0)=\mathrm{D}_0.$ In other words one has
\begin{itemize}
\item[$\bullet$] an additive morphism $\tau_1\colon\mathbb{C}\to\mathbb{C}$ such that
\begin{align*}
&\phi(z+\alpha)=z+\tau_1(\alpha), &&\forall\,\alpha\in\mathbb{C};
\end{align*}

\item[$\bullet$] a multiplicative one $\tau_2\colon\mathbb{C}^*\to\mathbb{C}^*$ such that
\begin{align*}
&\phi(\alpha z)=\tau_2(\alpha)z, &&\forall\,\alpha\in\mathbb{C}^*.
\end{align*}
\end{itemize}
On the one hand we have
\begin{align*}
&\phi(\alpha z+\alpha)=\phi(\alpha z)\phi(z+1)=\tau_2(\alpha)z+\tau_2(\alpha)\tau_1(1)
\end{align*}
and on the other hand
\begin{align*}
&\phi(\alpha z+\alpha)=\phi(z+\alpha)\phi(\alpha z)=\tau_2(\alpha)z+\tau_1(\alpha).
\end{align*}
Therefore $\tau_1(\alpha)=\tau_2(\alpha)\kappa$ where $\kappa=\tau_1(1).$ In particular $\tau_1$ is multiplicative and additive, {\it i.e.}~$\tau_1$ is an automorphism of the field 
$\mathbb{C}$ (and $\tau_2$ too).

  Then
\begin{eqnarray}
\phi(\alpha z+\beta)&=&\tau_2(\alpha)z+\tau_1(\beta)=\tau_2(\alpha)z+\tau_2(\beta)\kappa=\tau_2(\alpha z+\tau_2^{-1}(\kappa)\beta)\nonumber\\
&=& \tau_2(\tau_2^{-1}(\kappa)z\circ \alpha z+\beta\circ\tau_2(\kappa)z).\nonumber
\end{eqnarray}
\end{proof}

  Let us denote by $\mathrm{Aut}(\mathbb{C}^n)$ the group of polynomial automorphisms of $\mathbb{C}^n$. Ahern and Rudin show that the group of holomorphic 
automorphisms of~$\mathbb{C}^n$ and the group of holomorphic automorphisms of $\mathbb{C}^m$ have different finite subgroups when $n\not=m$ (\emph{see} \cite{Ru}); 
in particular the group of holomorphic automorphisms of $\mathbb{C}^n$ is isomorphic to the group of holomorphic automorphisms of $\mathbb{C}^m$ if and only 
if $n=m.$ The same argument holds for $\mathrm{Aut}(\mathbb{C}^n)$ and $\mathrm{Aut}(\mathbb{C}^m).$

\section{The group of polynomial automorphisms of the plane}

\subsection{Description of the automorphisms group of $\mathrm{Aut}(\mathbb{C}^2)$}\,

\begin{thm}[\cite{De}]\label{Thm:descautaut}
Let $\phi$ be an automorphism of $\mathrm{Aut}(\mathbb{C}^2).$ There exist~$\psi$ in $\mathrm{Aut}(\mathbb{C}^2)$ and an automorphism~$\tau$ of the field $\mathbb{C}$ such that 
\begin{align*}
& \phi(f)=\tau(\psi f \psi^{-1}), &&\forall\,f\in\mathrm{Aut}(\mathbb{C}^2).
\end{align*}
\end{thm}

\begin{rem}
Let us mention the existence of a similar result for the subgroup of tame automorphisms of $\mathrm{Aut}(\mathbb{C}^n)$:
every automorphism of the group of polynomial automorphisms of complex affine $n$-space inner up to field automorphisms when 
restricted to the subgroup of tame automorphisms~$($\cite{KS}$)$.
\end{rem}

The section is devoted to the proof of Theorem \ref{Thm:descautaut} which uses the well known amalgamated product structure of $\mathrm{Aut}(\mathbb{C}^2)$ (Theorem \ref{jung}).
Let us recall that a \textbf{\textit{H\'enon automorphism}} is an automorphism of the type $\varphi g_1\ldots g_p\varphi^{-1}$
\begin{align*}
& \varphi\in\mathrm{Aut}(\mathbb{C}^2),\, g_i=(y,P_i(y)-\delta_i x),\,P_i\in\mathbb{C}[y],\,\deg P_i\geq 2,\,\delta_i\in\mathbb{C}^*,
\end{align*}
and that 
$$\mathtt{A}=\big\{(a_1x+b_1y+c_1,a_2x+b_2y+c_2)\,\big\vert\, a_i,\,b_i,\,c_i\in \mathbb{C},\,a_1b_2-a_2b_1\not=0\big\},$$
$$\mathtt{E}=\big\{(\alpha x+P(y),\beta y+\gamma)\,\big\vert\, \alpha,\,\beta,\, \gamma\in\mathbb{C},\,\alpha\beta\not=0,\, 
P\in\mathbb{C}[y]\big\}.$$

  Let us also recall the two following statements.

\begin{pro}[\cite{FM}]\label{friedlandmilnor}
Let $f$ be an element of $\mathrm{Aut}(\mathbb{C}^2).$

  Either $f$ is conjugate to an element of $\mathtt{E},$ or $f$ is a H\'enon automorphism.
\end{pro}

\begin{pro}[\cite{La}]\label{lamy}
Let $f$ be a H\'enon automorphism; the centralizer of $f$ is coun\-table.
\end{pro}

  Proposition \ref{friedlandmilnor} and Proposition \ref{lamy} allow us to establish the following property:

\begin{pro}[\cite{De}]\label{Pro:henon}
Let $\phi$ be an automorphism of $\mathrm{Aut}(\mathbb{C}^2).$ Then~$\phi(\mathcal{H})=\mathcal{H}$ where $$\mathcal{H}=\Big\{f\in\mathrm{Aut}(\mathbb{C}^2)\,\big\vert\, f\text{ is a 
H\'enon automorphism}\Big\}.$$
\end{pro}

  We also have the following: for any $f$ in $\mathtt{E}$, $\phi(f)$ is up to conjugacy in~$\mathtt{E}$. But Lamy proved that a non-abelian subgroup whose each element is conjugate to 
an element of $\mathtt{E}$ is conjugate either to a subgroup of $\mathtt{A}$, or to a subgroup or $\mathtt{E}$. So we will try to "distinguish" $\mathtt{A}$ and $\mathtt{E}$.

  We set $\mathtt{E}^{(1)}=[\mathtt{E},\mathtt{E}]=\Big\{(x,y)\mapsto(x+P(y),y+\alpha)\,\big\vert\,\alpha\in\mathbb{C},\,P\in\mathbb{C}[y]\Big\}$ and $$\mathtt{E}^{(2)}=[
\mathtt{E}^{(1)},\mathtt{E}^{(1)}]=\Big\{(x,y)\mapsto(x+P(y),y)\,\big\vert\, P\in\mathbb{C}[y]\Big\}.$$

  The group $\mathtt{E}^{(2)}$ satisfies the following property.

\begin{lem}[\cite{De}]\label{ee}
The group $\mathtt{E}^{(2)}$ is a maximal abelian subgroup of $\mathtt{E}.$
\end{lem}

\begin{proof}
Let $K\supset\mathtt{E}^{(2)}$ be an abelian group. Let $g=(g_1,g_2)$ be in $K$. For any polynomial $P$ and for any $t$ in~$\mathbb{C}$ let us set $f_{tP}=(x+tP(y),y)$. We have
$$(\star)\,\,\,\,\, f_{tP}g=gf_{tP}.$$ If we consider the derivative of $(\star)$ with respect to $t$ at $t=0$ we obtain 
\begin{align*}
&(\diamond)\,\,\,\,\,\frac{\partial g_1}{\partial x}P(y)=P(g_2), &&(\diamond\diamond)\,\,\,\,\,\frac{\partial g_2}{\partial x}P(y)=0.
\end{align*}
The equality $(\diamond\diamond)$ implies that $g_2$ depends only on $y$. Thus from $(\star\star)$ we get: $\frac{\partial g_1}{\partial x}$ is a function of $y$, {\it i.e.} 
$\frac{\partial g_1}{\partial x}=R(y)$ and $g_1(x,y)=R(y)x+Q(y)$. As $g$ is an automorphism, $R$ is a constant $\alpha$ which is non-zero. Then $(\star\star)$  can be rewritten 
$\alpha P(y)=P(g_2)$. For $P\equiv 1$ we obtain that $\alpha=1$ and for $P(y)=y$ we have $g_2(y)=y$. In other words $g=(x+Q(y),y)$ belongs to~$\mathtt{E}^{(2)}$.
\end{proof}

Let $\mathrm{G}$ be a group; set 
\begin{align*}
&\mathrm{G}^{(0)}=\mathrm{G}, && \mathrm{G}^{(1)}=[\mathrm{G},\mathrm{G}],\,\ldots,\,\mathrm{G}^{(p)}=[\mathrm{G}^{(p-1)},\mathrm{G}^{(p-1)}],\,\ldots
\end{align*}
The group~$\mathrm{G}$ is \textbf{\textit{solvable}} if there exists an integer $k$ such that $\mathrm{G}^{(k)}=\mathrm{id};$ the smallest integer $k$ such that~$\mathrm{G}^{(k)}=
\mathrm{id}$ is the \textbf{\textit{length}} of $\mathrm{G}$. The Lemma~\ref{ee} allows us to establish the following statement.

\begin{pro}[\cite{De}]\label{Pro:Esol}
The group $\mathtt{E}$ is maximal among the solvable subgroups of $\mathrm{Aut}(\mathbb{C}^2)$ of length~$3.$
\end{pro}

\begin{proof}
Let $K$ be a solvable group of length $3$. Assume that $K\supset\mathtt{E}$. The group $K^{(2)}$ is abelian and contains $\mathtt{E}^{(2)}$. As $\mathtt{E}^{(2)}$ is maximal, 
$K^{(2)}=\mathtt{E}^{(2)}$. The group $K^{(2)}$ is a normal subgroup of $K$ so for all $f=(f_1,f_2)\in K$ and $g=(x+P(y),y)\in K^{(2)}=\mathtt{E}^{(2)}$ we have 
\begin{align*}
&(\star)\,\,\,\,\, f_1(x+P(y),y)=f_1(x,y)+\Theta(P)(f_2(x,y)) \\
&(\star\star)\,\,\,\,\, f_2(x+P(y),y)=f_2(x,y)
\end{align*}
where $\Theta\colon\mathbb{C}[y]\to\mathbb{C}[y]$ depends on $f$. The second equality implies that $f_2=f_2(y)$. The deri\-vative of $(\star)$ with respect to $x$ implies 
$\frac{\partial f_1}{\partial x}(x+P(y),y)=\frac{\partial f_1}{\partial x}(x,y)$ thus $\frac{\partial f_1}{\partial x}=R(y)$ and 
\begin{align*}
&f_1(x,y)=R(y)x+Q(y), && Q,\, R\in \mathbb{C}[y].
\end{align*} 
As $f$ is an automorphism we have $f_1(x,y)=\alpha x+Q(y)$, $\alpha\not=0$. In other words $K=\mathtt{E}$.
\end{proof}

  This algebraic characterization of $\mathtt{E}$ and the fact that a non-abelian subgroup whose each ele\-ment is conjugate to an element of $\mathtt{E}$ is conjugate either to 
a subgroup of $\mathtt{A}$ or to a subgroup or $\mathtt{E}$ (\emph{see} \cite{La}) allow us to establish a rigidity property concerning $\mathtt{E}$.

\begin{pro}[\cite{De}]\label{Pro:rigidite}
Let $\phi$ be an automorphism of $\mathrm{Aut}(\mathbb{C}^2).$ There exists a polynomial automorphism $\psi$ of $\mathbb{C}^2$ such that $\phi(\mathtt{E})=\psi \mathtt{E}\psi^{-1}.$
\end{pro}

  Assume that $\phi(\mathtt{E})=\mathtt{E}$; we can show that $\phi(\mathrm{D})=\mathrm{D}$ and $\phi(\mathrm{T}_i)=\mathrm{T}_i$ where 
\begin{align*}
& \mathrm{D}=\Big\{(x,y)\mapsto(\alpha x,\beta y)\,\big\vert\,\alpha,\,\beta\in\mathbb{C}^*\Big\},
\end{align*}
\begin{align*}
& \mathrm{T}_1=\Big\{(x,y)\mapsto(x+\alpha,y)\,\big\vert\,\alpha\in\mathbb{C}\Big\}, && \mathrm{T}_2=\Big\{(x,y)\mapsto(x,y+\beta)\,\big\vert\,\beta\in\mathbb{C}\Big\}.
\end{align*}

  With an argument similar to the one used in \S\ref{affine} we obtain the following statement.

\begin{pro}[\cite{De}]
Let $\phi$ be an automorphism of $\mathrm{Aut}(\mathbb{C}^2).$ Then up to inner conjugacies and up to the action of an automorphism of the field~$\mathbb{C}$ the group $\mathtt{E}$ 
is pointwise invariant by $\phi.$
\end{pro}

  It is thus not difficult to check that if $\mathtt{E}$ is pointwise invariant, then $\phi(x,x+y)=(x,x+y).$ We conclude using the following fact: $\mathtt{E}$ and~$(x,x+y)$ generate 
$\mathrm{Aut}(\mathbb{C}^2).$

\subsection{Corollaries}\,

\begin{cor}[\cite{De}]
An automorphism $\phi$ of $\mathrm{Aut}(\mathbb{C}^2)$ is inner if and only if for any $f$ in~$\mathrm{Aut}(\mathbb{C}^2)$ we have $$\mathrm{jac}\, \phi(f)=\mathrm{jac}\, f$$ where 
$\mathrm{jac}\, f$ is the determinant of the jacobian matrix of $f.$
\end{cor}

\begin{proof}
There exists an automorphism $\tau$ of the field $\mathbb{C}$ and a polynomial automorphism $\psi$ such that for any polynomial automorphism $f$ we have $\phi(f)=\tau(\psi^{-1}f\psi)$. 
Hence $$\mathrm{jac}\,\phi(f)=\mathrm{jac}\,\tau(f)=\tau(\mathrm{jac}\,f),$$ so $\mathrm{jac}\,\phi(f)=\mathrm{jac}\,f$ for any $f$ if and only if $\tau$ is trivial.
\end{proof}

\begin{cor}\label{Cor:end}
An isomorphism of the semi-group $\mathrm{End}(\mathbb{C}^2)$ in itself is inner up to the action of an automorphism of the field $\mathbb{C}.$
\end{cor}

\begin{proof}
Let $\phi$ be an isomorphism of the semi-group $\mathrm{End}(\mathbb{C}^2)$ in itself; $\phi$ induces an automorphism of $\mathbb{C}^2$. We can assume that, up to the action 
of an inner automorphism and up to the action of an automorphism of the field $\mathbb{C}$, the restriction of $\phi$ to $\mathrm{Aut}(\mathbb{C}^2)$ is trivial 
(Theorem~\ref{Thm:descautaut}).

For any $\alpha$ in $\mathbb{C}^2$, let us denote by $f_\alpha$ the constant endomorphism of $\mathbb{C}^2$, equal to $\alpha$. For any~$g$ in $\mathrm{End}(\mathbb{C}^2)$ we have 
$f_\alpha g=f_\alpha$. This equality implies that $\phi$ sends constant endomorphisms onto constant endomorphisms; this defines an invertible map $\kappa$ from $\mathbb{C}^2$ 
into itself such that $\phi(f_\alpha)=f_{\kappa(\alpha)}$. Since $gf_\alpha=f_{g(\alpha)}$ for any $g$ in $\mathrm{End}(\mathbb{C}^2)$ and any $\alpha$ in $\mathbb{C}^2$ 
we get: $\phi(g)=\kappa g\kappa^{-1}$. The restriction $\phi_{\vert\mathrm{Aut}(\mathbb{C}^2)}$ is trivial so $\kappa$ is trivial.
\end{proof}

\section{The Cremona group}

\subsection{Description of the automorphisms group of $\mathrm{Bir}(\mathbb{P}^2)$}

\begin{thm}[\cite{De2}]\label{keypoint}
Any automorphism of the Cremona group is the composition of an inner automorphism and an automorphism of the field $\mathbb{C}.$ 
\end{thm}

  Let us recall the definition of a \textbf{\textit{foliation on a compact complex surface}}. Let $\mathrm{S}$ be a compact complex surface; let
$(\mathcal{U}_i)$ be a collection of open sets which cover $\mathrm{S}.$  A foliation $\mathcal{F}$ on $\mathrm{S}$ is given by a family $(\chi_i)_i$ of holomorphic vector 
fields with isolated zeros defined on the $\mathcal{U}_i's.$ The vector fields $\chi_i$ satisfy some conditions 
\begin{align*}
& \text{ on } \mathcal{U}_i\cap\mathcal{U}_j \text{ we have } \chi_i=g_{ij}\chi_j,&& g_{ij}\in\mathcal{O}^*(\mathcal{U}_i\cap\mathcal{U}_j).
\end{align*}
Note that a non trivial vector field $\chi$ on $\mathrm{S}$ defines such a foliation.

  The keypoint of the proof of Theorem \ref{keypoint} is the following Lemma.

\begin{lem}[\cite{De2}]
Let $\mathrm{G}$ be an uncountable maximal abelian subgroup of~$\mathrm{Bir}(\mathbb{P}^2).$ There exists a rational vector
field $\chi$ such that
\begin{align*}
&f_*\chi=\chi, &&\forall\,f\in\mathrm{G}.
\end{align*}
In particular $\mathrm{G}$ preserves a foliation.
\end{lem}

\begin{proof}
The group $\mathrm{G}$ is uncountable so there exists an integer $n$ such that
\begin{align*}
\mathrm{G}_n=\big\{f\in\mathrm{G}\,\big\vert\,\deg f=n\big\}
\end{align*}
is uncountable. Then the Zariski's closure $\overline{\mathrm{G}_n}$ of $\mathrm{G}_n$ in
\begin{align*}
\mathrm{Bir}_n=\big\{f\in\mathrm{Bir}(\mathbb{P}^2)\,\big\vert\,\deg f\leq n\big\}
\end{align*}
is an algebraic set and $\dim\overline{\mathrm{G}_n}\geq 1.$ Let us consider a curve in $\overline{\mathrm{G}_n},$ {\it i.e.} a map
\begin{align*}
&\eta\colon\mathbb{D}\to\overline{\mathrm{G}_n}, && t\mapsto \eta(t).
\end{align*}

  Remark that the elements of $\overline{\mathrm{G}_n}$ are commuting birational maps.

   For each $p$ in $\mathbb{P}^2(\mathbb{C})\setminus\mathrm{Ind}\,\eta(0)^{-1}$ set
\begin{align*}
\chi(p)=\frac{\partial \eta(s)}{\partial s}\Big|_{s=0}(\eta(0)^{-1}(p)).
\end{align*}
This formula defines a rational vector field on $\mathbb{P}^2(\mathbb{C})$ which is non identically zero. By derivating the equality $f\eta(s)f^{-1}(p)=\eta(s)(p)$ we obtain 
$f_*\chi=\chi.$ Then $\chi$ is invariant by $\overline{\mathrm{G}_n};$ we note that in fact $\chi$ is invariant by $\mathrm{G}.$
\end{proof}

  So take an uncountable maximal abelian subgroup $\mathrm{G}$ of $\mathrm{Bir}(\mathbb{P}^2)$ without periodic element and an automorphism $\phi$ of $\mathrm{Bir}(\mathbb{P}^2).$ 
Then $\phi(\mathrm{G})$ is an uncountable maximal abelian subgroup of~$\mathrm{Bir}(\mathbb{P}^2)$ which preserves a foliation~$\mathcal{F}.$

  Let $\mathcal{F}$ be an holomorphic singular foliation on a compact complex projective surface $\mathrm{S}.$ Such foliations have been classified up to birational equi\-valence by 
Brunella, McQuillan and Mendes~(\cite{Br2, McQ, Me}). Let $\mathrm{Bir}(\mathrm{S}, \mathcal{F})$ (resp. $\mathrm{Aut}(\mathrm{S},\mathcal{F})$) be the group of birational (resp. 
biholomorphic) symmetries of $\mathcal{F},$ {\it i.e.} mappings $g$ which send leaf to leaf. For a foliation $\mathcal{F}$ of general type,~$\mathrm{Bir}(\mathrm{S},\mathcal{F})=
\mathrm{Aut}(\mathrm{S},\mathcal{F})$ is a finite group. In \cite{CF} the authors classify those tri\-ples~$(\mathrm{S},\mathcal{F},g)$ for which $\mathrm{Bir}(\mathrm{S},
\mathcal{F})$ (or $\mathrm{Aut}(\mathrm{S},\mathcal{F})$) is infinite. The classification leads to five classes of foliations listed below:
\begin{itemize}
\item[$\bullet$] $\mathcal{F}$ is left invariant by a holomorphic vector field;

\item[$\bullet$] $\mathcal{F}$ is an elliptic fibration;

\item[$\bullet$] $\mathrm{S}=\mathscr{T}/\mathrm{G}$ is the quotient of a complex $2$-torus $\mathscr{T}$ by a finite group and $\mathcal{F}$ is the projection of the stable 
foliation of some Anosov diffeomorphism of $\mathscr{T};$

\item[$\bullet$] $\mathcal{F}$ is a rational fibration;

\item[$\bullet$] $\mathcal{F}$ is a monomial foliation on $\mathbb{P}^1(\mathbb{C})\times\mathbb{P}^1(\mathbb{C})$ (or on the desingularisation of the quotient $\mathbb{P}^1(
\mathbb{C})\times\mathbb{P}^1(\mathbb{C})$ by the involution $(z,w)\mapsto(1/z,1/w)$).
\end{itemize}

  We prove that as $\phi(\mathrm{G})$ is uncountable, maximal and abelian without periodic element, $\mathcal{F}$ is a rational fibration\footnote{Here a rational fibration is 
a rational application from $\mathbb{P}^2(\mathbb{C})$ into $\mathbb{P}^1(\mathbb{C})$ whose fibers are rational curves.}. In other words $\phi(\mathrm{G})$ is up to conjugacy 
a subgroup of 
\begin{align*}
\mathrm{dJ}=\mathrm{PGL}_2(\mathbb{C}(y))\rtimes\mathrm{PGL}_2(\mathbb{C}).
\end{align*}

  The groups
\begin{align*}
&\mathrm{dJ}_a=\Big\{(x,y)\mapsto(x+a(y),y)\,\big\vert\,a\in\mathbb{C}(y)\Big\} 
\end{align*}
and 
\begin{align*}
&\mathrm{T}=\Big\{(x,y)\mapsto(x+\alpha,y+\beta)\,\big\vert\,\alpha,\,\beta\in
\mathbb{C}\Big\}
\end{align*}
are uncountable, maximal, abelian subgroups of the Cremona group; moreover they have no periodic element. So $\phi(\mathrm{dJ}_a)$ and $\phi(\mathrm{T})$ are contained 
in~$\mathrm{dJ}.$ After some computations and algebraic considerations we obtain that, up to conjugacy (by a birational map),
\begin{align*}
&\phi(\mathrm{dJ}_a)=\mathrm{dJ}_a && \text{and} && \phi(\mathrm{T})=\mathrm{T}.
\end{align*}

  As $\mathrm{D}=\Big\{(\alpha x,\beta y)\,\big\vert\,\alpha,\,\beta\in\mathbb{C}^*\Big\}$ acts by conjugacy on $\mathrm{T}$ we establish that $\phi(\mathrm{D})=\mathrm{D}.$ 
After conjugating $\phi$ by an inner automorphism and an automorphism of the field $\mathbb{C}$ the groups $\mathrm{T}$ and~$\mathrm{D}$ are pointwise invariant by~$\phi.$ 
Finally we show that $\phi$ preserves $(y,x)$ and $\left(\frac{1}{x},\frac{1}{y}\right);$ in particular we use the following identity due to Gizatullin (\cite{Gi})
\begin{align*}
&(h\sigma)^3=\mathrm{id}, && h=\left(\frac{x}{x-1},\frac{x-y}{x-1}\right).
\end{align*}

Since $\mathrm{Bir}(\mathbb{P}^2)$ is generated by $\mathrm{Aut}(\mathbb{P}^2)=\mathrm{PGL}_3(\mathbb{C})$ and $\left(\frac{1}{x},\frac{1}{y}\right)$
(Theorem \ref{nono}) we have after conjugating $\phi$ by an inner automorphism and an automorphism of the field $\mathbb{C}$:~$\phi_{\vert\mathrm{Bir}
(\mathbb{P}^2)}=~\mathrm{id}.$

\bigskip

  We will give another proof of Theorem \ref{keypoint} in Chapter \ref{Chap:Zimmer}.

\subsection{Corollaries}\,

We obtain a similar result as Corollary \ref{Cor:end}.

\begin{cor}[\cite{De2}]
An isomorphism of the semi-group of the rational maps from $\mathbb{P}^2(\mathbb{C})$ into itself is inner up to the action of an automorphism of the field $\mathbb{C}$.
\end{cor}

We also can prove the following statement.

\begin{cor}[\cite{De2}]
Let $\mathrm{S}$ be a complex projective surface and let $\varphi$ be an isomorphism between $\mathrm{Bir}(\mathrm{S})$ and $\mathrm{Bir}(\mathbb{P}^2)$. There exists a 
birational map $\psi\colon\mathrm{S}\dashrightarrow\mathbb{P}^2(\mathbb{C})$ and an automorphism of the field $\mathbb{C}$ such that 
\begin{align*}
& \varphi(f)=\tau(\psi f\psi^{-1}) &&\forall\,\,\,f\in\mathrm{Bir}(\mathrm{S}).
\end{align*}
\end{cor}

\chapter{Cremona group and Zimmer conjecture}\label{Chap:Zimmer}

\section{Introduction}

  In the $80$'s Zimmer suggests to generalise the works of Margulis on the linear representations of the lattices of simple, real Lie groups of real rank strictly greater 
than $1$ (\emph{see} \cite{[Ma], [VGS]}) to the non-linear ones. He thus establishes a program containing several conjectures (\cite{[Zi3], [Zi4], [Zi1], [Zi2]}); among 
them there is the following one. 
\bigskip

  {\bf Conjecture (Zimmer).} Let $\mathrm{G}$ be a real, simple, connected Lie group and let $\Gamma$ be a lattice of~$\mathrm{G}$. If there exists a morphism of infinite 
image from $\Gamma$ into the diffeomorphisms group of a compact manifold $\mathrm{M}$, the real rank of $\mathrm{G}$ is bounded by the dimension of $\mathrm{M}$.

\bigskip

  There are a lot of results about this conjecture (\emph{see} for example \cite{[Gh], [Wi], [Gh2], [BuMo2], [BuMo], [Na], [Po], [FrHa], [Ca]}). In the case of the Cremona 
group we have the following statement.

\begin{thm}[\cite{De4}]\label{gamma}
$1)$ The image of an embedding of a subgroup of finite index of $\mathrm{SL}_3(\mathbb{Z})$ into $\mathrm{Bir}(\mathbb{P}^2)$ is, up to conjugation, a subgroup of 
$\mathrm{PGL}_3(\mathbb{C}).$

More precisely let $\Gamma$ be a subgroup of finite index of $\mathrm{SL}_3(\mathbb{Z})$ and let $\rho$ be an embedding of~$\Gamma$ into $\mathrm{Bir}(\mathbb{P}^2)$. 
Then $\rho$ is, up to conjugation, either the canonical embedding or the involution $u\mapsto\transp(u^{-1})$.

$2)$ Let $\Gamma$ be a subgroup of finite index of $\mathrm{SL}_n(\mathbb{Z})$ and let $\rho$ be an embedding of $\Gamma$ into the Cremona group. If $\rho$ has infinite 
image, then $n$ is less or equal to $3$.
\end{thm}

In the same context Cantat proves the following statement.

\begin{thm}[\cite{Can3}]
Let $\Gamma$ be an infinite countable subgroup of $\mathrm{Bir}(\mathbb{P}^2)$. Assume that $\Gamma$ has Kazhdan's property\footnote{Let us recall
that $\mathrm{G}$ has Kazhdan's property if any continuous affine isometric action of $\mathrm{G}$ on a real Hilbert space has a fixed point.}; then up 
to birational conjugacy $\Gamma$ is a subgroup of $\mathrm{PGL}_3(\mathbb{C})$.
\end{thm}

The proof uses the tools presented in Chapter \ref{Chap:as} and in particular Theo\-rem \ref{clastypefini}. Let us give an idea of the proof: since $\Gamma$ has 
Kazhdan property the image of $\Gamma$ by any $\rho\colon\Gamma\to\mathrm{Bir}(\mathbb{P}^2)$ is a subgroup of $\mathrm{Bir}(\mathbb{P}^2)$ whose all 
elements are elliptic.
According to Theorem \ref{clastypefini} we have the following alternative: either $\rho(\Gamma)$ is conjugate to a subgroup of $\mathrm{PGL}_3(\mathbb{C})$,
or $\rho(\Gamma)$ preserves a rational fibration that implies that $\rho$ has finite image (Lemma~\ref{fi}).

  Let $\tau$ be an automorphism of the field $\mathbb{C}$~; we can associate to a birational map $f$ the birational map $\tau(f)$ obtained by the action of $\tau$ on the 
coefficients of $f$ given in a fixed system of homogeneous coordinates. Theorem~\ref{gamma} allows us to give another proof of the following result.

\begin{thm}[\cite{De2}]\label{autbirp2}
Let $\phi$ be an automorphism of the Cremona group. There exist a birational map $\psi$ and an automorphism $\tau$ of the field $\mathbb{C}$ such that 
\begin{align*}
&\phi(f)=\tau(\psi f\psi^{-1}), &&\forall\,\, f \in \mathrm{Bir}(\mathbb{P}^2).
\end{align*}
\end{thm}

  The Cremona group has a lot of common points with linear groups ne\-vertheless we have the following statement.

\begin{pro}[\cite{CD}]\label{nonlin}
The Cremona group cannot be embedded into~$\mathrm{GL}_n(\Bbbk)$ where $\Bbbk$ is a field of characteristic zero. 
\end{pro}

  First let us recall a result of linear algebra due to Birkhoff.

\begin{lem}[\cite{Bi}]\label{birkhoff}
Let $\Bbbk$ be a field of characteristic zero and let $A,$ $B,$~$C$ be three elements of~$\mathrm{GL}_n(\Bbbk)$ such that $[A,B]=C,$
$[A,C]=[B,C]=\mathrm{id}$ and~$C^p=\mathrm{id}$ with $p$ prime. Then $p\leq n.$
\end{lem}

\begin{proof}[Proof of Proposition \ref{nonlin}]
Assume that there exists an embedding $\varsigma$ of the Cremona group into $\mathrm{GL}_n(\Bbbk)$. For all prime $p$ let us consider in the affine chart $z=1$ the group 
\begin{align*}
\langle\left(\exp\left(-\frac{2\mathbf{i}\pi}{p}\right)x,y\right),\,(x,xy),\,\left(x,\exp\left(\frac{2\mathbf{i}\pi}{p}\right)y\right)\rangle.
\end{align*}
The images by $\varsigma$ of the three generators satisfy Lemma \ref{birkhoff} so $p\leq n$~; as it is possible for every prime $p$ we obtain a contradiction.
\end{proof}

\bigskip

This Chapter is devoted to the proof of Theorem \ref{gamma}. Let us describe the steps of the proof. First of all let us assume to simplify that 
$\Gamma=\mathrm{SL}_3(\mathbb{Z})$. Let $\rho$ denote an embedding of $\Gamma$ into $\mathrm{Bir}(\mathbb{P}^2)$. The group $\mathrm{SL}_3
(\mathbb{Z})$ contains many Heisenberg groups, {\it i.e.} groups having the following presentation
$$\mathcal{H}=\langle f,g,h\,\vert\, [f,g]=h,\,[f,h]=[g,h]=\mathrm{id}\rangle.$$ The key Lemma (Lemma \ref{degdyn}) says if $\varsigma$ is 
an embedding of $\mathcal{H}$ into $\mathrm{Bir}(\mathbb{P}^2)$ then $\lambda(\varsigma(h))=1$. Then either $\varsigma(h)$ is 
an elliptic birational map, or $\varsigma(h)$ is a de Jonqui\`eres or Halphen twist (Theorem \ref{dillerfavre}). Using the well-known presentation 
of $\mathrm{SL}_3(\mathbb{Z})$ (Proposition \ref{Pro:relations}) we know that the image of any generator $e_{ij}$ of $\mathrm{SL}_3(\mathbb{Z})$ 
satisfies this alternative; moreover the relations satisfied by the $e_{ij}$'s imply the following alternative
\begin{itemize}
 \item[$\bullet$] one of the $\rho(e_{ij})$ is a de Jonqui\`eres or Halphen twist;

 \item[$\bullet$] any $\rho(e_{ij})$ is an elliptic birational map.
\end{itemize}
In the first situation $\rho(\mathrm{SL}_3(\mathbb{Z}))$ thus preserves a rational or elliptic fibration that never happen because of the 
group properties of $\mathrm{SL}_3(\mathbb{Z})$ (Proposition~\ref{noninj}). In the second situation the first step is to prove that the Heisenberg group 
$\langle\rho(e_{12}),\,\rho(e_{13}),\,\rho(e_{23})\rangle$ is, up to finite index and up to conjugacy, a subgroup of $\mathrm{Aut}(\mathrm{S})$ 
where $\mathrm{S}$ is either $\mathbb{P}^2(\mathbb{C})$, or a Hirzebuch surface~(\S \ref{Sec:heisenberg}). In both cases we will prove that $\rho(\Gamma)$ is 
up to conjugacy a subgroup of~$\mathrm{Aut}(\mathbb{P}^2)=\mathrm{PGL}_3(\mathbb{C})$ (Lemmas \ref{PGL}, \ref{hirz}).

\section{First Properties}

\subsection{Zimmer conjecture for the group $\mathrm{Aut}(\mathbb{C}^2)$}\label{cala}\,

Let us recall the following statement that we use in the proof of Theorem~\ref{gamma}.

\begin{thm}[\cite{[Ca-La]}]\label{cantatlamy}
Let $\mathrm{G}$ be a real Lie group and let $\Gamma$ be a lattice of $\mathrm{G}$. If there exists embedding of~$\Gamma$ into the group of polynomial automorphisms of the plane, then
 $\mathrm{G}$ is isomorphic either to $\mathrm{PSO} (1,n)$ or to~$\mathrm{PSU}(1,n)$ for some integer $n$.
\end{thm}

  \textbf{Idea of the proof} (for details \emph{see} \cite{[Ca-La]}). The proof of this result uses the amalgamated pro\-duct structure of $\mathrm{Aut}(\mathbb{C}^2)$ (Theorem \ref{jung}). 
Let us recall 
that the group of affine automorphisms is given by $$\mathtt{A}= \Big\{(x,y)\mapsto(a_1x+b_1y+c_1,a_2x+b_2y+c_2)\,\big\vert\, a_i,\, b_i,\, c_i\in\mathbb{C},\, a_1b_2-a_2b_1\not=0
\Big\}$$ and the group of elementary automorphisms by $$\mathtt{E}=\Big\{(x,y)\mapsto (\alpha x+P(y),\beta y+\gamma)\,\big\vert\,\alpha,\,\beta\in\mathbb{C}^*,\,\gamma\in\mathbb{C},\, 
P\in\mathbb{C}[y]\Big\}.$$ 

\begin{thm}[\cite{Ju, La2}]
The group $\mathrm{Aut}(\mathbb{C}^2)$ is the amalgamated product of $\mathtt{A}$ and $\mathtt{E}$ along~$\mathtt{A}\cap~\mathtt{E}$.
\end{thm}

  There exists a tree on which $\mathrm{Aut}(\mathbb{C}^2)$ acts by translation (Bass-Serre theory, \emph{see} \S \ref{Sec:autpoly})~; the stabilizers of the vertex of the tree 
are conjugate either to $\mathtt{A}$ or to $\mathtt{E}$. So if a group $\mathrm{G}$ can be embedded into~$\mathrm{Aut}(\mathbb{C}^2)$, then~:

\begin{itemize}
\item[$\bullet$] either $\mathrm{G}$ acts on a tree without fixing a vertex; 

\item[$\bullet$] or $\mathrm{G}$ embeds into either $\mathtt{A}$ or $\mathtt{E}$.
\end{itemize}

  Using this fact, Cantat and Lamy study the embeddings of Kazhdan groups (\emph{see} \cite{[HV]}, chapter~I or \cite{[Ma]}, chapter III) having (FA) property and thus the 
embeddings of lattices of Lie groups with real rank greater or equal to~$2$.

\subsection{The groups $\mathrm{SL}_n(\mathbb{Z})$}\,

  Let us recall some properties of the groups $\mathrm{SL}_n(\mathbb{Z})$ (\emph{see} \cite{[St]} for more details).

  For any integer $q$ let us denote by $\Theta_q\,\colon\,\mathrm{SL}_n(\mathbb{Z})\to\mathrm{SL}_n(\mathbb{Z}/q \mathbb{Z})$ the morphism which sends $\mathrm{M}$ onto 
$\mathrm{M}$ modulo $q$. Let $\Gamma_n(q)$ be the kernel of $\Theta_q$ and let~$\widetilde{\Gamma}_n(q)$ be the reciprocical image of the diagonal group of 
$\mathrm{SL}_n(\mathbb{Z}/q \mathbb{Z})$ by~$\Theta_q$~; the $\Gamma_n(q)$ are normal subgroups of $\mathrm{SL}_n(\mathbb{Z})$, called \textbf{\textit{congruence groups}}. 

\begin{thm}[\cite{BaMiSe}]\label{cie}
Let $n\geq 3 $ be an integer and let $\Gamma$ be a subgroup of~$\mathrm{SL}_n(\mathbb{Z})$.

  If $\Gamma$ is of finite index, there exists an integer $q$ such that $\Gamma$ contains a subgroup $\Gamma_n(q)$ and is contained in~$\widetilde{\Gamma}_n(q)$.

  If $\Gamma$ is of infinite index, then $\Gamma$ is central and, in particular, finite.
\end{thm}

  Let $\delta_{ij}$ be the Kronecker matrix $3\times 3$ and let us set $e_{ij}=\mathrm{id}+\delta_{ij}.$

\begin{pro}\label{Pro:relations}
The group $\mathrm{SL}_3(\mathbb{Z})$ admits the following presentation~:  $$\langle e_{ij,\, i\not=j}\,\vert\, [e_{ij},e_{k\ell}]=\left\{ 
\begin{array}{lll}
\mathrm{id} \text{ if } i\not=\ell \,\&\, j\not=k \\
e_{i\ell} \text{ if } i\not=\ell \,\&\, j=k \\
e_{kj}^{-1} \text{ if } i=\ell \,\&\, j\not=k
\end{array},\, (e_{12}e_{21}^{-1}e_{12})^4=\mathrm{id}
\right.\rangle$$
\end{pro}

  The $e_{ij}^q$ generate $\Gamma_3(q)$ and satisfy equalities similar to those verified by the $e_{ij}$ except $(e_{12}e_{21}^{-1}e_{12})^4=~\mathrm{id}$~; we will call 
them \textbf{\textit{standard generators}}\label{Chap5:ind1} of $\Gamma_3(q)$. The system of roots of $\mathfrak{sl}_3(\mathbb{C})$ is of type 
$\mathrm{A_2}$ (\emph{see} \cite{[FuHa]})~:

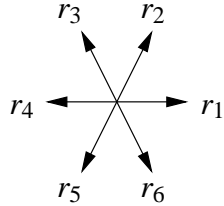
\begin{figure}[H]
\begin{center}
\input{sl.pstex_t}
\end{center}
\end{figure}

  Each standard generator of a $\Gamma_3(q)$ is an element of the group of one parameter associated to a root $r_i$ of the system~; the system of roots thus allows us to 
find most of the relations which appear in the presentation of~$\mathrm{SL}_3(\mathbb{Z})$. For example $r_1+r_3=r_2$ corresponds to $[e_{12},e_{23}]=e_{13}$, the relation 
$r_2+r_4=r_3$ to $[e_{13},e_{21}]=e_{23}^{-1}$ and the fact that $r_1+r_2$ is not a root to $[e_{12},e_{13}]=\mathrm{id}$.

\subsection{Heisenberg groups}

\begin{defi}\label{Chap5:ind2}
Let $k$ be an integer. We call \textbf{\textit{$k$-Heisenberg group}} a group with the presentation~: $$\mathcal{H}_k=\langle \mathrm{f},\mathrm{g},\mathrm{h} \,\vert\,
[\mathrm{f},\mathrm{h}]=[\mathrm{g},\mathrm{h}]= \mathrm{id},\,[\mathrm{f},\mathrm{g}]=\mathrm{h}^k\rangle.$$ By convention $\mathcal{H}=\mathcal{H}_1~;$ it is a Heisenberg group.
\end{defi}

  Let us remark that the Heisenberg group generated by $\mathrm{f}$, $\mathrm{g}$ and $\mathrm{h}^k$ is a subgroup of index $k$ of~$\mathcal{H}_k$. We  call $\mathrm{f}$,
 $\mathrm{g}$ and $\mathrm{h}$ the \textbf{\textit{standard generators}}\label{Chap5:ind3} of $\mathcal{H}_k$.

\begin{rem}\label{gen}
Each $e_{ij}^{q^2}$ can be written as the commutator of two $e_{k\ell}^q$ with whom it commutes. The group~$\mathrm{SL}_3(\mathbb{Z})$ thus contains a lot of $k$-Heisenberg 
groups~; for example $\langle e_{12}^q,e_{13}^q, e_{23}^q\rangle$ is one $($for $k=q)$.
\end{rem}

\section{Representations of Heisenberg groups}\label{Sec:heisenberg}\,

  As we said the groups $\mathrm{SL}_n(\mathbb{Z})$ contain Heisenberg groups, we thus naturally study the re\-presentations of those ones in the automorphisms groups of Hirzebruch 
surfaces and of $\mathbb{P}^2(\mathbb{C})$. Let us begin with some definitions and properties.

\begin{defi}\label{virtid}\label{Chap5:ind0}
Let $\mathrm{S}$ be a compact complex surface. The birational map $f\,\colon\, \mathrm{S}\dasharrow \mathrm{S}$ is an \textbf{\textit{elliptic birational map}} if there exist a birational 
map $\eta\,\colon\, \mathrm{S}\dasharrow \widetilde{\mathrm{S}}$ and an integer $n>0$ such that $\eta f^n\eta^{-1}$ is an automorphism of~$\widetilde{\mathrm{S}}$ isotopic to the 
identity $(${\it i.e.} $\eta f^n\eta^{-1}\in\mathrm{Aut}^0(\mathrm{S}))$. 

Two birational maps $f$ and $g$ on $\mathrm{S}$ are \textbf{\textit{simultaneously elliptic}} if the pair $(\eta,\widetilde{\mathrm{S}})$ is common to~$f$ and $g$. 
\end{defi}

\begin{rem}\label{iso}
Let $C_1$ and $C_2$ be two irreducible homologous curves of ne\-gative auto-intersection then $C_1$ and $C_2$ coincide. Thus an automorphism~$f$ of $\mathrm{S}$ isotopic to the identity fixes 
each curve of negative self-intersection; for any sequence of blow-downs $\psi$ from $\mathrm{S}$ to a minimal model $\widetilde{\mathrm{S}}$ of $\mathrm{S},$ the element $\psi 
f\psi^{-1}$ is an automorphism of $\widetilde{\mathrm{S}}$ isotopic to the identity.
\end{rem}

\begin{lem}[\cite{De4}]\label{commut}
Let $f$ and $g$ be two birational elliptic maps on a surface~$\mathrm{S}$. Assume that $f$ and $g$ commute; then $f$ and $g$ are simultaneously elliptic. 
\end{lem}

\begin{proof}
By hypothesis there exist a surface $\widetilde{\mathrm{S}}$, a birational map $\zeta\,\colon\, \mathrm{S}\dasharrow\widetilde{\mathrm{S}}$ and an integer~$n$ such that $\zeta^{-1}
f^n\zeta$ is an automorphism of $\widetilde{\mathrm{S}}$ isotopic to the identity. Let us work on $\widetilde{\mathrm{S}}~;$ to simplify we will still denote by $f$ (resp. $g$) the 
automorphism $\zeta^{-1}f^n\zeta$ (resp. $\zeta^{-1}g\zeta$).

  First let us prove that there exists a birational map $\eta\, \colon\, Y \dashrightarrow\widetilde{\mathrm{S}}$ such that~$\eta^{-1}f^ \ell\eta$ is an automorphism of~$Y$ isotopic 
to the identity for some integer $\ell$ and that~$\eta^{-1}g\eta$ is algebraically stable. Let us denote by $N(g)$ the minimal number of blow-ups needed to make $g$ algebraically stable.

  If $N(g)$ is zero, then we can take $\eta=\mathrm{id}$.

  Assume that the result is true for the maps $f$ and $g$ satisfying $N(g)\leq j$; let us consider the pair~$(\widetilde{f}, \widetilde{g})$  and assume that it satisfies the 
assumption of the statement and that $N(\widetilde{g})=j+1$. As $\widetilde{g}$ is not algebraically stable, there exists a curve $V$ in $\mathrm{Exc}\,\widetilde{g}$ and an 
integer $q$ such that $\widetilde{g}^q( V)$ is a point of indeterminacy $p$ of $\widetilde{g}$. As $\widetilde{f}$ and $\widetilde{g}$ commute, $\widetilde{f}^k$ fixes the 
irreducible components of $\mathrm{Ind}\,\widetilde{g}$ for some integer $k$. Let us consider $\kappa$ the blow-up of $p$; this point being fixed by $\widetilde{f}^k$, on the 
one hand $\kappa^{-1}\widetilde{f}^k\kappa$ is an automorphism and on the other hand $N(\kappa^{-1}\widetilde{g} \kappa)=j$. Then, by induction, there exists $\eta\,\colon\, 
Y\dashrightarrow\widetilde{\mathrm{S}}$ and $\ell$ such that $\eta^{-1}\widetilde{f}^\ell \eta$ is an automorphism isotopic to the identity and that $\eta^{-1} \widetilde{g}\eta$ 
is algebraically stable. \bigskip

  Let us set $\overline{f}=\eta^{-1}f^\ell\eta$ and $\overline{g}=\eta^{-1}g\eta$. Using \cite{DiFa}, Lemma~$4.1$, the maps $f$ and~$g$ are simultaneously elliptic. 
Indeed the first step to get an automorphism from $\overline{g}$ is to consider the blow-down $\varepsilon_1$ of a curve of~$\mathrm{Exc}\,
\overline{g}^{-1}~;$ as the curves contracted by $\overline{g}^{-1}$ are of negative self-intersection and as $\overline{f}$ is isotopic to the identity, these curves are fixed 
by $\overline{f}$ so by~$\varepsilon_1\overline{f} \varepsilon_1^{-1}$. The $i$-th step is to repeat the first one with $\varepsilon_{i-1} \ldots \varepsilon_1\overline{f}
\varepsilon_1^{-1}\ldots \varepsilon_{i-1}^{-1}$ and~$\varepsilon_{i-1} \ldots$ $\varepsilon_1\overline{g}\varepsilon_1^{-1}\ldots \varepsilon_{i-1}^{-1}$, we then obtain the 
result. According to \cite{DiFa} the process ends and a power of $\varepsilon^{-1}g\varepsilon$ is isotopic to the identity.
\end{proof}

  We have a similar result for the standard generators of a $k$-Heisenberg group. 

\begin{pro}[\cite{De4}]\label{pui}
Let $\varsigma$ be a representation of $\mathcal{H}_k$ into the Cremona group. Assume that each standard generator of $\varsigma(\mathcal{H}_k)$ is elliptic.
Then $\varsigma(\mathrm{f})$, $\varsigma( \mathrm{g})$ and~$\varsigma(\mathrm{h})$ are simultaneously elliptic.
\end{pro}

\begin{proof}
According to Lemma \ref{commut} the maps $\varsigma(\mathrm{f})$ and $\varsigma(\mathrm{h})$ are simultaneously elliptic. Since~$\mathrm{g}$ and $\mathrm{h}$ 
commute, $\mathrm{Exc}\,\varsigma(\mathrm{g})$ and $\mathrm{Ind}\,\varsigma(\mathrm{g})$ are invariant by $\varsigma(\mathrm{h}).$ The relation $[\mathrm{f},\mathrm{g}]=
\mathrm{h}^k$ implies that $\mathrm{Exc}\,\varsigma(\mathrm{g})$ and $\mathrm{Ind}\,\varsigma(\mathrm{g})$ are invariant by~$\varsigma(\mathrm{f}).$ Using the idea of the 
proof of Lemma \ref{commut} and (\cite{DiFa}, Lemma $4.1$), we obtain the result. 
\end{proof}

  In the sequel we are interested in the representations of $\mathcal{H}_k$ in the automorphisms groups of minimal surfaces which are $\mathbb{P}^1(\mathbb{C})\times
\mathbb{P}^1(\mathbb{C})$, $\mathbb{P}^2(\mathbb{C})$ and the Hirzebruch surfaces $\mathbb{F}_m$. In an affine chart $(x,y)$ of such a surface~$\mathrm{S}$, if $f$ is an 
element of $\mathrm{Bir}(\mathrm{S})$, we will denote $f$ by its two components $(f_1(x,y),f_2(x,y))$. Let us recall that in some good affine charts we have $$\mathrm{Aut}
(\mathbb{P}^1(\mathbb{C})\times\mathbb{P}^1(\mathbb{C}))=(\mathrm{PGL}_2(\mathbb{C})\times \mathrm{PGL}_2(\mathbb{C}))\rtimes(y,x)$$ and 
\begin{equation}\label{autfm}
\mathrm{Aut}(\mathbb{F}_m)=\Big\{\left(\frac{\zeta x+P(y)}{(cy+d)^m},
\frac{ay+b}{cy+d}\right)\,\Big\vert\, \left[\begin{array}{cc}
a & b\\
c & d
\end{array}\right]
\in\mathrm{PGL}_2(\mathbb{C}),\,\zeta\in\mathbb{C}^*,\, P\in\mathbb{C}[y],\,\deg P\leq
m\Big\}.
\end{equation}
\bigskip

\begin{lem}[\cite{De4}]\label{p1}
Let $\varsigma$ be a morphism from $\mathcal{H}_k$ into $\mathrm{Aut}(\mathbb{P}^1(\mathbb{C}) \times \mathbb{P}^1(\mathbb{C}))$. The morphism $\varsigma$ is not an embedding.
\end{lem}

\begin{proof}
We can assume that $\mathrm{f}$, $\mathrm{g}$ and $\mathrm{h}$ fixe the two standard fibrations (if it is not the case we can consider $\mathcal{H}_{2k}\subset\mathcal{H}_k$), 
{\it i.e.} $\mathrm{im}\,\varsigma$ is contained in $\mathrm{PGL}_2(\mathbb{C})\times\mathrm{PGL}_2(\mathbb{C})$. For $j=1$, $2$ let us denote by $\pi_j$ the $j$-th projection. 
The image of $\varsigma(\mathcal{H}_{2k})$ by $\pi_j$ is a solvable subgroup of $\mathrm{PGL}_2(\mathbb{C});$ as~$\pi_j(\varsigma(\mathrm{h}^k))$ is a commutator, this 
homography is conjugate to the transla\-tion~$z+~\beta_j$. Assume that $\beta_j$ is nonzero~; then $\pi_j(\varsigma(\mathrm{f}))$ and $\pi_j(\varsigma(\mathrm{g}))$ are 
also some translations (they commute with~$\pi_j(\varsigma(\mathrm{h}^k))$). The relation $[\pi_j(\varsigma (\mathrm{f})),\pi_j(\varsigma(\mathrm{g}))]=\pi_j(\varsigma(\mathrm{h}^k))$ 
thus implies that $\beta_j$ is zero~: contradiction. So $\beta_j$ is zero and the image of $\mathrm{h}^{2k}$ by $\varsigma$ is trivial~: $\varsigma$ is not an embedding.
\end{proof}

  Concerning the morphisms from $\mathcal{H}_k$ to $\mathrm{Aut}(\mathbb{F}_m),$ $m\geq 1,$ we obtain a different statement.  Let us note that we can see $\mathrm{Aut}
(\mathbb{C}^2)$ as a subgroup of~$\mathrm{Bir}(\mathbb{P}^2)$; indeed  any automorphism $(f_1(x,y),f_2(x,y))$ of $\mathbb{C}^2$ can be extended to a birational map: 
\begin{align*}
&(z^nf_1(x/z,y/z):z^nf_2(x/z, y/z):z^n)&& \text{where }n=\max(\deg f_1,\deg f_2).
\end{align*}

\begin{lem}[\cite{De4}]\label{highirz}
Let $\varsigma$ be a morphism from $\mathcal{H}_k$ into $\mathrm{Aut}(\mathbb{F}_m)$ with~$m\geq~1$. Then $\varsigma(\mathcal{H}_k)$ is birationally conjugate to a subgroup of 
$\mathtt{E}.$ Moreover, $\varsigma(\mathrm{h}^{2k})$ can be written $(x+P(y),y)$ where~$P$ denotes a polynomial.
\end{lem}

\begin{rem}
The abelian subgroups of $\mathrm{PGL}_2( \mathbb{C})$ are, up to conjugation, some subgroups of $\mathbb{C}$, $\mathbb{C}^*$ or the group of order $4$ 
generated by $-y$ and $\frac{1}{y}.$
\end{rem}

\begin{proof}
Let us consider the projection $\pi$ from $\mathrm{Aut}({\mathbb{F}}_m)$ into $\mathrm{PGL}_2(\mathbb{C})$. We can assume that $\pi(\varsigma(\mathcal{H}_k))$ is not conjugate 
to $\Big\{y, -y,\frac{1}{y},-\frac{1}{y}\Big\}$ (if it is the case let us consider $\mathcal{H}_{2k}$). Therefore $\pi(\varsigma(\mathcal{H}_k))$ is, up to conjugation, a 
subgroup of the group of the affine maps of the line; so $\varsigma(\mathcal{H}_k)$ is, up to conjugation, a subgroup of $\mathtt{E}$ (\emph{see}~(\ref{autfm})). The relations 
satisfied by the generators imply that $\varsigma(\mathrm{h}^{2k})$ can be written $(x+P(y),y)$.
\end{proof}

\begin{lem}[\cite{De4}]\label{heisen}
Let $\varsigma$ be an embedding of $\mathcal{H}_k$ into $\mathrm{PGL}_3(\mathbb{C})$. Up to linear conjugation, we have 
\begin{align*}
& \varsigma(\mathrm{f})=(x+\zeta y,y+\beta),&&\varsigma(\mathrm{g})=(x+\gamma y,y+\delta)&& \text{and} &&\varsigma(\mathrm{h}^k)=(x+k,y)
\end{align*}
with $\zeta\delta-\beta\gamma=k$.
\end{lem}

\begin{proof}
The Zariski closure $\overline{\varsigma(\mathcal{H}_k)}$ of $\varsigma(\mathcal{H}_k)$ is an algebraic unipotent subgroup of $\mathrm{PGL}_3(\mathbb{C})$~; as $\varsigma$ is an 
embedding, the Lie algebra of $\overline{\varsigma(\mathcal{H}_k)}$ is isomorphic to:
$$\mathfrak{h}=\left\{\left[
\begin{array}{ccc}
0 & \zeta & \beta\\
0 & 0 & \gamma\\
0 & 0 & 0
\end{array}
\right]\,\Big\vert\,\zeta,\,\beta,\,\gamma\in\mathbb{C}\right\}.$$ Let us denote by $\pi$  the canonical projection from $\mathrm{SL}_3(\mathbb{C})$ into $\mathrm{PGL}_3(\mathbb{C})$. 
The Lie algebra of~$\pi^{-1}(\overline{\varsigma(\mathcal{H}_k)})$ is, up to conjugation, equal to $\mathfrak{h}$. The exponential map sends $\mathfrak{h}$ in the group $\mathrm{H}$ 
of the upper triangular matrices which is a connected algebraic group. Therefore the identity component of $\pi^{-1}(\overline{\varsigma(\mathcal{H}_k)})$ coincides with $\mathrm{H}$. 
Any element $\mathrm{g}$ of $\pi^{-1}( \overline{\varsigma(\mathcal{H}_k)})$ acts by conjugation on~$\mathrm{H}$ so belongs to the group generated by $\mathrm{H}$ and $\mathbf{j}.
\mathrm{id}$ where $\mathbf{j}^3=\mathrm{id}.$ Since $\pi(\mathbf{j}.\mathrm{id})$ is trivial, the restriction of $\pi$ to $\mathrm{H}$ is surjective on $\overline{\varsigma(
\mathcal{H}_k)}$~; but it is injective so it is an isomorphism.  Therefore~$\varsigma$ can be lifted in a representation $\widetilde{ \varsigma}$ from $\mathcal{H}_k$
into~$\mathrm{H}$~: $$\xymatrix{\mathcal{H}_k\ar[r]^{\widetilde{\varsigma}}\ar[dr]_{\varsigma} & \mathrm{H}\ar[d]^{\pi_{\vert\mathrm{H}}}\\
&\overline{\varsigma(\mathcal{H}_k)}}$$

  As $\widetilde{\varsigma} (\mathrm{h}^k)$ can be written as a commutator, it is unipotent. The relations satisfied by the generators imply that we have up to conjugation in 
$\mathrm{SL}_3(\mathbb{C})$ $$\widetilde{\varsigma}(\mathrm{h}^k)=(x+k,y),\,\hspace{3mm}\widetilde{\varsigma}(\mathrm{f})=(x+\zeta y,y+\beta) \,\hspace{3mm}\text{and}\,\hspace{3mm}
\widetilde{\varsigma}(\mathrm{g}) =(x+\gamma y,y+\delta)$$ with $\zeta\delta-\beta\gamma=k$.
\end{proof}

\section{Quasi-rigidity of $\mathrm{SL}_3(\mathbb{Z})$}\label{SL3}

\subsection{Dynamic of the image of an Heisenberg group}\,

\begin{defi}
Let $\mathrm{G}$ be a finitely generated group,  let $\big\{a_1,\,\ldots,\,a_n\big\}$ be a part which gene\-rates~$\mathrm{G}$ and let $f$ be an element of $\mathrm{G}$.

\begin{itemize}
\item[$\bullet$] The \textbf{\textit{length}}\label{Chap5:ind4} of $f$, denoted by $\vert f\vert$, is the smallest integer $k$ such that there exists a sequence $(s_1,\ldots,s_k)$, $s_i\in\big\{a_1,
\ldots,a_n,a_1^{-1}, \ldots,a_n^{-1}\big\}$, with $f=s_1\ldots s_k$. \bigskip

\item[$\bullet$] The quantity $\displaystyle\lim_{k \to +\infty}\frac{\vert f^k\vert}{k}$ is the \textbf{\textit{stable length}}\label{Chap5:ind5} of $f$ $($\emph{see} \cite{[Ha]}$)$. \bigskip

\item[$\bullet$] An element $f$ of $\mathrm{G}$ is \textbf{\textit{distorted}}\label{Chap5:ind6} if it is of infinite order and if its stable length is zero. This notion is invariant by conjugation.
\end{itemize}
\end{defi}

\begin{lem}[\cite{De4}]\label{dist}
Let $\mathcal{H}_k=\langle\mathrm{f},\mathrm{g},\mathrm{h}\rangle$ be a $k$-Heisenberg group. The element $\mathrm{h}^k$ is distorted.
In particular the standard generators of $\mathrm{SL}_n(\mathbb{Z})$ are distorded.
\end{lem}

\begin{proof}
As $[\mathrm{f},\mathrm{h}]=[\mathrm{g},\mathrm{h}]=\mathrm{id}$, we have $\mathrm{h}^{knm}= [\mathrm{f}^n,\mathrm{g}^m]$ for any pair $(n,m)$ of integers. For~$n=~m$ we obtain 
$\mathrm{h}^{kn^2}=[\mathrm{f}^n,\mathrm{g}^n]$~; therefore $\vert\mathrm{h}^{kn^2}\vert \leq 4n$.

  Each standard generator $e_{ij}$ of $\mathrm{SL}_n(\mathbb{Z})$ can be written as follows $e_{ij}=[e_{ik},e_{kj}]$, moreover we have $[e_{ij},e_{ik}]=[e_{ij},e_{kj}]=\mathrm{id}$ 
(Remark~\ref{gen}).
\end{proof}

\begin{lem}[\cite{De4}]\label{degdyn}
Let $\mathrm{G}$ be a finitely generated group and let $\big\{a_1,\,\ldots,\,a_n\big\}$ be a set which genera\-tes~$\mathrm{G}$. Let $f$ be an element of $\mathrm{G}$ and let 
$\varsigma$ be an embedding of $G$ into $\mathrm{Bir}(\mathbb{P}^2)$. There exists a constant $m\geq 0$ such that $$1\leq\lambda(\varsigma(f))\leq\exp\left(m\frac{\vert f^n\vert}{n}
\right).$$ In particular, if $f$ is distorted, the stable length of $f$ is zero and the first dynamical degree of~$\varsigma(f)$ is $1$. 
\end{lem}

\begin{proof}
The inequalities $\lambda(\varsigma(f))^n\leq\deg \varsigma(f)^n\leq\max_i(\deg\varsigma(a_i))^{ \vert f^n\vert}$ imply $$0\leq \log\lambda(\varsigma(f))\leq\frac{\vert f^n\vert}{n}
\log(\max_i(\deg\varsigma(a_i))).$$ If $f$ is distorted, the quantity $\displaystyle\lim_{k \to \infty}\frac{\vert f^k\vert}{k}$ is zero and the first dynamical degree of $\varsigma(f)$ 
is $1$.
\end{proof}

\subsection{Notations}\,

  In the sequel, $\rho$ will denote an embedding of $\mathrm{SL}_3(\mathbb{Z})$ into $\mathrm{Bir}(\mathbb{P}^2)$. Lemmas \ref{dist} and~\ref{degdyn} imply that 
$\lambda(\rho(e_{ij}))=~1$. Thanks to Proposition \ref{Pro:relations} and Theorem~\ref{dillerfavre}, we have~:

\begin{itemize}
\item[$\bullet$] either one of the $\rho(e_{ij})$ preserves a unique fibration, rational or elliptic;

\item[$\bullet$] or each standard generator of $\Gamma_3(q)$ is an elliptic birational map. \bigskip
\end{itemize}

  We will study these two possibilities.

\subsection{Invariant fibration}\,

\begin{lem}[\cite{De4}]\label{fi}
Let $\Gamma$ be a finitely generated group with the Kazhdan's property (T). Let $\rho$ be a morphism from $\Gamma$ to $\mathrm{PGL}_2(\mathbb{C}(y))$ $($resp. $\mathrm{PGL}_2
(\mathbb{C}))$. Then the image of $\rho$ is finite.
\end{lem}

\begin{proof}
Let us denote by $\gamma_i$ the generators of $\Gamma$ and let $\left[\begin{array}{cc}a_i(y)&b_i(y)\\c_i(y)&d_i(y) \end{array}\right]$ be their image by $\rho$. A finitely generated 
$\mathbb{Q}$-group is isomorphic to a subfield of $\mathbb{C}$ so $\mathbb{Q}(a_i(y),b_i(y),c_i(y),d_i(y))$ is isomorphic to a subfield of~$\mathbb{C}$ and we can assume that 
$\mathrm{im}\,\rho\subset\mathrm{PGL}_2(\mathbb{C}) =\mathrm{Isom}(\mathbb{H}_3)$. As $\Gamma$ has property (T), each continuous action of $\Gamma$ by isometries on a real or complex 
hyperbolic space has a fixed point~; the image of~$\rho$ is thus, up to conjugacy, a subgroup of~$\mathrm{SO}_3(\mathbb{R})$. A result of Zimmer implies that the image of $\rho$ is 
finite (\emph{see}~\cite{[HV]}).
\end{proof}

\begin{pro}[\cite{De4}]\label{noninj}
Let $\rho$ be a morphism from a congruence subgroup~$\Gamma_3(q)$ of $\mathrm{SL}_3(\mathbb{Z})$ into $\mathrm{Bir}(\mathbb{P}^2)$. If one of the $\rho(e_{ij}^q)$ preserves a 
unique fibration, then the image of $\rho$ is finite. 
\end{pro}

\begin{proof}
Let us denote by $\widetilde{e}_{ij}^q$ the image of $e_{ij}^q$ by $\rho$~; Remark \ref{gen} implies that the different generators play a similar role; we can thus assume, without 
loss of generality, that $\widetilde{e}_{12}^q$ preserves a unique fibration $\mathcal{F}.$

  The relations imply that $\mathcal{F}$ is invariant by all the $\widetilde{e}_{ij}^{ q^2}$'s. Indeed as $\widetilde{e}_{12}^q$ commutes with $\widetilde{e}_{1 3}^q$ and 
$\widetilde{e}_{3 2}^q$, the elements $\widetilde{e}_{13}^q$ and $\widetilde{e}_{32}^q$ preserve $\mathcal{F}$ (it's the unicity)~; then the relation $[\widetilde{e}_{12}^q, 
\widetilde{e}_{23}^q]=\widetilde{e}_{ 13}^{q^2}$, which can also be written $\widetilde{e}_{23}^q\widetilde{e}_{12}^q\widetilde{e}_{23}^{-q}= \widetilde{e}_{13}^{q^2}
\widetilde{e}_{12}$, implies that $\widetilde{e}_{2 3}^q$ preserves $\mathcal{F}$. Thanks to $[\widetilde{e}_{12}^q,\widetilde{ e}_{31}^q]=~\widetilde{e}_{ 32}^{-q^2}$ we 
obtain that $\mathcal{F}$ is invariant by $\widetilde{e}_{31}^q$. Finally as $[\widetilde{e}_{23}^q, \widetilde{ e}_{31}^q]=~\widetilde{ e}_{21}^{q^2},$ the element 
$\widetilde{e}_{21}^{q^2}$ preserves $\mathcal{F}$.

  Then, for each $\widetilde{e}_{ij} ^{q^2}$, there exists $h_{ij}$ in $\mathrm{PGL}_2(\mathbb{C})$ and $$F\,\colon\,\mathbb{P}^2(\mathbb{C})\to\mathrm{Aut}(\mathbb{P}^1(\mathbb{C}))$$ 
defining $\mathcal{F}$ such that $F\circ \widetilde{e}_{ij}^{q^2} =h_{ij}\circ F$. Let us consider the morphism $\varsigma$ given by
\begin{align*}
&\Gamma_3(q^2)\to\mathrm{PGL}_2(\mathbb{C}) , && \widetilde{e}_{ij}^{q^2}\mapsto h_{ij}.
\end{align*}
As $\Gamma_3(q^2)$ has Kazhdan's property (T) the group $\Gamma=\ker\varsigma$ is of finite index (Lemma \ref{fi}) so it also has  Kazhdan's property (T). If $\mathcal{F}$ is rational, 
we can assume that $\mathcal{F}=(y= \text{ cte})$ where~$y$ is a coordinate in an affine chart of $\mathbb{P}^2(\mathbb{C})$~; as the group of birational maps which preserve the 
fibration $y=$ cte can be identified with $\mathrm{PGL}_2( \mathbb{C}(y))\rtimes\mathrm{PGL}_2(\mathbb{C})$, the image of $\Gamma$ by $\rho$ is contained in $\mathrm{PGL}_2(
\mathbb{C}(y))$. In this case $\rho(\Gamma)$ is thus finite (Lemma \ref{fi}) which implies that~$\rho(\Gamma_3(q^2))$ and $\rho(\Gamma_3(q))$ are also finite. The fibration~$\mathcal{F}$ 
cannot be elliptic~; indeed the group of birational maps which preserve pointwise an elliptic fibration is metabelian and a subgroup of $\Gamma_3(q^2)$ cannot be metabelian.
\end{proof}

\subsection{Factorisation in an automorphism group}\,

  Assume that every standard generator of $\mathrm{SL}_3(\mathbb{Z})$ is elliptic; in particular every standard generator of $\mathrm{SL}_3(\mathbb{Z})$ is isotopic to the identity. 
According to Remark  \ref{iso}, Proposition \ref{pui}, Lemmas \ref{dist} and \ref{degdyn}, the images of $e_{12}^n$, $e_{13}^n$ and $e_{2 3}^n$ by $\rho$ are, for some $n$, 
automorphisms of a minimal surface $\mathrm{S}$. First of all let us consider the case $\mathrm{S}=\mathbb{P}^2(\mathbb{C})$. 

\begin{lem}[\cite{De4}]\label{PGL}
Let $\rho$ be an embedding of $\mathrm{SL}_3(\mathbb{Z})$ into $\mathrm{Bir}(\mathbb{P}^2)$. If $\rho(e_{12}^n)$, $\rho(e_{13}^n)$ and~$\rho(e_{23}^n)$ belongs, for some integer $n$, 
to $\mathrm{PGL}_3(\mathbb{C})$, then $\rho(\Gamma_3(n^2))$ is a subgroup of $\mathrm{PGL}_3(\mathbb{C})$.
\end{lem}

\begin{proof}[Idea of the proof]
According to Lemma \ref{heisen} we have normal 
forms for $\rho(e_{12}^n)$, $\rho(e_{13}^n)$ and $\rho(e_{23}^n)$ up to conjugation. A computation gives the following alternative
\begin{itemize}
\item[$\bullet$] either all $\rho(e_{ij}^{n^2})$ are polynomial automorphisms of $\mathbb{C}^2$;

\item[$\bullet$] of all $\rho(e_{ij}^{n^2})$ are in $\mathrm{PGL}_3(\mathbb{C})$.
\end{itemize}
The first case cannot occur (Theorem \ref{cantatlamy}).
\end{proof}

  The following statement deals with the case of Hirzebruch surfaces. 

\begin{lem}[\cite{De4}]\label{hirz}
Let $\rho$ be a morphism from $\mathrm{SL}_3(\mathbb{Z})$ to $\mathrm{Bir}(\mathbb{P}^2)$. 
Assume that $\rho(e_{12}^n)$, $\rho(e_{13}^n)$ and $\rho(e_{23}^n)$ are, for some integer
$n$, simultaneously conjugate to some elements of~$\mathrm{Aut}(\mathbb{F}_m)$ with $m\geq 1$~; 
then the image of $\rho$ is either finite, or contained, up to conjugation, in $\mathrm{PGL}_3(\mathbb{C})$.
\end{lem}

\subsection{Proof of Theorem \ref{gamma} 1)}\label{con}\,

  According to Proposition \ref{noninj} any standard generator of $\mathrm{SL}_3(\mathbb{Z})$ is virtually isotopic to the identity. The maps $\rho( e_{12}^n)$, $\rho(e_{13}^n)$ 
and~$\rho(e_{23}^n)$ are, for some integer $n$, conjugate to automorphisms of a minimal surface $\mathrm{S}$ (Proposition \ref{pui}); we don't have to consider the case 
$\mathrm{S}=\mathbb{P}^1(\mathbb{C})\times \mathbb{P}^1(\mathbb{C})$ (Lemma~\ref{p1}). We finally obtain that~$\rho(\Gamma_3(n^2))$ is, up to conjugation, a subgroup 
of~$\mathrm{PGL}_3(\mathbb{C})$ (Lemmas \ref{PGL} and \ref{hirz}).

  The restriction of $\rho$ to $\Gamma_3(n^2)$ can be extended to an endomorphism of Lie group of $\mathrm{PGL}_3(\mathbb{C})$ (\emph{see} \cite{[St]}); as $\mathrm{PGL}_3(\mathbb{C})$ 
is simple, this extension is injective and thus  surjective. According to \cite{Die}, chapter IV, the automorphisms of $\mathrm{PGL}_3(\mathbb{C})$ are obtained from inner automorphisms,
 automorphisms of the field $\mathbb{C}$ and the involution $u\mapsto\transp(u^{-1})$~; since automorphisms of the field~$\mathbb{C}$ don't act on $\Gamma_3(n^2)$, we can assume, up to 
linear conjugation, that the restriction of $\rho$ to $\Gamma_3(n^2)$ coincides, up to conjugation, with the identity or the involution $u\mapsto\transp(u^{-1})$.

  Let $f$ be an element of $\rho(\mathrm{SL}_3(\mathbb{Z}))\setminus \rho(\Gamma_3(n^2))$ which contracts at least one curve~$\mathcal{
C}=\mathrm{Exc}\, f$. The group $\Gamma_3(n^2)$ is normal in $\Gamma~;$ therefore the curve~$\mathcal{C}$ is invariant by $\rho(\Gamma_3(n^2))$ and so by $\overline{\rho( \Gamma_3(n^2))}
=\mathrm{PGL}_3(\mathbb{C})$ (where the closure is the Zariski closure) which is impossible. So $f$ belongs to~$\mathrm{PGL}_3(\mathbb{C})$ and~$\rho(\mathrm{SL}_3(\mathbb{Z}))$ is 
contained in~$\mathrm{PGL}_3(\mathbb{C}).$

\subsection{Proof of Theorem \ref{gamma} 2)}\label{con}\,

\begin{thm}[\cite{De4}]\label{finie}
Each morphism from  a subgroup of finite index of~$\mathrm{SL}_4(\mathbb{Z})$ in the Cremona group is of finite image.
\end{thm}

\begin{proof}
Let $\Gamma$ be a subgroup of finite index of $\mathrm{SL}_4(\mathbb{Z})$ and let $\rho$ be a morphism from $\Gamma$ into~$\mathrm{Bir}(\mathbb{P}^2)$. To simplify we will assume that 
$\Gamma=\mathrm{SL}_4(\mathbb{Z})$. Let us denote by $E_{ij}$ the images of the standard generators of $\mathrm{SL}_4(\mathbb{Z})$ by~$\rho$. The morphism $\rho$ induces a faithful 
representation $\widetilde{\rho}$ from~$\mathrm{SL}_3(\mathbb{Z})$ into $\mathrm{Bir}(\mathbb{P}^2)$~:
$$\mathrm{SL}_4(\mathbb{Z})\supset \left[
\begin{array}{cc}
   \mathrm{SL}_3(\mathbb{Z}) & 0 \\
   0 & 1 \end{array}\right]
\to\mathrm{Bir}(\mathbb{P}^2).$$ According to the first assertion of Theorem \ref{gamma}, the map $\widetilde{\rho}$ is, up to conjugation, either the identity or the involution 
$u\mapsto\transp(u^{-1})$.

  Let us begin with the first case. The element $E_{34}$ commutes with $E_{31}$ and~$E_{32}$ so $\rho (E_{14})$ commutes with $(x,y,ax+by+z)$ where $a$ and $b$ are two complex numbers 
and $\mathrm{Exc}\,\rho(E_{34})$ is invariant by $(x,y,ax+by+z)$. Moreover~$E_{34}$ commutes with $E_{12}$ and $E_{21}$, in other words with the follo\-wing~$\mathrm{SL}_2(\mathbb{Z})$:
$$\mathrm{SL}_4(\mathbb{Z})\supset \left[
\begin{array}{ccc}
   \mathrm{SL}_2(\mathbb{Z}) & 0 & 0 \\
   0 & 1 & 0 \\
   0 & 0 & 1\end{array}\right]
\to\mathrm{Bir}(\mathbb{P}^2).$$ But the action of $\mathrm{SL}_2(\mathbb{Z})$ on $\mathbb{C}^2$ has no invariant curve; the curves contracted by $\rho(E_{34})$ are contained in the 
line at infinity. The image of this one by $(x,y,ax+by+z)$ intersects $\mathbb{C}^2$; so~$\mathrm{Exc}\,\rho(E_{34})$ is empty and $\rho(E_{34})$ belongs to $\mathrm{PGL}_3(\mathbb{C})$. 
With a similar argument we show that~$\rho(E_{43})$ belongs to $\mathrm{PGL}_3(\mathbb{C})$. The relations thus imply that $\rho(\Gamma_4(q))$ is in $\mathrm{PGL}_3(\mathbb{C})~;$ so 
the image of $\rho$ is finite.

  We can use a similar idea when $\widetilde{\rho}$ is the involution $u\mapsto\transp(u^{-1})$.
\end{proof}

\begin{proof}[Conclusion of the proof of Theorem \ref{gamma}] Let $n$ be an integer greater or equal to $4$ and let $\Gamma$ be a subgroup of finite index of $\mathrm{SL}_n(\mathbb{Z})$. 
Let 
$\rho$ be a morphism from~$\Gamma$ to $\mathrm{Bir}(\mathbb{P}^2)$~; let us denote by $\Gamma_n(q)$ the congruence subgroup contained in $\Gamma$ (Theorem \ref{cie}). The morphism $\rho$ 
induces a representation from  $\Gamma_4(q)$ to $\mathrm{Bir}(\mathbb{P}^2)$; according to Theorem \ref{finie} its kernel is finite, so $\ker\rho$ is finite.
\end{proof}

\section{Automorphisms and endomorphisms of the Cremona group}

  We will prove Theorem \ref{autbirp2}. To do it we will use that (Theorem \ref{nono}) $$\mathrm{Bir}(\mathbb{P}^2)=\langle 
\mathrm{Aut}(\mathbb{P}^2)=\mathrm{PGL}_3(\mathbb{C}),\,\left(\frac{1}{x},\frac{1}{y}\right)\rangle.$$

\begin{lem}[\cite{De4}]\label{impgl3}
Let $\phi$ be an automorphism of the Cremona group. If~$\phi_{\vert\mathrm{SL}_3(\mathbb{Z})}$ is trivial, then, up to the action of an automorphism of the field~$\mathbb{C}$, 
$\phi_{\vert\mathrm{PGL}_3(\mathbb{C})}$ is trivial.
\end{lem}

\begin{proof}
Let us denote by $\mathrm{H}$ the group of upper triangular matrices~:
$$\mathrm{H}=\left\{\left[
\begin{array}{ccc}
1 & a & b\\
0 & 1 & c\\
0 & 0 & 1
\end{array}
\right]\,\big\vert\, a,\, b,\, c\in\mathbb{C}\right\}.$$ The groups $\mathrm{H}$ and $\mathrm{SL}_3(\mathbb{Z})$ generate $\mathrm{PGL}_3(\mathbb{C})$ so $\mathrm{PGL}_3(\mathbb{C})$ is 
invariant by~$\phi$ if and only if~$\phi(\mathrm{H})=\mathrm{H}$. Let us set~: 
\begin{align*}
& f_b(x,y)=\phi(x+b,y),&& g_a(x,y)=\phi(x+ay,y)&& \text{and} && h_c(x,y)=\phi(x,y+c).
\end{align*}
The birational map $f_b$ (resp. $h_c$) commutes with $(x+1,y)$ and $(x,y+1)$ so $f_b$ (resp. $h_c$) can be written as $(x+\eta(b),y+\zeta(b))$ (resp. $(x+ \gamma(c), y+\beta(c))$) where 
$\eta$ and $\zeta$ (resp. $\gamma$ and $\beta$) are two additive morphisms; as~$g_a$ commute with $(x+y,y)$ and $(x+1,y)$ we have: $g_a=(x+A_a(y),y)$. The equality $$(x+ay,y)(x,y+c)
(x+ay, y)^{-1} (x,y+c)^{-1}= (x+ac,y)$$ implies that, for any complex numbers $a$ and $c$, we have: $g_ah_c=f_{ac}h_cg_a$. Therefore $f_b=(x+\eta(b),y)$, $g_a=(x+\mu(a)y+\delta(a),y)$ 
and $\mu(a)\beta(c)=\eta(ac)$. In particular $\phi(\mathrm{H})$ is contained in $\mathrm{H}$. Since $\mu(a)\beta(c)=\eta(ac)$ we have $\eta=\mu=\beta$ (because $\eta(1)=\mu(1)=
\beta(1)=1$); let us note that this equality also implies that
 $\eta$ is mul\-tiplicative.

  Let $\mathrm{T}$ denote by the group of translations in $\mathbb{C}^2$~; each element of $\mathrm{T}$ can be written $$(x+a,y)(x,y+b).$$ As $f_b$, resp. $h_c$  is of the type 
$(x+\eta(b),y)$, resp. $(x+\eta(c), y+\eta(c))$, the image of $\mathrm{T}$ by $\phi$ is a subgroup of $\mathrm{T}$. The group of translations is a maximal abelian subgroup of 
$\mathrm{Bir}(\mathbb{P}^2)$, so does $\phi(\mathrm{T})$ and the inclusion $\phi(\mathrm{T})\subset \mathrm{T}$ is an equality. The map $\eta$ is thus surjective and $\phi(\mathrm{H})= 
\mathrm{H}$. So $\phi$ induces an automorphism of  $\mathrm{PGL}_3(\mathbb{C})$ trivial on $\mathrm{SL}_3(\mathbb{Z})$. But the automorphisms of  $\mathrm{PGL}_3(\mathbb{C})$ are 
generated by inner automorphisms, automorphisms of the field $\mathbb{C}$ and the involution $u\mapsto\transp(u^{-1})$ (\emph{see} \cite{Die}). Then up to conjugation and up to the 
action of an automorphism of the field $\mathbb{C}$, $\phi_{\vert\mathrm{PGL}_3(\mathbb{C})}$ is trivial (the involution $u\mapsto\transp(u^{-1})$ on $\mathrm{SL}_3(\mathbb{Z})$ is not 
the restriction of an inner automorphism).
\end{proof}

\begin{cor}[\cite{De4}]\label{contra}
Let $\phi$ be an automorphism of the Cremona group. If $\phi_{\vert\mathrm{SL}_3(\mathbb{Z})}$ is the involution $u\mapsto\transp(u^{-1})$ then $\phi_{\vert\mathrm{PGL}_3(\mathbb{C})}$ 
also.
\end{cor}

\begin{proof}
Let us denote by $\psi$ the composition of $\phi_{\vert\mathrm{SL}_3(\mathbb{Z})}$ with the restriction $C$ of the involution $u\mapsto\transp(u^{-1})$ to $\mathrm{SL}_3( \mathbb{Z})$. The 
morphism $\psi$ can be extended to a morphism $\widetilde{\psi}$ from $\mathrm{PGL}_3(\mathbb{C})$ into $\mathrm{Bir}(\mathbb{P}^2)$ by $\widetilde{\psi}=\phi_{\vert\mathrm{PGL}_3
(\mathbb{C})}\circ C$. The kernel of $\widetilde{\psi}$ contains $\mathrm{SL}_3(\mathbb{Z})$~; as the group $\mathrm{PGL}_3(\mathbb{C})$ is simple, $\widetilde{\psi}$ is trivial.
\end{proof}

\begin{lem}[\cite{De4}]\label{invcremona}
Let $\phi$ be an automorphism of the Cremona group such that $\phi_{\vert\mathrm{PGL}_3(\mathbb{C})}$ is trivial or is the involution $u\mapsto\transp(u^{-1})$. There exist $a$,~$b$ two 
nonzero complex numbers such that $\phi(\sigma)=\left(\frac{ a}{x},\frac{b}{y}\right)$ where $\sigma$ is the involution $\left(\frac{1}{x},\frac{1}{y}\right)$.
\end{lem}

\begin{proof}
  Assume that $\phi_{\vert\mathrm{PGL}_3(\mathbb{C})}$ is trivial. The map $\phi(\sigma)$ can be written $\left(\frac{F}{x},\frac{G}{y}\right)$ where $F$ and~$G$ are rational. The 
equality $\sigma(\beta x,\mu y)=(\beta^{-1}x,\mu^{-1} y) \sigma$ implies $(F,G)(\beta x,\mu y)=(F,G)$~; as this equality is true for any pair $(\beta,\mu)$ of nonzero complex numbers, 
the functions $F$ and $G$ are constant.\bigskip

  The involution $u\mapsto\transp(u^{-1})$ preserves the diagonal group; so $\phi_{\vert\mathrm{PGL}_3(\mathbb{C})}$ coincides with $u\mapsto\transp(u^{-1})$.
\end{proof}

\medskip

\begin{proof}[Proof of Theorem \ref{autbirp2}]
Theorem \ref{gamma}, Corollary \ref{contra} and Lemma \ref{impgl3} allow us to assume that up to conjugation and up to the action of an automorphism of the field $\mathbb{C}$, 
$\phi_{\vert\mathrm{PGL}_3(\mathbb{C})}$ is trivial or is the involution $u\mapsto\transp(u^{-1})$. Assume we are in the last case and let us set $h=~(x,x-y,x-z)$~; the map 
$(h\sigma)^3$ is trivial (\emph{see}~\cite{[Gi]}). But $\phi(h)=(x+y+z,-y,-z)$ and $\phi(\sigma)=\left(\frac{a}{x},\frac{b}{y},\frac{1}{z}\right)$ (Lemma~\ref{invcremona}) so 
$\phi(h\sigma)^3\not=~\mathrm{id}$: contradiction. We thus can assume that~$\phi_{\vert\mathrm{PGL}_3(\mathbb{C})}$ is trivial~; the equality $(h\sigma)^3=\mathrm{id}$ implies 
$\phi(\sigma)=\sigma$ and Theorem \ref{nono} allows us to conclude. 
\end{proof}

  Using the same type of arguments we can describe the endomorphisms of the Cremona group.

\begin{thm}[\cite{De3}]
Let $\phi$ be a non-trivial endomorphism of $\mathrm{Bir}(\mathbb{P}^2)$. There exists an embedding $\tau$ of the field $\mathbb{C}$ into itself and a birational map~$\psi$ of 
$\mathbb{P}^2(\mathbb{C})$ such that
\begin{align*}
& \phi(f)=\tau(\psi f\psi^{-1}), && \forall\,\,\,f \in\mathrm{Bir}(\mathbb{P}^2).
\end{align*}
\end{thm}

  This allows us to state the following corollary.

\begin{cor}[\cite{De3}]
The Cremona group is hopfian: any surjective endomorphism of~$\mathrm{Bir}(\mathbb{P}^2)$ is an automorphism.
\end{cor}

\chapter{Centralizers in the Cremona group}\label{Chap:centralizer}

\section{Introduction}

  The description of the centralizers of the discrete dynamical systems is an important problem in real and complex dynamic. Julia (\cite{Julia, Julia2}) and then Ritt (\cite{Ritt}) show 
that the set $$\mathrm{Cent}(f,\mathrm{Rat}\,\mathbb{P}^1)=\big\{\psi\colon\mathbb{P}^1\to\mathbb{P}^1\,\big\vert\, f\psi=\psi f\big\}$$ of rational functions commuting with a fixed 
rational function $f$ is in gene\-ral $f_0^\mathbb{N}=\big\{f_0^n\,\big\vert\,n\in\mathbb{N}\big\}$ for some $f_0$ in~$\mathrm{Cent}(f,\mathrm{Rat}\,\mathbb{P}^1)$ except in some 
special cases (up to conjugacy $z \mapsto z^k$,  Tchebychev polynomials, Latt\`es examples...) In the $60$'s Smale asks if the centralizer of a generic diffeomorphism $f\colon \mathrm{M}
\to\mathrm{M}$ of a compact manifold is trivial, {\it i.e.} if $$\mathrm{Cent}(f,\mathrm{Diff}^\infty(\mathrm{M}))=\big\{g\in\mathrm{Diff}^\infty(\mathrm{M})\,\big\vert\, f\psi=\psi 
f\big\}$$ coincides with 
$f^\mathbb{Z}=\big\{f^n\,\big\vert\,n\in\mathbb{Z}\big\}$.  A lot of mathematicians have worked on this problem, for example Bonatti, Crovisier, Fisher, Palis, Wilkinson, Yoccoz 
(\cite{Kop, BCWi, Fish, Fish2, Pa, PaYo2, PaYo}). 

  Let us precise some of these works. In \cite{Kop} Kopell proves the existence of a dense open subset~$\Omega$ of $\mathrm{Diff}^\infty(\mathbb{S}^1)$ having the following property: 
the centralizer of any element of $\Omega$ is trivial. 

Let $f$ be a $\mathcal{C}^r$-diffeomorphism of a compact manifold $\mathrm{M}$ without boun\-dary. A point $p$ of $\mathrm{M}$ is 
\textbf{\textit{non-wandering}}\label{Chap12:ind1} if for any neighborhood $\mathcal{U}$ of $p$ and for any integer $n_0>0$ there exists an integer $n>n_0$ such that 
$f^n\mathcal{U}\cap\mathcal{U}\not=\emptyset$. The set of such points is denoted by $\Omega(f)$, it is a closed invariant set; $\Omega(f)$ is 
\textbf{\textit{hyperbolic}}\label{Chap12:ind2} if
\begin{itemize}
\item[$\bullet$] the tangent bundle of $\mathrm{M}$ restricted to $\Omega(f)$ can be written as a continuous direct sum of two subbundles $\mathrm{T}_{\Omega(f)}\mathrm{M}=E^s\oplus E^u$ 
which are invariant by the differential $\mathrm{D}f$ of $f$;

\item[$\bullet$] there exists a riemannian metric on $\mathrm{M}$ and a constant $0<\mu<1$ such that for any $p\in\Omega(f)$, $v\in E_p^s$, $w\in E_p^u$ 
\begin{align*}
&\vert\vert \mathrm{D}f_{p}v\vert\vert\leq \mu \vert\vert v\vert\vert, && \vert\vert \mathrm{D}f_{p}^{-1}w\vert\vert\leq \mu \vert\vert w\vert\vert.
\end{align*}
\end{itemize}
In this case the sets 
$$\mathrm{W}^s(p)=\big\{z\in \mathrm{M}\,\big\vert\,d(f^n(p),f^n(z))\to 0 \text{ as } n\to \infty\big\}$$ and $$\mathrm{W}^u(p)=\big\{z\in \mathrm{M}\,\big\vert\,d(f^{-n}(p),
f^{-n}(z))\to 0 \text{ as }n\to\infty\big\}$$
are some immersed submanifolds of $\mathrm{M}$ called \textbf{\textit{stable}}\label{Chap12:ind3} and \textbf{\textit{unstable manifolds}}\label{Chap12:ind4} of $p\in\Omega(f)$. 
We say that~$f$ \textbf{\textit{satisfies axiom A}}\label{Chap12:ind5} if $\Omega(f)$ is hyperbolic and if $\Omega(f)$ coincides with the closure of periodic points of~$f$ 
(\emph{see}~\cite{Sm}). Finally we impose a \textbf{\textit{"strong" transversality condition}}\label{Chap12:ind6}: for any $p\in\Omega(f)$ the stable $\mathrm{W}^s(p)$ and 
unstable~$\mathrm{W}^u(p)$ manifolds are transverse. In \cite{Pa} Palis proves that the set of diffeomorphisms of $\mathrm{M}$ satisfying axiom~A and the strong transversality condition 
contains a dense open subset $\Lambda$ such that: the centralizer of any~$f$ in $\Lambda$ is trivial. Anderson shows a similar result for the Morse-Smale diffeomorphisms (\cite{An}).

  In the study of the elements of the group $\mathrm{Diff}(\mathbb{C},0)$ of the germs of holomorphic diffeomorphism  at the origin of $\mathbb{C}$, the description of the 
centrali\-zers is very important. Ecalle proves that if~$f\in~\mathrm{Diff}(\mathbb{C},0)$ is tangent to the identity, then, except for some exceptional cases, its centra\-lizer 
is a $f_0^\mathbb{Z}$ (\emph{see} \cite{Ec, Ec2}); it allows for example to describe the solvable non abelian subgroups of~$\mathrm{Diff}(\mathbb{C},0)$ (\emph{see} \cite{CM}). 
Conversely Perez-Marco gets the existence of uncountable, non linearizable abelian subgroups of~$\mathrm{Diff}(\mathbb{C},0)$ related to some difficult questions of small 
divisors~(\cite{PM}).

  In the context of polynomial automorphisms of the plane, Lamy obtains that the centralizer of a H\'enon automorphism is almost trivial. More precisely we have the following 
statement: let $f$ be a polynomial automorphism of $\mathbb{C}^2$; then 
\begin{itemize}  
\item[$\bullet$] either $f$ is conjugate to an element of the type
\begin{align*}
&(\alpha x+P(y),\beta y+\gamma), &&P\in\mathbb{C}[y],\,\alpha,\,\beta,\,\gamma\in\mathbb{C},\,\alpha\beta\not=0
\end{align*}
and its centralizer is uncountable, 
\item[$\bullet$] or $f$ is a H\'enon automorphism $\psi g_1\ldots g_n\psi^{-1}$ where
\begin{align*}
&\psi\in\mathrm{Aut}(\mathbb{C}^2), \, g_i=(y,P_i(y)-\delta_i x),\,P_i\in\mathbb{C}[y], \,\deg P_i\geq 2,\,\delta_i\in\mathbb{C}^*
\end{align*}
and its centralizer is isomorphic to $\mathbb{Z}\rtimes\mathbb{Z}/p\mathbb{Z}$ (\emph{see} \cite[Proposition~4.8]{La}). 
\end{itemize}
We will not give the proof of Lamy but will give a ``related`` result due to Cantat (Corollary~\ref{ameliolamy})

Let us also mention the recent work \cite{DS2} of Dinh and Sibony.

\section{Dynamics and centralizer of hyperbolic diffeomorphisms}

  Let $\mathrm{S}$ be a complex surface and let $f\colon \mathrm{S}\to \mathrm{S}$ be a holomorphic map. Let~$q$ be a periodic point of period~$k$ for~$f$, {\it i.e.} $f^k(q)=q$ 
and $f^\ell(q)\not=q$ for all $1\leq\ell\leq k-1$. Let $\lambda^u(q)$ and $\lambda^s(q)$ be the eigenvalues of $\mathrm{D}f_{(q)}$. We say that $f$ is 
\textbf{\textit{hyperbolic}}\label{Chap12:ind7} if $$\vert\lambda^s(q)\vert<1<\vert\lambda^u(q)\vert.$$

  Let us denote by $\mathrm{P}_k(f)$ the set hyperbolic periodic points of period~$k$ of $f$. 

  Let us consider $q\in\mathrm{P}_k(f)$; locally around $q$ the map $f$ is well defined. We can linearize~$f^k$. The \textbf{\textit{local stable manifold}}\label{Chap12:ind8} 
$\mathrm{W}_{\text{loc}}^s(q)$ and \textbf{\textit{local unstable manifold}}\label{Chap12:ind9} $\mathrm{W}_{\text{loc}}^u(q)$ of $f^k$ in $q$ are the image by the linearizing 
map of the eigenvectors of $\mathrm{D}f^k_{q}$. To simplify we can assume that up to conjugation~$\mathrm{D}f^k_q$ is given by $\left[\begin{array}{cc}\alpha & 0 \\ 
0 & \beta\end{array}\right]$ with $\vert\alpha\vert<1<\vert\beta\vert$; there exists a holomorphic diffeomorphism $\kappa\colon(\mathcal{U},q)\to(\mathbb{C}^2,0)$ where 
$\mathcal{U}$ is a neighborhood of $q$ such that $\kappa f^k\kappa^{-1}=\left[\begin{array}{cc}\alpha & 0 \\ 0 & \beta\end{array}\right]$. Then $\mathrm{W}^s_{\text{loc}}(q)
=\kappa^{-1}(y=0)$ and~$\mathrm{W}^u_{\text{loc}}(q)=\kappa^{-1}(x=0)$:

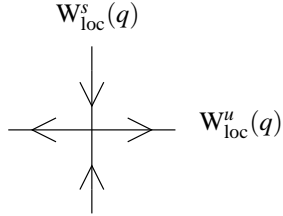
\begin{figure}[H]
\begin{center}
\input{stable.pstex_t}
\end{center}
\end{figure}

  In the sequel, to simplify, we will denote $f$ instead of $f^k$.

\begin{lem}
There exist entire curves $\xi_q^s,\,\xi_q^u\colon\mathbb{C}\to\mathrm{S}$ such that 
\begin{itemize}
\item[$\bullet$] $\xi_q^u(0)=\xi_q^s(0)=q$;

\item[$\bullet$] the \textbf{\textit{global stable}}\label{Chap12:ind10} and \textbf{\textit{global unstable manifolds}}\label{Chap12:ind11} of $f$ in $q$ are defined by
\begin{align*}
&\mathrm{W}^s(q)=\displaystyle\bigcup_{n>0}f^n(\mathrm{W}_{\text{loc}}^s(q)),&&\mathrm{W}^u(q)=\displaystyle\bigcup_{n>0}f^n(\mathrm{W}_{\text{loc}}^u(q)).
\end{align*}

\item[$\bullet$] $f(\xi_q^u(z))=\xi_q^u(\alpha^u(z))$, $f(\xi_q^s(z))=\xi_q^s(\alpha^s(z))$ for all $z\in\mathbb{C}$;

\item[$\bullet$] if $\eta_q^u\colon\mathbb{C}\to\mathrm{S}$ $($resp. $\eta_q^s\colon\mathbb{C}\to\mathrm{S})$ satisfies the first three properties, then $\eta_q^u(z)=\xi_q^u(\mu z)$ 
$($resp. $\eta_q^s(z)=\xi_q^s(\mu'z))$ for some $\mu\in\mathbb{C}^*$ $($resp.~$\mu'\in~\mathbb{C}^*)$.
\end{itemize}
\end{lem}

\begin{proof}
As we just see there exists a holomorphic diffeomorphism $\kappa\colon(\mathcal{U},q)\to\mathbb{D}$ where~$\mathcal{U}$ is a neighborhood of $q$ and $\mathbb{D}$ a small disk 
centered at the origin such that $\kappa f^k\kappa^{-1}=\left[\begin{array}{cc}\alpha & 0 \\ 0 & \beta\end{array}\right]$. Moreover $\mathrm{W}^u_{\text{loc}}(q)=\kappa^{-1}(x=0)$ 
and $\mathrm{W}^s_{\text{loc}}(q)=\kappa^{-1}(y=0)$. Let us extend $\kappa$. Let~$z$ be a point which does not belong to $\mathbb{D}$; there exist an integer $m$ such that 
$z/\alpha^m$ belongs to $\mathbb{D}$. We then set $\xi_q^u(z)=f^m\left(\kappa^{-1}\left(\frac{z}{\alpha^m}\right)\right)$. Let us note that if $\frac{z}{\alpha^m}$ and 
$\frac{z}{\alpha^k}$ both belong to $\mathbb{D}$ we have~$$f^m\left(\kappa^{-1}\left(\frac{z}{\alpha^m}\right)\right)=f^k\left(\kappa^{-1}\left(\frac{z}{\alpha^k}\right)\right)$$ 
and $\xi_q^s(z)$ is well-defined. By construction we get 
\smallskip
\begin{itemize}
\item[$\bullet$] $\xi_q^u(0)=\xi_q^s(0)=q$;

\item[$\bullet$] $\mathrm{W}^s(q)=\displaystyle\bigcup_{n>0}f^n(\mathrm{W}_{\text{loc}}^s(q))$, $\mathrm{W}^u(q)=\displaystyle\bigcup_{n>0}f^n(\mathrm{W}_{\text{loc}}^u(q)).$

\item[$\bullet$] $f(\xi_q^u(z))=\xi_q^u(\alpha^u(z))$, $f(\xi_q^s(z))=\xi_q^s(\alpha^s(z))$ for all $z\in\mathbb{C}$.
\end{itemize}

\medskip

The map $\xi_q^s$ is the analytic extension of $\kappa^{-1}_{\vert y=0}$. Let $\Delta$ be a subset of~$\big\{y=~0\big\}$ containing $0$. Set $q=\xi_q^s(1)$. Let $\eta_q^s\colon\Delta
\to\mathrm{W}_{\text{loc}}^s(q)$ be a non-constant map such that 
\begin{itemize}
\item[$\bullet$] $\eta_q^s(0)=q$,

\item[$\bullet$] $\eta_q^s(\alpha z)=f(\eta_q^s(z))$ for any $z$ in $\Delta$ such that $\alpha z$ belongs to $\Delta$.
\end{itemize}

Working with $\eta_q^s\circ(z\mapsto \mu z)$ for some good choice of $\mu$ instead of $\eta_q^s$ we can assume that~$\eta_q^s(1)=~q$. Since 
\begin{align*}
& \eta_q^s(0)=\xi_q^s(0), && \eta_q^s(1)=\xi_q^s(1), && \eta_q^s\left(\frac{1}{\alpha^n}\right)=\xi_q^s\left(\frac{1}{\alpha^n}\right)\,\,\forall\, n\in \mathbb{Z}
\end{align*}
we have $\eta_q^s=\xi_q^s$.
\end{proof}

  Let $\psi$ be an automorphism of $\mathrm{S}$ which commutes with $f$. The map $\psi$ permutes the elements of~$\mathrm{P}_k(f)$. If $\mathrm{P}_k(f)$ is finite, of cardinal $N_k>0$, 
the map $\psi^{N_k!}$ fixes any element of $\mathrm{P}_k(f)$. The stable and unstable manifolds of the points $q$ of $\mathrm{P}_k(f)$ are also invariant under the action of $\psi$. 
When the union of $\mathrm{W}^u(q)$ and~$\mathrm{W}^s(q)$ is Zariski dense in $\mathrm{S}$, then the restrictions of $\psi$ to $\mathrm{W}_{\text{loc}}^u(q)$
 and~$\mathrm{W}_{\text{loc}}^s(q)$ completely determine the map $\psi\colon \mathrm{S}\to \mathrm{S}$.

  Let us denote by $A_k$ the subgroup of $\mathrm{Cent}(f,\mathrm{Aut}(\mathrm{S}))$ which contains the automorphisms of $\mathrm{S}$ fixing any of the $N_k$ points of $\mathrm{P}_k(f)$. 
Then $\psi$ preserves~$\mathrm{W}^u(q)$ and $\mathrm{W}^s(q)$. We thus can define the morphism
\begin{align*}
&\alpha\colon A_k\to\mathbb{C}^*\times\mathbb{C}^*, &&\psi \mapsto \alpha(\psi)=(\alpha^s(\psi),\alpha^u(\psi))
\end{align*}
such that 
\begin{align*}
& \forall\, z\in\mathbb{C}, && \xi^s_q(\alpha^s(\psi)z)=\psi(\xi^s_q(z)) && \text{and} && \xi^u_q(\alpha^u(\psi)z)=\psi(\xi^u_q(z)).
\end{align*}
When the union of $\mathrm{W}^s(q)$ and $\mathrm{W}^u(q)$ is Zariski dense, this morphism is injective. In particular~$A_k$ is abelian and $\mathrm{Cent}(f,\mathrm{Aut}(\mathrm{S}))$ 
contains an abelian subgroup of finite index with index~$\leq N_k!$.

\begin{lem}[\cite{Can3}]\label{Lem:util}
The subset $\Lambda$ of $\mathbb{C}\times\mathbb{C}$ defined by $$\Lambda=\big\{(x,y)\in\mathbb{C}\times\mathbb{C}\,\big\vert\,\xi_q^u(x)=\xi_q^s(y)\big\}$$ is a discrete subset of 
$\mathbb{C}\times\mathbb{C}$.

  The set $\Lambda$ intersects $\{0\}\times\mathbb{C}$ $($resp. $\mathbb{C}\times\{0\})$ only at $(0,0)$.
\end{lem}

\begin{proof}
Let $(x,y)$ be an element of $\Lambda$ and let $m$ be the point of $\mathrm{S}$ defined by $m=\xi_q^s(x)=~\xi_q^u(y)$. In a sufficiently small neighborhood of $m$, the connected 
components of $\mathrm{W}^s(q)$ and $\mathrm{W}^u(q)$ which contain $m$ are two distinct complex submanifolds and so intersect in a finite number of points. Therefore there exist 
a neighborhood $\mathcal{U}$ of $x$ and a neighborhood $\mathcal{V}$ of $y$ such that~$\xi_q^s(\mathcal{U})\cap\xi_q^u(\mathcal{V})=\{m\}$. The point $(x,y)$ is thus the unique 
point of $\Lambda$ in $\mathcal{U}\times\mathcal{V}$ so $\Lambda$ is discrete.

  Since $\xi_q^u$ and $\xi_q^s$ are injective, we have the second assertion.
\end{proof}

\begin{pro}[\cite{Can3}]\label{Pro:centdiff}
Let $f$ be a holomorphic diffeomorphism of a connected complex surface~$\mathrm{S}$. Assume that there exists an integer $k$ such that
\begin{itemize}
\item[$\bullet$]  the set $\mathrm{P}_k(f)$ is finite and non empty;

\item[$\bullet$]  for at least one point $q$ in $\mathrm{P}_k(f)$ we have $\#\,(\mathrm{W}^s(q)\cap\mathrm{W}^u(q))\geq 2$.
\end{itemize}
Then the cyclic group generated by $f$ is of finite index in the group of holomorphic diffeomorphisms of $\mathrm{S}$ which commute to $f$.
\end{pro}

\begin{proof}
Let us take the notations introduced previously and let us set $A:=\alpha(A_k)$. Sin\-ce~$\#\,(\mathrm{W}^s(q)\cap\mathrm{W}^u(q))\geq 2$, the manifolds $\mathrm{W}^s(q)$ and 
$\mathrm{W}^u(q)$ intersect in an infinite number of points and there exists a neighborhood $\mathcal{U}$ of $q$ such that any holomorphic function on $\mathcal{U}$ which vanishes 
on~$\mathcal{U}\cap~\mathrm{W}^u(q)$ vanishes everywhere. The morphism $\alpha$ is thus injective and $\Lambda$ is a discrete and infinite subset of $\mathbb{C}\times\mathbb{C}$ 
invariant under the diagonal action of $A$.

Let us show that $A$ is discrete. Let $\overline{A}$ be the closure of $A$ in $\mathbb{C}^*\times~\mathbb{C}^*$. Since~$\Lambda$ is discrete, $\Lambda$ is $\overline{A}$-invariant. 
Let us assume that $A$ is not discrete; then $\overline{A}$ contains a $1$-parameter non-trivial subgroup of the type $t\mapsto(e^{tu},e^{tv})$. Since $\Lambda$ is discrete, one of 
the following property holds:
\begin{itemize}
\item[$\bullet$] $\Lambda=\{(0,0)\}$,

\item[$\bullet$] $u=0$ and $\Lambda\subset\mathbb{C}\times\{0\}$,

\item[$\bullet$] $v=0$ and $\Lambda\subset\{0\}\times\mathbb{C}$.
\end{itemize}
But according to Lemma \ref{Lem:util} none of this possibilities hold. So $\overline{A}$ doesn't contain a $1$-parameter non-trivial subgroup and $A$ is discrete. In particular there 
is a finite index abelian free subgroup~$A'$ of $A$ such that the rank of $A'$ is less or equal to $2$.
Since $f$ is an element of infinite order of $\mathrm{Cent}(f,\mathrm{Aut}(\mathrm{S}))$, the group $\langle f^k\rangle$ is a free subgroup of rank $1$ of 
$A_k$ so the lower bound of the rank of $A'$ is $1$ and if this lower bound is reached then~$\langle f\rangle$ is of finite index in $\mathrm{Cent}(f,\mathrm{Aut}(\mathrm{S}))$. Let 
us consider $$\exp\colon\mathbb{C}\times\mathbb{C}\to\mathbb{C}^*\times\mathbb{C}^*,$$ then $\exp^{-1}(\Lambda\cap(\mathbb{C}^*\times\mathbb{C}^*))$ is a discrete subgroup 
of~$\mathbb{C}^2\simeq\mathbb{R}^4$. Its rank is $3$ or $4$; indeed the kernel of $\exp$ contains $2\mathbf{i}\pi\mathbb{Z}\times2\mathbf{i}\pi\mathbb{Z}$ and also $(\alpha^u(f),
\alpha^s(f))$.

If $A'$ is of rank $2$, then $A'$ is a discrete and co-compact subgroup of~$\mathbb{C}^*\times\mathbb{C}^*$ and there exists an element $\psi$ in $\mathrm{Cent}(f,\mathrm{Aut}
(\mathrm{S}))$ such that 
\begin{align*}
& \vert\alpha^u(\psi)\vert<1, &&\vert\alpha^s(\psi)\vert<1, && (\alpha^u(\psi),\alpha^s(\psi))\in A.
\end{align*}
Let $(x,y)$ be a point of $\Lambda\setminus\{(0,0)\}$; the sequence $$\psi^n(x,y)=\big((\alpha^u(\psi))^nx,(\alpha^s(\psi))^ny\big)$$ is thus an infinite sequence of elements of
 $\Lambda$ and $\psi^n(x,y)\to (0,0)$ as $n\to +\infty$: contradiction. This implies that $A'$ is of rank $1$.
\end{proof}

\begin{cor}[\cite{Can3}]\label{ameliolamy}
Let $f$ be a H\'enon automorphism. The cyclic group generated by $f$ is of finite index in the group of biholomorphisms of $\mathbb{C}^2$ which commute with $f$.
\end{cor}

\begin{proof}
According to \cite{BLS} if $k$ is large enough, then the automorphism $f$ has $n>0$ hyperbolic periodic points of period $k$ whose unstable and stable manifolds intersect each other. 
Proposition \ref{Pro:centdiff} allows us to conclude.
\end{proof}

\section{Centralizer of hyperbolic birational maps}

  In this context we can also define global stable and unstable manifolds but this time we take the union of strict transforms of $\mathrm{W}_{\text{loc}}^s(q)$ and 
$\mathrm{W}_{\text{loc}}^u(q)$ by $f^n$. They are parametrized by holomorphic applications $\xi_q^u$, $\xi_q^s$ which are not necessarily injective: if a curve $\mathcal{C}$ is 
contracted on a point~$p$ by $f$ and if $\mathrm{W}^s(q)$ intersects $E$ infinitely many times, then $\mathrm{W}^s(q)$ passes through~$p$ infinitely many times.

\begin{lem}[\cite{Can3}]
Let $\Lambda$ be the set of pairs $(x,y)$ such that $\xi_q^u(x)=\xi_q^s(y)$. The set $\Lambda$ is a discrete subset of~$\mathbb{C}\times\mathbb{C}$ which intersects the 
coordinate axis only at the origin.
\end{lem}

\begin{proof}
Let $(x,y)$ be a point of $\Lambda$ and set $m=\xi^u_q(x)=\xi^s_q(y)$. The unstable and stable mani\-folds can a priori pass through $m$ infinitely many times. But since each of 
these manifolds is the union of the $f^{\pm n}(\mathrm{W}^{u/s}_{loc}(q))$, there exist two open subsets $\mathcal{U}\ni x$ and $\mathcal{V}\ni y$ of $\mathbb{C}$ and an open 
subset $\mathcal{W}$ of $\mathrm{S}$ containing $m$ such that $\xi_q^u(\mathcal{U})\cap~\mathcal{W}$ and $\xi_q^s(\mathcal{V})\cap\mathcal{W}$ are two distinct analytic curves of 
$\mathcal{W}$. We can assume that $\#\,\xi_q^u(\mathcal{U})\cap\xi_q^s(\mathcal{V})=~1$ (if it is not the case we can consi\-der~$\mathcal{U}'\subset\mathcal{U}$ and 
$\mathcal{V}'\subset\mathcal{V}$ such that $\#\,\xi_q^u(\mathcal{U}')\cap\xi_q^s(\mathcal{V}')=1$); therefore $(x,y)$ is the only point of $\Lambda$ contained in $\mathcal{U}
\times\mathcal{V}$. The set $\Lambda$ is thus discrete. Since $q$ is periodic there is no curve contracted onto $q$ by an iterate of $f$, the map $\xi_q^u$ (resp. $\xi_q^s$) doesn't 
pass again through $q$. So $\Lambda$ intersects the axis-coordinates only at $(0,0)$.
\end{proof}

  Let us recall that if a map $f$ is algebraically stable then the positive orbits $f^n(p)$, $n\geq 0$, of the elements $p$ of $\mathrm{Ind}\,f^{-1}$ do not intersect $\mathrm{Ind}\, f$. 
 We say that $f$ satisfies the \textbf{\textit{Bedford-Diller condition}}\label{Chap12:ind11a} if the sum $$\sum_{n\geq 0}\frac{1}{\lambda(f)^n}
\log(\mathrm{dist}(f^n(p),\mathrm{Ind}\,f))$$ is finite for any $p$ in $\mathrm{Ind}\,f^{-1}$; in other words the positive orbit $f^n(p)$, $n\geq 0$, of the elements $p$ of 
$\mathrm{Ind}\,f^{-1}$ does not go too fast to $\mathrm{Ind}\, f$. Note that this condition is verified by automorphisms of~$\mathbb{P}^2(\mathbb{C})$ 
or also by birational maps whose points of indeterminacy have finite orbit. Let us mention the following statement.

\begin{thm}[\cite{BD, Du}]\label{dujardin}
Let $f$ be a hyperbolic birational map of complex projective surface. Assume that $f$ satisfies the Bedford-Diller condition. Then there is a infinite number of hyperbolic periodic 
points whose stable and unstable manifolds intersect. 
\end{thm}

\subsection{Birational maps satisfying Bedford-Diller condition}\,

\begin{pro}[\cite{Can3}]
Let $f$ be a hyperbolic birational map of a complex projective surface $\mathrm{S}$. If $f$ satisfies the Bedford-Diller condition, then the cyclic subgroup generated by~$f$ is of 
finite index in the group of birational maps of $\mathrm{S}$ which commute with $f$.
\end{pro}

\begin{proof}
The set of hyperbolic periodic points of $f$ of period $k$ is a finite set. According to Theorem \ref{dujardin} there exists an integer $k$ such that 
\begin{itemize}
\item[$\bullet$] $q$ is a hyperbolic periodic point of period $k$;

\item[$\bullet$] $\mathrm{W}^s(q)$ and $\mathrm{W}^u(q)$ are Zariski dense in $\mathrm{S}$;

\item[$\bullet$] $\#\,(\mathrm{W}^s(q)\cap\mathrm{W}^u(q))$ is not finite. 
\end{itemize}

Let $\psi$ be a birational map of $\mathrm{S}$ which commutes with $f$. The map $\psi$ permutes the unstable and stable manifolds of hyperbolic periodic points of $f$ even if these 
manifolds pass through a point of indeterminacy of $\psi$. Indeed, if $q$ is a periodic point of $f$ and $\mathrm{W}^u(q)$ is Zariski-dense, then $\psi$ is holomorphic in any generic 
point of $\mathrm{W}^u(q)$ so we can extend $\psi$ analytically along $\mathrm{W}^u(q)$. Since $f$ has $\nu_k$ hyperbolic periodic points of period $k$, there exists a subgroup $B_k$ 
of $\mathrm{Cent}(f,\mathrm{Bir}(\mathrm{S}))$ of index less than $\nu_k !$; any element of~$B_k$ fixes $\mathrm{W}^s(q)$ and $\mathrm{W}^u(q)$. More precisely there exists a morphism
\begin{align*}
&\alpha\colon B_k\to\mathbb{C}^*\times\mathbb{C}^*, && \psi\mapsto (\alpha^u(\psi),\alpha^s(\psi))
\end{align*}
such that $\psi(\xi_q^{u/s}(z))=\xi_q^{u/s}(\alpha^{u/s}(\psi)z)$ for any $\psi$ of $B_k$ and for any $z$ of $\mathbb{C}$ such that $\psi$ is holomorphic on a neighborhood of 
$\xi_q^{u/s}(z)$.

As $\mathrm{W}^s(q)$ and $\mathrm{W}^u(q)$ are Zariski dense, $\alpha$ is injective. Then we can apply the arguments of Proposition \ref{Pro:centdiff}.
\end{proof}

\subsection{Birational maps that don't satisfy Bedford-Diller condition}\,

  Let $f$ be a birational map of a complex surface $\mathrm{S}$; assume that $f$ is algebraically stable. Let~$p$ be a point of indeterminacy of $f$. If $\mathcal{C}$ is a curve 
contracted on $p$ by an iterate $f^{-n}$, $n>0$, of $f$, then we say that \textsl{\textbf{$\mathcal{C}$ comes from $p$}}. If $q$ is a point of $\mathrm{S}$ for which there exists 
an integer $m$ such that 
\begin{align*}
& \forall\,\, 0\leq\ell<m,\,\, f^\ell(q)\not\in\mathrm{Ind}\,f, && f^m(q)=p
\end{align*}
we say that $q$ is a point of indeterminacy of $f$ \textsl{\textbf{passing through $p$ at the time $m$}}. Since $f$ is algebraically stable, the iterates $f^{-m}$ of $f$, $m\geq 0$, 
are all holomorphic in a neighborhood of $p$ so the unique point passing through~$p$ at the time $m$ is $f^{-m}(p)$. We say that $p$ has an infinite negative orbit if the set 
$\big\{f^{-m}(p)\,\vert\, m\geq 0\big\}$ is infinite.

\begin{lem}[\cite{Can3}]\label{Lem:pers}
Let $f$ be a birational map of $\mathrm{S}$. Assume that $f$ is algebraically stable. Let~$p$ be a point of indeterminacy of~$f$ having an infinite negative orbit. One of the 
following holds:
\begin{itemize}
\item[{\it i.}] there exist an infinite number of irreducible curves contracted on $p$ by the iterates $f^{-n}$ of~$f$, $n\in\mathbb{N}$;

\item[{\it ii.}] there exists a birational morphism $\pi\colon\mathrm{S}\to\mathrm{S}'$ such that $\pi f\pi^{-1}$ is an algebraically stable birational map of $\mathrm{S}'$ whose all 
iterates are holomorphic in a neighborhood of $\pi(p)$.
\end{itemize}
\end{lem}

  We will say that a point of indeterminacy $p$ is \textsl{\textbf{persistent}}\label{Chap12:ind12} if there exists no birational mor\-phism $\pi\colon\mathrm{S}\to\mathrm{S}'$ 
satisfying property {\it ii.}

\begin{proof}
Assume that the union of the curves contracted by $f^{-n}$, $n\geq 0$, onto~$p$ is a finite union $\mathcal{C}$ of curves. 

Let us consider a curve $C$ in $\mathcal{C}$ such that
\begin{itemize}
\item[$\bullet$] $f^m$ is holomorphic on $C$;

\item[$\bullet$] $f^m(C)$ is a point.
\end{itemize}
We can then contract the divisor $C$ by a birational map $\pi\colon\mathrm{S}\to\mathrm{S}'$ and the map $\pi f\pi^{-1}$ is still algebraically stable. By induction we can suppose 
that there is no such curve $C$ in $\mathcal{C}$.

If $\mathcal{C}$ is empty the second assertion of the statement is satisfied.

Assume that $\mathcal{C}$ is not empty. If $C$ belongs to $\mathcal{C}$ and $f^m(C)$ does not belong to $\mathcal{C}$ then $f^m(C)$ is a point which does not belong to $C$ and $f^m$ is 
holomorphic along $C$: contradiction.  So for any curve $C$ of $\mathcal{C}$, $f^m(C)$, $m\geq 0$, belongs to $\mathcal{C}$. We can hence assume that $\mathcal{C}$ is invariant by 
any $f^m$ with $m\geq 0$. The set $\mathcal{C}$ is invariant by $f^n$ for any $n$ in $\mathbb{Z}$ so $f^{-n}(p)$, $n>0$, is a sequence of points of $\mathcal{C}$. Let $C$ be an 
irreducible component of $\mathcal{C}$ passing through $p$. Since~$\mathcal{C}$ contains curves coming from $p$ there exists an integer $k$ such that $f^{-k}$ is holomorphic along 
$C$ and contracts $C$ onto $p$. Therefore the negative orbit of $p$ passes periodically through $p$ and cannot be infinite: contradiction.
\end{proof}

\begin{lem}[\cite{Coo, DJS}]\label{Lem:fib}
Let $\mathrm{S}$ be a compact complex surface and let $f$ be a birational map of $\mathrm{S}$. If $f$ preserves an infinite number of curves, then $f$ preserves a fibration.
\end{lem}

\begin{pro}[\cite{Can3}]
Let $f$ be an algebraically stable birational map of a compact complex surface $\mathrm{S}$. Let~$p$ be a persistent point of indeterminacy of $f$ whose negative orbit is infinite. 
If~$\psi$ is a birational map of $\mathrm{S}$ which commutes with $f$ then 
\begin{itemize}
\item[$\bullet$]  either $\psi$ preserves a pencil of rational curves;

\item[$\bullet$]  or an iterate $\psi^m$ of $\psi$, $m\not=0$, coincides with an iterate $f^n$ of $f$.
\end{itemize}
\end{pro}

\begin{proof}
Let us set $\nu:=\#\,\mathrm{Ind}\,f$, and consider $\psi^{\nu !}$ instead of $\psi$. Since the negative orbit of~$p$ is infinite, there exists an integer $k_0$ such that $\psi$ 
is holomorphic around the points $f^{-k}(p)$ for any~$k\geq k_0$. For any $n\geq 0$ let us denote by $\mathcal{C}_n$ the union of curves coming from $p$. The periodic point~$p$ is 
persistent, so according to Lemma \ref{Lem:pers} there is an infinite number of curves coming from $p$. Hence there exists an integer $n_0$ such that for any $n\geq n_0$ the map $\psi$ 
does not contract~$\mathcal{C}_n$. Since $f$ and $\psi$ commute,~$\psi(f^{-k}(p))$ is a point of indeterminacy of $f^m$ for at least an inte\-ger~$$0\leq m\leq n_0+k+1\, 
(\forall\,k\geq k_0).$$ This point of indeterminacy passes through~$p$. Let us consider~$\psi f^\ell$ for some good choice of $\ell$; we can thus assume that $\psi(f^{-k}(p))$ is a 
point of indeterminacy of $f$ passing through $p$ at the time~$k$ and so $\psi(f^{-k}(p))=~f^{-k}(p)$ for any $k\geq k_0$. Moreover for $n$ sufficiently large we have 
$\psi(\mathcal{C}_n)=\mathcal{C}_n$. We conclude with Lemma \ref{Lem:fib}.
\end{proof}

\begin{cor}[\cite{Can3}]
Let $f$ be a birational map of a compact complex surface $\mathrm{S}$ which is algebraically stable. Assume that
\begin{itemize}
\item[$\bullet$]  the map $f$ is hyperbolic;

\item[$\bullet$]  $f$ has a persistent point of indeterminacy whose negative orbit is infinite.
\end{itemize}

  If $\psi$ is a birational map of $\mathrm{S}$ which commutes with $f$, there exists $m\in\mathbb{Z}\setminus\{0\}$ and $n\in\mathbb{Z}$ such that $\psi^m=f^n$.
\end{cor}

\begin{proof}
Let $\psi$ be in $\mathrm{Cent}(f,\mathrm{Bir}(\mathbb{P}^2))$. Assume that $\psi$ preserves a pencil of curves $\mathcal{P}$. As $f$ is hyperbolic, $f$ doesn't preserve a pencil of 
curves so $\psi$ preserves two distinct pencils $\mathcal{P}$ and~$f(\mathcal{P})$. According to \cite{DiFa} an iterate of~$\psi$ is conjugate to an automorphism isotopic to the 
identity on a minimal rational surface $\mathrm{S}'$; let us still denote by $f$ and by $\psi$ the maps of $\mathrm{S}'$ obtained from~$f$ and~$\psi$ by conjugation. Assume that $\psi$ 
has infinite order; let us denote by~$\mathrm{G}$ the Zariski closure of the cyclic group generated by $\psi$ in $\mathrm{Aut}(\mathrm{S}')$. It is an abelian Lie group which commutes 
with $f$. Any subgroup of one parameter of $\mathrm{G}$ determines a flow which commutes with $f$: $f \phi_t=\phi_tf$. If the orbits of $\phi_t$ are algebraic curves, $f$ preserves a 
pencil of curves: contradiction with $\lambda(f)>1$. Otherwise $\phi_t$ fixes a finite number of algebraic curves and among these we find all the curves contracted by $f$ or by 
some $f^n$; hence there is a finite number of such curves: contradiction with the second assumption.
\end{proof}

Since then Blanc and Cantat got a more precise statement.

\begin{thm}[\cite{BlCa}]
Let $f$ be a hyperbolic birational map. Then $$\mathrm{Cent}(f, \mathrm{Bir}(\mathbb{P}^2))\simeq\mathbb{Z}~\rtimes~F$$ where $F$ denotes a finite group.
\end{thm}

\section{Centralizer of elliptic birational maps of infinite order}

Let us recall (\cite[Proposition 1.3]{BlDe}) that an elliptic birational map $f$ of~$\mathbb{P}^2(\mathbb{C})$ of infinite order is conjugate to an automorphism of $\mathbb{P}^2
(\mathbb{C})$ which restricts to one of the following automorphisms on some open subset isomorphic to $\mathbb{C}^2$:
\begin{itemize}
\item[$\bullet$]
 $(\alpha x,\beta y)$, where  $\alpha$, $\beta\in \mathbb{C}^{*}$, and where the kernel of the group homomorphism $\mathbb{Z}^2 \to \mathbb{C}^{*}$ given by $(i,j)\mapsto \alpha^i 
\beta^j$ is generated by $(k,0)$  for some $k\in \mathbb{Z}$.
 \item[$\bullet$]
$(\alpha x, y+1)$, where $\alpha\in \mathbb{C}^{*}$.
\end{itemize}

We can describe the centralizers of such maps.

\begin{lem}[\cite{BlDe}]
Let us consider $f=(\alpha x,\beta y)$ where $\alpha$, $\beta$ are in $\mathbb{C}^*$, and where the kernel of the group homomorphism $\mathbb{Z}^2 \to \mathbb{C}^{*}$ given by 
$(i,j)\mapsto \alpha^i \beta^j$ is generated by $(k,0)$  for some $k\in \mathbb{Z}$. Then the centralizer of~$f$ in~$\mathrm{Bir}(\mathbb{P}^2)$ is 
$$\mathrm{Cent}(f,\mathrm{Bir}(\mathbb{P}^2))=\big\{(\eta (x),y R(x^k))\, \big\vert\, R\in \mathbb{C}(x), \eta \in \mathrm{PGL}_2(\mathbb{C}), \eta(\alpha x)=\alpha\eta(x)\big\}.$$
\end{lem}

\begin{lem}[\cite{BlDe}]
Let us consider $f=(\alpha x,y+\beta)$ where $\alpha$, $\beta\in\mathbb{C}^*$.
Then $\mathrm{Cent}(f,\mathrm{Bir}(\mathbb{P}^2))$ is equal to
 $$\big\{(\eta(x),y+R(x))\,\big\vert\,\eta\in \mathrm{PGL}_2(\mathbb{C}), \eta(\alpha x)=\alpha \eta(x), R\in \mathbb{C}(x), R(\alpha x)=R(x)\big\}.$$
\end{lem}

\section{Centralizer of de Jonqui\`eres twists}

Let us denote by $\pi_2$ the morphism from $\mathrm{dJ}$ (\emph{see} Chapter \ref{Chap:jung}, \S \ref{Sec:Jonq}) into $\mathrm{PGL}_2(\mathbb{C})$, {\it i.e.}~$\pi_2(f)$ is the second component of 
$f\in\mathrm{dJ}$. The elements 
of~$\mathrm{dJ}$ which preserve the fibration with a trivial action on the basis of the fibration form a normal subgroup $\mathrm{dJ}_0$ of $\mathrm{dJ}$ (kernel of the morphism 
$\pi_2$); of course~$\mathrm{dJ}_0\simeq~\mathrm{PGL}_2(\mathbb{C}(y))$. Let $f$ be an element of $\mathrm{dJ}_0$; it is, up to conjugacy, of one of the following form (\emph{see} 
for example \cite{De2})
\begin{align*}
&\mathfrak{a}\,\,\,\,\,(x+a(y),y), && \mathfrak{b}\,\,\,\,\,(b(y)x,y),&& \mathfrak{c}\,\,\,\,\,\left(\frac{c(y)x+F(y)}{x+c(y)},y\right), 
\end{align*}
with $a$ in $\mathbb{C}(y)$, $b$ in $\mathbb{C}(y)^*$ and $c$, $F$ in $\mathbb{C}[y]$, $F$ being not a square (if $F$ is a square, then $f$ is conjugate to an element of type 
$\mathfrak{b}$).

The non finite maximal abelian subgroups of $\mathrm{dJ}_0$ are 
\begin{align*}
&\mathrm{dJ}_a=\big\{(x+a(y),y)\,\big\vert\, a\in\mathbb{C}(y)\big\}, &&\mathrm{dJ}_m=\big\{(b(y)x,y)\,\big\vert\, b\in\mathbb{C}(y)^*\big\},
\end{align*}
\begin{align*}
\mathrm{dJ}_F=\left\{(x,y),\,\left(\frac{c(y)x+F(y)}{x+c(y)},y\right)\,\Big\vert\, c\in\mathbb{C}(y)\right\}
\end{align*}
where $F$ denotes an element of $\mathbb{C}[y]$ which is not a square (\cite{De2}).
We can assume that $F$ is a polynomial with roots of multiplicity one (up to conjugation by a map $(a(y)x,y)$). Therefore if~$f$ belongs to $\mathrm{dJ}_0$ and 
if~$\mathrm{Ab}(f)$ is the non finite maximal abelian subgroup of $\mathrm{dJ}_0$ that contains~$f$ then, up to conjugacy, $\mathrm{Ab}(f)$ is either $\mathrm{dJ}_a$, 
or $\mathrm{dJ}_m$, or $\mathrm{dJ}_F$. More precisely if $f$ is of type~$\mathfrak{a}$ (resp.~$\mathfrak{b}$, resp. $\mathfrak{c}$), then $\mathrm{Ab}(f)=\mathrm{dJ}_a$ 
(resp.~$\mathrm{Ab}(f)=\mathrm{dJ}_m$, resp.~$\mathrm{Ab}(f)=\mathrm{dJ}_F$).

\medskip

In \cite{CD3} we first establish the following property.
 
\begin{pro}[\cite{CD3}]\label{carac}
Let $f$ be an element of $\mathrm{dJ}_0$. Then
\begin{itemize}
\item[$\bullet$] either $\mathrm{Cent}(f,\mathrm{Bir}(\mathbb{P}^2))$ is contained in $\mathrm{dJ};$

\item[$\bullet$] or $f$ is periodic.
\end{itemize}
\end{pro}

\begin{proof}
Let $f=(\psi(x,y),y)$ be an element of $\mathrm{dJ}_0$, {\it i.e.} $\psi\in\mathrm{PGL}_2(\mathbb{C}(y))$.

Let $\varphi=(P(x,y),Q(x,y))$ be a rational map that commutes with $f$. If~$\varphi$ does not belong to $\mathrm{dJ}$, then $Q=$ cte is a fibration invariant by $f$ which 
is not $y=$ cte. Hence $f$ preserves two distinct fibrations and the action on the basis is trivial in both cases so $f$ is periodic. 
\end{proof}

This allows us to prove the following statement.

\begin{thm}[\cite{CD3}]\label{ababelien}
Let $f$ be a birational map which preserves a rational fibration, the action on the basis being trivial. If $f$ is a Jonqui\`eres twist, then $\mathrm{Cent}(f,\mathrm{Bir}
(\mathbb{P}^2))$ is a finite extension of $\mathrm{Ab}(f)$.
\end{thm}

This result allows us to describe, up to finite index, the centralisers of the elements of $\mathrm{dJ}\setminus\mathrm{dJ}_0$, question related to classical problems of 
difference equations. A generic element of $\mathrm{dJ}\setminus\mathrm{dJ}_0$ has a trivial centralizer.

In this section we will give an idea of the proof of Theorem \ref{ababelien}.

\subsection{Maps of $\mathrm{dJ}_a$}\,

\begin{pro}[\cite{CD3}]\label{centtypea}
The centralizer of $f=(x+1,y)$ is $$\big\{\left(x+b(y),\nu(y)\right)\,\big\vert\,b\in\mathbb{C}(y),\,\nu\in\mathrm{PGL}_2(\mathbb{C})\big\}\simeq\mathrm{dJ}_a\rtimes
\mathrm{PGL}_2(\mathbb{C}).$$
\end{pro}

\begin{proof}
The map $f$ is not periodic and so, according to Proposition~\ref{carac}, any map $\psi$ which commutes with $f$ can be written as $(\psi_1(x,y),\nu(y))$ with~$\nu$ in $\mathrm{PGL}_2
(\mathbb{C})$. The equality $f \psi=~\psi f$ implies $\psi_1(x+1,y)=\psi_1(x,y)+1$. Thus $\frac{\partial \psi_1}{\partial x}(x+1,y)=\frac{\partial \psi_1}{\partial x}(x,y)$ and 
$\frac{\partial \psi_1}{\partial x}$ depends only on~$y$, {\it i.e.} $$\psi_1(x,y)=A(y)x+B(y).$$ Writing again $\psi_1(x+1,y)=\psi_1(x,y)+1$ we get $A=1$. Hence
\begin{align*}
&\psi=(x+B(y),\nu(y)), &&B \in \mathbb{C}(y)\,\nu\in\mathrm{PGL}_2(\mathbb{C}).
\end{align*}
\end{proof}

\begin{cor}
The centralizer of a non trivial element $(x+b(y),y)$ is thus conjugate to~$\mathrm{dJ}_a\rtimes\mathrm{PGL}_2(\mathbb{C})$.
\end{cor}

\begin{proof}
Let $f=(x+a(y),y)$ be a non trivial element of $\mathrm{dJ}_a$, {\it i.e.} $a\not=0$; up to conjugation by $(a(y)x,y)$ we can assume that $f=(x+1,y)$. 
\end{proof}

\subsection{Maps of $\mathrm{dJ}_m$}\,

If $a\in\mathbb{C}(y)$ is non constant, we denote $\mathrm{stab}(a)$ the finite subgroup of~$\mathrm{PGL}_2(\mathbb{C})$ defined by 
$$\mathrm{stab}(a)=\big\{\nu\in\mathrm{PGL}_2(\mathbb{C})\,\big\vert\,a(\nu(y))=a(y)\big\}.$$ Let us also introduce the subgroup $$\mathrm{Stab}(a)=\big\{\nu\in
\mathrm{PGL}_2(\mathbb{C})\,\big\vert\,a(\nu(y))=a(y)^{\pm 1}\big\}.$$ We remark that $\mathrm{stab}(a)$ is a normal subgroup of $\mathrm{Stab}(a)$. 

\begin{eg}
If $k$ is an integer and if $a(y)=y^k$, then 
\begin{align*}
&\mathrm{stab}(a)=\big\{\omega^ky\,\big\vert\,\omega^k=1\big\} &&\& &&\mathrm{Stab}(a)=\big\langle\frac{1}{y},\,
\omega^ky\,\big\vert\,\omega^k=1\big\rangle.
\end{align*}
\end{eg}

Let us denote by $\overline{\mathrm{stab}(a)}$ the linear group $$\overline{\mathrm{stab}(a)}=\big\{(x,\nu(y))\,\big\vert\,\nu\in\mathrm{stab}(a)\big\}.$$ By definition the group 
$\overline{\mathrm{Stab}(a)}$ is generated by $\overline{\mathrm{stab}(a)}$ and the elements $\left(\frac{1}{x},\nu(y)\right)$, with $\nu$ in~$\mathrm{Stab}(a)\setminus\mathrm{stab}(a)$.

\begin{pro}[\cite{CD3}]\label{centtypeb}
Let $f=(a(y)x,y)$ be a non periodic element of~$\mathrm{dJ}_m$. 

If $f$ is an elliptic birational map, {\it i.e.} $a$ is a constant, the centralizer of~$f$ is $$\big\{\big(b(y)x,\nu(y)\big)\,\big\vert\, b\in\mathbb{C}(y)^*, \, \nu\in
\mathrm{PGL}_2(\mathbb{C})\big\}.$$ 

If $f$ is a Jonqui\`eres twist, then $\mathrm{Cent}(f,\mathrm{Bir}(\mathbb{P}^2))=\mathrm{dJ}_m\rtimes\overline{\mathrm{Stab}(a)}$.
\end{pro}

\begin{rems}
\begin{itemize}
\item[$\bullet$] For generic $a$ the group $\overline{\mathrm{Stab}(a)}$ is trivial; so for generic $f\in~\mathrm{dJ}_m$, the group $\mathrm{Cent}(f,\mathrm{Bir}(\mathbb{P}^2))$ 
coincides with $\mathrm{dJ}_m=\mathrm{Ab}(f)$.

\item[$\bullet$] If $f=(a(y)x,y)$ with $a$ non constant, then $\mathrm{Cent}(f,\mathrm{Bir}(\mathbb{P}^2))$ is a finite extension of $\mathrm{dJ}_m=\mathrm{Ab}(f)$.

\item[$\bullet$] If $f=(ax,y)$, $a\in\mathbb{C}^*$, we have $\mathrm{Cent}(f,\mathrm{Bir}(\mathbb{P}^2))=\mathrm{dJ}_m\rtimes\overline{\mathrm{Stab}(a)}$ $($here we can define 
$\mathrm{Stab(a)}=\mathrm{PGL}_2(\mathbb{C}))$.
\end{itemize}
\end{rems}

\subsection{Maps of $\mathrm{dJ}_F$}\,

Let us now consider the elements of $\mathrm{dJ}_F$; as we said we can assume that~$F$ only has roots with multiplicity one. We can thus write $f$ as follows: 
\begin{align*}
&f=\left(\frac{c(y)x+F(y)}{x+c(y)},y\right) && c\in\mathbb{C}(y);
\end{align*}
the curve of fixed points $\mathcal{C}$ of $f$ is given by $x^2=F(y)$. Since the eigenvalues of 
$\left[\begin{array}{cc}c(y) & F(y)\\ 1 & c(y)\end{array}\right]$ are $c(y)\pm\sqrt{F(y)}$ we note that $f$ is periodic if and only if $c$ is zero; in that case $f$ is 
periodic of period $2$. Assume now that $f$ is not periodic. As $F$ has simple roots the genus of $\mathcal{C}$ is $\geq 2$ for $\deg F\geq 5$, is equal to $1$ for 
$\deg F\in\{3,\,4\}$; finally $\mathcal{C}$ is rational when~$\deg F\in\{1,\,2\}$.

\subsubsection{Assume that the genus of $\mathcal{C}$ is positive} 

Since $f$ is a Jonqui\`eres twist, $f$ is not periodic. The map $f$ has two fixed points on a generic fiber which correspond to the two points on the cur\-ve~$x^2=F(y)$. The 
curves~$x^2=F(y)$ and the fibers $y=$ constant are invariant by $f$ and there is no other invariant curve. Indeed an invariant curve which is not a fiber $y=$ constant intersects a 
generic fiber in a finite number of points necessary invariant by $f$; since $f$ is of infinite order it is impossible (a Moebius transformation which preserves a set of 
more than three elements is periodic). 

\begin{pro}[\cite{CD3}]\label{genregd}
Let $f=\left(\frac{c(y)x+F(y)}{x+c(y)},y\right)$ be a non periodic map $(${\it i.e.} $c\not=0)$, where~$F$ is a polynomial of degree $\geq 3$ with simple roots $(${\it i.e.} 
the genus of $\mathcal{C}$ is $\geq 1)$. Then if $F$ is generic, $\mathrm{Cent}(f,\mathrm{Bir}(\mathbb{P}^2))$ coincides with $\mathrm{dJ}_F$; if it is not, $\mathrm{Cent}(f,
\mathrm{Bir}(\mathbb{P}^2))$ is a finite extension of $\mathrm{dJ}_F=\mathrm{Ab}(f)$.
\end{pro}

\subsubsection{Suppose that $\mathcal{C}$ is rational}\label{Subsec:crationnelle} 

Let $f$ be an element of $\mathrm{dJ}_F$; assume that $f$ is a Jonqui\`eres twist.

The curve of fixed points $\mathcal{C}$ of $f$ is given by $x^2=F(y)$. Let $\psi$ be an element of $\mathrm{Cent}(f,\mathrm{Bir}(\mathbb{P}^2))$; either~$\psi$ contracts 
$\mathcal{C}$, or~$\psi$ preserves~$\mathcal{C}$. According to Proposition \ref{carac} the map $\psi$ preserves the fibration $y=$ cte; the curve $\mathcal{C}$ is transverse 
to the fibration so $\psi$ cannot contract $\mathcal{C}$. Therefore~$\psi$ belongs to $\mathrm{dJ}$ and preserves~$\mathcal{C}$. As soon as $\deg F\geq 3$ the assumptions of 
Proposition~\ref{genregd} are satisfied; so assume that~$\deg F\leq 2$. The case $\deg F=2$ can be deduced from the case $\deg F=1$. Indeed let us consider $f=~\left(\frac{c(y)x+y}
{x+c(y)},y\right)$. Let us set $\varphi=\left(\frac{x}{cy+d}, \frac{ay+b}{cy+d}\right)$. We can check that $\varphi^{-1}f \varphi$ can be written $$\left(\frac{\widetilde{c}(y)x
+(ay+b)(cy+d)}{x+\widetilde{c}(y)},y\right),$$ and this allows to obtain all polynomials of degree $2$ with simple roots. If $\deg F=1$, {\it i.e.} $F(y)=ay+b$, we have, up to 
conjugation by $\left(x,\frac{y-b}{a}\right)$, $F(y)=y$.

\begin{lem}[\cite{CD3}]\label{secondcomp}
Let $f$ be a map of the form $\left(\frac{c(y)x+y}{x+c(y)},y\right)$ with $c$ in~$\mathbb{C}(y)^*$. If $\psi$ is an element of~$\mathrm{Cent}(f,\mathrm{Bir}(\mathbb{P}^2))$, 
then~$\pi_2(\psi)$ is either $\frac{\alpha}{y}$, $\alpha\in\mathbb{C}^*$, or~$\xi y$, $\xi$ root of unity; moreover,~$\pi_2(\psi)$ belongs to $\mathrm{stab}\left(\frac{4c(y)^2}
{c(y)^2-y}\right)$.
\end{lem}

For $\alpha$ in $\mathbb{C}^*$ we denote by $\mathrm{D}_\infty(\alpha)$ the infinite dihedral group $$\mathrm{D}_\infty(\alpha)=\Big\langle\frac{\alpha}{y},\,\omega y\,
\big\vert\,\omega\text{ root of unity}\Big\rangle;$$ let us remark that any $\mathrm{D}_\infty(\alpha)$ is conjugate to $\mathrm{D}_\infty(1)$. 

If $c$ is a non constant element of $\mathbb{C}(y)^*$, then $\mathrm{S}(c;\alpha)$ is the finite subgroup of $\mathrm{PGL}_2(\mathbb{C})$ given by $$\mathrm{S}(c;\alpha)=
\mathrm{stab}\left(\frac{4c(y)^2}{c(y)^2-y}\right)\cap\mathrm{D}_\infty(\alpha).$$

The description of $\mathrm{Cent}(f,\mathrm{Bir}(\mathbb{P}^2))$ with $f$ in $\mathrm{dJ}_F$ and $\mathcal{C}=\mathrm{Fix}\,f$ rational is given by:

\begin{pro}[\cite{CD3}]\label{genrenul}
Let us consider $f=\left(\frac{c(y)x+y}{x+c(y)},y\right)$ with $c$ in $\mathbb{C}(y)^*$, $c$ non constant. There exists $\alpha$ in $\mathbb{C}^*$ such that 
$$\mathrm{Cent}(f,\mathrm{Bir}(\mathbb{P}^2))=\mathrm{dJ}_y\rtimes\mathrm{S}(c;\alpha).$$ 
\end{pro}

Propositions \ref{centtypea}, \ref{centtypeb}, \ref{genregd} and \ref{genrenul} imply Theorem  \ref{ababelien}.

\section{Centralizer of Halphen twists}

For the definition of Halphen twists, see Chapter \ref{Chap:as}, \S \ref{Sec:twist}.

\begin{pro}[\cite{Can3, [Gi]}]
Let $f$ be an Halphen twist. The centralizer of $f$ in $\mathrm{Bir}(\mathbb{P}^2)$ contains a subgroup of finite index which is abelian, free and of rank $\leq 8$. 
\end{pro}

\begin{proof}
Up to a birational change of coordinates, we can assume that $f$ is an element of a rational surface with an elliptic fibration $\pi\colon \mathrm{S}\to \mathbb{P}^1$ 
and that this fibration is $f$-invariant. Moreover we can assume that this fibration is minimal (there is no smooth curve of self intersection~$-1$ in the fibers) and 
so $f$ is an automorphism. The elliptic fibration is the unique fibration invariant by $f$ (\emph{see} \cite{DiFa}) so it is invariant by $\mathrm{Cent}(f,\mathrm{Bir}
(\mathbb{P}^2))$; thus $\mathrm{Cent}(f,\mathrm{Bir}(\mathbb{P}^2))$ is contained in $\mathrm{Aut}(\mathrm{S})$. 

As the fibration is minimal, the surface $\mathrm{S}$ is obtained by blowing up $\mathbb{P}^2(\mathbb{C})$ in the nine base-points of an Halphen pencil\footnote{An Halphen 
pencil is a pencil of plane algebraic curves of degree $3n$ with nine $n$-tuple base-points.} and the rank of its Neron-Severi group is equal to $10$ (Proposition~\ref{hart}). 
The automorphism group of~$\mathrm{S}$ can be embedded in the endomorphisms of $\mathrm{H}^2(\mathrm{S},\mathbb{Z})$ for the intersection form and preserves the class 
$[\mathrm{K}_S]$ of the canonical divisor, {\it i.e.} the class of the elliptic fibration. The dimension of the orthogonal hyperplane to~$[\mathrm{K}_\mathrm{S}]$ is $9$ 
and the restriction of the intersection form on its hyperplane is semi-negative: its kernel coincides with $\mathbb{Z}[\mathrm{K}_\mathrm{S}]$. Hence 
$\mathrm{Aut}(\mathrm{S})$ contains an abelian group of finite index whose rank is $\leq 8$.
\end{proof}

\chapter{Automorphisms with positive entropy, first definitions and properties}\label{Chap:introaut}

  Let $\mathrm{V}$ be a complex projective manifold. Let $\phi$ be a rational or holomorphic map on $\mathrm{V}.$ When we iterate this map we obtain 
a ``dynamical system'': a point $p$ of $\mathrm{V}$ moves to $p_1=\phi(p),$ then to $p_2=\phi(p_1),$ to $p_3=\phi(p_2)$ $\ldots$ So $\phi$ ``induces a 
movement on $\mathrm{V}$''. The set $\big\{p,\,p_1,\,p_2,\,p_3,\,\ldots\big\}$ is the \textbf{\textit{orbit}}\label{Chap8:ind1} of $p.$

 Let $A$ be a projective manifold; $A$ is an 
\textbf{\textit{Abelian variety}}\label{Chap8:ind2} of dimension~$k$ if $A(\mathbb{C})$ is isomorphic to a compact quotient of $\mathbb{C}^k$ by an
additive subgroup. 

  Multiplication by an integer $m>1$ on an Abelian variety, endomorphisms of degree $d>1$ on projective spaces are studied since XIXth century in particular 
by Julia and Fatou (\cite{Al}). These two families of maps ``have an interesting dynamic''. Consider the first case; let $f_m$ denote the multiplication by $m.$ 
Periodic points of $f_m$ are repulsive and dense in $A(\mathbb{C}):$ a point is periodic if and only if it is a torsion point of $A;$ the differential of $f_m^n$ 
at a periodic point of period $n$ is an homothety of ratio $m^n>1.$ 

Around $1964$ Adler, Konheim and McAndrew introduce a new way to measure the complexity of a 
dynamical system: the topological entropy~(\cite{AKM}). 
  Let $X$ be a compact metric space. Let $\phi$ be a continuous map from~$X$ into itself. Let $\varepsilon$ be a strictly positif real number. For all integer $n$ let~$N(n,\varepsilon)$ 
be the minimal cardinal of a part $X_n$ of $X$ such that for all $y$ in $X$ there exists $x$ in $X$ satisfying
\begin{align*}
& \mathrm{dist}(f^j(x),f^j(y))\leq\varepsilon, &&\forall\,\,\, 0\leq j\leq n.
\end{align*}

  We introduce $\mathrm{h}_{\text{top}}(f,\varepsilon)$ defined by $$\mathrm{h}_{\text{top} }(f,\varepsilon)=\displaystyle\limsup_{n\to +\infty}\frac{1}{n}\log\, N(n,\varepsilon).$$ The 
\textbf{\textit{topological entropy}}\label{Chap8:ind3} of $f$ is given by $$\mathrm{h}_{\text{top}}(f)=\displaystyle\lim_{\varepsilon \to 0} \mathrm{h}_{\text{top}}(f,\varepsilon).$$ 
For an isometry of $X$ the topological entropy is 
zero. For the multiplication by~$m$ on a complex Abelian variety of dimension $k$ we have: $\mathrm{h}_{\text{top}}(f)=2k\log\,m.$ For an endomorphism 
of~$\mathbb{P}^k(\mathbb{C})$ defined by homogeneous polynomials of degree $d$ we have: $\mathrm{h}_{\text{top}}(f)=k\log\,d$ (\emph{see} \cite{Gr}).

  Let $\mathrm{V}$ be a complex projective manifold. On which conditions do rational maps with chaotic behavior exist~? The existence of such rational maps implies a 
lot of constraints on $\mathrm{V}:$

\begin{thm}[\cite{Be2}]
A smooth complex projective hypersurface of dimension greater than~$1$ and
degree greater than $2$ admits no endomorphism of degree greater than $1.$
\end{thm}

  Let us consider the case of compact homogeneous manifolds $\mathrm{V}:$ the group of holomorphic diffeomorphisms acts faithfully on $\mathrm{V}$ and there are a 
lot of holomorphic maps on it. Meanwhile in this context all endomorphisms with topological degree strictly greater than $1$ come from endomorphisms on projective 
manifolds and nilvarieties.

  So the "idea'' is that complex projective manifolds with rich polynomial dynamic are rare; moreover it is not easy to describe the set of rational or holomorphic 
maps on such manifolds.

\section{Some dynamics}

\subsection{Smale horseshoe}\,

  The Smale horsehoe is the hallmark of chaos. Let us now describe it (\emph{see for example} \cite{Sh}). Consider the embedding $f$ of the disc $\Delta$ into 
itself. Assume that
\begin{itemize}
\item[$\bullet$] $f$ contracts the semi-discs $f(A)$ and $f(E)$ in $A;$

\item[$\bullet$] $f$ sends the rectangles $B$ and $D$ linearly to the rectangles $f(B)$ and~$f(D)$ stretching them vertically and shrinking them horizontally, in 
the case of $D$ it also rotates by $180$ degrees.
\end{itemize}

  We don't care what the image $f(C)$ of $C$ is, as long as $f(C)\cap(B\cup C\cup D)=\emptyset.$ In other words we have the following situation

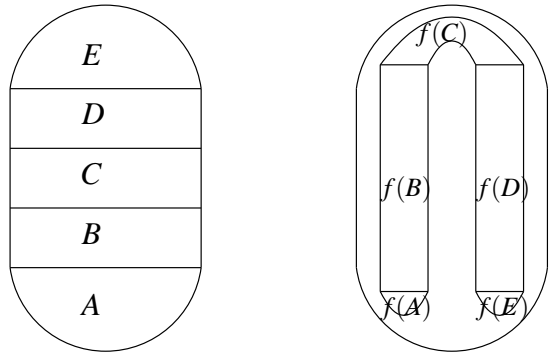
\begin{figure}[H]
\begin{center}
\input{horseshoe2.pstex_t}
\end{center}
\end{figure}

  There are three fixed points: $p\in f(B),$ $q\in A,$ $s\in f(D).$ The points $q$ is a \textbf{\textit{sink}}\label{Chap8:ind4} in the sense that for all 
$z\in A\cup C\cup E$ we have $\displaystyle\lim_{n\to +\infty} f^n(z)=q.$ The points $p$ and $s$ are \textbf{\textit{saddle points}}\label{Chap8:ind5}: if 
$m$ lies on the horizontal through~$p$ then~$f^n$ squeezes it to $p$ as $n\to +\infty,$ while if $m$ lies on the vertical through~$p$ then $f^{-n}$ squeezes 
it to $p$ as $n\to +\infty.$ In some coordinates centered in $p$ we have 
\begin{align*}
& \forall (x,y)\in B, && f(x,y)=(kx,my)
\end{align*}
for some $0<k<1<m;$ similarly $f(x,y)=(-kx,-my)$ on $D$ for some coordinates centered at $s.$ 
Let us recall that the sets 
\begin{align*}
& W^s(p)=\big\{z\,\big\vert\, f^n(z)\to p \text{ as }n\to +\infty\big\}, \\
& W^u(p)=\big\{z\,\big\vert\, f^n(z)\to p \text{ as }n\to -\infty\big\}
\end{align*}
are called stable and unstable manifolds of $p.$ They intersect at $r,$ which is what Poincar\'e called a \textbf{\textit{homoclinic point}}\label{Chap8:ind8}. 
Homoclinic points are dense in $\big\{m\in\Delta\,\big\vert \, f^n(m)\in\Delta,\, n\in\mathbb{Z}\big\}$. 

  The keypart of the dynamic of $f$ happens on the horseshoe $$\Lambda=\big\{z\,\big\vert\, f^n(z) \in B\cup D\,\,\, \forall\, n\in\mathbb{Z}\big\}.$$ Let us 
introduce the shift map on the space of two symbols. Take two symbols $0$ and $1,$ and look at the set $\Sigma=\big\{0,1\big\}^\mathbb{Z}$ of all bi-infinite 
sequences $a=(a_n)_{n\in\mathbb{Z}}$ where, for each $n,$ $a_n$ is~$0$ or $1.$ The map $\sigma\colon\Sigma\to\Sigma$ that sends $a=(a_n)$ to $\sigma(a)=(a_{n+1})$ 
is a homeomorphism called the \textbf{\textit{shift map}}\label{Chap8:ind9}. Let us consider the itinerary map $i\colon\Lambda\to\Sigma$ defined as follows: 
$i(p)=(s_n)_{n\in\mathbb{Z}}$ where $s_n=1$ if $f^n(p)$ is in $B$ and $s_n=0$ if $f^n(p)$ belongs to $D.$ The diagram 
$$\xymatrix{\Sigma\ar[d]_{i} \ar[r]^\sigma &\Sigma\ar[d]^{i} \\
\Lambda \ar[r]^f &\Lambda }$$
commutes so every dynamical property of the shift map is possessed equally by $f_{\vert\Lambda}.$ Due to conjugacy the chaos of $\sigma$ is reproduced exactly in the 
horseshoe: the map $\sigma$ has positive entropy: $\log 2;$ it has $2^n$ periodic orbits of period $n,$ and so must be the set of periodic orbits of~$f_{\vert\Lambda}.$ 

  To summarize: every dynamical system having a transverse homoclinic point also has a horseshoe and thus has a shift chaos, even in higher dimensions. The mere 
existence of a transverse intersection between the stable and unstable manifolds of a periodic orbit implies a horseshoe; since transversality persists under 
perturbation, it follows that so does the horseshoe and so does the chaos.

  The concepts of horseshoe and hyperbolicity are related. In the description of the horseshoe the derivative of $f$ stretches tangent vectors that are parallel to the 
vertical and contracts vectors parallel to the horizontal, not only at the saddle points, but uniformly throughout $\Lambda.$ In general, 
\textbf{\textit{hyperbolicity}}\label{Chap8:ind10} of a compact invariant set such as $\Lambda$ is expressed in terms of expansion and contraction of the derivative on 
subbundles of the tangent bundle. 

\subsection{Two examples}\,

  Let us consider $P_c(z)=z^2+c.$ A periodic point $p$ of $P_c$ with period~$n$ is \textbf{\textit{repelling}}\label{Chap8:ind11} if $\vert(P_c^n(p))'\vert>1$ and the 
\textbf{\textit{Julia set}}\label{Chap8:ind12} of $P_c$ is the closure of the set of repelling periodic points. $P_c$ is a complex horseshoe if it is hyperbolic 
({\it i.e.} uniformly expanding on the Julia set) and conjugate to the shift on two symbols. The \textbf{\textit{Mandelbrot set}}\label{Chap8:ind13} $M$ is defined as 
the set of all points $c$ such that the sequence $(P_c^n(0))_n$ does not escape to infinity $$M=\big\{c\in\mathbb{C}\,\big\vert\, \exists\,s\in\mathbb{R},\,\forall\,n
\in\mathbb{N},\, \big\vert P_c^n(0)\big\vert\leq s\big\}.$$ The complex horseshoe locus is the complement of the Mandelbrot set.

\medskip

  Let us consider the H\'enon family of quadratic maps 
\begin{align*}
& \phi_{a,b}\colon\mathbb{R}^2\to\mathbb{R}^2, && \phi_{a,b}(x,y)=(x^2+a-by,x).
\end{align*}
For fixed parameters $a$ and $b,$ $\phi_{a,b}$ defines a dynamical system, and we are interested in the way that the dynamic varies with the parameters. The parameter 
$b$ is equal to $\det\mathrm{jac}\, \phi_{a,b};$ when~$b=0,$ the map has a one-dimensional image and is equivalent to $P_c.$ As soon as $b$ is non zero, these maps 
are diffeomorphisms, and maps similar to Smale's horseshoe example occur when $a<<0$ (\emph{see} \cite{DN}).

\bigskip

  In the $60$'s it was hoped that uniformly hyperbolic dynamical systems might be in some sense typical. While they form a large open sets on all manifolds, they are 
not dense. The search for typical dynamical systems continues to be a great problem, in order to find new phenomena we try the framework of compact complex surfaces.

\section{Some algebraic geometry}

\subsection{Compact complex surfaces}\,

  Let us recall some notions introduced in Chapters \ref{Chap:firststep} and \ref{Chap:as} and some others.

  To any surface $\mathrm{S}$ we associate its Dolbeault cohomology groups $\mathrm{H}^{ p,q}(\mathrm{S})$ and the cohomological groups $\mathrm{H}^k(\mathrm{S},\mathbb{Z}),$ 
$\mathrm{H}^k(\mathrm{S},\mathbb{R})$ and $\mathrm{H}^k(\mathrm{S},\mathbb{C}).$ Set $$\mathrm{H}^{1,1}_\mathbb{R}(\mathrm{S})=\mathrm{H}^{1,1}(\mathrm{S})\cap\mathrm{H}^2(\mathrm{S},
\mathbb{R}).$$ Let~$f\colon\mathrm{X}\dashrightarrow~\mathrm{S}$ be a dominating meromorphic map between compact complex surfaces, let $\Gamma$ be a desingularization of its graph and 
let $\pi_1,$ $\pi_2$ be the natural projections. A smooth form $\alpha$ in $\mathcal{C}^\infty_{p,q}(\mathrm{S})$ can be pulled back as a smooth form $
\pi_2^*\alpha\in\mathcal{C}_{p,q}^\infty(\Gamma)$ 
and then pushed forward as a current. We define $f^*$ by $$f^*\alpha=\pi_{1*}\pi_2^*\alpha$$ which gives a $\mathrm{L}_{\text{loc}}^1$ form on $\mathrm{X}$ that is smooth outside 
$\mathrm{Ind}\,f.$ The action of~$f^*$ satisfies: $f^*(\mathrm{d}\alpha)=\mathrm{d}(f^*\alpha)$ so descends to a linear action on Dolbeault cohomology. 

  Let $\{\alpha\}\in\mathrm{H}^{p,q}(\mathrm{S})$ be the Dolbeault class of some smooth form $\alpha.$ We set $$f^*\{\alpha\}=\{\pi_{1*}\pi_2^*\alpha\}\in\mathrm{H}^{p,q}(\mathrm{X}).$$ 
This defines a linear map $f^*$ from $\mathrm{H}^{p,q}(\mathrm{S})$ into $\mathrm{H}^{p,q}( \mathrm{X}).$ Similarly we can define the push-forward $f_*=\pi_{2*}\pi_1^*$ from 
$\mathrm{H}^{p,q}(\mathrm{X})$ into $\mathrm{H}^{p,q}(\mathrm{S}).$ When $f$ is bimeromorphic, we have $f_* =(f^{-1})^*.$ The operation $(\alpha,\beta) \mapsto\int\alpha\wedge
\overline{\beta}$ on smooth $2$-forms induced a quadratic intersection form, called \textbf{\textit{product intersection}}\label{Chap8:ind16}, denoted by $(\cdot,\,\cdot)$ on 
$\mathrm{H}^2(\mathrm{S}, \mathbb{C}).$ Its structure is given by the following fundamental statement.

\begin{thm}[\cite{BHPV}]\label{signature}
Let $\mathrm{S}$ be a compact K\"{a}hler surface and let $\mathrm{h}^{1,1}$ denote the dimension
of $\mathrm{H}^{1,1}(\mathrm{S},\mathbb{R})\subset \mathrm{H}^2(\mathrm{S},\mathbb{R}).$ Then the signature of the restriction of the intersection pro\-duct to $\mathrm{H}^{1,1}(\mathrm{S},
 \mathbb{R})$ is $(1,\mathrm{h}^{1,1}-1).$ In particular, there is no $2$-dimensional linear subspace
$\mathrm{L}$ in~$\mathrm{H}^{1,1}(\mathrm{S},\mathbb{R})$ with the property that $(v,\, v)= 0$ forall~$v$ in $\mathrm{L}.$
\end{thm}

  The Picard group $\mathrm{Pic}(\mathbb{P}^2)$ is isomorphic to $\mathbb{Z}$ (\emph{see} Chapter \ref{Chap:firststep}, Example~\ref{picproj}); similarly $\mathrm{H}^2(\mathbb{P}^2
(\mathbb{C}),\mathbb{Z})$ is isomorphic to $\mathbb{Z}.$ We may identi\-fy~$\mathrm{Pic}(\mathbb{P}^2)$ and $\mathrm{H}^2(\mathbb{P}^2(\mathbb{C}), \mathbb{Z}).$

\subsection{Exceptional configurations and characteristic matrices}\,

  Let $f\in\mathrm{Bir}(\mathbb{P}^2)$ be a birational map of degree $\nu.$ By Theorem~\ref{Zariski} there exist a smooth projective surface $\mathrm{S}'$ and $\pi$, $\eta$ two 
sequences of blow-ups such that
$$
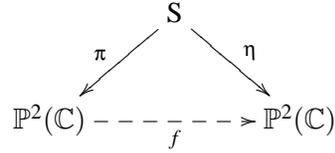
\xymatrix{& \mathrm{S}\ar[dl]_{\pi}\ar[dr]^{\eta} &\\
\mathbb{P}^2(\mathbb{C})\ar@{-->}[rr]_f & & \mathbb{P}^2(\mathbb{C}) }$$
We can rewrite $\pi$ as follows $$\pi\colon\mathrm{S}=\mathrm{S}_k\stackrel{\pi_k}{\to} \mathrm{S}_{k-1}\stackrel{\pi_{k-1}}{\to}\ldots\stackrel{\pi_2}{\to}\mathrm{S}_1\stackrel{\pi_1}
{\to} \mathrm{S}_0=\mathbb{P}^2(\mathbb{C})$$ where $\pi_i$ is the blow-up of the point $p_{i-1}$ in $\mathrm{S}_{i-1}.$ Let us set
\begin{align*}
&\mathrm{E}_i=\pi_i^{-1}(p_i), && \mathcal{E}_i=(\pi_{i+1}\circ\ldots\circ\pi_k)^*\mathrm{E}_i.
\end{align*}

  The divisors $\mathcal{E}_i$ are called the \textbf{\textit{exceptional configurations}}\label{Chap8:ind19} of $\pi$ and the~$p_i$ base-points of $f.$

  For any effective divisor $\mathrm{D}\not=0$ on $\mathbb{P}^2(\mathbb{C})$ let $\mathrm{mult}_{p_i} \mathrm{D}$ be defined inductively in the following way. 
We set $\mathrm{mult}_{p_1} \mathrm{D}$ to be the usual multiplicity of~$\mathrm{D}$ at $p_1:$ it is defined as the largest integer~$m$ such that the local 
equation of~$\mathrm{D}$ at $p_1$ belongs to the $m$-th power of the maximal ideal $\mathfrak{m}_{\mathbb{P}^2,p_1}.$ Suppose that $\mathrm{mult}_{p_1}\mathrm{D}$ 
is defined. We take the proper inverse transform $\pi_i^{-1}\mathrm{D}$ of $\mathrm{D}$ in $\mathrm{S}_i$ and define $\mathrm{mult}_{p_{i+1}} 
\mathrm{D}=\mathrm{mult}_{p_{i+1}} \pi_i^{-1}\mathrm{D}.$ It follows
from the definition that $$\pi^{-1}\mathrm{D}=\pi^*(\mathrm{D})-\sum_{i=1}^k m_i\mathcal{E}_i$$
where $m_i=\mathrm{mult}_{p_i} \mathrm{D}.$

There are two relationships between $\nu$ and the $m_i$'s (Chapter \ref{Chap:firststep}, \S \ref{Sec:firstdef}):
\begin{align*}
& 1=\nu^2-\sum_{i=1}^km_i^2, && 3=3\nu-\sum_{i=1}^km_i.
\end{align*}

\medskip

An \textbf{\textit{ordered 
resolution}}\label{Chap8:ind20} of $f$ is a decomposition $f=\eta\pi^{-1}$ where $\eta$ and~$\pi$ are ordered sequences of blow-ups. An ordered resolution of $f$ induces two basis of 
$\mathrm{Pic}(\mathrm{S})$
\begin{itemize}
\item[$\bullet$] $\mathcal{B}=\big\{e_0=\pi^*\mathrm{H},\, e_1=[\mathcal{E}_1],\,\ldots,\, e_k=[\mathcal{E}_k]\big\},$

\item[$\bullet$] $\mathcal{B}'=\big\{e'_0=\eta^*\mathrm{H},\, e'_1=[\mathcal{E}'_1],\,\ldots,\, e'_k=[\mathcal{E}'_k]\big\},$
\end{itemize}
where $\mathrm{H}$ is a generic line. We can write $e'_i$ as follows
\begin{align*}
&e'_0=\nu e_0-\sum_{i=1}^k m_ie_i, && e'_j=\nu_je_0-\sum_{i=1}^k m_{ij}e_i,\, j\geq 1.
\end{align*}

  The matrix of change of basis $$M=\left[\begin{array}{cccc} \nu&\nu_1&\ldots&\nu_k\\ -m_1 & -m_{11}&\ldots&-m_{1k}\\
\vdots&\vdots& &\vdots\\ -m_k&-m_{k1}&\ldots&-m_{kk}\end{array}\right]$$ is called \textbf{\textit{characteristic matrix}}\label{Chap8:ind21} of $f.$ The first column of $M,$ which is 
the \textbf{\textit{characteristic vector}}\label{Chap8:ind22} of $f,$ is the vector $(\nu,-m_1,\ldots,-m_k).$ The other columns $(\nu_i,-m_{1i},\ldots,-m_{ki})$ describe the 
\og behavior of~$\mathcal{E}'_i$\fg: if $\nu_j>0,$ then~$\pi(\mathcal{E}'_j)$ is a curve of degree $\nu_j$ in $\mathbb{P}^2(\mathbb{C})$ through the points $p_\ell$ of $f$ with 
multiplicity $m_{\ell j}.$  

\begin{eg}
Consider the birational map 
\begin{align*}
&\sigma\colon\mathbb{P}^2(\mathbb{C})\dashrightarrow\mathbb{P}^2(\mathbb{C}), &&(x:y:z) \dashrightarrow(yz:xz:xy).
\end{align*}

  The points of indeterminacy of $\sigma$ are $P=(1:0:0),$ $Q=(0:1:0)$ and~$R=(0:0:1);$ the exceptional set is the union of the three lines $\Delta=\{x=0\},$ $\Delta'=\{y=0\}$ and 
$\Delta''= \{z=0\}.$ 

  First we blow up $P;$ let us denote by $\mathrm{E}$ the exceptional divisor and $\mathcal{D}_1$ the strict transform of $\mathcal{D}.$ Set

\begin{align*}
&\left\{\begin{array}{ll} y=u_1\\ z=u_1v_1\end{array}\right. && \begin{array}{ll} \mathrm{E}=\{u_1 =0\}\\ \Delta''_1=\{v_1=0\}\end{array} && \hspace{1cm} &&\left\{\begin{array}{ll} 
y=r_1s_1\\ z=s_1\end{array}\right. && \begin{array}{ll} \mathrm{E}=\{s_1 =0\}\\ \Delta'_1=\{r_1=0\}\end{array}
\end{align*}

  On the one hand $$(u_1,v_1)\to(u_1,u_1v_1)_{(y,z)}\to(u_1v_1:v_1:1)=\left(\frac{1}{u_1}, \frac{1}{u_1v_1}\right)_{(y,z)}\to \left(\frac{1}{u_1},\frac{1}{v_1}\right)_{(u_1,v_1)};$$ 
on the other hand $$(r_1,s_1)\to(r_1s_1,s_1)_{(y,z)}\to(r_1s_1:1:r_1)=\left(\frac{1}{r_1s_1}, \frac{1}{s_1} \right)_{(y,z)} \to\left(\frac{1}{r_1},\frac{1}{s_1}\right)_{(r_1,s_1)}.$$ 
Hence $\mathrm{E}$ is sent on $\Delta_1;$ as $\sigma$ is an involution $\Delta_1$ is sent on $\mathrm{E}.$

\medskip

  Now blow up $Q_1;$ this time let us denote by $\mathrm{F}$ the exceptional divisor and~$\mathcal{D}_2$ the strict transform of~$\mathcal{D}_1:$

\begin{align*}
&\left\{\begin{array}{ll} x=u_2\\ z=u_2v_2\end{array}\right. && \begin{array}{ll} \mathrm{F}=\{u_2 =0\}\\ \Delta''_2=\{v_2=0\}\end{array} && \hspace{1cm} &&\left\{\begin{array}{ll} 
x=r_2s_2\\ z=s_2\end{array}\right. && \begin{array}{ll} \mathrm{E}=\{s_2 =0\}\\ \Delta_2=\{r_2=0\}\end{array}
\end{align*}

  We have $$(u_2,v_2)\to(u_2,u_2v_2)_{(x,z)}\to(v_2:u_2v_2:1)=\left(\frac{1}{u_2},\frac{1}{u_2v_2}\right)_{(x,z)}\to\left(\frac{1}{u_2},\frac{1}{v_2}\right)_{(u_2,v_2)}$$ and 
$$(r_2,s_2)\to(r_2s_2,s_2)_{( x,z)}\to(1:r_2s_2:r_2)=\left(\frac{1}{r_2s_2},\frac{1}{s_2}\right)_{(x,z)}\to\left(\frac{1}{r_2},\frac{1}{s_2}\right)_{(r_2,s_2)}.$$

  Therefore $\mathrm{F}\to\Delta'_2$ and $\Delta'_2\to\mathrm{F}.$

  Finally we blow up $R_2;$ let us denote by $\mathrm{G}$ the exceptional divisor and set 
\begin{align*}
&\left\{\begin{array}{ll} x=u_3\\ y=u_3v_3\end{array}\right. && \begin{array}{ll} \mathrm{G}=\{u_3 =0\}\\ \Delta''_3=\{v_3=0\}\end{array} && \hspace{1cm} &&\left\{\begin{array}{ll} 
x=r_3s_3\\ z=s_3\end{array}\right. && \begin{array}{ll} \mathrm{E}=\{s_3 =0\}\\ \Delta_2=\{r_3=0\}\end{array}
\end{align*}

  Note that $$(u_3,v_3)\to(u_3,u_3v_3)_{(x,y)}\to(v_3:1:u_3v_3)=\left(\frac{1}{u_3}, \frac{1}{u_3v_3}\right)_{(x,y)} \to\left(\frac{1}{u_3},\frac{1}{v_3}\right)_{(u_3,v_3)}$$ and 
$$(r_3 ,s_3)\to(r_3s_3,s_3)_{(x,y)}\to(1:r_3:r_3s_3)=\left(\frac{1}{r_3s_3},\frac{1}{s_3}\right)_{(x,y)}\to \left(\frac{1}{r_3},\frac{1}{s_3}\right)_{(r_3,s_3)}.$$ Thus $\mathrm{G}\to 
\Delta'_3$ and $\Delta'_3\to\mathrm{G}.$ There are no more points of indeterminacy, no more exceptional curves; in other words $\sigma$ is conjugate to an automorphism of 
$\mathrm{Bl}_{P,Q_1,R_2}\mathbb{P}^2.$

\bigskip

  Let $\mathrm{H}$ be a generic line. Note that $\mathcal{E}_1=\mathrm{E},$ $\mathcal{E}_2=\mathrm{F},$ $\mathcal{E}_3=\mathrm{H}.$ Consider the basis $\{\mathrm{H},\,\mathrm{E},\,
\mathrm{F},\,\mathrm{G}\}.$ After the first blow-up $\Delta$ and $\mathrm{E}$ are swapped; the point blown up is the intersection of $\Delta'$ and~$\Delta''$ so $\Delta\to\Delta+
\mathrm{F}+\mathrm{G}.$ Then $\sigma^*\mathrm{E}=\mathrm{H}-\mathrm{F}- \mathrm{G}.$ Similarly we have: 
\begin{align*}
&\sigma^*\mathrm{F}=\mathrm{H}-\mathrm{E}-\mathrm{G} && \text{and} &&\sigma^*\mathrm{G}=\mathrm{H}- \mathrm{E}-\mathrm{F}.
\end{align*}
It remains to determine $\sigma^*\mathrm{H}.$ The image of a generic line by $\sigma$ is a conic hence $\sigma^*\mathrm{H}=2\mathrm{H}-m_1\mathrm{E}-m_2\mathrm{F}-m_3\mathrm{G}.$ 
Let $\mathrm{L}$ be a generic line described by~$a_0x+a_1y+a_2z.$ A computation shows that $$(u_1,v_1)\to(u_1,u_1v_1)_{(y,z)}\to(u_1^2v_1:u_1v_1:u_1)\to u_1(a_0v_2+a_1u_2v_2+a_2)$$ 
vanishes to order $1$ on $\mathrm{E}=\{u_1 =0\}$ thus $m_1=1.$ Note also that $$(u_2,v_2)\to(u_2,u_2v_2)_{(x,z)}\to(u_2v_2:u_2^2v_2:u_2)\to u_2(a_0v_2+a_1u_2v_2+a_2),$$ respectively 
$$(u_3,v_3)\to(u_3,u_3v_3)_{(x,y)}\to(u_3v_3:u_3:u_3^2v_3)\to u_3(a_0v_3+a_1+a_2u_3v_3)$$ vanishes to order $1$ on $\mathrm{F}=\{u_2=0\},$ resp. $\mathrm{G}=\{u_3=0\}$ so $m_2=1,$
 resp.~$m_3=1.$ Therefore $\sigma^*\mathrm{H}=2\mathrm{H}-\mathrm{E}-\mathrm{F}-\mathrm{G}$ and the characteristic matrix of $\sigma$ in the basis $\big\{\mathrm{H},\,\mathrm{E},\,
\mathrm{F},\,\mathrm{G}\big\}$ is $$M_\sigma=\left[\begin{array}{cccc} 2 & 1 &1 & 1\\ -1 & 0 & -1 & -1 \\ -1 & -1 & 0 & -1 \\ -1 & -1 & -1 & 0 \end{array}\right].$$ 
\end{eg}

\begin{eg}
  Let us consider the involution given by
\begin{align*}
&\rho\colon\mathbb{P}^2(\mathbb{C})\dashrightarrow\mathbb{P}^2(\mathbb{C}), &&(x:y:z)\dashrightarrow(xy:z^2:yz).
\end{align*}

\smallskip

  We can show that $M_\rho=M_\sigma.$
\end{eg}

\begin{eg}
  Consider the birational map 
\begin{align*}
& \tau\colon\mathbb{P}^2(\mathbb{C})\dashrightarrow\mathbb{P}^2(\mathbb{C}), && (x:y:z) \dashrightarrow(x^2:xy:y^2-xz).
\end{align*}

  We can verify that $M_\tau=M_\sigma.$
\end{eg}

\section{Where can we find automorphisms with po\-sitive entropy ?}

\subsection{Some properties about the entropy}\,

  Let $f$ be a map of class $\mathcal{C}^\infty$ on a compact manifold $\mathrm{V};$ the topological entropy is greater than the logarithm of the spectral radius of the linear map 
induced by $f$ on $\mathrm{H}^*(\mathrm{V},\mathbb{R}),$ direct sum of the cohomological groups of~$\mathrm{V}$: $$\mathrm{h}_{\text{top}}(f)\geq\log\, r(f^*).$$ Remark that the 
inequali\-ty~$\mathrm{h}_{\text{top}}(f)\geq\log\, r(f^*)$ is still true in the meromorphic case (\cite{DS}). Before stating a more precise result when~$\mathrm{V}$ is K\"{a}hler we 
introduce some notation: for all integer $p$ such that $0\leq p\leq \dim_\mathbb{C}\mathrm{V}$ we denote by~$\lambda_p(f)$ the spectral radius of the map $f^*$ acting on the Dolbeault 
cohomological group $\mathrm{H}^{p,p}(\mathrm{V},\mathbb{R}).$ 

\begin{thm}[\cite{Gr, Gr2, Yo}]
Let $f$ be a holomorphic map on a compact complex K\"{a}hler manifold $\mathrm{V};$ we have $$\mathrm{h}_{\text{top}}(f)=\max_{0\leq p\leq\dim_\mathbb{C}\mathrm{V}}\log\,\lambda_p(f).$$
\end{thm}

\begin{rem}
The spectral radius of $f^*$ is strictly greater than $1$ if and only if one of the~$\lambda_p(f)$'s is and, in fact, if and only if $\lambda(f)=\lambda_1(f)>1$. In other words in order 
to know if the entropy of~$f$ is positive we just have to study the growth of $(f^n)^*\{\alpha\}$ where $\{\alpha\}$ is a K\"{a}hler form.
\end{rem}

\begin{egs}
\begin{itemize}
\item[$\bullet$] Let $\mathrm{V}$ be a compact K\"{a}hler manifold and $\mathrm{Aut}^0(\mathrm{V})$ be the connected component of~$\mathrm{Aut}(\mathrm{V})$ which contains the identity 
element. The topological entropy of each element of~$\mathrm{Aut}^0(\mathrm{V})$ is zero.

\item[$\bullet$] The topological entropy of an holomorphic endomorphism $f$ of the projective sapce is equal to the logarithm of the topological degree of $f.$

\item[$\bullet$] Whereas the topological entropy of an elementary automorphism is zero, the topological entropy of an H\'enon automorphism is positive.
\end{itemize}
\end{egs}

\subsection{A theorem of Cantat}\,

Before describing the pairs $(\mathrm{S},f)$ of compact complex surfaces $\mathrm{S}$ carrying an automorphism $f$ with positive entropy, let us recall that a surface 
$\mathrm{S}$ is \textbf{\textit{rational}}\label{Chap8:ind23} if it is birational to $\mathbb{P}^2(\mathbb{C})$.
A rational surface is always projective~(\cite{BHPV}). A \textbf{\textit{K$3$ surface}}\label{Chap8:ind24} is a complex, compact, simply 
connected surface~$\mathrm{S}$ with a trivial canonical bundle. Equivalently there exists a holomorphic $2$-form $\omega$ on $\mathrm{S}$ which is never zero; $\omega$ is unique 
modulo multiplication by a scalar. Let $\mathrm{S}$ be a K$3$ surface with a holomorphic involution $\iota.$ If $\iota$ has no fixed point the quotient is an \textbf{\textit{Enriques 
surface}}\label{Chap8:ind25}, otherwise it is a rational surface. As Enriques surfaces are quotients of~K$3$ surfaces by a group of order $2$ acting without fixed points, their theory 
is similar to that of algebraic K$3$ surfaces.

\begin{thm}[\cite{Can1}]\label{Thm:serge}
Let $\mathrm{S}$ be a compact complex surface. Assume that $\mathrm{S}$ has an automorphism $f$ with positive entropy. Then
\begin{itemize}
\item[$\bullet$] either $f$ is conjugate to an automorphism on the unique minimal model of $\mathrm{S}$ which is either a torus, or a K$3$ surface, or an Enriques surface;

\item[$\bullet$] or $\mathrm{S}$ is rational, obtained from $\mathbb{P}^2(\mathbb{C})$ by blowing up $\mathbb{P}^2(\mathbb{C})$ in at least~$10$ points and $f$ is birationally conjugate 
to a birational map of $\mathbb{P}^2(\mathbb{C}).$
\end{itemize}

  In particular $\mathrm{S}$ is k\"{a}hlerian. 
\end{thm}

\begin{egs}
\begin{itemize}
\item[$\bullet$] Set $\Lambda=\mathbb{Z}[\mathbf{i}]$ and $E=\mathbb{C}/\Lambda.$ The group $\mathrm{SL}_2 (\Lambda)$ acts linearly on $\mathbb{C}^2$ and preserves the lattice 
$\Lambda\times\Lambda;$ then each element $\mathrm{A}$ of $\mathrm{SL}_2(\Lambda)$ induces an automorphism~$f_\mathrm{A}$ on $E\times E$ which commutes with~$\iota(x,y)=(\mathbf{i}x,
\mathbf{i}y).$ Each automorphism $f_\mathrm{A}$ can be lifted to an automorphism $\widetilde{f_\mathrm{A}}$ on the desingula\-rization of $(E\times E)/\iota$ which is a K$3$ surface.
 The entropy of $\widetilde{f_\mathrm{A}}$ is positive as soon as the modulus of one eigenvalue of $\mathrm{A}$ is strictly greater than $1.$

\item[$\bullet$] We have the following statement due to Torelli.

\begin{thm}
Let $\mathrm{S}$ be a K$3$ surface. The morphism 
\begin{align*}
& \mathrm{Aut}(\mathrm{S})\to\mathrm{GL}(\mathrm{H}^2(\mathrm{S},\mathbb{Z})), && f\mapsto f^*
\end{align*}

  is injective.

  Conversely assume that $\psi$ is an element of $\mathrm{GL}(\mathrm{H}^2(\mathrm{S}, \mathbb{Z}))$ which preserves the intersection form on~$\mathrm{H}^2(\mathrm{S},\mathbb{Z}),$ the 
Hodge decomposition of~$\mathrm{H}^2(\mathrm{S},\mathbb{Z})$ and the K\"{a}hler cone of~$\mathrm{H}^2(\mathrm{S},\mathbb{Z}).$ Then there exists an automorphism $f$ on $\mathrm{S}$ such 
that $f^*=\psi.$
\end{thm}
\end{itemize}
\end{egs}

  The case of K$3$ surfaces has been studied by Cantat, McMullen, Silverman, Wang and others (\emph{see} for example \cite{Can2, Mc2, Si, Wa}). The context of rational surfaces produces 
much more examples (\emph{see} for example \cite{Mc, BK1, BK2, BK3, DeGr}).

\subsection{Case of rational surfaces}\,

Let us recall the following statement due to Nagata.

\begin{pro}[\cite{Na}, Theorem $5$]
Let $\mathrm{S}$ be a rational surface and let $f$ be an automorphism on $\mathrm{S}$ such that $f_*$ is of infinite order; then there exists a sequence of holomorphic maps~$\pi_{j+1}
\colon \mathrm{S}_{j+1}\to \mathrm{S}_j$ such that $\mathrm{S}_1=\mathbb{P}^2(\mathbb{C}),$ $\mathrm{S}_{N+1}=\mathrm{S}$ and $\pi_{j+1}$ is the blow-up of $p_j\in\mathrm{S}_j.$ 
\end{pro}

  Remark that a surface obtained from $\mathbb{P}^2(\mathbb{C})$ via generic blow-ups has no nontrivial automorphism (\cite{Hi, Ko}). Moreover we have the following statement which can 
be found for example in \cite[Proposition 2.2.]{Di2}.

\begin{pro}
Let $\mathrm{S}$ be a surface obtained from $\mathbb{P}^2(\mathbb{C})$ by blowing up $n\leq 9$ points. Let~$f$ be an automorphism on $\mathrm{S}.$ The topological entropy of~$f$ is zero.

  Moreover, if $n\leq 8$ then there exists an integer $k$ such that $f^k$ is birationally conjugate to an automorphism of the complex projective plane. 
\end{pro}

\begin{proof}
Assume that $f$ has positive entropy $\log\lambda(f)>0$. According to~\cite{Can1} there exists a non-trivial cohomology class $\theta$ in $\mathrm{H}^2(\mathrm{S},\mathbb{R})$ such that 
$f^*\theta=~\lambda(f)\theta$ and $\theta^2=0$. Moreover $f_*\mathrm{K}_\mathrm{S}=f^*\mathrm{K}_\mathrm{S}=\mathrm{K}_\mathrm{S}$. Since $$(\theta,\mathrm{K}_\mathrm{S})=(f^*\theta,
f^*\mathrm{K}_\mathrm{S})=(\lambda(f)\theta,\mathrm{K}_\mathrm{S})$$ we have $(\theta,\mathrm{K}_\mathrm{S})=0$. The intersection form on $\mathrm{S}$ has signature $(1,n-1)$ 
and~$\mathrm{K}_{\mathrm{S}}^2\geq 0$ for $n\leq 9$ so $\theta=c\mathrm{K}_\mathrm{S}$ for some $c<0$. But then $f^*\theta=\theta\not=\lambda(f)\theta$: contradiction. The map $f$ thus 
has zero entropy.

If $n\leq 8$, then $\mathrm{K}_\mathrm{S}^2>0$.  The intersection form is thus strictly negative on the orthogonal complement $H\subset\mathrm{H}^2(\mathrm{S},\mathbb{R})$ of 
$\mathrm{K}_{\mathrm{S}}$. But $\dim H$ is finite, $H$ is invariant under $f^*$ and $f^*$ preserves $\mathrm{H}^2(\mathrm{S},\mathbb{Z})$ so $f^*$ has finite order on $H$. Therefore 
$f^{k*}$ is trivial for some integer $k$. In particular $f^k$ preserves each of the exceptional divisors in $X$ that correspond to the $n\leq 8$ points blown up in $\mathbb{P}^2(
\mathbb{C})$. So $f^k$ descends to a well-defined automorphism of $\mathbb{P}^2(\mathbb{C})$.
\end{proof}

  Let $f$ be an automorphism with positive entropy on a K\"{a}hler surface. The following statement gives properties on the eigenvalues of~$f^*.$

\begin{thm}[\cite{Bedford}, Theorem 2.8, Corollary 2.9]\label{Thm:Salem}
Let $f$ be an automorphism with posi\-tive entropy $\log\lambda(f)$ on a K\"{a}hler surface. The first dyna\-mical degree $\lambda(f)$ is an eigenvalue of~$f^*$ with multiplicity $1$ and 
this is the unique eigenvalue with modulus strictly greater than $1.$ 

  If $\eta$ is an eigenvalue of $f^*,$ then either $\eta$ belongs to $\{\lambda(f), \lambda(f)^{-1}\},$ or~$\vert\eta\vert$ is equal to $1.$   
\end{thm}

\begin{proof}
Let $v_1$, $\ldots$, $v_k$ denote the eigenvectors of $f^*$ for which the associated eigenvalue $\mu_\ell$ has modulus $>1$. We have
\begin{align*}
&(v_j,v_k)=(f^*v_j,f^*v_k)=\mu_j\overline{\mu_k}(v_j,v_k), &&\forall\, 1\leq j\leq k
\end{align*}
so $(v_j,v_k)=0$. Let $L$ be the linear span of $v_1$, $\ldots$, $v_k$. Each element $v=\sum_i\alpha_iv_i$ in $L$ satisfies $(v,v)=0$. According to Theorem \ref{signature} $\dim L\leq 1$. 
But since $\lambda(f)>1$, $L$ is spanned by a unique nontrivial eigenvector. If $v$ has eigenvalue $\mu$, then $\overline{v}$ has eigenvalue $\overline{\mu}$ so we must have 
$\mu=\overline{\mu}=\lambda(f)$.

Let us see that $\lambda(f)$ has multiplicity one. Assume that it has not; then there exists $\theta$ such that $f^*\theta=\lambda(f)\theta+cv$. In this case $$(\theta,v)=(f^*\theta,
f^*v)=(\lambda(f)\theta+cv,\lambda v)=\lambda^2(\theta,v)$$ so $(\theta,v)=0$. Similarly we have $(\theta,\theta)=0$ so by Theorem \ref{signature} again, the space spanned by $\theta$ 
and $v$ must have dimension $1$; in other words $\lambda(f)$ is a simple eigenvalue.

\smallskip

We know that $\lambda(f)$ is the only eigenvalue of modulus $>1$. Since $(f^*)^{-1}=(f^{-1})^*$, if $\eta$ is an eigenvalue of $f^*$, then $\frac{1}{\eta}$ is an eigenvalue of 
$(f^{-1})^*$. Applying the first statement to $f^{-1}$ we obtain that $\lambda$ is the only eigenvalue of $(f^{-1})^*$ with modulus strictly larger than $1$.
\end{proof}

  Let $\chi_f$ denote the characteristic polynomial of $f^*.$ This is a monic polynomial whose constant term is $\pm 1$ (constant term is equal to the determinant of $f^*$). Let $\Psi_f$ 
be the minimal polynomial of $\lambda(f).$ Except for $\lambda(f)$ and~$\lambda(f)^{-1}$ all zeroes of $\chi_f$ (and thus of $\Psi_f$) lie on the unit circle. Such polynomial is a 
\textbf{\textit{Salem polynomial}}\label{Chap8:ind26} and such a $\lambda(f)$ is a \textbf{\textit{Salem number}}\label{Chap8:ind27}. So Theorem \ref{Thm:Salem} says that if $f$ is
conjugate to an automorphism then $\lambda(f)$ is a Salem number; in fact the converse is true (\cite{BlCa}). There exists another birational invariant which allows us to characterize
birational maps that are conjugate to automorphisms (\emph{see} \cite{BlDe, BlCa}).

\section{Linearization and Fatou sets}\label{Sec:dyn}

\subsection{Linearization}\,

Let us recall some facts about linearization of germs of holomorphic diffeomorphism in dimension $1$ when the modulus of the multipliers is $1.$ Let us consider
\begin{equation}\label{cremer}
f(z)=\alpha z+a_2z^2+a_3z^3+ \ldots,\,\,\,\,\,\,\,\, \alpha=\mathrm{e}^{2\mathbf{i}\pi\theta}, \,\,\,\,\,\,\,\, \theta \in\mathbb{R}\setminus \mathbb{Q}
\end{equation}

We are looking for $\psi(z)=z+b_2z^2+\ldots$ such that $f\psi(z)=\psi(\alpha z).$ Since we can formally compute the coefficients~$b_i$
$$b_2=\frac{a_2}{\alpha^2-\alpha},\,\,\,\ldots,\,\,\, b_n=\frac{a_n+Q_n}{\alpha^n-\alpha}$$ with $Q_n\in\mathbb{Z}[a_i,\,\,i\leq n-1,\,\, b_i,\,\, i\leq n]$ we say that 
$f$ is \textbf{\textit{formally lineari\-zable}}\label{Chap8:ind28}. If $\psi$ converges, we say that the germ $f$ is \textbf{\textit{analytically lineari\-zable}}\label{Chap8:ind29}.

\begin{thm}[Cremer]
If $\liminf\vert\alpha^q-\alpha\vert^{1/q}=0,$ there exists an analytic germ $f$ of the type $($\ref{cremer}$)$ which is not analytically linearizable. 

More precisely if $\liminf\vert\alpha^q-\alpha\vert^{\frac{1}{\nu^q}}=0,$ then no polynomial germ $$f(z)=\alpha z+a_2z^2+\ldots+z^\nu$$ of degree $\nu$ is linearizable.
\end{thm}

\begin{thm}[Siegel]
If there exist two constants $c$ and $M$ strictly po\-sitive such that $\vert\alpha^q -\alpha\vert\geq\frac{c}{q^M}$ then any germ $f(z)=\alpha z+a_2z^2+\ldots$ is analytically linearizable.
\end{thm}

Let us now deal with the case of two variables. Let us consider $$f(x,y)=(\alpha x,\beta y)+ \text{ h.o.t. }$$ with $\alpha,$~$\beta$ of modulus $1$ but not root of unity. The pair $(\alpha,\beta)$
is \textbf{\textit{resonant}}\label{Chap8:ind30a} if there exists a relation of the form $\alpha=\alpha^a \beta^b$ or~$\beta=~\alpha^a\beta^b$ where $a,$ $b$ are some positive integers 
such that~$a+b\geq 2.$ A \textbf{\textit{resonant monomial}}\label{Chap8:ind31} is a monomial of the form~$x^ay^b.$ We say that $\alpha$ and $\beta$ are \textbf{\textit{multiplicatively 
independent}}\label{Chap8:ind32} if the unique solution of $\alpha^a \beta^b=1$ with $a,$ $b$ in $\mathbb{Z}$ is $(0,0).$ The numbers $\alpha$ and $\beta$ are 
\textbf{\textit{simultaneously diophantine}}\label{Chap8:ind33} if there exist two positive constants $c$ and $M$ such that 
\begin{align*}
&\min\Big(\vert\alpha^a\beta^b-\alpha\vert,\,\vert\alpha^a\beta^b-\beta\vert\Big)\geq \frac{c}{\vert a+b \vert^M}&& \forall a,\, b\in\mathbb{N},\, a+b\geq 2. 
\end{align*}

\begin{thm}
If $\alpha$ and $\beta$ are simultaneously diophantine then $f$ is li\-nearizable.

If $\alpha$ and $\beta$ are algebraic and multiplicatively independent then they are simultaneously diophantine. 
\end{thm}

For more details \emph{see} \cite{Ar, Blanchard, H, Sie}.

\subsection{Fatou sets}\,

\subsubsection{Definitions and properties}

Let $f$ be an automorphism on a compact complex ma\-nifold $\mathrm{M}$. Let us recall that the \textbf{\textit{Fatou set}}\label{Chap8:ind34} $\mathcal{F}(f)$ of $f$ is the set of points which 
own a neighborhood $\mathcal{V}$ such that~$\big\{f^n_{\vert\mathcal{V}},\, n\geq~0\big\}$ is a normal family. Let us consider $$\mathcal{G}=\mathcal{G}(\mathcal{U})=\big\{\psi\colon 
\mathcal{U}\to\overline{\mathcal{U}}\,\big\vert\, \psi=\lim_{n_j\to +\infty} f^{n_j}\big\}.$$ We say that $\mathcal{U}$ is a \textbf{\textit{rotation domain}}\label{Chap8:ind35} if 
$\mathcal{G}$ is a subgroup of $\mathrm{Aut}(\mathcal{U})$, that is, if any element of~$\mathcal{G}$ defines an automorphism of $\mathcal{U}.$ An equivalent definition is the following: 
if $\mathcal{U}$ is a component of~$\mathcal{F}(f)$ which is invariant by $f$, we say that $\mathcal{U}$ is a rotation domain if $f_{\vert\mathcal{U}}$ is conjugate to a linear 
rotation; in dimension $1$ this is equivalent to have a Siegel disk. We have the following properties (\cite{BK4}).
\begin{itemize}
\item[$\bullet$] If $f$ preserves a smooth volume form, then any Fatou component is a rotation domain. 

\item[$\bullet$] If $\mathcal{U}$ is a rotation domain, $\mathcal{G}$ is a subgroup of $\mathrm{Aut}(\mathrm{M}).$

\item[$\bullet$] A Fatou component $\mathcal{U}$ is a rotation domain if and only there exists a subsequence such that $(n_j)\to+\infty$ and such that $(f^{n_j})$ converges uniformly 
to the identity on compact subsets of $\mathcal{U}.$

\item[$\bullet$]  If $\mathcal{U}$ is a rotation domain, $\mathcal{G}$ is a compact Lie group and the action of~$\mathcal{G}$ on $\mathcal{U}$ is analytic real.
\end{itemize}

Let $\mathcal{G}_0$ be the connected component of the identity of $\mathcal{G}.$ Since $\mathcal{G}$ is a compact, infinite, abelian Lie group,~$\mathcal{G}_0$ is a torus of dimension 
$d\geq 0;$ let us note that $d\leq \dim_\mathbb{C} \mathrm{M}.$ We say that~$d$ is the \textbf{\textit{rank of the rotation domain}}\label{Chap8:ind36}. The rank is equal to the dimension of 
the closure of a generic orbit of a point in $\mathcal{U}.$

We have some geometric information on the rotation domains: if $\mathcal{U}$ is a rotation domain then it is pseudo-convex (\cite{BK4}). 

Let us give some details when $\mathrm{M}$ is a k\"{a}hlerian surface carrying an automorphism with positive entropy.

\begin{thm}[\cite{BK4}]\label{rangrang}
Let $\mathrm{S}$ be a compact, k\"{a}hlerian surface and let $f$ be an automorphism of $\mathrm{S}$ with positive entropy. Let $\mathcal{U}$ be a rotation domain of rank $d$. Then 
$d\leq 2.$

If $d=2$ the $\mathcal{G}_0$-orbit of a generic point of $\mathcal{U}$ is a real $2$-torus.

If $d=1,$ there exists a holomorphic vector field which induces a foliation by Riemann surfaces on $\mathrm{S}$ whose any leaf is invariant by $\mathcal{G}_0.$
\end{thm}

We can use an argument of local linearization to show that some fixed points belong to the Fatou set. Conversely we can always linearize a fixed point of the Fatou set. 

\subsubsection{Fatou sets of H\'enon automorphisms}

Let $f$ be a H\'enon automorphism. Let us denote by $\mathcal{K}^\pm$ the subset of $\mathbb{C}^2$ whose positive/negative orbit is bounded: $$\mathcal{K}^\pm=\big\{(x,y)\in\mathbb{C}^2\,
\big\vert \,\big\{f^{\pm n}(x,y)\,\vert\, n\geq 0\big\} \text{ is bounded}\big\}.$$ Set 
\begin{align*}
& \mathcal{K}=\mathcal{K}^+\cap \mathcal{K}^-, && \mathcal{J}^\pm=\partial \mathcal{K}^\pm, && \mathcal{J}=\mathcal{J}^+\cap \mathcal{J}^-, && \mathcal{U}^+=\mathbb{C}^2\setminus 
\mathcal{K}^+.
\end{align*}

Let us state some properties.
\begin{itemize}
\item[$\bullet$]  The family of the iterates $f^n,$ $n\geq 0,$ is a normal family in the interior of $\mathcal{K}^+.$

\item[$\bullet$]  If $(x,y)$ belongs to $\mathcal{J}^+$ there exists no neighborhood $U$ of $(x,y)$ on which the family $\big\{f_{\vert U}^n\,\big\vert\, n\geq 0\big\}$ is normal.
\end{itemize}

We have the following statement.

\begin{pro}
The Fatou set of a H\'enon map is $\mathbb{C}^2\setminus \mathcal{J}^+.$
\end{pro}

\begin{defis}
Let $\Omega$ be a Fatou component; $\Omega$ is \textbf{\textit{recurrent}}\label{Chap8:ind37} if there exist a compact subset $C$ of $\Omega$ and a point $m$ in $C$ such that 
$f^{n_j}(m)$ belongs to $C$ for an infinite number of~$n_j\to~+\infty.$ A recurrent Fatou component is periodic.

A fixed point $m$ of $f$ is a \textbf{\textit{sink}}\label{Chap8:ind38} if $m$ belongs to the interior of the stable manifold $$\mathrm{W}^s(m)=\big\{p\,\big\vert\,\lim_{n\to +\infty} 
\text{dist}(f^n(m),f^n(p))=0 \big\}.$$ We say that $\mathrm{W}^s(m)$ is the \textbf{\textit{basin}}\label{Chap8:ind39} of $m.$ If $m$ is a sink, the eigenvalues of~$Df_{m}$ have all 
modulus less than $1.$ 

A \textbf{\textit{Siegel disk}}\label{Chap8:ind40} $($resp. \textbf{\textit{Herman ring}}\label{Chap8:ind41}$)$ is the image of a disk $($resp. of an annulus$)$~$\Delta$ by an injective 
holomorphic map $\varphi$ having the following property: for any $z$ in $\Delta$ we have
\begin{align*}
& f\varphi(z)=\varphi(\alpha z), && \alpha=\mathrm{e}^{2\mathbf{i}\pi\theta},\,\theta\in\mathbb{R} \setminus\mathbb{Q}.
\end{align*} 
\end{defis}

We can describe the recurrent Fatou components of a H\'enon map.

\begin{thm}[\cite{BS}]
Let $f$ be a H\'enon map with jacobian $<1$ and let $\Omega$ be a recurrent Fatou component. Then $\Omega$ is
\begin{itemize}
\item[$\bullet$] either the basin of a sink;

\item[$\bullet$]  or the basin of a Siegel disk;

\item[$\bullet$]  or a Herman ring.
\end{itemize}
\end{thm}

Under some assumptions the Fatou component of a H\'enon automorphisms are recurrent.

\begin{pro}
The Fatou component of a H\'enon map which preserves the volume are periodic and recurrent. 
\end{pro}

\subsection{Fatou sets of automorphisms with positive entropy on torus, (quotients of) K$3$, rational surfaces}\,

If $\mathrm{S}$ is a complex torus, an automorphism of positive entropy is essentially an element of $\mathrm{GL}_2(\mathbb{Z});$ since the entropy is positive, the eigenvalues 
satisfy: $\vert\lambda_1\vert<~1<~\vert\lambda_2\vert$ and the Fatou set is empty. 

Assume that $\mathrm{S}$ is a K$3$ surface or a quotient of a K$3$ surface. Since there exists a volume form, the only possible Fatou components are rotation domains. McMullen proved 
there exist non algebraic K$3$ surfaces with rotation domains of rank $2$ (\emph{see} \cite{Mc2}); we can also look at~\cite{Og}.  

The other compact surfaces carrying automorphisms with positive entropy are rational ones; in this case there are rotation domains of rank $1$, $2$ (\emph{see} \cite{BK2, Mc}). Other 
phenomena like attractive, repulsive basins can happen (\cite{BK2, Mc}).

\chapter{Weyl groups and automorphisms of positive entropy}\label{Chap:mcmdg}

  In \cite{Mc} McMullen, thanks to Nagata's works and Harbourne's works, establishes a result similar to Torelli's theorem for K$3$ surfaces: he constructs automorphisms on some 
rational surfaces prescribing the action of the automorphisms on cohomological groups of the surface. These rational surfaces own, up to multiplication by a constant, a 
unique meromorphic nowhere vanishing $2$-form $\Omega.$ If $f$ is an automorphism on $\mathrm{S}$ obtained via this construction,~$f^*\Omega$ is proportional to $\Omega$ and $f$ 
preserves the poles of $\Omega.$ When we project $\mathrm{S}$ on the complex projective plane,~$f$ induces a birational map preserving a cubic.

\medskip

  The relationship of the Weyl group to the birational geometry of the plane, used by McMullen, is discussed since $1895$ in \cite{Ka} and has been much developed since then 
(\cite{DV, Na, Na2, Co, [Gi], Lo, Ha2, Manin, Ha1, Ni, Ha3, DoOr, Hi, Zh, DZ}).

\section{Weyl groups}\,

  Let $\mathrm{S}$ be a surface obtained by blowing up the complex projective plane in a finite number of points. Let $\big\{\mathbf{e}_0,\,\ldots,\,\mathbf{e}_n\big\}$ be a basis of
 $\mathrm{H}^2(\mathrm{S}, \mathbb{Z});$ if 
\begin{align*}
&\mathbf{e}_0\cdot \mathbf{e}_0=1, && \mathbf{e}_j\cdot \mathbf{e}_j=-1,\,\,\forall\,\, 1\leq j\leq k, && \mathbf{e}_i\cdot \mathbf{e}_j=0,\,\,\forall\,\,0\leq i\not=j \leq n
\end{align*}  
then $\big\{\mathbf{e}_0,\,\ldots,\,\mathbf{e}_n\big\}$ is a \textbf{\textit{geometric basis}}\label{Chap8:ind28}. Consider $\alpha$ in $\mathrm{H}^2(\mathrm{S},\mathbb{Z})$ such that 
$\alpha\cdot\alpha=-2,$ then~$R_\alpha(x)=x+(x\cdot\alpha)\alpha$ sends $\alpha$ on~$-\alpha$ and $R_\alpha$ fixes each element of $\alpha^\perp;$ in other words~$R_\alpha$ is a 
reflection in the direction $\alpha.$ 

  Consider the vectors given by 
\begin{align*}
&\alpha_0=\mathbf{e}_0-\mathbf{e}_1-\mathbf{e}_2-\mathbf{e}_3, && \alpha_j=\mathbf{e}_{j+1}-\mathbf{e}_j,\, 1\leq j\leq n-1.
\end{align*}

  For all $j$ in $\{0,\ldots,n-1\}$ we have $\alpha_j\cdot\alpha_j=- 2.$ When $j$ is nonzero the reflection~$R_{\alpha_j}$ induces a permutation on $\{\mathbf{e}_j,\, 
\mathbf{e}_{j+1}\}.$ The subgroup genera\-ted by the $R_{\alpha_j}$'s, with $1\leq j\leq~n-~1,$ is the set of permutations on the elements $\{\mathbf{e}_1,\,\ldots,\,\mathbf{e}_n\}.$ 
Let $\mathrm{W}_n\subset\mathrm{O}(\mathbb{Z}^{1,n})$ denote the group $$\langle R_{\alpha_j}\,\vert\, 0\leq j\leq n-1\rangle$$ which is called 
\textbf{\textit{Weyl group}}\label{Chap8:ind29}. 

  The Weyl groups are, for $3\leq n\leq 8,$ isomorphic to the following finite groups
\begin{align*}
&A_1\times A_2, && A_4, && D_5, && E_6, && E_7, &&E_8
\end{align*}
and are associated to del Pezzo surfaces. For $n\geq 9$ Weyl groups are infinite and for~$n\geq 10$ Weyl groups contain elements 
with a spectral radius strictly greater than~$1.$

If $\mathrm{Y}$ and $\mathrm{S}$ are two projective surfaces, let us recall that $\mathrm{Y}$ \textbf{\textit{dominates}}\label{Chap13:ind1}~$\mathrm{S}$ if there exists a surjective 
algebraic birational morphism from $\mathrm{Y}$ to $\mathrm{S}.$

\begin{thm}[\cite{Dol}]
Let $\mathrm{S}$ be a rational surface which dominates $\mathbb{P}^2(\mathbb{C}).$ 
\begin{itemize}
\item[$\bullet$]  The Weyl group $\mathrm{W}_k\subset\mathrm{GL}(\mathrm{Pic}(\mathrm{S}))$ does not depend on the chosen exceptional configuration. 

\item[$\bullet$]  If $\mathcal{E}$ and $\mathcal{E}'$ are two distinct exceptional configurations, there exists $w$ in $\mathrm{W}_k$ such that $w(\mathcal{E})=\mathcal{E}'.$

\item[$\bullet$]  If $\mathrm{S}$ is obtained by blowing up $k$ generic points and if $\mathcal{E}$ is an exceptional configuration, then for any $w$ in the Weyl group $w(\mathcal{E})$ is an 
exceptional configuration. 
\end{itemize}
\end{thm}

  If $f$ is an automorphism of $\mathrm{S},$ by a theorem of Nagata there exists a unique element~$w$ in $\mathrm{W}_n$ such that  $$\xymatrix{\mathbb{Z}^{1,n}\ar[d]_{\varphi}\ar[r]^{w} 
&\mathbb{Z}^{1,n}\ar[d]^{\varphi} \\
\mathrm{H}^2(\mathrm{S},\mathbb{Z})\ar[r]^{f_*} &\mathrm{H}^2(\mathrm{S},
\mathbb{Z})}$$ commutes; we said that the automorphism $f$ \textbf{\textit{realizes}}\label{Chap8:ind30} $\omega.$ 

  A product of generators $R_{\alpha_j}$is a \textbf{\textit{Coxeter element}}\label{Chap8:ind31} of $\mathrm{W}_n.$ Note that all Coxeter elements are conjugate so the spectral radius 
of a Coxeter element is well defined.

  The map $\sigma$ is represented by the reflection $\kappa_{ijk}=R_{\alpha_{ijk}}$ where $\alpha_{ijk}=\mathbf{e}_0-\mathbf{e}_i-\mathbf{e}_j-\mathbf{e}_k$ and~$i$, $j$, $k\geq~1$ are 
distinct elements; it acts as follows 
\begin{align*}
& \mathbf{e}_0\to 2 \mathbf{e}_0-\mathbf{e}_i-\mathbf{e}_j-\mathbf{e}_k, && \mathbf{e}_i\to \mathbf{e}_0-\mathbf{e}_j-\mathbf{e}_k, && \mathbf{e}_j\to \mathbf{e}_0-\mathbf{e}_i-
\mathbf{e}_k
\end{align*}
\begin{align*}
& \mathbf{e}_k\to \mathbf{e}_0-\mathbf{e}_i-\mathbf{e}_j, && \mathbf{e}_\ell\to \mathbf{e}_\ell \text{ if }\ell\not\in\{0,\,i,\,j,\,k\} 
\end{align*}

  When $n=3,$ we say that $\kappa_{123}$ is the \textbf{\textit{standard element}} of $\mathrm{W}_3.$  Consider the cyclic permutation $$(123\ldots n)=\kappa_{123}R_{\alpha_1}\ldots 
R_{\alpha_{n-1}}\in\Sigma_n\subset\mathrm{W}_n;$$ let us denote it by $\pi_n.$ For $n\geq 4$ we define the \textbf{\textit{standard element}}\label{Chap8:ind32} $w$ of $\mathrm{W}_n$ by 
$w=\pi_n\kappa_{123}.$ It satisfies
\begin{align*}
& w(\mathbf{e}_0)=2\mathbf{e}_0-\mathbf{e}_2-\mathbf{e}_3-\mathbf{e}_4, && w(\mathbf{e}_1)=\mathbf{e}_0-\mathbf{e}_3-\mathbf{e}_4, &&w(\mathbf{e}_2)=\mathbf{e}_0-\mathbf{e}_2-
\mathbf{e}_4,
\end{align*}
\begin{align*}
& w(\mathbf{e}_3)=\mathbf{e}_0-\mathbf{e}_2-\mathbf{e}_3,&&w(\mathbf{e}_j)=\mathbf{e}_{j+1}, \,\, 4\leq j \leq n-2, && w(\mathbf{e}_{n-1})=\mathbf{e}_1.
\end{align*}

\section{Statements}\,

  In \cite{Mc} McMullen constructs examples of automorphisms with posi\-tive entropy ``thanks to'' elements of Weyl groups. 

\begin{thm}[\cite{Mc}]
For $n\geq 10,$ the standard element of $\mathrm{W}_n$ can be realizable by an automorphism~$f_n$ with positive entropy $\log(\lambda_n)$ of a rational surface $\mathrm{S}_n.$
\end{thm}

  More precisely the automorphism $f_n\colon\mathrm{S}_n\to\mathrm{S}_n$ can be chosen to have the following additional properties:
\begin{itemize}
\item[$\bullet$] $\mathrm{S}_n$ is the complex projective plane blown up in $n$ distinct points $p_1,$ $\ldots,$ $p_n$ lying on a cuspidal cubic curve $\mathcal{C},$ 

\item[$\bullet$] there exists a nowhere vanishing meromorphic $2$-form $\eta$ on $\mathrm{S}_n$ with a simple
pole along the proper transform of $\mathcal{C},$ 

\item[$\bullet$] $f_n^*(\eta)=\lambda_n\cdot\eta,$

\item[$\bullet$] $(\langle f_n\rangle,\mathrm{S}_n)$ is minimal in the sense of Manin\footnote{Let $\mathrm{Z}$ be a surface and $\mathrm{G}$ be a subgroup of $\mathrm{Aut}(\mathrm{S}).$ 
A birational map $f\colon\mathrm{S}\dashrightarrow\widetilde{\mathrm{S}}$ is $\mathrm{G}$-equivariant if $\widetilde{\mathrm{G}}=f \mathrm{G}f^{-1}\subset 
\mathrm{Aut}(\widetilde{\mathrm{S}}).$ The pair $(\mathrm{G},\mathrm{S})$ is minimal if every $\mathrm{G}$-equivariant birational morphism is an isomorphism.}.
\end{itemize}

  The first three properties determine $f_n$ uniquely. The points $p_i$ admit a simple description which leads to concrete formulas for $f_n.$ 

  The smallest known Salem number is a root $\lambda_{\text{Lehmer}}\sim 1.17628081$ of Lehmer's polynom $$L(t)=t^{10}+t^9-t^7-t^6-t^5-t^4-t^3+t+1.$$

\begin{thm}[\cite{Mc}]
If $f$ is an automorphism of a compact complex surface with po\-sitive entropy, then $\mathrm{h}_{\text{top}}(f)\geq\log\lambda_{\text{Lehmer}}.$
\end{thm}

\begin{cor}[\cite{Mc}]
The map $f_{10}\colon\mathrm{S}_{10}\to\mathrm{S}_{10}$ is an automorphism of~$\mathrm{S}_{10}$ with the smallest possible positive entropy.
\end{cor}

\begin{thm}[\cite{Mc}]
There is an infinite number of $n$ for which the standard element of~$\mathrm{W}_n$ can be realized as an automorphism of $\mathbb{P}^2(\mathbb{C})$ blown up in a finite number of points 
having a Siegel disk. 
\end{thm}

  Let us also mention a more recent work in this direction (\cite{Ue}). Diller also find examples using plane cubics~(\cite{Di2}).

\section{Tools}\,

\subsection{Marked cubic curves}

A \textbf{\textit{cubic curve}}\label{Chap8:ind32a} $\mathcal{C}\subset\mathbb{P}^2(\mathbb{C})$ is a reduced curve of degree $3$. It can be singular or reducible; let us denote by 
$\mathcal{C}^*$ its smooth part. Let us recall some properties of the Picard group of such a curve (\emph{see} \cite{HM} for more details). We have the following exact sequence 
$$0\longrightarrow\mathrm{Pic}_0(\mathcal{C})\longrightarrow \mathrm{Pic}(\mathcal{C})\longrightarrow\mathrm{H}^2(\mathcal{C},\mathbb{Z}) \longrightarrow 0$$ where 
$\mathrm{Pic}_0(\mathcal{C})$ is isomorphic to
\begin{itemize}
\item[$\bullet$] either a torus $\mathbb{C}/\Lambda$ (when $\mathcal{C}$ is smooth);

\item[$\bullet$] or to the multiplicative group $\mathbb{C}^*$ (it corresponds to the following case:~$\mathcal{C}$ is either a nodal cubic or the union of a cubic curve and a 
transverse line, or the union of three line in general position); 

\item[$\bullet$] or to the additive group $\mathbb{C}$ (when $\mathcal{C}$ is either a cuspidal cubic, or the union of a conic and a tangent line, or the union of three lines 
through a single point). 
\end{itemize}

A \textbf{\textit{cubic marked curve}}\label{Chap8:ind32b} is a pair $(\mathcal{C},\eta)$ of an abstract curve $\mathcal{C}$ equipped with a homomorphism $\eta\colon\mathbb{Z}^{1,n}
\to\mathrm{Pic} (\mathcal{C})$ such that
\begin{itemize}
\item[$\bullet$] the sections of the line bundle $\eta(\mathbf{e}_0)$ provide an embedding of $\mathcal{C}$ into~$\mathbb{P}^2(\mathbb{C});$ 

\item[$\bullet$] there exist distinct base-points $p_i$ on $\mathcal{C}^*$ for which $\eta(\mathbf{e}_i)=[p_i]$ for any $i=2,$ $\ldots,$ $n.$
\end{itemize}

The base-points $p_i$ are uniquely determined by $\eta$ since $\mathbb{C}^*$ can be embedded into $\mathrm{Pic}(\mathcal{C}).$ Conversely a cubic curve $\mathcal{C}$ which embeds 
into $\mathbb{P}^2(\mathbb{C})$ and a collection of distinct points on $\mathcal{C}^*$ determine a marking of $\mathcal{C}.$

\begin{rem}
Different markings of $\mathcal{C}$ can yield different projective embeddings $\mathcal{C}\hookrightarrow~\mathbb{P}^2(\mathbb{C})$ but all these embeddings are equivalent under the 
action of $\mathrm{Aut}(\mathcal{C}).$ 
\end{rem}

Let $(\mathcal{C},\eta)$ and $(\mathcal{C}',\eta')$ be two cubic marked curves; an \textbf{\textit{isomorphism}}\label{Chap8:ind32c} between $(\mathcal{C},\eta)$ and $(\mathcal{C}',
\eta')$ is a biholomorphic application $f\colon\mathcal{C}\to\mathcal{C}'$ such that $\eta'=f_*\circ\eta.$

Let $(\mathcal{C},\eta)$ be a cubic marked curve; let us set $$W(\mathcal{C},\eta)=\big\{w\in\mathrm{W}_n\,\big\vert\,(\mathcal{C},\eta w)\text{ is a cubic marked curve}\big\},$$ 
$$\mathrm{Aut}(\mathcal{C},\eta)=\big\{w\in W(\mathcal{C},\eta)\,\big\vert \,(\mathcal{C},\eta) \,\&\, (\mathcal{C}',\eta') \text{ are isomorphic}\big\}.$$ 

We can decompose the marking $\eta$ of $\mathcal{C}$ in two pieces 
\begin{align*}
& \eta_0\colon\ker(\deg\circ\eta)\to\mathrm{Pic}_0(\mathcal{C}), && \deg\circ\eta\colon \mathbb{Z}^{1,n}\to\mathrm{H}^2(\mathcal{C},\mathbb{Z}).
\end{align*}

We have the following property.

\begin{thm}[\cite{Mc}]
Let $(\mathcal{C},\eta)$ be a marked cubic curve. The applications~$\eta_0$ and~$\deg\circ\eta$ determine $(\mathcal{C},\eta)$ up to isomorphism.
\end{thm}

A consequence of this statement is the following. 

\begin{cor}[\cite{Mc}]\label{cubcub}
An irreducible marked cubic curve $(\mathcal{C},\eta)$ is determined, up to isomorphism, by $\eta_0\colon L_n\to\mathrm{Pic}_0(\mathcal{C}).$
\end{cor}

\subsection{Marked blow-ups}

A \textbf{\textit{marked blow-up}}\label{Chap8:ind32d} $(\mathrm{S},\Phi)$ is the data of a smooth projective surface $\mathrm{S}$ and an isomorphism 
$\Phi\colon\mathbb{Z}^{1,n}\to \mathrm{H}^2(\mathrm{S},\mathbb{Z})$ such that
\begin{itemize}
\item[$\bullet$] $\Phi$ sends the Minkowski inner product $(x\cdot x)=x^2=x_0^2-x_1^2-\ldots-x_n^2$ on the intersection pairing on $\mathrm{H}^2(\mathrm{S},\mathbb{Z});$
                   
\item[$\bullet$] there exists a birational morphism $\pi\colon \mathrm{S}\to\mathbb{P}^2(\mathbb{C})$ presenting $\mathrm{S}$ as the blow-up of~$\mathbb{P}^2(\mathbb{C})$ 
in $n$ distinct base-points $p_1,$ $\ldots,$ $p_n;$

\item[$\bullet$] $\Phi(\mathbf{e}_0)=[\mathrm{H}]$ and $\Phi(\mathbf{e}_i)=[\mathrm{E}_i]$ for any $i=1$, $\ldots$, $n$ where $\mathrm{H}$ is the pre-image of a generic line 
in $\mathbb{P}^2(\mathbb{C})$ and $\mathrm{E}_i$ the divisor obtained by blowing up  $p_i.$
\end{itemize}

The marking determines the morphism $\pi\colon \mathrm{S}\to\mathbb{P}^2(\mathbb{C})$ up to the action of an automorphism of $\mathbb{P}^2(\mathbb{C}).$

Let $(\mathrm{S},\Phi)$ and $(\mathrm{S}',\Phi)$  be two marked blow-ups; an \textbf{\textit{isomorphism}}\label{Chap8:ind32e} between $(\mathrm{S},\Phi)$ and $(\mathrm{S}',\Phi')$ 
is a biholomorphic application $F\colon \mathrm{S}\to \mathrm{S}'$ such that the following diagram $$\xymatrix{& \mathbb{Z}^{1,n}\ar[dl]_{\Phi}\ar[dr]^{\Phi'} &\\
\mathrm{H}^2(\mathrm{S},\mathbb{Z})\ar[rr]_{F_*} & &\mathrm{H}^2(\mathrm{S}',\mathbb{Z})}$$ commutes. If $(\mathrm{S},\Phi)$ and $(\mathrm{S}',\Phi')$ are isomorphic, there 
exists an automorphism $\varphi$ of $\mathbb{P}^2(\mathbb{C})$ such that $p'_i=~\varphi(p_i).$

Assume that there exist two birational morphisms $\pi,$ $\pi'\colon \mathrm{S}\to \mathbb{P}^2(\mathbb{C})$ such that $\mathrm{S}$ is the surface obtained by blowing up 
$\mathbb{P}^2(\mathbb{C})$ in $p_1,$ $\ldots,$ $p_n$ (resp. $p'_1,$ $\ldots,$ $p'_n$) via $\pi$ (resp. $\pi'$).There exists a birational map $f\colon\mathbb{P}^2(\mathbb{C}) 
\dashrightarrow\mathbb{P}^2(\mathbb{C})$ such that the diagram $$\xymatrix{& \mathrm{S}\ar[dl]_{\pi}\ar[dr]^{\pi'} &\\
\mathbb{P}^2(\mathbb{C})\ar@{-->}[rr]_f & &\mathbb{P}^2(\mathbb{C})}$$ commutes; moreover there exists a unique element $w$ in $\mathbb{Z}^{1,n}$ such that $\Phi'=\Phi w.$

\medskip

The Weyl group satisfies the following property due to Nagata: let $(\mathrm{S},\Phi)$ be a marked blow-up and let $w$ be an element of $\mathbb{Z}^{1,n}.$ If $(\mathrm{S},\Phi w)$ 
is still a marked blow-up, then~$w$ belongs to the Weyl group $\mathrm{W}_n.$ Let $(\mathrm{S},\Phi)$ be a marked blow-up; let us denote by~$W(\mathrm{S},\Phi)$ the set of elements 
$w$ of $\mathrm{W}_n$ such that $(\mathrm{S},\Phi w)$ is a marked blow-up: $$W(\mathrm{S},\Phi)=\big\{w\in\mathrm{W}_n\,\big\vert\,(\mathrm{S}, \Phi w)\text{ is a marked blow-up}\big\}.$$
The right action of the symmetric group reorders the base-points of a blow-up so the group of permutations is contained in $W(\mathrm{S},\Phi).$ The following statement gives other 
examples of elements of $W(\mathrm{S},\Phi).$ 

\begin{thm}[\cite{Mc}]\label{eclmar}
Let $(\mathrm{S},\Phi)$ be a marked blow-up and let $\sigma$ be the involution $(x:y:z)\dashrightarrow(yz:xz:xy)$. Let us denote by $p_1,$ $\ldots,$ $p_n$ the base-points of $(\mathrm{S},\Phi).$ If, for 
any $4\leq k\leq n,$ the point~$p_k$ does not belong to the line through $p_i$ and $p_j,$ where $1\leq i,j\leq 3,$ $i\not=j,$ then $(\mathrm{S},\Phi\kappa_{123})$ is a marked blow-up.
\end{thm}

\begin{proof}
Let $\pi\colon \mathrm{S}\to\mathbb{P}^2(\mathbb{C})$ be the birational morphism associated to the marked blow-up $(\mathrm{S},\Phi).$ Let us denote by $q_1,$ $q_2$ and $q_3$ the points 
of indeterminacy of $\sigma.$ Let us choose some coordinates for which $p_i=q_i$ for $i=1,$ $2,$~$3;$ then $\pi'=\sigma\pi\colon \mathrm{S}\to\mathbb{P}^2(\mathbb{C})$ is a birational 
morphism which allows us to see $(\mathrm{S},\Phi\kappa_{123})$ as a marked blow-up with base-points $p_1,$ $p_2,$ $p_3$ and $\sigma(p_i)$ for $i\geq 4.$ These points are distinct 
since, by hypothesis,~$p_4,$ $\ldots,$ $p_n$ do not belong to the lines contracted by $\sigma$.
\end{proof}

A root $\alpha$ of $\Theta_n$ is a \textbf{\textit{nodal root}}\label{Chap8:ind32f} for $(\mathrm{S},\Phi)$ if $\Phi(\alpha)$ is represented by an effective divisor $D.$ In this case $D$ 
projects to a curve of degree $d>0$ on $\mathbb{P}^2(\mathbb{C});$ thus~$\alpha=d\mathbf{e}_0-\sum_{i\geq 1}m_i\mathbf{e}_i$ is a positive root. A nodal root is 
\textbf{\textit{geometric}}\label{Chap8:ind32g} if we can write $D$ as a sum of smooth rational curves. 

\begin{thm}[\cite{Mc}]\label{racnod}
Let $(\mathrm{S},\Phi)$ be a marked blow-up. If three of the base-points are colinear, $(\mathrm{S},\Phi)$ has a geometric nodal root.
\end{thm}

\begin{proof}
After reordering the base-points $p_1,$ $\ldots,$ $p_n,$ we can assume that $p_1,$ $p_2$ and $p_3$ are colinear; let us denote by $L$ the line through these three points. We can 
suppose that the base-points which belong to $L$ are $p_1,$ $\ldots,$ $p_k.$ The strict transform~$\widetilde{L}$ of~$L$ induces a smooth rational curve on $\mathrm{S}$ with 
$[\widetilde{L}]=[\mathrm{H}-\sum_{i=1}^k\mathrm{E}_i]$ so $$\Phi(\alpha_{123})=[ \widetilde{L}+\sum_{i=1}^k\mathrm{E}_i].$$
\end{proof}

\begin{thm}[\cite{Mc}]\label{egal}
Let $(\mathrm{S},\Phi)$ be a marked blow-up. If $(\mathrm{S},\Phi)$ has no geometric nodal root, then $$W(\mathrm{S}, \Phi)= \mathrm{W}_n.$$
\end{thm}

\begin{proof}
If  $(\mathrm{S},\Phi)$ has no geometric nodal root and if $w$ belongs to $W(\mathrm{S},\Phi),$ then $(\mathrm{S},\Phi w)$ has no geometric nodal root. It is so sufficient to prove 
that the generators of 
$\mathrm{W}_n$ belong to $W(\mathrm{S},\Phi).$ Since the group of permutations is contained in $W(\mathrm{S},\Phi)$, it is clear for the transpositions; for~$\kappa_{123}$ it is a 
consequence of 
Theorems \ref{eclmar} and \ref{racnod}.
\end{proof}

\begin{cor}[\cite{Mc}]
A marked surface has a nodal root if and only if it has a geometric nodal root. 
\end{cor}

\subsection{Marked pairs}

\subsubsection{First definitions}

Let $(\mathrm{S},\Phi)$ be a marked blow-up. Let us recall that an \textbf{\textit{anticanonical curve}}\label{Chap8:ind32h} is a reduced curve~$Y\subset \mathrm{S}$ such that its class in 
$\mathrm{H}^2(\mathrm{S},\mathbb{Z})$ satisfies
\begin{equation}\label{antican}
[Y]=[3\mathrm{H}-\sum_i\mathrm{E}_i]=-\mathrm{K}_\mathrm{S}.
\end{equation}

A \textbf{\textit{marked pair}}\label{Chap8:ind32i} $(\mathrm{S},\Phi,Y)$ is the data of a marked blow-up $(\mathrm{S},\Phi)$ and an anticanonical curve~$Y.$ An 
{\textbf{\textit{isomorphism}}\label{Chap8:ind32j} between marked pairs $(\mathrm{S},\Phi,Y)$ and $(\mathrm{S}',\Phi',Y')$ is a biholomorphism $f$ from~$\mathrm{S}$ into $\mathrm{S}',$ 
compatible with markings and 
which sends $Y$ to $Y'.$ If~$n\geq 10,$ then~$\mathrm{S}$ contains at most one irreducible anticanonical curve; indeed if such a curve  $Y$ exists, then $Y^2=9-n<0.$

\bigskip

\subsubsection{From surfaces to cubic curves}

Let us consider a marked pair $(\mathrm{S},\Phi,Y).$ Let $\pi$ be the projection of $\mathrm{S}$ to $\mathbb{P}^2(\mathbb{C})$ compatible with $\Phi.$ The equality (\ref{antican}) 
implies that $\mathcal{C}=\pi(Y)$ is a cubic curve through any base-point~$p_i$ with multiplicity $1.$ Moreover, $\mathrm{E}_i\cdot Y=1$ implies that $\pi\colon Y\to\mathcal{C}$ is 
an isomorphism. 
The identification of~$\mathrm{H}^2(\mathrm{S}, \mathbb{Z})$ and $\mathrm{Pic}(\mathrm{S})$ allows us to obtain the natural marking $$\eta\colon \mathbb{Z}^{1,n}\stackrel{\Phi}{\longrightarrow}
\mathrm{H}^2(\mathrm{S}, \mathbb{Z})=\mathrm{Pic} (\mathrm{S})\stackrel{r}{\longrightarrow}\mathrm{Pic}(Y)\stackrel{\pi_*}{\longrightarrow}\mathrm{Pic}(\mathcal{C})$$ where $r$ 
is the restriction 
$r\colon\mathrm{Pic}(\mathrm{S})\to\mathrm{Pic}(Y).$ Therefore a marked pair $(\mathrm{S},Y,\Phi)$ determines canonically a marked cubic curve $(\mathcal{C},\eta).$

\bigskip

\subsubsection{From cubic curves to surfaces}

Conversely let us consider a marked cubic curve~$(\mathcal{C},\eta).$ Then we have base-points $p_i\in\mathcal{C}$ determined by $(\eta(\mathbf{e}_i))_{1\leq i\leq n}$ and an embedding 
$\mathcal{C}\subset~\mathbb{P}^2(\mathbb{C})$ determined by $\eta(\mathbf{e}_0).$ Let $(\mathrm{S},\Phi)$ be the marked blow-up with base-points~$p_i$ and $Y\subset \mathrm{S}$ 
the strict transform of $\mathcal{C}.$ Hence we obtain a marked pair $(\mathrm{S},\Phi,Y)$ called blow-up of $(\mathcal{C},\eta)$ and denoted $\mathrm{Bl}(\mathcal{C},\eta).$ 

\smallskip

This construction inverts the previous one, in other words we have the following statement. 

\begin{pro}[\cite{Mc}]\label{foncteur}
A marked pair determines canonically a marked cubic curve and conversely. 
\end{pro}

\section{Idea of the proof}\,

The automorphisms constructed to prove the previous results are obtained from a birational map by blowing up base-points on a cubic curve~$\mathcal{C}$; the cubic curves play a very 
special role because its transforms $Y$ are anticanonical curves. 

Assume that $w\in\mathrm{W}_n$ is realized by an automorphism $F$ of a rational surface $\mathrm{S}$ which preserve an anticanonical curve $Y$. A marked cubic curve $(\mathcal{C},\eta)$ 
is canonically associated to a marked pair $(\mathrm{S},\Phi,Y)$ (Theorem \ref{foncteur}). Then there exists a birational map $f\colon\mathbb{P}^2(\mathbb{C}) 
\dashrightarrow~\mathbb{P}^2(\mathbb{C})$ such that: 
\begin{itemize}
\item[$\bullet$] the lift of $f$ to $\mathrm{S}$ coincides with $F,$ 

\item[$\bullet$] $f$ preserves $\mathcal{C}$ ,

\item[$\bullet$] and $f$ induces an automorphism $f_*$ of $\mathrm{Pic}_0(\mathcal{C})$ which satisfies $\eta_0 w=f_*\eta_0.$ In other words~$[\eta_0]$ is a fixed point for the natural 
action of $w$ on the moduli space of markings. 
\end{itemize}

\smallskip

Conversely to realize a given element $w$ of the group $\mathrm{W}_n$ we search a fixed point $\eta_0$ in the moduli space of markings. We can associate to $\eta_0$ a marked cubic 
$(\mathcal{C},\eta)$ up to isomorphism (Corollary \ref{cubcub}). Let us denote by $(\mathrm{S},\Phi,Y)$ the marked pair canonically determined by $(\mathcal{C},\eta).$ Assume that, for any 
$\alpha$ in $\Theta_n,$ $\eta_0(\alpha)$ is non zero (which is a generic condition); the base-points~$p_i$ do not satisfy some nodal relation (they all are distinct, no three are 
on a line, no six are on a conic, etc). According to a theorem of Nagata there exists a second projection $\pi'\colon \mathrm{S}\to\mathbb{P}^2(\mathbb{C})$ which corresponds to the marking 
$\Phi w.$ Let us denote by~$\mathcal{C}'$ the cubic $\pi'(Y).$ Since $[\eta_0]$ is a fixed point of $w,$ the marked cubics~$(\mathcal{C}', \eta w)$ and~$(\mathcal{C},\eta)$ are 
isomorphic. But such an isomorphism is an automorphism $F$ of $\mathrm{S}$ satisfying~$F_*\Phi=\Phi w.$

\smallskip

Let us remark that in \cite{Hi, Ha2, PS, Di2} there are also constructions with automorphisms of surfaces and cubic curves. 

\section{Examples}\,

Let us consider the family of birational maps $f\colon\mathbb{P}^2(\mathbb{C}) \dashrightarrow\mathbb{P}^2(\mathbb{C})$ given in the affine chart $z=1$ by 
\begin{align*}
&f(x,y)=\left(a+y,b+\frac{y}{x}\right),&& a,\,b\in\mathbb{C}. 
\end{align*}

Let us remark that the case $b=-a$ has been studied in \cite{PS} and \cite{BaRo}.

The points of indeterminacy of $f$ are $p_1=(0:0:1),$ $p_2=(0:1:0)$ and~$p_3=(1:0:0)$. Let us set $p_4=(a:b:1)$ and let us denote by $\Delta$ (resp.~$\Delta'$) the triangle 
whose vertex are $p_1,$ $p_2,$ $p_3$ (resp. $p_2,$ $p_3,$ $p_4$). The map $f$ sends $\Delta$ onto $\Delta':$ the point $p_1$ (resp. $p_2,$ resp. $p_3$) is blown up on the 
line $(p_1p_4)$ (resp. $(p_2p_3),$ resp. $(p_3p_4)$) and the lines $(p_1p_2)$ (resp.~$(p_1p_3),$ resp. $(p_2p_3)$) are contracted on $p_2$ (resp. $p_4,$ resp. $p_3$).

If $a$ and $b$ are chosen such that $p_1=p_4,$ then $\Delta$ is invariant by $f$ and if we blow up~$\mathbb{P}^2(\mathbb{C})$ at~$p_1,$ $p_2,$ $p_3$ we obtain a realization 
of the standard Coxeter element of $\mathrm{W}_3.$ Indeed,~$f$ sends a generic line onto a conic through the $p_i;$ so $w(\mathbf{e}_0)=2\mathbf{e}_0-\mathbf{e}_1-\mathbf{e}_2
-\mathbf{e}_3.$ The point $p_1$ (resp.~$p_2,$ resp. $p_3$) is blown up on the line through $p_2$ and $p_3$ (resp. $p_1$ and $p_3,$ resp. $p_1$ and $p_2$). Therefore 
\begin{align*}
& w(\mathbf{e}_1)=\mathbf{e}_0-\mathbf{e}_2-\mathbf{e}_3, && w(\mathbf{e}_2)=\mathbf{e}_0-\mathbf{e}_1-\mathbf{e}_3, && w(\mathbf{e}_3)=\mathbf{e}_0-\mathbf{e}_1-\mathbf{e}_2.
\end{align*}

More generally we have the following statement.

\begin{thm}[\cite{Mc}]\label{refMc}
Let us denote by $p_{i+4}$ the $i$-th iterate $f^i(p_4)$ of~$p_4.$ 

The realization of the standard Coxeter element of $\mathrm{W}_n$ corresponds to the pairs $(a,b)$ of~$\mathbb{C}^2$ such that 
\begin{align*}
& p_i\not\in (p_1p_2)\cup(p_2p_3)\cup(p_3p_1), && p_{n+1}=p_1.
\end{align*}
\end{thm}

\begin{proof}
Assume that there exists an integer $i$ such that $f^i(p_4)=p_{i+4}.$ Let~$(\mathrm{S},\pi)$ be the marked blow-up with base-points $p_i.$ The map  $f$ lifts to a morphism $F_0\colon 
\mathrm{S}\to 
\mathbb{P}^2(\mathbb{C}).$ Since any $p_i$ is now the image $F_0(\ell_i)$ of a line in~$\mathrm{S},$ the morphism $F_0$ lifts to an automorphism $F$ of $\mathrm{S}$ such that $f$ 
lifts to $F.$ Let us 
find the element $w$ realized by $F.$ Let us remark that $f$ sends a generic line onto a conic through $p_2,$ $p_3$ and $p_4$ thus~$w(\mathbf{e}_0)=2\mathbf{e}_0-\mathbf{e}_2-
\mathbf{e}_3-\mathbf{e}_4.$ The point~$p_1$ is blown up to the line through $p_3$ and $p_4$ so~$w(\mathbf{e}_1)=\mathbf{e}_0-\mathbf{e}_3-\mathbf{e}_4;$ similarly we obtain 
\begin{small}
\begin{align*}
&w(\mathbf{e}_2)=\mathbf{e}_0-\mathbf{e}_2-\mathbf{e}_4,&& w(\mathbf{e}_3)=\mathbf{e}_0-\mathbf{e}_2-\mathbf{e}_3,\\
 & w(\mathbf{e}_i)=\mathbf{e}_{i+1}\text{ for $4\leq i<n$ }, &&w(\mathbf{e}_n)=\mathbf{e}_1.
\end{align*}
\end{small}

Conversely if an automorphism $F\colon \mathrm{S}\to \mathrm{S}$ realizes the standard Coxeter element $w=\pi_n\kappa_{123},$ we can normalize the base-points such that 
$$\big\{p_1,\,p_2,\,p_3\big\}=
\big\{(0:0:1),\, (0:1:0),\, (1:0:0)\big\};$$ the birational map $f\colon\mathbb{P}^2(\mathbb{C})\dashrightarrow\mathbb{P}^2(\mathbb{C})$ covered by~$F$ is a composition of the standard 
Cremona involution and an automorphism sending $(p_1,\, p_2)$ onto~$(p_2,\,p_3).$ Such a map $f$ has the form in the affine chart $z=1$ $$f(x,y)=(a',b')+(Ay,By/x)$$ so up to conjugacy by 
$(Bx,By/A),$ we have $f(x,y)=(a,b)+(y,y/x).$  
\end{proof}

\chapter{Automorphisms of positive entropy: some examples}\label{chapbedkim1}

A possibility to produce an automorphism $f$ on a rational surface $\mathrm{S}$ is the following: starting with a birational map $f$ of $\mathbb{P}^2( \mathbb{C}),$ we find a 
sequence of blow-ups $\pi\colon\mathrm{S}\to\mathbb{P}^2(\mathbb{C})$ such that the induced map $f_\mathrm{S}=\pi f\pi^{-1}$ is an automorphism of~$\mathrm{S}.$ The difficulty is 
to find such a sequence~$\pi$... If $f$ is not an automorphism of the complex projective plane, then $f$ contracts a curve~$\mathcal{C}_1$ onto a point $p_1;$ the first thing to 
do to obtain an automorphism from $f$ is to blow up the point~$p_1$ via $\pi_1\colon\mathrm{S}_1\to\mathbb{P}^2(\mathbb{C}).$ In the best case $f_{\mathrm{S}_1}=~\pi_1f\pi_1^{-1}$ 
sends the strict transform of $\mathcal{C}_1$ onto the exceptional divisor~$\mathrm{E}_1.$ But if~$p_1$ is not a point of indeterminacy, $f_{\mathrm{S}_1}$ contracts $\mathrm{E}_1$ 
onto $p_2=f(p_1).$ This process thus finishes only if $f$ is not algebraically stable. 

 In \cite{BK3} Bedford and Kim exhibit a continuous family of birational maps $(f_a)_{a\in\mathbb{C}^{k-2}}$.
We will see that this family is conjugate to automorphisms with positive entropy on some rational surface $\mathrm{S}_a$ (Theorem \ref{bk3}). Let us hold the parameter $c$ fixed; 
the family $f_a$ induces a family of dynamical systems of dimension $k/2-1$: there exists a neighborhood $\mathcal{U}$ of $0$ in 
$\mathbb{C}^{k/2-1}$ such that if $a=(a_0,a_2,\ldots,a_{k-2})$, $b=(b_0,b_2,\ldots,b_{k-2})$ are in $\mathcal{U}$ then 
$f_a$ and $f_b$ are not smoothly conjugate (Theorem \ref{thm:sysdyn}).
 Moreover they show, for $k\geq 4$, the existence of a neighborhood $\mathcal{U}$ of $0$ in $\mathbb{C}^{k/2-1}$ such that if~$a,$~$b$ are two distinct points 
of $\mathcal{U}$, then $\mathrm{S}_a$ is not biholomorphically equivalent to~$\mathrm{S}_b$ (Theorem~\ref{Thm:surf}).

The results evoked in the last section are also due to Bedford and Kim (\cite{BK4}); they concern the Fatou sets of automorphisms with positive entropy on rational non-minimal 
surfaces obtained from birational maps of the complex projective plane. Bedford and Kim prove that such automorphisms can have large rotation domains (Theorem \ref{rot1}).

\section{Description of the sequence of blow-ups (\cite{BK1})}

Let $f_{a,b}$ be the birational map of the complex projective plane given by $$f_{a,b}(x,y,z)=\big(x(bx+y):z(bx+y):x(ax+z)\big),$$ or in the affine chart $x=1$ $$f_{a,b}(y,z)=\left(z,
\frac{a+z}{b+y} \right).$$ We note that $\mathrm{Ind}\,f_{a,b}=\{p_1,\, p_2,\,p_*\}$ and $\mathrm{Exc}\,f_{a,b}=\Sigma_0\cup \Sigma_\beta\cup\Sigma_\gamma$ with 
\begin{align*}
& p_1=(0:1:0),&& p_2=(0:0:1), &&p_*=(1:-b:-a),\\
&\Sigma_0=\{x=0\}, && \Sigma_\beta=\{bx+y=0\}, &&\Sigma_\gamma=\{ax+z=0\}. 
\end{align*}

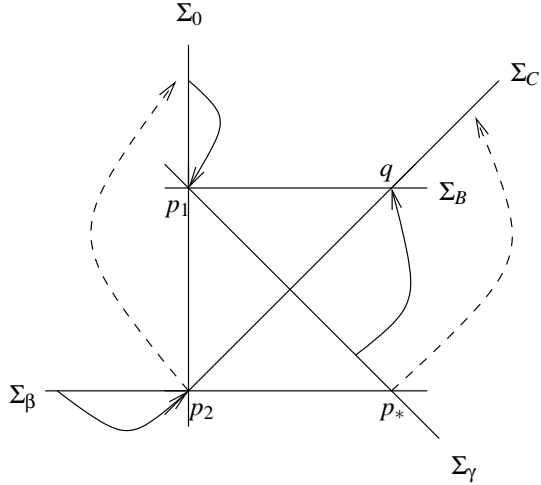
\begin{figure}[H]
\begin{center}
\input{bedkim.pstex_t}
\end{center}
\end{figure}

Set $Y=\mathrm{Bl}_{p_1,p_2}\mathbb{P}^2,$ $\pi\colon Y\to\mathbb{P}^2(\mathbb{C})$ and $f_{a,b,\,Y}=\pi^{-1} f_{a,b}\pi.$ Let us prove that after these two blow-ups~$\Sigma_0$ 
does not belong to $\mathrm{Exc}\,f_{a,b,\,Y}.$  

To begin let us blow up $p_2.$ Let us set $x=r_2$ and $y=r_2s_2;$ then $(r_2,s_2)$ is a system of local coordinates in which~$\Sigma_\beta=\{s_2+b=0\}$ and $\mathrm{E}_2=\{r_2=0\}.$ 
We remark that 
\begin{align*}
& (r_2,s_2)\to(r_2,r_2s_2)_{(x,y)}\to(r_2(b+s_2):b+s_2:ar_2+1)= \left(\frac{r_2(b+s_2)}{ar_2+1},\frac{b+s_2}{ar_2+1}\right)_{(x,y)}\\
&\hspace{1cm}\to \left(\frac{r_2(b+s_2)}{ar_2+1},\frac{1}{r_2}\right)_{(r_2,s_2)}.
\end{align*}

Thus $\Sigma_\beta$ is sent onto $\mathrm{E}_2$ and $\mathrm{E}_2$ sur $\Sigma_0.$ 

Let us now blow up $p_1.$ Set $x=u_2v_2$ and $y=v_2;$ the exceptional divisor~$\mathrm{E}_2$ is given by $v_2=0$ and $\Sigma_0$ by $u_2=0.$ We have 
\begin{align*}
&(u_2,v_2)\to(u_2v_2,v_2)_{(x,y)}\to(u_2v_2(bu_2+1):bu_2+1:u_2(au_2v_2+1))\\
&\hspace{1cm}=\left(\frac{v_2(bu_2+1)}{au_2v_2+1},\frac{bu_2+1}{u_2(au_2v_2+1)}\right)_{(x,y)}
\to\left(u_2v_2,\frac{bu_2+1}{u_2(au_2v_2+1)}\right)_{(u_2,v_2)};
\end{align*}
therefore $\mathrm{E}_2$ is sent onto $\Sigma_0.$ 

Let us set $x=r_1,$ $z=r_1s_1;$ in the coordinates $(r_1,s_1)$ we have $\mathrm{E}_1=\{r_1=0\}.$ Moreover
\begin{align*}
& (r_1,s_1)\to(r_1,r_1s_1)_{(x,z)}\to(br_1+1:b+s_1(br_1+1):r_1(a+s_1)).
\end{align*}

Hence $\mathrm{E}_1$ is sent onto $\Sigma_B.$ 

Set $x=u_1v_1$ and $z=v_1;$ in these coordinates $\Sigma_0=\{u_1=0\},$ $\mathrm{E}_1=\{v_1=0\}$ and 
\begin{align*}
&(u_1,v_1)\to(u_1v_1,v_1)_{(x,z)}\to(u_1(bu_1v_1+1):bu_1v_1+1:u_1v_1(au_1+1))\\
&\hspace{1cm}=\left(u_1, \frac{u_1v_1(au_1+1)}{bu_1v_1+1}\right)_{(x,z)}\to\left(u_1,\frac{v_1(au_1+1)}{bu_1v_1+1}\right)_{(r_1,s_1)}.
\end{align*} 

So $\Sigma_0\to\mathrm{E}_1$ and $\Sigma_\beta\to\mathrm{E}_2\to\Sigma_0\to\mathrm{E}_1\to \Sigma_B.$ In particular 
\begin{align*}
&\mathrm{Ind}\,f_{a,b,\,Y}=\{p_*\} &&\& && \mathrm{Exc}\,f_{a,b,\,Y}=~\{\Sigma_\gamma\}.
\end{align*}

We remark that $\big\{\mathrm{H},\,\mathrm{E}_1,\,\mathrm{E}_2\big\}$ is a basis of $\mathrm{Pic}(Y).$ The exceptional divisor~$\mathrm{E}_1$ is sent on $\Sigma_B;$ since~$p_1$ 
belongs to $\Sigma_B$ we have $\mathrm{E}_1\to\Sigma_B\to\Sigma_B+\mathrm{E}_1.$ On the other hand $\mathrm{E}_2$ is sent onto~$\Sigma_0;$ as~$p_1$ and $p_2$ belong to $\Sigma_0$ 
we have $$\mathrm{E}_2\to\Sigma_0\to\Sigma_0+\mathrm{E}_1+\mathrm{E}_2.$$ Let $\mathrm{H}$ be a generic line of~$\mathbb{P}^2(\mathbb{C})$; it is given by~$\ell=0$ with $\ell=a_0x
+a_1y+a_2z.$ Its image by $f_{a,b,\,Y}$ is a conic thus $$f_{a,b,\,Y}^*\mathrm{H}=2\mathrm{H}-\sum_{i=1}^2m_i \mathrm{E}_i.$$ Let us find the $m_i$'s. As
\begin{align*}
&(r_2,s_2)\to(r_2,r_2s_2)_{(x,y)}\to 
(r_2(b+s_2):b+s_2:ar_2+1)\\
&\hspace{1cm}\to r_2\Big(a_0r_2(b+s_2) +a_1(b+s_2)+a_2(ar_2+1)\Big)
\end{align*}
and $\mathrm{E}_2=\{r_2=0\}$ the integer $m_2$ is equal to $1.$ Since
\begin{align*}
&(r_1,s_1)\to(r_1,r_1s_1)_{(x,z)}\to(br_1+1:b+s_1(br_1+1):r_1(a+s_1))\\
&\hspace{1cm}\to s_1r_1\Big(a_0(bs_1r_1+1)+a_1s_1(bs_1r_1+1)+s_1r_1(a+s_1)\Big)
\end{align*}
and $\mathrm{E}_1=\{s_1=0\}$ we get $m_1=1.$ That's why $$M_{f_{a,b,\,Y}}=\left[\begin{array}{ccc} 2 & 1 & 1 \\ -1 & -1 & -1 \\ -1 & 0 & -1\end{array}\right].$$ The characteristic 
polynomial of $M_{f_{a,b,\,Y}}$ is $1+t-t^3.$ Let us explain all the information contained in $M_{f_{a,b,\,Y}}.$ Let $\mathrm{L}$ be a line and $\mathrm{L}$ its class 
in~$\mathrm{Pic}(Y).$ If $\mathrm{L}$ does not intersect neither $\mathrm{E}_1,$ nor~$\mathrm{E}_2,$ then $\mathrm{L}=\mathrm{H}.$ As $f_{a,b,\,Y}^*\mathrm{H}=2\mathrm{H}-\mathrm{E}_1-
\mathrm{E}_2$ the image of~$\mathrm{L}$ by $f_{a,b,\,Y}$ is a conic which intersects~$\mathrm{E}_1$ and~$\mathrm{E}_2$ with multiplicity $1.$ If $\mathrm{L}$ contains $p_*,$ then 
$f_{a,b,\,Y}(\mathrm{L})$ is the union of~$\Sigma_C$ and a second line. Assume that $p_*$ does not belong to $\mathrm{L}\cup f_{a,b,\,Y}(\mathrm{L}),$ then $$f^2_{a,b,\,Y}(\mathrm{L})
=M_{f_{a,b}}^2\left[\begin{array}{ccc} 1 \\ 0 \\ 0 \end{array}\right]=2\mathrm{H}-\mathrm{E}_2;$$ in other words $f_{a,b,\,Y}^2(\mathrm{L})$ is a conic which intersects $\mathrm{E}_2$ 
but not $\mathrm{E}_1.$ If $p_*$ does not belong to~$\mathrm{L}~\cup~f_{a,b,\,Y}(\mathrm{L})~\cup~f_{a,b,\,Y}^2(\mathrm{L}),$ then $$f^3_{a,b,\,Y}(\mathrm{L})=M_{f_{a,b}}^3
\left[\begin{array}{ccc} 1 \\ 0 \\ 0 \end{array}\right]=3 \mathrm{H}-\mathrm{E}_1-\mathrm{E}_2,$$ {\it i.e.} $f_{a,b,\,Y}^3(\mathrm{L})$ is a cubic which intersects $\mathrm{E}_1$ and 
$\mathrm{E}_2$ with multiplicity $1.$ If~$p_*$ does not belong to $$\mathrm{L}\cup f_{a,b,\,Y}(\mathrm{L})\cup\ldots\cup f_{a,b,\,Y}^{n-1}(\mathrm{L}),$$ the iterates of $f_{a,b,\,Y}$ 
are holomorphic on the neighborhood of $\mathrm{L}$ and~$$(f_{a,b,\,Y}^*)^n (\mathrm{H})=f_{a,b,\,Y}^n\mathrm{L}.$$ The parameters $a$ and~$b$ are said \textbf{\textit{generic}} if $p_*$ 
does not belong to $\displaystyle\bigcup_{j=0}^\infty f_{a,b,\,Y}^j(\mathrm{L}).$

\begin{thm}\label{abgen}
Assume that $a$ and $b$ are generic; $f_{a,b,\,Y}$ is algebraically stable and~$\lambda(f_{a,b})\sim 1.324$ is the largest eigenvalue of the characteristic polynomial $t^3-t-1.$
\end{thm}

\section{Construction of surfaces and automorphisms (\cite{BK1})}

Let us consider the subset $\mathcal{V}_n$ of $\mathbb{C}^2$ given by $$\mathcal{V}_n=\big\{(a,b)\in \mathbb{C}^2\,\big\vert\, f_{a,b,\,Y}^j(q)\not=p_*\,\,\,\forall\,0\leq j\leq n-1,
\,f_{a,b,\,Y}^n(q)=p_*\big\}.$$ 

\begin{thm}
The map $f_{a,b,\,Y}$ is conjugate to an automorphism on a rational surface if and only if $(a,b)$ belongs to $\mathcal{V}_n$ for some~$n.$
\end{thm}

\begin{proof}
If $(a,b)$ does not belong to $\mathcal{V}_n,$ Theorem \ref{abgen} implies that $\lambda( f_{a,b})$ is the largest root of $t^3-t-1;$ we note that $\lambda(f_{a,b})$ is not a Salem 
number so~$f_{a,b}$ is not conjugate to an automorphism (Theorem \ref{Thm:Salem}).

Conversely assume that there exists an integer $n$ such that $(a,b)$ belongs to $\mathcal{V}_n.$ Let $\mathrm{S}$ be the surface obtained from $Y$ by blowing up the points $q,$ $f_{a,b,\,Y}(q),$ 
$\ldots,$ $f_{a,b,\,Y}^n(q)=p_*$ of the orbit of $q.$ We can check that the induced map $f_{a,b,\,\mathrm{S}}$ is an automorphism of $\mathrm{S}.$
\end{proof} 

Let us now consider $f_{a,b,\,\mathrm{S}}^*$ which will be denoted by $f_{a,b}^*.$

\begin{thm}
Assume that $(a,b)$ belongs to $\mathcal{V}_n$ for some integer $n$. If $n\leq~5,$ the map $f_{a,b}$ is periodic of period $\leq 30.$ If $n$ is equal to~$6,$ the degree growth of 
$f_{a,b}$ is quadratic. Finally if  $n\geq 7,$ then $\big\{\deg f_{a,b}^k\big\}_k$ grows exponentially and $\lambda(f_{a,b})$ is the largest eigenvalue of the characteristic polynomial 
$$\chi_n(t)=t^{n+1}(t^3-t-1)+t^3+t^2-1.$$ Moreover, when $n$ tends to infinity, $\lambda(f_{a,b})$ tends to the largest eigenvalue of $t^3-t-1.$
\end{thm}

The action $f_{a,b,\,\mathrm{S}*}$ on the cohomology is given by $$\mathrm{E}_2\to\Sigma_0=\mathrm{H}-\mathrm{E}_1-\mathrm{E}_2\to\mathrm{E}_1\to\Sigma_B=\mathrm{H}- \mathrm{E}_1-\mathrm{Q}$$ 
where $\mathrm{Q}$ denotes the divisor obtained by blowing up the point $q$ which is on $\Sigma_B.$ As $p_*$ is blown-up by $f_{a,b}$ on~$\Sigma_C,$ we have $$\mathrm{Q}\to 
f_{a,b}(\mathrm{Q})\to\ldots\to f_{a,b}^n(\mathrm{Q})\to\Sigma_C=\mathrm{H}-\mathrm{E}_2- \mathrm{Q}.$$ Finally a generic line $\mathrm{L}$ intersects $\Sigma_0,$ $\Sigma_\beta$ and
 $\Sigma_\gamma$ with multiplicity $1;$ the image of $\mathrm{L}$ is thus a conic through $q,$ $p_1$ and $p_2$ so $\mathrm{H}\to 2\mathrm{H}-\mathrm{E}_1-\mathrm{E}_2 -\mathrm{Q}.$ 
In the basis $$\big\{\mathrm{H},\,\mathrm{E}_1,\,\mathrm{E}_2,\,\mathrm{Q},\, f_{a,b}(\mathrm{Q}),\,\ldots,\,f_{a,b}^n(\mathrm{Q})\big\}$$ we have $$M_{f_{a,b}}=\left[
\begin{array}{ccccccccc} 
2& 1 &1 &0 &0 &\ldots &\ldots & 0 &1\\
-1 & -1 & -1 & 0 &0 &\ldots& \ldots& 0& 0\\
-1 & 0 & -1 & 0 & 0 &\ldots&\ldots& 0& -1\\
-1 & -1 & 0 & 0 & 0 &\ldots&\ldots& 0& -1\\
0 & 0 & 0 & 1 & 0 &\ldots & \ldots & 0 & 0\\
0 & 0 & 0 & 0 & 1 &0 & \ldots & 0 & \vdots\\
\vdots & \vdots & \vdots & \vdots & 0 &\ddots & \ddots & \vdots & \vdots\\
\vdots & \vdots & \vdots & \vdots & \vdots &\ddots & \ddots & 0 & 0\\
0 & 0 & 0 & 0 & 0 & \ldots &0 & 1 & 0\\
\end{array}\right].$$

\bigskip

\section{Invariant curves (\cite{BK2})}

In the spirit of \cite{DJS} (\emph{see} Chapter \ref{Chap:folinv}, \S \ref{Sec:courbinv}) 
Bedford and Kim study the curves invariant by~$f_{a,b}.$ There exists rational maps $\varphi_j\colon\mathbb{C}\to\mathbb{C}^2$ such that if $(a,b)=
~\varphi_j(t)$ for some complex number $t,$ then $f_{a,b}$ has an invariant curve $\mathcal{C}$ with $j$ irreducible components. Let us set
\begin{small} 
\begin{align*}
&\varphi_1(t)=\left(\frac{t-t^3-t^4}{1+2t+t^2},\frac{1-t^5}{t^2+t^3}\right), &&\varphi_2(t)=\left(\frac{t+t^2+t^3}{1+2t+t^2},\frac{t^3-1}{t+t^2}\right),
\end{align*}
$$\varphi_3(t)=\left(1+t,t-\frac{1}{t}\right).$$
\end{small}

\begin{thm}
Let $t$ be in $\mathbb{C}\setminus\{-1,\,1,\,0,\,\mathbf{j},\,\mathbf{j}^2\}.$ There exists a cubic $\mathcal{C}$ invariant by~$f_{a,b}$ if and only if $(a,b)=\varphi_j(t)$ for a 
certain $1\leq j\leq 3$; in that case $\mathcal{C}$ is described by an homogeneous polynomial $P_{t,a,b}$ of degree $3.$

Moreover, if $P_{t,a,b}$ exists, it is given, up to multiplication by a constant, by 
\begin{small}
\begin{align*}
&P_{t,a,b}(x,y,z)=ax^3(t-1)t^4+yz(t-1)t(z+ty)\\&\hspace{2cm}+x\Big(2byzt^3+y^2(t-1)t^3+z^2(t-1)(1+bt)\Big)\\ &\hspace{2cm}+x^2(t-1)t^3\Big(a(y+tz)+t(y+(t-2b)z)\Big).
\end{align*}
\end{small}
\end{thm}

More precisely we have the following description.

\begin{itemize}
\item[$\bullet$] If $(a,b)=\varphi_1(t),$ then $\Gamma_1=(P_{t,a,b}=0)$ is a irreducible cuspidal cubic. The map $f_{a,b}$ has two fixed points, one of them is the singular point 
of $\mathcal{C}.$ 

\item[$\bullet$] If $(a,b)=\varphi_2(t),$ then $\Gamma_2=(P_{t,a,b}=0)$ is the union of a conic and a tangent line to it. The map $f_{a,b}$ has two fixed points.

\item[$\bullet$] If $(a,b)=\varphi_3(t),$ then $\Gamma_3=(P_{t,a,b}=0)$ is the union of three concurrent lines; $f_{a,b}$ has two fixed points, one of them is the intersection of 
the three components of~$\mathcal{C}.$ 
\end{itemize}

There is a relationship between the parameters $(a,b)$ for which there exists a complex number $t$ such that~$\varphi_j(t)=(a,b)$ and the roots of the characteristic polynomial 
$\chi_n.$

\begin{thm}
Let $n$ be an integer, let $1\leq j\leq 3$ be an integer and let $t$ be a complex number. Assume that $(a,b):=\varphi_j(t)$ does not belong to any~$\mathcal{V}_k$ for $k<n.$ Then~$(a,b)$ 
belongs to $\mathcal{V}_n$ if and only if $j$ divides $n$ and $t$ is a root of $\chi_n.$
\end{thm}

We can write $\chi_n$ as $C_n\psi_n$ where $C_n$ is the product of cyclotomic factors and $\psi_n$ is the minimal polynomial of $\lambda(f_{a,b}).$

\begin{thm}
Assume that $n\geq 7$. Let $t$ be a root of~$\chi_n$ not equal to $1.$ Then either~$t$ is a root of $\psi_n,$ or $t$ is a root of $\chi_j$ for some $0\leq j\leq 5.$ 
\end{thm}

Bedford and Kim prove that $\#\,(\Gamma_j\cap\mathcal{V}_n)$ is, for $n\geq 7,$ determined by the number of Galois conjugates of the unique root of $\psi_n$ strictly greater than~$1:$ 
if~$n\geq 7$ and $1\leq j\leq 3$ divides $n,$ then $$\Gamma_j\cap\mathcal{V}_n=\big\{\varphi_j(t)\,\big\vert\, t\text{ root of }\psi_n\big\};$$ in particular $\Gamma_j\cap\mathcal{V}_n$ 
is not empty.

Let $X$ be a rational surface and let $g$ be an automorphism of $X.$ The pair~$(X,g)$ is said \textbf{\textit{minimal}} if any birational morphism $\pi\colon X\to X'$ which sends $(X,g)$ 
on $(X',g'),$ where $g'$ is an automorphism of $X',$ is an isomorphism. Let us recall a question of \cite{Mc}. Let $X$ be a rational surface and let $g$ be an automorphism of $X$. Assume 
that $(X,g)$ is minimal. Does there exist a negative power of the class of the canonical divisor $\mathrm{K}_X$ which admits an holomorphic section~? We know since \cite{Ha} that the 
answer is no if we remove the assumption \og $(X,g)$ minimal\fg. 

\begin{thm}\label{bedfordkim}
There exists a surface $\mathrm{S}$ and an automorphism with positive entropy $f_{a,b}$ on $\mathrm{S}$ such that~$(\mathrm{S},f_{a,b})$ is minimal and such that $f_{a,b}$ has no 
invariant curve. 
\end{thm}

If $g$ is an automorphism of a rational surface $X$ such that a negative power of~$\mathrm{K}_X$ admits an holomorphic section, $g$ preserves a curve; so Theorem~\ref{bedfordkim} gives 
an answer to McMullen's question.

\section{Rotation domains (\cite{BK2})}

Assume that $n\geq 7$ (so $f$ is not periodic); if there is a rotation domain, then its rank is $1$ or $2$  (Theorem \ref{rangrang}). We will see that both happen; let us begin with 
rotation domains of rank $1$. 

\begin{thm}
Assume that $n\geq 7.$ Assume that $j$ divides $n$ and that $(a,b)$ belongs to~$\Gamma_j\cap~\mathcal{V}_n.$ There exists a complex number $t$ such that $(a,b)= \varphi_j(t).$ If $t$ is 
a Galois conjugate of $\lambda(f_{a,b})$, not equal to~$\lambda( f_{a,b})^{\pm 1},$ then $f_{a,b}$ has a rotation domain of rank $1$ centered in  
\begin{align*}
& \left(\frac{t^3}{1+t},\frac{t^3}{1+t}\right)\text{ if } j=1,&& \left(-\frac{t^2}{1+t},-\frac{t^2}{1+t}\right)\text{ if } j=2,&& (-t,-t)\text{ if } j=3.
\end{align*}
\end{thm}

Let us now deal with those of rank $2.$

\begin{thm}
Let us consider an integer $n\geq 8$, an integer $2\leq j\leq 3$ which divides~$n.$ Assume that $(a,b)=\varphi_j(t)$ and that $\vert t\vert=1;$ moreover suppose that $t$ is a root of 
$\psi_n.$ Let us denote by $\eta_1,$ $\eta_2$ the eigenvalues of~$Df_{a,b}$ at the point
\begin{align*}
&m=\left(\frac{1+t+t^2}{t+t^2},\frac{1+t+t^2}{t+t^2}\right) \text{ if }j=2, && m=\left(1+\frac{1}{t},1+\frac{1}{t}\right) \text{ if }j=3.
\end{align*}

If $\vert\eta_1\vert=\vert\eta_2\vert=1$ then $f_{a,b}$ has a rotation domain on rank $2$ centered at~$m$.
\end{thm}

There are examples where rotation domains of rank $1$ and $2$ coexist. 

\begin{thm}
Assume that $n\geq 8$, that $j=2$ and that $j$ divides $n.$ There exists $(a,b)$ in~$\Gamma_j\cap\mathcal{V}_n$ such that $f_{a,b}$ has a rotation domain of rank~$2$ centered at 
\begin{align*}
&\left(\frac{1+t+t^2}{t+t^2},\frac{1+t+t^2}{t+t^2}\right) \text{ if }j=2, && \left(1+\frac{1}{t},1+\frac{1}{t}\right) \text{ if }j=3
\end{align*}
and a rotation domain of rank $1$ centered at 
\begin{align*}
& \left(-\frac{t^2}{1+t},-\frac{t^2}{1+t}\right)\text{ if } j=2,&& (-t,-t)\text{ if } j=3.
\end{align*}
\end{thm}

\section{Weyl groups (\cite{BK2})}

Let us recall that $\mathrm{E}_1$ and $\mathrm{E}_2$ are the divisors obtained by blowing up~$p_1$ and~$p_2.$ To simplify let us introduce some notations: $\mathrm{E}_0=
\mathrm{H}$, $\mathrm{E}_3=~\mathrm{Q},$ $\mathrm{E}_4=f(\mathrm{Q}),$ $\ldots,$ $\mathrm{E}_n=f^{n-3}(\mathrm{Q})$ and let $\pi_i$ be the blow-up associated to $\mathrm{E}_i.$ Let us 
set
\begin{align*}
& e_0=\mathrm{E}_0, && e_i=(\pi_{i+1}\ldots\pi_n)^*\mathrm{E}_i,&& 1\leq i\leq n;
\end{align*}
the basis $\big\{e_0,\ldots,e_n\big\}$ of $\mathrm{Pic}(\mathrm{S}=)$ is geometric.

 Bedford and Kim prove that they can apply Theorem \ref{refMc} and deduce from it the follo\-wing statement. 

\begin{thm}
Let $X$ be a rational surface obtained by blowing up $\mathbb{P}^2(\mathbb{C})$ in a finite number of points $\pi\colon X\to\mathbb{P}^2(\mathbb{C})$ and let $F$ be an automorphism 
on~$X$ which represents the standard element of the Weyl group~$\mathrm{W}_n$, $n\geq 5.$ There exists an automorphism $A$ of $\mathbb{P}^2(\mathbb{C})$ and some complex numbers $a$ 
and $b$ such that $$f_{a,b} A\pi= A\pi F.$$
\end{thm}

Moreover they get that a representation of the standard element of the Weyl group can be obtained from  $f_{a,b,\,Y}.$

\begin{thm}
Let $X$ be a rational surface and let $F$ be an automorphism on~$X$ which represents the standard element of the Weyl group $\mathrm{W}_n.$ There exist 
\begin{itemize}
\item[$\bullet$] a surface $\widetilde{Y}$ obtained by blowing up $Y$ in a finite number of distinct points $\pi\colon\widetilde{Y}\to Y,$

\item[$\bullet$] an automorphism $g$ on $\widetilde{Y},$

\item[$\bullet$] $(a,b)$ in $\mathcal{V}_{n-3}$
\end{itemize}
such that $(F,X)$ is conjugate to $(g,\widetilde{Y})$ and $\pi g=f_{a,b,\,Y}\pi.$
\end{thm}

\section{Continuous families of automorphisms with positive entropy (\cite{BK3})}

In \cite{BK3} Bedford and Kim introduce the following family:
\begin{equation}\label{contfam}
f_a(y,z)=\Big(z,-y+cz+\sum_{\stackrel{j=1}{j \text{ pair}}}^{k-2}\frac{a_j} {y^j}+\frac{1}{y^k}\Big),\\
a=(a_1,\ldots,a_{k-2})\in\mathbb{C}^{k-2},\,\,\, c\in\mathbb{R},\,\,\,k\geq 2. 
\end{equation}

\begin{thm}\label{bk3}
Let us consider the family $(f_a)$ of birational maps given by~$($\ref{contfam}$)$.

Let $j,$ $n$ be two integers relatively prime and such that $1\leq j\leq n$. There exists a non-empty subset $C_n$ of $\mathbb{R}$ such that, for any even $k\geq 2$ and for any $(c,a_j)$ 
in $C_n\times\mathbb{C},$ the map $f_a$ is conjugate to an automorphism of a rational surface $\mathrm{S}_a$ with entropy $\log\lambda_{n,k}$ where $\log\lambda_{n,k}$ is the largest root of the 
polynomial $$\chi_{n,k}=1-k\displaystyle\sum_{j=1}^{n-1}x^j+x^n.$$
\end{thm}

Let us explain briefly the construction of $C_n.$ The line $\Delta=\{x=0\}$ is invariant by $f_a.$ An element of $\Delta\setminus\{(0:0:1)\}$ can be written as $(0:1:w)$ and $f(0:1:w) 
=\left(0:1:c-\frac{1}{w}\right).$ The restriction of $f_a$ to $\Delta$ coincides with $g(w)=c-\frac{1}{w}.$ The set of values of $c$ for which~$g$ is periodic of period $n$ is 
$$\big\{2\cos(j\pi/n)\,\big\vert\,0<j<n,\,(j,n)=1\big\}.$$ Let us set $w_s=g^{s-1} (c)$ for $1\leq s\leq n-1,$ in other words the $w_i$'s encode the orbit of $(0:1:0)$ under the action of 
$f.$ The $w_j$ satisfy the following properties:
\begin{itemize}
\item[$\bullet$] $w_jw_{n-1-j}=1;$

\item[$\bullet$] if $n$ is even, then $w_1\ldots w_{n-2}=1;$

\item[$\bullet$] if $n$ is odd, let us set $w_*(c)=w_{(n-1)/2}$ then $w_1\ldots w_{n-2}=w_*.$
\end{itemize}


\bigskip

Let us give details about the case $n=3,$ $k=2,$ then $C_3=\{-1,\,1\}.$ Assume that~$c=~1;$ in other words $$f_a=f=\big(xz^2:z^3:x^3+z^3-yz^2\big).$$ The map $f$ contracts only one 
line $\Delta''=\{z=0\}$ onto the point $R=(0:0:1)$ and blows up exactly one point, $Q=(0:1:0).$ Let us describe the sequence of blow-ups that allows us to ``solve indeterminacy'':

\begin{itemize}
\item[$\bullet$] {\sf first blow-up.} First of all let us blow up $Q$ in the domain and $R$ in the range. Let us denote by~$\mathrm{E}$ (resp. $\mathrm{F}$) the exceptional divisor obtained by blowing up 
$Q$ (resp. $R$). 
%
One can check that $\mathrm{E}$ is sent onto $\mathrm{F},$ $\Delta''_1$ is contracted onto $S=(0,0)_{(a_1,b_1)}$ and 
$Q_1=(0,0)_{(u_1,v_1)}$ is a point of indeterminacy;

\item[$\bullet$] {\sf second blow-up.} Let us then blow up $Q_1$ in the domain and $S$ in the range; let $\mathrm{G},$ resp. $\mathrm{H}$ be the exceptional divisors.
%
One can verify that the exceptional divisor $\mathrm{G}$ is contracted onto $T=(0,0)_{(c_2,d_2)}$, $\Delta''_2$ onto $T$ and $U=(0,0)_{(r_2,s_2)}$ is a point of indeterminacy;

\item[$\bullet$] {\sf third blow-up.} Let us continue by blowing up $U$ in the domain and $T$ in the range, where~$\mathrm{K}$ and $\mathrm{L}$ denote the associated exceptional divisors.
%
One can check that $W=(1,0)_{(r_3,s_3)}$ is a point of indeterminacy, $\mathrm{K}$ is sent onto $\mathrm{L}$ 
and $\mathrm{G}_1$ is contracted on~$V=(1,0)_{(c_3,d_3)}$ 
and $\Delta''_3$ on $V$;

\item[$\bullet$] {\sf fourth blow-up.} Let us blow up $W$ in the domain and $V$ in the range, let $\mathrm{M}$ and $\mathrm{N}$ be the associated exceptional divisors. 
%
%
Then $\Delta''_4$ is contracted on $X=(0,0)_{(c_4,d_4)}$,
 $Y=(0,0)_{(r_4, s_4)}$ is a point of indeterminacy, $\mathrm{G}_1$ is sent onto $\mathrm{N}$ and $\mathrm{M}$ onto $\mathrm{H}$;

\item[$\bullet$] {\sf fifth blow-up.} Finally let us blow up $Y$ in the domain and $X$ in the range, where $\Lambda,$ $\Omega$ are the associated exceptional divisors.
%
%
%
So $\Delta''_5$ is sent onto $\Omega$ and $\Lambda$ onto $\Delta''_5.$ 
\end{itemize}

\begin{thm}
The map $f=\big(xz^2:z^3:x^3+z^3-yz^2\big)$ is conjugate to an automorphism of~$\mathbb{P}^2(\mathbb{C})$ blown up in $15$ points.

The first dynamical degree of $f$ is $\frac{3+\sqrt{5}}{2}$.
\end{thm}

\begin{proof}
Let us denote by $\widehat{P}_1$ $($resp. $\widehat{P}_2)$ the point infinitely near obtained by blowing up $Q,$ $Q_1,$ $U,$ $W$ and $Y$ $($resp. $R,$ $S,$ $T,$ $V$ and $X)$. 
By following the sequence of blow-ups we get that $f$ induces an isomorphism between~$\mathrm{Bl}_{\widehat{P}_1}\mathbb{P}^2$ and $\mathrm{Bl}_{\widehat{P}_2}\mathbb{P}^2$, 
the components being switched as follows 
\begin{align*}
& \mathrm{E}\to\mathrm{F}, &&\Delta''\to\Omega, && \mathrm{K}\to\mathrm{L}, &&\mathrm{M}\to\mathrm{H}, && \Lambda\to\Delta'', &&\mathrm{G}\to\mathrm{N}.
\end{align*}

A conjugate of $f$ has positive entropy on $\mathbb{P}^2(\mathbb{C})$ blown up in $\ell$ points if~$\ell\geq 10;$ we thus search an automorphism $A$ of $\mathbb{P}^2(\mathbb{C})$ such 
that $(Af)^2A$ sends $\widehat{P}_2$ onto $\widehat{P}_1.$ We remark that $f(R)=(0:1:1)$ and $f^2(R)=Q$ then that $f^2(\widehat{P}_2)=\widehat{P}_1$ so~$A=\mathrm{id}$ is such that 
$(Af)^2A$ sends $\widehat{P}_2$ onto $\widehat{P}_1.$

The components are switched as follows 
\begin{align*}
& \Delta''\to f\Omega, &&\mathrm{E}\to f\mathrm{F}, &&\mathrm{G}\to f\mathrm{N}, &&\mathrm{K}\to f\mathrm{L}, && \mathrm{M}\to f\mathrm{H},\\
& \Lambda\to f\Delta'', &&f\mathrm{F}\to f^2\mathrm{F},&&f\mathrm{N}\to f^2\mathrm{N}, &&f\mathrm{L}\to f^2\mathrm{L},&&f\mathrm{H}\to f^2\mathrm{H},\\
&f\Omega\to f^2\Omega, &&f^2\mathrm{F}\to \mathrm{E},&&f^2\mathrm{N}\to \mathrm{G},&&f^2\mathrm{L}\to \mathrm{K},&&f^2\mathrm{H}\to \mathrm{M},\\
&f^2\Omega\to \Lambda.
\end{align*}

Therefore the matrix of $f^*$ is given in the basis $$ \{\Delta'',\,\mathrm{E},\,\mathrm{G},\,\mathrm{K},\,\mathrm{M},\,\Lambda,\,f\mathrm{F},\,f\mathrm{N},\, f\mathrm{L},\,
f\mathrm{H},\,f\Omega,\,f^2\mathrm{F},\,f^2\mathrm{N},\,f^2\mathrm{L},\,f^2\mathrm{H},\,f^2\Omega\}$$ by
$$\left[\begin{array}{cccccccccccccccc}
0 & 0 & 0 & 0 & 0 & 1  & 0 & 0 & 0 & 0 & 0 & 0 & 0 & 0 & 0 & 0\\  
0 & 0 & 0 & 0 & 0 & 1  & 0 & 0 & 0 & 0 & 0 & 1 & 0 & 0 & 0 & 0\\ 
0 & 0 & 0 & 0 & 0 & 2  & 0 & 0 & 0 & 0 & 0 & 0 & 1 & 0 & 0 & 0\\ 
0 & 0 & 0 & 0 & 0 & 3  & 0 & 0 & 0 & 0 & 0 & 0 & 0 & 1 & 0 & 0\\ 
0 & 0 & 0 & 0 & 0 & 3  & 0 & 0 & 0 & 0 & 0 & 0 & 0 & 0 & 1 & 0\\ 
0 & 0 & 0 & 0 & 0 & 3  & 0 & 0 & 0 & 0 & 0 & 0 & 0 & 0 & 0 & 1\\ 
0 & 1 & 0 & 0 & 0 & -1 & 0 & 0 & 0 & 0 & 0 & 0 & 0 & 0 & 0 & 0\\ 
0 & 0 & 0 & 0 & 1 & -3 & 0 & 0 & 0 & 0 & 0 & 0 & 0 & 0 & 0 & 0\\ 
0 & 0 & 0 & 1 & 0 & -3 & 0 & 0 & 0 & 0 & 0 & 0 & 0 & 0 & 0 & 0\\ 
0 & 0 & 1 & 0 & 0 & -2 & 0 & 0 & 0 & 0 & 0 & 0 & 0 & 0 & 0 & 0\\ 
1 & 0 & 0 & 0 & 0 & -3 & 0 & 0 & 0 & 0 & 0 & 0 & 0 & 0 & 0 & 0\\ 
0 & 0 & 0 & 0 & 0 & 0  & 1 & 0 & 0 & 0 & 0 & 0 & 0 & 0 & 0 & 0\\ 
0 & 0 & 0 & 0 & 0 & 0  & 0 & 1 & 0 & 0 & 0 & 0 & 0 & 0 & 0 & 0\\ 
0 & 0 & 0 & 0 & 0 & 0  & 0 & 0 & 1 & 0 & 0 & 0 & 0 & 0 & 0 & 0\\ 
0 & 0 & 0 & 0 & 0 & 0  & 0 & 0 & 0 & 1 & 0 & 0 & 0 & 0 & 0 & 0\\ 
0 & 0 & 0 & 0 & 0 & 0  & 0 & 0 & 0 & 0 & 1 & 0 & 0 & 0 & 0 & 0\\ 
\end{array}\right];$$ the largest root of the characteristic polynomial $$(X^2-3X+1)(X^2-X+1)(X+1)^2(X^2+X+1)^3(X-1)^4$$ is $\frac{3+\sqrt{5}}{2},$ {\it i.e.} the first dynamical 
degree of $f$ is $\frac{3+\sqrt{5}}{2}.$ Let us remark that the polynomial $\chi_{3,2}$ introduced in Theorem \ref{bk3} is $1-2X-2X^2+X^3$ whose the largest root is $\frac{3+\sqrt{5}}
{2}.$ 
\end{proof}

 The considered family of birational maps is not trivial, {\it i.e.} parameters are effective.

\begin{thm}\label{thm:sysdyn}
Let us hold the parameter $c\in C_n$ fixed. The family of maps~$(f_a)$ defined by~$($\ref{contfam}$)$
induces a family of dynamical systems of dimension $k/2-1.$ In other words there is a neighborhood $\mathcal{U}$ of $0$ in 
$\mathbb{C}^{k/2-1}$ such that if $a=(a_0,a_2,\ldots,a_{k-2})$, $b=(b_0,b_2,\ldots,b_{k-2})$ are in~$\mathcal{U}$ then 
$f_a$ and $f_b$ are not smoothly conjugate.
\end{thm}

\begin{proof}[Idea of the proof]
Such a map $f_a$ has $k+1$ fixed points $p_1,\,\ldots,\,p_{k+1}.$ Let us set $a=(a_1,\ldots,a_{k-2}).$ Bedford and Kim show that the eigenvalues of $Df_a$ at~$p_j(a)$ depend on $a;$ it follows 
that the family varies non trivially with $a$. More precisely they prove that the trace of $Df_{a}$ varies in a non-trivial way. Let~$\tau_j(a)$ denote the trace of $Df_{a}$ at $p_j(a)$ 
and let us consider the map $T$ defined by 
\begin{align*}
& a\mapsto T(a)=(\tau_1(a),\ldots,\tau_{k+1}(a)).
\end{align*}

\medskip

The rank of the map $T$ is equal to $\frac{k}{2}-1$ at $a=0.$ In fact the fixed points of $f_a$ can be written $(\xi_s,\xi_s)$ where $\xi_s$ is a root of 
\begin{equation}\label{ptsfixes}
\xi=(c-1)\xi+\sum_{\stackrel{j=1}{j \text{ pair}}}^{k-2}\frac{a_j}{\xi^j}+\frac{1}{\xi^k}.
\end{equation}
When $a$ is zero, we have for any fixed point $\xi^{k+1}=\frac{1}{2-c}.$ By differentiating (\ref{ptsfixes}) with respect to~$a_\ell$ we get for $a=0$ the equality $$\left(2-
c+\frac{k}{\xi^{k+1}}\right)\frac{\partial\xi}{\partial a_\ell}=\frac{1}{\xi^\ell};$$ this implies that $$\frac{\partial\xi}{\partial a_\ell}\Big\vert_{a=0}=\frac{1}{(2-c)(k+1)
\xi^\ell}.$$ The trace of $Df_{a_{(y,z)}}$ is given by $$\tau=c-\sum_{\stackrel{j=1}{j \text{ pair}}}^{k-2}\frac{ja_j}{y^{j+1}}-\frac{k}{y^{k+1}}.$$ For $y=\xi_a$ we have 
\begin{eqnarray}
\frac{\partial\tau(\xi_a)}{\partial a_\ell}\Big\vert_{a=0}&=&-\frac{\ell}{y^{\ell+1}}+\frac{k(k+1)}{y^{k+2}}\frac{\partial\xi_a}{\partial a_\ell}=-\frac{\ell}{y^{\ell+1}}+\frac{k}
{2-c}\frac{1}{\xi^{k+1}\xi^{\ell+1}}\nonumber\\
&=&-\frac{\ell}{y^{\ell+1}}+\frac{k}{y\xi^\ell}=\frac{k-\ell}{\xi^{\ell+1}}.\nonumber
\end{eqnarray}

If we let $\xi_j$ range over $\frac{k}{2}-1$ distinct choices of roots $\frac{1}{(2-c)^{k+1}},$ the matrix essentially is a~$(\frac{k}{2}-1)\times(\frac{k}{2}-1)$ Vandermondian and so 
of rank $\frac{k}{2}-1.$

\medskip
There exists a neighborhood $\mathcal{U}$ of $0$ in $\mathbb{C}^{\frac{k}{2}-1}$ such that, for any $a$, $b$ in~$\mathcal{U}$ with $a\not=b$, the map~$f_{a}$ is not 
diffeomorphic to $f_{b}.$ In fact the map $\mathbb{C}^{\frac{k}{2}-1}\to\mathbb{C}^{k+1},$ $a\mapsto T(a)$ is locally injective in a neighborhood of $0.$ Moreover, for $a=0$, 
the fixed points $p_1,$ $\ldots,$ $p_{k+1},$ and so the values $\tau_1(0),$ $\ldots,$ $\tau_{k+1}(0),$ are distinct. Thus $\mathbb{C}^{\frac{k}{2}-1}\ni a \mapsto\{ \tau_1(a),\ldots,
\tau_{k+1}(a)\}$ is locally injective in~$0.$ So if $\mathcal{U}$ is a sufficiently small neighborhood of $0$ and if $a$ and $b$ are two distinct elements of~$\mathcal{U}$, the sets 
of multipliers at the fixed  points are not the same; it follows that~$f_{a}$ and~$f_{b}$ are not smoothly conjugate.
\end{proof}

Let $f_a$ be a map which satisfies Theorem \ref{bk3}.  
Bedford and Kim show that in all the cases under their consideration the representation 
\begin{align*}
&\mathrm{Aut}(\mathrm{S}_a)\to\mathrm{GL}(\mathrm{Pic}(\mathrm{S}_a)), && \phi\mapsto \phi_*
\end{align*}
is at most $((k^2-1):1);$ moreover if $a_{k-2}$ is non zero, it is faithful. When~$n=2,$ 
the image of~$\mathrm{Aut}(\mathrm{S}_a)\to\mathrm{GL}(\mathrm{Pic}(\mathrm{S}_a)),$ $\phi\mapsto \phi_*$ coincides with elements of $\mathrm{GL}(\mathrm{Pic}(\mathrm{S}_a))$ 
that are isometries with respect to the intersection
 product, and which preserve the canonical class of $\mathrm{S}_a$ as well as the semigroup of effective divisors; this subgroup is the infinite dihedral group with 
generators~$f_{a_*}$ and $\iota_*$
where $\iota$ denotes the reflection $(x,y)\mapsto(y,x)$.
They deduce from it that, always for $n=2,$ the surfaces $\mathrm{S}_a$ are, in general, not biholomorphically equivalent.

\begin{thm}\label{Thm:surf}
Assume that $n=2$ and that $k\geq 4$ is even. Let $a$ be in~$\mathbb{C}^{k/2-1}$ and $c$ be in~$C_2$. 
There exists a neighborhood $\mathcal{U}$ of $0$ in $\mathbb{C}^{k/2-1}$ such that if $a,$ $b$ are two distinct points of~$\mathcal{U}$ and if~$a_{k-1}$ is nonzero, then~$\mathrm{S}_a$ 
is not biholomorphically equivalent to $\mathrm{S}_{b}.$
\end{thm}

\section{Dynamics of automorphisms with positive entropy: rotation domains (\cite{BK4})}

If $\mathrm{S}$ is a compact complex surface carrying an automorphism with positive entropy $f$, a theorem of Cantat (Theorem \ref{Thm:serge}) says that 
\begin{itemize}
\item[$\bullet$] either $f$ is conjugate to an automorphism of the unique minimal model of $\mathrm{S}$ which has to be a torus, a 
K$3$ surface or an Enriques surface;  

\item[$\bullet$] or $f$ is birationally conjugate to a birational map of the complex projective plane (\cite{Can1}). 
\end{itemize}

We also see that if $\mathrm{S}$ is a complex torus, the Fatou set of $f$ is empty. If~$\mathrm{S}$ is a K$3$ surface or a quotient of a K$3$ surface, the existence of a volume 
form implies that 
the only possible Fatou components are the rotation domains. McMullen proved the existence of non-algebraic~K$3$ surfaces with rotation domains of rank $2$ (\emph{see} \cite{Mc2}).
 What happen if $\mathrm{S}$ is a rational non-minimal surface ? The automorphisms with positive entropy on rational non-minimal surfaces can have large rotation domains. 


\begin{thm}\label{rot1}
There exists a rational surface $\mathrm{S}$ carrying an automorphism with positive entropy $h$ and a rotation domain $\mathcal{U}.$ Moreover, $\mathcal{U}$ is a union of invariant 
Siegel disks,~$h$ acting as an irrational rotation on any of these disks.
\end{thm}

The linearization is a very good tool to prove the existence of rotation domains but it is a local technique. In order to understand the global nature of the Fatou component 
$\mathcal{U},$ Bedford and Kim introduce a global model and get the following statement. 

\begin{thm}\label{rot2}
There exist a surface $\mathcal{L}$ obtained by blowing up $\mathbb{P}^2(\mathbb{C})$ in a finite number of points, an automorphism $L$ on $\mathcal{L},$ a domain $\Omega$ of 
$\mathcal{L}$ and a biholomorphic conjugacy~$\Phi\colon\mathcal{U}\to~\Omega$ which sends $(h,\mathcal{U})$ onto $(L,\mathcal{L}).$ 

In particular, $h$ has no periodic point on $\mathcal{U}\setminus\{z=0\}.$ 
\end{thm}

Let us consider for $n,$ $m\geq 1$ the polynomial $$\mathcal{P}_{n,m}(t)=\frac{t(t^{nm}-1)(t^n-2t^{n-1}+1)}{(t^n-1)(t-1)}+1.$$ If $n\geq 4,$ $m\geq 1$ or if $n=3,$ $m\geq 2$ this 
polynomial is a Salem polynomial. 

\begin{thm}\label{rot3}
Let us consider the birational map $f$ given in the affine chart $z~=~1$~by $$f(x,y)=\left(y,-\delta x+cy+\frac{1}{y}\right)$$ where
$\delta$ is a root of~$\mathcal{P}_{n,m}$ which is not a root of unity and $c=2\sqrt{\delta} \cos(j\pi/n)$ with $1\leq j\leq n-1,$ $(j,n)=1$.
 
There exists a rational surface $\mathrm{S}$ obtained by blowing up $\mathbb{P}^2(\mathbb{C})$ in a finite number of points $\pi\colon \mathrm{S}\to\mathbb{P}^2(\mathbb{C})$ 
such that $\pi^{-1} f\pi$ 
is an automorphism on~$\mathrm{S}.$

Moreover, the entropy of $f$ is the largest root of the polynomial $\mathcal{P}_{n,m}.$
\end{thm}

 Bedford and Kim use the pair $(f^k,\mathrm{S})$ to prove the statements \ref{rot1} and~\ref{rot2}.

\chapter{A ``systematic'' way to construct automorphisms of positive entropy}\label{Chap:juju}

This section is devoted to a ``systematic'' construction of examples of rational surfaces with biholomorphisms of positive entropy. The strategy is the following: start with a 
birational map~$f$ of~$\mathbb{P}^2(\mathbb{C}).$ By the standard factorization theorem for birational maps on surfaces as a composition of blow-ups and blow-downs, there exist two sets of (possibly 
infinitely near) points~$\widehat{P}_1$ and $\widehat{P}_2$ in $\mathbb{P}^2(\mathbb{C})$ such that $f$ can be lifted to an automorphism between $\mathrm{Bl}_{\widehat{P}_1}\mathbb{P}^2$ 
and~$\mathrm{Bl}_{\widehat{P}_2}\mathbb{P}^2.$ The data of $\widehat{P}_1$ and $\widehat{P}_2$ allows to get automorphisms of rational surfaces in the 
left~$\mathrm{PGL}_3(\mathbb{C})$-orbit of~$f:$ assume that $k\in\mathbb{N}$ is fixed and let $\varphi$ be an element of~$\mathrm{PGL}_3(\mathbb{C})$ such that $\widehat{P}_1,$ 
$\varphi\widehat{P}_2,$ $(\varphi f)\varphi\widehat{P}_2,$ $\ldots,$ $(\varphi f)^{k-1}\varphi\widehat{P}_2$ have all distinct supports in~$\mathbb{P}^2(\mathbb{C})$ and $(\varphi f)^k
\varphi\widehat{P}_2=~\widehat{P}_1.$ Then~$\varphi f$ can be lifted to an automorphism of $\mathbb{P}^2(\mathbb{C})$ blown up at $\widehat{P}_1,$ $\varphi\widehat{P}_2,$ $(\varphi f)
\varphi\widehat{P}_2,$ $\ldots,$ $(\varphi f)^{k-1}\varphi\widehat{P}_2.$ Furthermore, if the conditions above are satisfied for a holomorphic family of~$\varphi,$ we get a holomorphic 
family of rational surfaces (whose dimension is at most eight).
Therefore, we see that the problem of lifting an element in the $\mathrm{PGL}_3(\mathbb{C})$-orbit of $f$ to an automorphism is strongly related to the equation $u(\widehat{P}_2)=
\widehat{P}_1,$ where~$u$ is a germ of biholomorphism of $\mathbb{P}^2(\mathbb{C})$ mapping the support of $\widehat{P}_2$ to the support of $\widehat{P}_1.$ In concrete examples, 
when $\widehat{P}_1$ and $\widehat{P}_2$ are known, this equation can actually be solved and involves polynomial equations in the Taylor expansions of $u$ at the various points of 
the support of $\widehat{P}_2.$ It is worth pointing out that in the generic case, $\widehat{P}_1$ and $\widehat{P}_2$ consist of the same number $d$ of distinct points in the projective 
plane, and the equa\-tion~$u(\widehat{P}_2)=\widehat{P}_1$ gives $2d$ independent conditions on $u$ (which is the maximum possible number if $\widehat{P}_1$ and $\widehat{P}_2$ have 
length $d$). Conversely, infinitely near points can considerably decrease the number of conditions on $u$ as shown in our examples. This explains why holomorphic families of 
automorphisms of rational surfaces occur when blow-ups on infinitely near point are made. We illustrate the method on two examples.

We end the chapter with a summary about the current knowledge on automorphisms of
rational surfaces with positive entropy.

\medskip

 \section{Birational maps whose exceptional locus is a line}

  Let us consider the birational map defined by
\begin{align*}
&\Phi_n=\big(xz^{n-1}+y^n:yz^{n-1}:z^n\big), &&n\geq 3.
\end{align*}
The sequence $(\deg \Phi_n^k)_{k\in \mathbb{N}}$ is bounded (it's easy to see in 
the affine chart $z=~1$), so $\Phi_n$ is conjugate to an automorphism on some rational surface~$\mathrm{S}$ and an iterate of $\Phi_n$ is conjugate to an automorphism isotopic to the 
identity (\cite{DiFa}). The map $\Phi_n$ blows up one point~$P=(1:0:0)$ and blows down one curve $\Delta=\{z=0\}.$ 

  Here we will assume that $n=3$ but the construction is similar for $n\geq 4$ (\emph{see} \cite{DeGr}). We first construct two infinitely near points $\widehat{P}_1$ and $\widehat{P}_2$ 
such  that $\Phi_3$ induces an isomorphism between~$\mathrm{Bl}_{\widehat{P}_1}\mathbb{P}^2$ and $\mathrm{Bl}_{\widehat{P}_2}\mathbb{P}^2.$ Then we give ``theoretical'' conditions to produce 
automorphisms $\varphi$ of~$\mathbb{P}^2(\mathbb{C})$ such that~$\varphi\Phi_3$ is conjugate to an automorphism on a surface obtained from~$\mathbb{P}^2(\mathbb{C})$ by successive 
blow-ups.

\medskip

 \subsection{First step: description of the sequence of blow-ups}\label{cons}

\subsubsection{} First blow up the point $P$ in the domain and in the range. Set $y=u_1$ and $z=u_1v_1;$ remark that $(u_1,v_1)$ are coordinates near~$P_1=(0,0)_{(u_1,v_1)},$ coordinates 
in which the exceptional divisor is given by $\mathrm{E}=\{u_1=0\}$ and the strict transform of $\Delta$ is given by $\Delta_1=\{v_1=0\}.$ Set~$y=r_1s_1$ and~$z=~s_1;$ note that $(r_1,s_1)$
 are coordinates near $Q=(0,0)_{(r_1,s_1)},$ coordinates in which $\mathrm{E}=\{s_1=0\}.$ We have
\begin{align*}
&(u_1,v_1)\to(u_1,u_1v_1)_{(y,z)}\to\big(v_1^2+u_1:v_1^2u_1:v_1^3u_1\big)\\
&\hspace{1cm}=\left(\frac{v_1^2u_1}{v_1^2+u_1},\frac{v_1^3u_1}{v_1^2+u_1}\right)_{(y,z)}
\to\left(\frac{v_1^2u_1}{v_1^2+u_1},v_1\right)_{(u_1,v_1)}
\end{align*}
 and
\begin{align*} &(r_1,s_1)\to(r_1s_1,s_1)_{(y,z)}\to\big(1+r_1^3s_1:r_1s_1:s_1\big)\\
&\hspace{1cm}=\left(\frac{r_1s_1}{1+r_1^3s_1},\frac{s_1}{1+r_1^3s_1}\right)_{(y,z)}\to\left(r_1,\frac{s_1}{1+r_1^3s_1}\right)_{(r_1,s_1)};
\end{align*}
therefore $P_1$ is a point of indeterminacy, $\Delta_1$ is blown down to $P_1$ and $\mathrm{E}$ is fixed.

\medskip

\subsubsection{} Let us blow up $P_1$ in the domain and in the range. Set $u_1=u_2$ and $v_1=u_2v_2.$ Note that
$(u_2,v_2)$ are coordinates around $P_2=(0,0)_{(u_2,v_2)}$ in which $\Delta_2=\{v_2=0\}$ and $\mathrm{F}=\{u_2=0\}.$ If we set $u_1=r_2s_2$ and $v_1=s_2$ then $(r_2,s_2)$ are coordinates 
near  $A=(0,0)_{(r_2,s_2)};$ in these coordinates $\mathrm{F}=\{s_2=0\}.$ Moreover
$$(u_2,v_2)\to(u_2,u_2v_2)_{(u_1,v_1)}\to\big(1+u_2v_2^2:u_2^2v_2^2:u_2^3v_2^3\big)$$ and
$$(r_2,s_2)\to(r_2s_2,s_2)_{(r_1,s_1)}\to\big(r_2+s_2:r_2s_2^2:r_2s_2^3\big).$$

  Remark that $A$ is a point of indeterminacy. We also have
\begin{align*}
&(u_2,v_2)\to(u_2,u_2v_2)_{(u_1,v_1)}\to\big(1+u_2v_2^2:u_2^2v_2^2:u_2^3v_2^3\big)
\to\left(\frac{u_2^2v_2^2}{1+u_2v_2^2},\frac{u_2^3v_2^3}{1+u_2v_2^2}\right)_{(y,z)}\\
&\hspace{1cm}\to\left(\frac{u_2^2v_2^2}{1+u_2v_2^2},u_2v_2\right)_{(u_1,v_1)}
\to\left(\frac{u_2v_2}{1+u_2v_2^2},u_2v_2\right)_{(r_2,s_2)}
\end{align*}
so $\mathrm{F}$ and $\Delta_2$ are blown down to $A.$

\medskip

 \subsubsection{} Now let us blow up $A$ in the domain and in the range. Set $r_2=u_3$ and $s_2=u_3v_3;$  
$(u_3,v_3)$ are coordinates near $A_1=(0,0)_{(u_3,v_3)},$ coordinates in which 
$\mathrm{F}_1=\{v_3=0\}$ and $\mathrm{G}~=~\{u_3=~0\}.$ If $r_2=
r_3s_3$ and $s_2=s_3,$ then~$(r_3,s_3)$ is a system of coordinates in which $\mathrm{E}_2=\{r_3=0\}$ and~$\mathrm{G}=~\{s_3=0\}.$ We have
$$(u_3,v_3)\to(u_3,u_3v_3)_{(r_2,s_2)}\to\big(1+v_3:u_3^2v_3^2:
u_3^3v_3^3\big),$$ $$(r_3,s_3)\to(r_3s_3,s_3)_{(r_2,s_2)}\to\big(1+r_3:
r_3s_3^2:r_3s_3^3\big).$$
The point $T=(-1,0)_{(r_3,s_3)}$ is a point of indeterminacy. Moreover
\begin{align*}
&(u_3,v_3)\to\left(\frac{u_3^2v_3^2}{1+v_3},\frac{u_3^3v_3^3}{1+v_3}\right)_{(y,z)}
\to\left(\frac{u_3^2v_3^2}{1+v_3},u_3v_3\right)_{(u_1,v_1)}\\
&\hspace{1cm}\to\left(\frac{u_3v_3}{1+v_3},u_3v_3\right)_{(r_2,s_2)}\to\left(\frac{1}{1+v_3},u_3v_3\right)_{(r_3,s_3)};
\end{align*}
so $\mathrm{G}$ is fixed and $\mathrm{F}_1$ is blown down to $S=(1,0)_{(r_3,
s_3)}.$

\medskip

 \subsubsection{} Let us blow up $T$ in the domain and $S$ in the range. Set $r_3=u_4-1$ and $s_3=u_4v_4;$ in the system of coordinates $(u_4,v_4)$ we have
$\mathrm{G}_1=\{v_4=0\}$ and $\mathrm{H}=\{u_4=0\}.$ Note that $(r_4,
s_4),$ where $r_3=r_4s_4-1$ and $s_3=s_4,$ is a system of coordinates in which $\mathrm{H}=\{s_4=0\}.$ On the one hand
\begin{align*}
&(u_4,v_4)\to(u_4-1,u_4v_4)_{(r_3,s_3)}\to\big((u_4-1)u_4v_4^2,(u_4-1)u_4^2
v_4^3\big)_{(y,z)}\\
&\hspace{1cm}\to\big((u_4-1)u_4v_4^2,u_4v_4\big)_{(u_1,v_1)}
\to\big((u_4-1)v_4,u_4v_4\big)_{(r_2,s_2)}\\
&\hspace{1cm}
\to\left((u_4-1)v_4,\frac{u_4}{u_4-1}\right)_{(u_3,v_3)}
\end{align*}
  so $\mathrm{H}$ is sent on $\mathrm{F}_2$. On the other hand
$$(r_4,s_4)\to(r_4s_4-1,s_4)_{(r_3,s_3)}\to\big(r_4:(r_4s_4-1)s_4:(r_4s_4-1)s_4^2\big);$$
hence $B=(0,0)_{(r_4,s_4)}$ is a point of indeterminacy.

\medskip

  Set $r_3=a_4+1,$ $s_3=a_4b_4;$ $(a_4,b_4)$
are coordinates in which $\mathrm{G}_1=\{b_4=0\}$ and 
$\mathrm{K}=~\{a_4=~0\}.$ We can also set $r_3=c_4d_4+1$ and $s_3=d_4;$ in the system of coordinates 
$(c_4,d_4)$ the exceptional divisor $\mathrm{K}$ is given by~$d_4=0.$

  Note that
$$(u_3,v_3)\to\left(\frac{1}{1+v_3},u_3v_3\right)_{(r_3,s_3)}\to\left(-\frac{v_3}{1+v_3},
-u_3(1+v_3)\right)_{(a_4,b_4)};$$
thus $\mathrm{F}_2$ is sent on $\mathrm{K}.$

  We remark that
\begin{align*}
&(u_1,v_1)\to\big(v_1^2+u_1:u_1v_1^2:u_1v_1^3\big)=\left(\frac{u_1v_1^2}{u_1+v_1^2},\frac{u_1v_1^3}{u_1+v_1^2}\right)_{(y,z)}\\
&\hspace{1cm}\to\left(\frac{u_1v_1^2}{u_1+v_1^2},v_1\right)_{(u_1,v_1)}\to\left(\frac{u_1v_1}{u_1+v_1^2},v_1\right)_{(r_2,s_2)}\\
&\hspace{1cm} 
\to\left(\frac{u_1}{u_1+v_1^2},v_1\right)_{(r_3,s_3)}\to\left(-\frac{v_1}{u_1+v_1^2},v_1\right)_{(c_4,d_4)};
\end{align*}
so $\Delta_4$ is blown down to $C=(0,0)_{(c_4,d_4)}.$

\medskip

 \subsubsection{} Now let us blown up $B$ in the domain and $C$ in the range. Set $r_4=u_5,$ $s_4=u_5v_5$ and $r_4=r_5s_5,$ $s_4=s_5.$
Then $(u_5,v_5)$ (resp. $(r_5,s_5)$) is a system of coordinates in which $\mathrm{L}=\{u_5=0\}$ (resp. $\mathrm{H}_1=\{v_5=0\}$ and $\mathrm{L}=\{s_5=0\}$). We note that
$$(u_5,v_5)\to(u_5,u_5v_5)_{(r_4,s_4)}\to\big(1:v_5(u_5^2v_5-1):u_5v_5^2(u_5^2v_5-1)\big)$$ and 
$$(r_5,s_5)\to(r_5s_5,s_5)_{(r_4,s_4)}\to\big(r_5:r_5s_5^2-1:s_5(r_5s_5^2-1)\big).$$

  Therefore $\mathrm{L}$ is sent on $\Delta_5$ and there is no point of indeterminacy.

  Set $c_4=a_5,$ $d_4=a_5b_5$ and $c_4=c_5d_5,$ $d_4=d_5.$ In the first (resp. second) system of coordinates the exceptional divisor $\mathrm{M}$ is given by $\{a_5=0\}$ (resp. 
$\{d_5=0\}$). We have
$$(u_1,v_1)\to\left(-\frac{v_1}{u_1+v_1^2},v_1\right)_{(c_4,d_4)}\to\left(-\frac{1}{u_1+v_1^2},v_1\right)_{(c_5,d_5)};$$
in particular $\Delta_5$ is sent on $\mathrm{M}.$

\begin{pro}[\cite{DeGr}]\label{isomorphism}
Let $\widehat{P}_1$ $($resp. $\widehat{P}_2)$ be the point infinitely near $P$ 
obtained by blo\-wing up $\mathbb{P}^2(\mathbb{C})$ at $P,$ $P_1,$ $A,$ $T$ and $U$ $($resp.
$P,$ $P_1,$ $A,$ $S$ and~$U')$.

  The map $\Phi_3$ induces an isomorphism between $\mathrm{Bl}_{\widehat{P}_1}\mathbb{P}^2$ and $\mathrm{Bl}_{\widehat{P}_2}\mathbb{P}^2.$
\end{pro}

  The different components are swapped as follows 
\begin{align*}
& \Delta\to \mathrm{M}, && \mathrm{E}\to \mathrm{E}, && \mathrm{F}\to \mathrm{K}, && \mathrm{G}\to \mathrm{G},
&& \mathrm{H}\to \mathrm{F}, && \mathrm{L}\to\Delta.
\end{align*}

\medskip

 \subsection{Second step: gluing conditions}

  The gluing conditions reduce to the following problem: if $u$ is a germ of biholomorphism in a neighborhood of $P,$ find the conditions on $u$ in order that $u(\widehat{P}_2)=
\widehat{P}_1.$

\begin{pro}[\cite{DeGr}]\label{gluglu}
Let $u(y,z)=\left(\displaystyle\sum_{(i,j)\in\mathbb{N}^2} m_{i,j}y^iz^j, \displaystyle\sum_{ (i,j)\in\mathbb{N}^2} n_{i,j}y^iz^j\right)$ be a germ of biholomorphism at $P.$

  Then $u$ can be lifted to a germ of biholomorphism between $\mathrm{Bl}_{\widehat{P}_2}\mathbb{P}^2$ and~$\mathrm{Bl}_{\widehat{P}_1}\mathbb{P}^2$ if and only if
\begin{align*}
&m_{0,0}=n_{0,0}=n_{1,0}=m_{1,0}^3+n_{0,1}^2=0, && n_{2,0}=\frac{3m_{0,1}n_{0,1}}{2m_{1,0}}.
\end{align*}
\end{pro}

\medskip

 \subsection{Examples}

  In this section, we will use the two above steps to produce explicit examples of automorphisms of rational surfaces obtained from birational maps in the 
$\mathrm{PGL}_3(\mathbb{C})$-orbit of $\Phi_3.$ As we have to blow up $\mathbb{P}^2(\mathbb{C})$ at least ten times to have non zero-entropy, we want to find an automorphism 
$\varphi$ of~$\mathbb{P}^2(\mathbb{C})$ such that 

\begin{equation}\label{conditions}
(\varphi\Phi_3)^k\varphi(\widehat{P}_2)=\widehat{P}_1\text{ with }\, (k+1)(2n-1)\geq 10 \\
(\varphi\Phi_3)^i\varphi( P)\not=P \,\text{ for } \, 0\leq i\leq k-1
\end{equation}

  First of all let us introduce the following definition.

\begin{defi}
Let $U$ be an open subset of $\mathbb{C}^n$ and let $\varphi\colon U\to\mathrm{PGL}_3(\mathbb{C})$ be a holomorphic map. If $f$ is a birational map of the projective plane, we say that 
the family of birational maps $(\varphi_{\alpha_1,\,\ldots,\,\alpha_n}f)_{(\alpha_1,\,\ldots,\,\alpha_n)\in U}$ is \textbf{\textit{holomorphically trivial}}\label{Chap8:ind33} if 
for every $\alpha^0=(\alpha_1^0,\,\ldots,\,\alpha_n^0)$ in $U$ there exists a holomorphic map from a neighborhood $U_{\alpha^0}$ of $\alpha^0$ to~$\mathrm{PGL}_3(\mathbb{C})$ such that 
\begin{itemize}
\item[$\bullet$] $M_{\alpha_1^0,\,\ldots,\,\alpha_n^0}=\mathrm{Id},$

\item[$\bullet$] $\forall\,(\alpha_1,\,\ldots,\,\alpha_n)\in U_{\alpha^0},$ $\varphi_{\alpha_1,\,\ldots,\,\alpha_n}f=M_{\alpha_1,\,\ldots,\,\alpha_n}(\varphi_{\alpha_1^0,\,\ldots,\,
\alpha_n^0}f)M_{\alpha_1,\,\ldots,\,\alpha_n}^{-1}.$ 
\end{itemize}
\end{defi}

\begin{thm}\label{nqcq3}
Let $\varphi_\alpha$ be the automorphism of the complex projective plane given by
\begin{align*}
&\varphi_\alpha=\left[\begin{array}{ccc}
\alpha &2(1-\alpha)&(2+\alpha-\alpha^2)\\
-1&0&(\alpha+1)\\
1&-2&(1-\alpha)\end{array}\right], &&\alpha\in\mathbb{C}\setminus\{0,\,1\}.
\end{align*}

  The map $\varphi_\alpha\Phi_3$ is conjugate to an automorphism of $\mathbb{P}^2(\mathbb{C})$ blown up in~$15$ points.

  The first dynamical degree of $\varphi_\alpha\Phi_3$ is $\frac{3+\sqrt{5}}{2}>1.$

  The family $\varphi_\alpha\Phi_3$ is holomorphically trivial.
\end{thm}

\begin{proof}[Proof] 
The first assertion is given by Proposition \ref{gluglu}.

  The different components are swapped as follows (\S\ref{cons})
\begin{align*}
& \Delta\to \varphi_\alpha \mathrm{M},&& \mathrm{E}\to \varphi_\alpha \mathrm{E}, && \mathrm{F}\to \varphi_\alpha \mathrm{K}, \\
& \mathrm{G}\to \varphi_\alpha \mathrm{G}, 
&& \mathrm{H}\to \varphi_\alpha \mathrm{F}, 
&& \mathrm{L}\to \varphi_\alpha \Delta,\\
& \varphi_\alpha \mathrm{E}\to\varphi_\alpha\Phi_3 \varphi_\alpha \mathrm{E},&&
\varphi_\alpha \mathrm{F} \to\varphi_\alpha \Phi_3\varphi_\alpha \mathrm{F}, 
&&\varphi_\alpha \mathrm{G} \to\varphi_\alpha \Phi_3\varphi_\alpha \mathrm{G},\\
 &\varphi_\alpha \mathrm{K} \to\varphi_\alpha \Phi_3\varphi_\alpha \mathrm{K}, 
&&\varphi_\alpha \mathrm{M} \to\varphi_\alpha \Phi_3\varphi_\alpha \mathrm{M}, &&
\varphi_\alpha \Phi_3 \varphi_\alpha \mathrm{E}\to \mathrm{E},\\
&\varphi_\alpha  \Phi_3\varphi_\alpha \mathrm{F}\to \mathrm{F},&&
\varphi_\alpha \Phi_3\varphi_\alpha \mathrm{G}\to \mathrm{G},&&
\varphi_\alpha \Phi_3\varphi_\alpha \mathrm{K}\to \mathrm{H},\\
&\varphi_\alpha \Phi_3\varphi_\alpha \mathrm{M}\to \mathrm{L}. && &&
\end{align*}

  So, in the basis 
\begin{align*}
&\big\{\Delta,\,\mathrm{E},\,
\mathrm{F},\,\mathrm{G},\,\mathrm{H},\,\mathrm{L},\,\varphi_\alpha\mathrm{E},\,\varphi_\alpha\mathrm{F},
\,\varphi_\alpha\mathrm{G},\,\varphi_\alpha\mathrm{K},\,\varphi_\alpha\mathrm{M}\,\varphi_\alpha\Phi_3 \varphi_\alpha\mathrm{E},\\
&\hspace{0.8cm}\,\varphi_\alpha\Phi_3 \varphi_\alpha\mathrm{F},
\,\varphi_\alpha\Phi_3 \varphi_\alpha\mathrm{G},\,\varphi_\alpha\Phi_3 \varphi_\alpha\mathrm{K},\,\varphi_\alpha\Phi_3 \varphi_\alpha\mathrm{M}\big\},
\end{align*} 
the matrix of $(\varphi_\alpha\Phi_3)_*$ is
\begin{small}
$$\left[\begin{array}{cccccccccccccccc}
0 & 0 & 0 & 0 & 0 & 1  & 0 & 0 & 0 & 0 & 0 & 0 & 0 & 0 & 0 & 0\\
0 & 0 & 0 & 0 & 0 & 1  & 0 & 0 & 0 & 0 & 0 & 1 & 0 & 0 & 0 & 0\\
0 & 0 & 0 & 0 & 0 & 2  & 0 & 0 & 0 & 0 & 0 & 0 & 1 & 0 & 0 & 0\\
0 & 0 & 0 & 0 & 0 & 3  & 0 & 0 & 0 & 0 & 0 & 0 & 0 & 1 & 0 & 0\\
0 & 0 & 0 & 0 & 0 & 3  & 0 & 0 & 0 & 0 & 0 & 0 & 0 & 0 & 1 & 0\\
0 & 0 & 0 & 0 & 0 & 3  & 0 & 0 & 0 & 0 & 0 & 0 & 0 & 0 & 0 & 1\\
0 & 1 & 0 & 0 & 0 & -1 & 0 & 0 & 0 & 0 & 0 & 0 & 0 & 0 & 0 & 0\\
0 & 0 & 0 & 0 & 1 & -2 & 0 & 0 & 0 & 0 & 0 & 0 & 0 & 0 & 0 & 0\\
0 & 0 & 0 & 1 & 0 & -3 & 0 & 0 & 0 & 0 & 0 & 0 & 0 & 0 & 0 & 0\\
0 & 0 & 1 & 0 & 0 & -3 & 0 & 0 & 0 & 0 & 0 & 0 & 0 & 0 & 0 & 0\\
1 & 0 & 0 & 0 & 0 & -3 & 0 & 0 & 0 & 0 & 0 & 0 & 0 & 0 & 0 & 0\\
0 & 0 & 0 & 0 & 0 & 0  & 1 & 0 & 0 & 0 & 0 & 0 & 0 & 0 & 0 & 0\\
0 & 0 & 0 & 0 & 0 & 0  & 0 & 1 & 0 & 0 & 0 & 0 & 0 & 0 & 0 & 0\\
0 & 0 & 0 & 0 & 0 & 0  & 0 & 0 & 1 & 0 & 0 & 0 & 0 & 0 & 0 & 0\\
0 & 0 & 0 & 0 & 0 & 0  & 0 & 0 & 0 & 1 & 0 & 0 & 0 & 0 & 0 & 0\\
0 & 0 & 0 & 0 & 0 & 0  & 0 & 0 & 0 & 0 & 1 & 0 & 0 & 0 & 0 & 0\\
\end{array}\right]$$
\end{small}
and its characteristic polynomial is $$(X^2-3X+1)(X^2-X+1)(X+1)^2(X^2+X+1)^3(X-1)^4.$$ Thus
$$\lambda(\varphi_\alpha\Phi_3)=\frac{3+\sqrt{5}}{2}>1.$$

  Fix a point $\alpha_0$ in $\mathbb{C}\setminus\{0,\,1\}.$ We can find locally around $\alpha_0$ a matrix~$M_\alpha$ depending holomorphically on $\alpha$ such that for all $\alpha$ 
near $\alpha_0$ we have $$\varphi_\alpha\Phi_3=M_\alpha^{-1}\varphi_{\alpha_0}\Phi_3 M_\alpha:$$
if $\mu$ is a local holomorphic solution of the equation  $\alpha=\mu^n\alpha_0$ such that $\mu_0=1$ we can take $$M_\alpha=\left[\begin{array}{ccc} 1 & 0 & \alpha_0-\alpha\\ 0 &1& 0 \\
 0 & 0 & 1\end{array}\right].$$
\end{proof}

 \section{A birational cubic map blowing down one conic and one line}

  Let $\psi$ denote the following birational map $$\psi=\big(y^2z:x(xz+y^2):y(xz+y^2)\big);$$ it blows up two points and blows down two curves, more precisely
\begin{align*}
& \mathrm{Ind}\, \psi=\big\{R=(1:0:0),\,P=(0:0:1)\big\},\\
& \mathrm{Exc}\, \psi=\big(\mathcal{C}=\big\{xz+y^2=0\big\}\big)\cup\big(\Delta'=\big\{y=0\big\}\big).
\end{align*}
We can verify that $\psi^{-1}=(y(z^2-xy):z(z^2-xy):xz^2)$ and
\begin{align*}
& \mathrm{Ind}\, \psi^{-1}=\big\{Q=(0:1:0),\,R\big\},\\
& \mathrm{Exc}\, \psi^{-1}=\big(\mathcal{C}'=\big\{z^2-xy=0\big\}\big)\cup\big(\Delta''=\big\{z=0\big\}\big).
\end{align*}
The sequence of blow-ups is a little bit different; let us describe it. Denote by $\Delta$ the line $x=0$.

\begin{itemize}
\item[$\bullet$] First we blow up $R$ in the domain and in the range and denote by 
$\mathrm{E}$ the exceptional divisor. We can show that $\mathcal{C}_1=\{u_1+v_1=0\}$ is sent on $\mathrm{E},$ $\mathrm{E}$ is blown down to 
$Q=(0:1:0)$ and $S=\mathrm{E}\cap\Delta''_1$ is a point of indeterminacy. 

\item[$\bullet$] Next we blow up $P$ in the domain and $Q$ in the range and denote by~$\mathrm{F}$ (resp. $\mathrm{G}$) the exceptional divisor associated 
with $P$ (resp. $Q$). We can verify that $\mathrm{F}$ is sent on $\mathcal{C}'_2,$ $\mathrm{E}_1$ is blown down to $T=\mathrm{G}\cap\Delta_2$ 
and $\Delta'_2$ is blown down to $T.$

\item[$\bullet$] Then we blow up $S$ in the domain and $T$ in the range and denote by~$\mathrm{H}$ (resp. $\mathrm{K}$) the exceptional divisor obtained by 
blowing up $S$ (resp. $T$). We can show that
$\mathrm{H}$ is sent on $\mathrm{K};$ $\mathrm{E}_2,$ $\Delta'_3$ are blown down to a point~$V$ on $\mathrm{K}$ and there is a 
point of indeterminacy $U$ on $\mathrm{H}.$ 

\item[$\bullet$] We will now blow up $U$ in the domain and $V$ in the range; let $\mathrm{L}$ (resp.~$\mathrm{M}$) be the exceptional divisor obtained by 
blowing up $U$ (resp.~$V$). 
There is a point of indeterminacy $Y$ on~$\mathrm{L},$ $\mathrm{L}$ is sent on $\mathrm{G}_2,$  $\mathrm{E}_3$ on $\mathrm{M}$ and $\Delta'_4$ is
 blown down to a point $Z$ of $\mathrm{M}.$

\item[$\bullet$] Finally we blow up $Y$ in the domain and $Z$ in the range. We have:~$\Delta'_5$ is sent on $\Omega$ and $\mathrm{N}$ on $\Delta''_5,$ where
 $\Omega$ (resp. $\mathrm{N}$) is the exceptional divisor obtained by blowing up $Z$ (resp. $Y$).
\end{itemize}

\bigskip

\begin{pro}
Let $\widehat{P}_1$ $($resp. $\widehat{P}_2)$ denote the point infinitely near $R$ $($resp. $Q)$ obtained by blowing up $R,$ $S,$ $U$ and $Y$ 
$($resp. $Q,$ $T,$ $V$ and $Z)$. The map $\psi$ induces an isomorphism between $\mathrm{Bl}_{\widehat{P}_1,P}\,\mathbb{P}^2$ and $\mathrm{Bl}_{\widehat{P}_2,R}\,
\mathbb{P}^2.$ The different components are swapped as follows:
\begin{align*}
&&\mathcal{C}\to\mathrm{E}, &&\mathrm{F}\to\mathcal{C}', && \mathrm{H}\to\mathrm{K}, &&\mathrm{L}\to\mathrm{G}, && \mathrm{E}\to\mathrm{M}, &&\Delta'\to\Omega, 
  && \mathrm{N}\to \Delta''.
\end{align*}
\end{pro}

\bigskip

The following statement gives the gluing conditions.

\begin{pro}\label{recolconfig9}
Let $u(x,z)=\left(\displaystyle\sum_{(i,j)\in\mathbb{N}^2} m_{i,j}x^iz^j, \displaystyle\sum_{ (i,j)\in\mathbb{N}^2} n_{i,j}x^iz^j\right)$ be a germ of biholomorphism at $Q.$

\noindent Then $u$ can be lifted to a germ of biholomorphism between $\mathrm{Bl}_{\widehat{P}_2}\mathbb{P}^2$ and~$\mathrm{Bl}_{\widehat{P}_1}\mathbb{P}^2$  if and only if
\begin{itemize}
\item[$\bullet$] $m_{0,0}=n_{0,0}=0;$

\item[$\bullet$] $n_{0,1}=0;$

\item[$\bullet$] $n_{0,2}+n_{1,0}+m_{0,1}^2=0;$

\item[$\bullet$] $n_{0,3}+n_{1,1}+2m_{0,1}(m_{0,2}+m_{1,0})=0.$
\end{itemize}
\end{pro}

Let $\varphi$ be an automorphism of $\mathbb{P}^2.$ We will adjust $\varphi$ such that $(\varphi \psi)^k\varphi$ sends~$\widehat{P}_2$ onto $\widehat{P}_1$ and $R$ 
onto $P.$ As we have to blow up~$\mathbb{P}^2$ at least ten times to have nonzero entropy, $k$ must be larger than two, $$\widehat{P}_1,\,\varphi\widehat{P}_2,\, 
\varphi \psi\varphi\widehat{P}_2,\,(\varphi \psi)^2\varphi\widehat{P}_2,\,\ldots,\,(\varphi \psi)^{k-1}\varphi\widehat{P}_2$$ must all have distinct supports and 
$(\varphi \psi)^k\varphi\widehat{P}_2=\widehat{P}_1.$ We provide such matrices for $k=3;$ then by Proposition~\ref{recolconfig9} we have the following statement.

\begin{thm}\label{conique}
Assume that $\psi=\Big(y^2z:x(xz+y^2):y(xz+y^2)\Big)$ and that
\begin{align*}
\varphi_\alpha=\left[\begin{array}{ccc}\frac{2\alpha^3}{343}(37\mathrm{i}\sqrt{3}+3)& \alpha& -\frac{2\alpha^{2}}{49}(5\mathrm{i}\sqrt{3}+11)\\[2ex]
\frac{\alpha^2}{49}(-15+11\mathrm{i}\sqrt{3}) & 1 & -\frac{\alpha}{14}(5\mathrm{i}\sqrt{3}+11)\\[2ex]
 -\frac{\alpha}{7}(2\mathrm{i}\sqrt{3}+3)& 0 & 0\end{array} \right], && \alpha \in \mathbb{C}^*.
\end{align*}

\noindent The map $\varphi_\alpha \psi$ is conjugate to an automorphism of $\mathbb{P}^2$ blown up in $15$ points.

The first dynamical degree of~$\varphi_\alpha \psi$ is $\lambda(\varphi_\alpha \psi)=\frac{3+\sqrt{5}}{2}.$

The family $\varphi_\alpha \psi$ is locally holomorphically trivial.
\end{thm}

\begin{proof}
In the basis
\begin{align*}
&\big\{\Delta',\,\mathrm{E},\,
\mathrm{F},\,\mathrm{H},\,\mathrm{L},\,\mathrm{N},\,\varphi_\alpha \mathrm{E},\,\varphi_\alpha \mathrm{G},
\,\varphi_\alpha \mathrm{K},\,\varphi_\alpha \mathrm{M},\,\varphi_\alpha \Omega,\\
&\hspace{0.8cm}\varphi_\alpha \psi\varphi_\alpha\mathrm{E},\,\varphi_\alpha \psi\varphi_\alpha\mathrm{G},
\,\varphi_\alpha \psi\varphi_\alpha\mathrm{K},\,\varphi_\alpha \psi\varphi_\alpha\mathrm{M},\,\varphi_\alpha \psi\varphi_\alpha\Omega\big\}\end{align*}
the matrix $M$ of $(\varphi_\alpha \psi)_*$ is
$$\left[\begin{array}{cccccccccccccccc}
0 & 0 & 2 & 0 & 0 & 1 & 0 & 0 & 0 & 0 & 0 & 0 & 0 & 0 & 0 & 0\\
0 & 0 & 2 & 0 & 0 & 1 & 0 & 0 & 0 & 0 & 0 & 0 & 1 & 0 & 0 & 0\\
0 & 0 & 2 & 0 & 0 & 1 & 0 & 0 & 0 & 0 & 0 & 1 & 0 & 0 & 0 & 0\\
0 & 0 & 2 & 0 & 0 & 1 & 0 & 0 & 0 & 0 & 0 & 0 & 0 & 1 & 0 & 0\\
0 & 0 & 2 & 0 & 0 & 1 & 0 & 0 & 0 & 0 & 0 & 0 & 0 & 0 & 1 & 0\\
0 & 0 & 2 & 0 & 0 & 1 & 0 & 0 & 0 & 0 & 0 & 0 & 0 & 0 & 0 & 1\\
0 & 0 & -1 & 0 & 0 & -1 & 0 & 0 & 0 & 0 & 0 & 0 & 0 & 0 & 0 & 0\\
0 & 0 & -1 & 0 & 1 & -1 & 0 & 0 & 0 & 0 & 0 & 0 & 0 & 0 & 0 & 0\\
0 & 0 & -2 & 1 & 0 & -1 & 0 & 0 & 0 & 0 & 0 & 0 & 0 & 0 & 0 & 0\\
0 & 1 & -3 & 0 & 0 & -1 & 0 & 0 & 0 & 0 & 0 & 0 & 0 & 0 & 0 & 0\\
1 & 0 & -4 & 0 & 0 & -1 & 0 & 0 & 0 & 0 & 0 & 0 & 0 & 0 & 0 & 0\\
0 & 0 & 0 & 0 & 0 & 0 & 1 & 0 & 0 & 0 & 0 & 0 & 0 & 0 & 0 & 0\\
0 & 0 & 0 & 0 & 0 & 0 & 0 & 1 & 0 & 0 & 0 & 0 & 0 & 0 & 0 & 0\\
0 & 0 & 0 & 0 & 0 & 0 & 0 & 0 & 1 & 0 & 0 & 0 & 0 & 0 & 0 & 0\\
0 & 0 & 0 & 0 & 0 & 0 & 0 & 0 & 0 & 1 & 0 & 0 & 0 & 0 & 0 & 0\\
0 & 0 & 0 & 0 & 0 & 0 & 0 & 0 & 0 & 0 & 1 & 0 & 0 & 0 & 0 & 0
\end{array}\right].$$

Its characteristic polynomial is $$(X-1)^4 (X+1)^2 (X^2-X+1) (X^2+X+1)^3 (X^2-3X+1).$$ Hence $\lambda(\varphi_\alpha \psi)=\frac{3+\sqrt{5}}{2}.$

\medskip

Fix a point $\alpha_0$ in $\mathbb{C}^*.$
We can find  locally around $\alpha_0$ a matrix $M_\alpha$ depending holomorphically on $\alpha$ such that for all $\alpha$ near $\alpha_0,$ 
we have $\varphi_\alpha \psi= M_{\alpha}^{-1}\varphi_{\alpha_0}\psi M_{\alpha}:$ take
$$M_\alpha=\left[\begin{array}{ccc} 1 & 0 & 0\\ 0 &\frac{\alpha}{\alpha_0}& 0 \\ 0 & 0 &\frac{\alpha^2}{\alpha_0^2}\end{array}\right].$$
\end{proof}

\section{Scholium}

There are now two different points of view to construct automorphisms with positive entropy on rational non-minimal surfaces obtained from
birational maps of the complex projective plane. 

\smallskip

The first one is to start with birational maps of $\mathbb{P}^2(\mathbb{C})$ and to adjust their coefficients such that after a 
finite number of blow-ups the maps become automorphisms on some rational surfaces $\mathrm{S}$. Then we compute the action of these 
maps on the Picard group of $\mathrm{S}$ and in particular obtain the entropy. There is 
a systematic way to do explained in \cite{DeGr} and applied to produce examples. Using examples coming from physicists Bedford and Kim 
\begin{itemize}
\item[$\bullet$] exhibit continuous families of birational maps conjugate to automorphisms with positive entropy on some 
rational surfaces;

\item[$\bullet$] show that automorphisms with positive entropy on rational non-mini\-mal 
surfaces obtained from birational maps of $\mathbb{P}^2(\mathbb{C})$ can have large rotation domains and that 
rotation domains of rank $1$ and $2$ coexist.
\end{itemize}
Let us also mention the idea of \cite{Di2}: the author begins with a quadratic
birational map that fixes some cubic curve and then use the ``group law'' on the
cubic to understand when the indeterminacy and exceptional behavior of the transformation
can be eliminated by repeated blowing up.

\smallskip

The second point of view is to construct automorphisms on some 
rational surfaces prescribing the action of the automorphisms on cohomological groups; this is exactly what does McMullen in \cite{Mc}:
for $n\geq 10,$ the standard element of the Weyl group $\mathrm{W}_n$ can be realized by an automorphism~$f_n$ with positive entropy 
$\log(\lambda_n)$ of a rational surface $\mathrm{S}_n.$ This result has been improved in \cite{Ue}: 
\begin{align*}
&\big\{\lambda(f)\,\vert\, f\text{ is an automorphism on some rational surface}\big\}\\
&=\big\{
\text{spectral radius of } w\geq 1\,\vert\, w\in\mathrm{W}_n,\,n\geq 3\big\}.
\end{align*}

In \cite{CaDo} the authors classify rational surfaces for which the image of the automorphisms
group in the group of linear transformations of the Picard group
 is the largest possible; it can be rephrased in terms of periodic orbits of birational actions of infinite
Coxeter groups.

\backmatter

\chapter{Index}

\newlength{\largeur}
\setlength{\largeur}{\textwidth} \addtolength{\largeur}{1cm}

\hspace*{-1.2cm}\begin{tabular}{p{\largeur}p{4.2cm}}

\noteB{Abelian variety}{Chap8:ind2}
\noteB{adjoint linear system}{Chap7:ind20aa}
\noteB{affine group}{Chap2:ind25}
\noteB{algebraically stable}{Chap8:ind17}
\noteB{anticanonical curve}{Chap8:ind32h}
\noteB{axiom A}{Chap12:ind5}
\noteB{base-points of a birational map}{Chap2:ind20}
\noteB{base-points of a linear system}{Chap2:ind21}
\noteB{basic surface}{Chap13:ind2}
\noteB{basin}{Chap8:ind39}
\noteB{Bedford-Diller condition}{Chap12:ind11a}
\noteB{Bertini involution}{Chap7:ind4}
\noteB{Bertini type}{Chap7:ind9}
\noteB{birational map}{Chap2:ind10}
\noteB{birational maps simultaneously elliptic}{Chap5:ind0}
\noteB{blow-up}{Chap2:ind17}
\noteB{characteristic matrix}{Chap8:ind21}
\noteB{characteristic vector}{Chap8:ind22}
\noteB{conic bundle}{Chap7:ind3a}
\noteB{Coxeter element}{Chap8:ind31}
\noteB{Cremona group}{Chap2:ind11}
\noteB{Cremona transformation}{Chap2:ind12}
\noteB{cubic curve}{Chap8:ind32a}
\noteB{degree of a birational map}{Chap2:ind13}
\noteB{degree of foliation}{Chap7:ind11}
\noteB{degree of a polynomial automorphism}{Chap2:ind23}
\noteB{del Pezzo surface}{Chap7:ind3}
\noteB{de Jonqui\`eres group}{Chap2:ind29}
\noteB{de Jonqui\`eres involution}{Chap7:ind5}
\noteB{de Jonqui\`eres map}{Chap2:ind28}
\noteB{de Jonqui\`eres type}{Chap7:ind7}
\end{tabular}

\hspace*{-1.2cm}\begin{tabular}{p{\largeur}p{4.2cm}}
\noteB{distorted}{Chap5:ind6}
\noteB{divisor}{Chap2:ind1}
\noteB{dominate}{Chap13:ind1}
\noteB{elementary group}{Chap2:ind26}
\noteB{elliptic (birational) map}{Chap2:ind29a}
\noteB{Enriques surface}{Chap8:ind25}
\noteB{exceptional configuration}{Chap8:ind19}
\noteB{exceptional divisor}{Chap2:ind18}
\noteB{exceptional locus, exceptional set}{Chap2:ind15}
\noteB{Fatou set}{Chap8:ind34}
\noteB{first dynamical degree of a polynomial automorphism}{Chap2:ind24}
\noteB{first dynamical degree of a birational map of the plane}{Chap2:ind27a}
\noteB{first dynamical degree of a birational map of a rational surface}{Chap8:ind17a}
\noteB{Geiser involution}{Chap7:ind1}
\noteB{Geiser type}{Chap7:ind8}
\noteB{geometric basis}{Chap8:ind28}
\noteB{geometric nodal root}{Chap8:ind32g}
\noteB{global stable manifold}{Chap12:ind10}
\noteB{global unstable manifold}{Chap12:ind11}
\noteB{Halphen twist}{Chap2:ind29aa}
\noteB{$k$-Heisenberg group}{Chap5:ind2}
\noteB{H\'enon automorphism}{Chap2:ind27}
\noteB{Herman ring}{Chap8:ind41}
\noteB{Hirzebruch surfaces}{ind1272}
\noteB{holomorphic foliation}{Chap7:ind10}
\noteB{holomorphically trivial}{Chap8:ind33}
\noteB{homoclinic point}{Chap8:ind8}
\noteB{hyperbolic (birational) map}{Chap2:ind29aaa}
\noteB{hyperbolic set}{Chap12:ind2}
\noteB{hyperbolicity}{Chap8:ind10}
\noteB{indeterminacy locus, indeterminacy set}{Chap2:ind14}
\noteB{inertia group}{Chap7:ind20a}
\noteB{inflection point}{Chap7:ind15}
\noteB{isomorphism between marked blow-ups}{Chap8:ind32e} 
\noteB{isomorphism between marked cubics}{Chap8:ind32c} 
\noteB{isomorphism bewteen marked pairs}{Chap8:ind32j}
\noteB{isotropy group}{Chap7:ind16}
\noteB{length of an element of a finitely generated group}{Chap5:ind4}
\noteB{linear system}{Chap2:ind19}
\noteB{linearly equivalent}{Chap2:ind5}
\noteB{local stable manifold}{Chap12:ind8}
\noteB{local unstable manifold}{Chap12:ind9}
\end{tabular}

\hspace*{-1.2cm}\begin{tabular}{p{\largeur}p{4.2cm}}
\noteB{Jonqui\`eres twist}{Chap2:ind29aa}
\noteB{Julia set}{Chap8:ind12}
\noteB{K$3$ surface}{Chap8:ind24}
\noteB{Mandelbrot set}{Chap8:ind13}
\noteB{marked blow-up}{Chap8:ind32d}
\noteB{marked cubic}{Chap8:ind32b}
\noteB{marked pair}{Chap8:ind32i}
\noteB{multiplicity}{Chap2:ind3}
\noteB{multiplicity of a curve at a point}{Chap2:ind8}
\noteB{nef cone}{Chap2:ind8c}
\noteB{nodal root}{Chap8:ind32f}
\noteB{non-wandering point}{Chap12:ind1}
\noteB{orbit}{Chap8:ind1}
\noteB{ordered resolution}{Chap8:ind20}
\noteB{persistent point}{Chap12:ind12}
\noteB{Picard group}{Chap2:ind6}
\noteB{Picard number}{Chap2:ind31}
\noteB{Picard-Manin space}{Chap2:ind33}
\noteB{Pisot number}{Chap8:ind17c}  
\noteB{point of tangency}{Chap7:ind14}
\noteB{polynomial automorphism}{Chap2:ind22}
\noteB{principal divisor}{Chap2:ind4}
\noteB{rational map of $\mathbb{P}^2(\mathbb{C})$}{Chap2:ind9}
\noteB{rational map}{Chap2:ind16}
\noteB{repelling}{Chap8:ind11}
\noteB{saddle points}{Chap8:ind5}
\noteB{shift map}{Chap8:ind9}
\noteB{singular locus}{Chap7:ind12}
\noteB{rank of the rotation domain}{Chap8:ind36}
\noteB{rational surface}{Chap8:ind23}
\noteB{realized}{Chap8:ind30}
\noteB{recurrent (Fatou component)}{Chap8:ind37}
\noteB{rotation domain}{Chap8:ind35}
\noteB{Salem number}{Chap8:ind17d}
\noteB{Salem polynomial}{Chap8:ind26}
\noteB{Siegel disk}{Chap8:ind40}
\noteB{sink}{Chap8:ind4}
\noteB{stable length of an element of a finitely generated group}{Chap5:ind5}
\noteB{stable manifold}{Chap12:ind3}
\noteB{standard element}{Chap8:ind32}
\noteB{standard generators of $\mathrm{SL}_3(\mathbb{Z})$}{Chap5:ind1} 
\end{tabular}

\hspace*{-1.2cm}\begin{tabular}{p{\largeur}p{4.2cm}}
\noteB{standard generators of a $k$-Heisenberg group}{Chap5:ind3}
\noteB{strict transform}{Chap2:ind7}
\noteB{"strong" transversality condition}{Chap12:ind6}
\noteB{tight}{Chap2:ind33z}
\noteB{topological entropy}{Chap8:ind3}
\noteB{transversal}{Chap7:ind13}
\noteB{unstable manifold}{Chap12:ind4}
\noteB{Weil divisor}{Chap2:ind2}
\noteB{Weyl group}{Chap8:ind29}
\end{tabular}

\bibliographystyle{plain}
\bibliography{biblio}
\nocite{*}

\noindent {{Julie D\'eserti}}

\noindent {{Universit\"{a}t Basel, Mathematisches Institut, Rheinsprung $21$, CH-$4051$ Basel, Switzerland}}

\noindent {{On leave from Institut de Math\'ematiques de Jussieu, Universit\'e Paris~$7$, Projet G\'eom\'etrie 
et Dynamique, Site Chevaleret, Case $7012$, $75205$ Paris Cedex~$13$, France}}

\noindent {{deserti@math.jussieu.fr}}

\noindent {{Author supported by the Swiss National Science Foundation grant no PP00P2\_128422 /1}}

\end{document}